%% file: main.tex
\documentclass[psamsfonts]{amsart}
\usepackage{a4}
\usepackage[utf8]{inputenc}
\usepackage[english]{babel}
\usepackage{graphicx}
\usepackage{amsfonts,amstext,amsmath,amsthm,amssymb}
\usepackage[labelfont=bf]{caption}
\usepackage{yfonts}
\usepackage{dsfont}
\usepackage{tikz-cd}
\usepackage{tipa}
\usepackage{endnotes}
\usepackage{xfrac}
\usepackage{tabularx}
\usepackage{abstract}
\usepackage{stmaryrd}
\usepackage{slashed}
\usepackage{mathrsfs}
\usepackage{mathtools}
\usepackage{leftidx,tensor}
\usepackage{lscape}
\usepackage{tikz}
\usetikzlibrary{calc}
\usepackage{blkarray}
\usepackage{adjustbox}
\usepackage{url}
\usepackage{multirow}
\usepackage{pgfplots}
\usepackage[hidelinks]{hyperref}
\usepackage[top=3cm, bottom=3cm, right=3cm, left=3cm]{geometry}
\usepackage{dirtytalk}
\usepackage{dynkin-diagrams}
\usepackage{empheq}
\usepackage{xcolor}

\theoremstyle{definition}{
\newtheorem{defi}{Definition}[section]

\newtheorem*{nota}{Notation}}
\theoremstyle{remark}{
\newtheorem{rem}{Remark}
\newtheorem{ex}{Example}}
\theoremstyle{plain}{
\newtheorem{thm}{Theorem}[section]
\newtheorem{lem}{Lemma}[section]
\newtheorem{cor}{Corollary}[section]
\newtheorem{prop}{Proposition}[section]
}

\AtEndEnvironment{defi}{\hfill$\lozenge$}
\AtEndEnvironment{ex}{\hfill$\clubsuit$}

\newtheorem{manualtheoreminner}{Theorem}
\newenvironment{manualtheorem}[1]{%
  \renewcommand\themanualtheoreminner{#1}%
  \manualtheoreminner
}{\endmanualtheoreminner}

\newtheorem{manualpropositioninner}{Proposition}
\newenvironment{manualproposition}[1]{%
  \renewcommand\themanualpropositioninner{#1}%
  \manualpropositioninner
}{\endmanualpropositioninner}

\newtheorem{manualassumptioninner}{Assumption}
\newenvironment{manualassumption}[1]{%
  \renewcommand\themanualassumptioninner{#1}%
  \manualassumptioninner
}{\endmanualassumptioninner}

\tikzset{
  curarrow/.style={
  rounded corners=8pt,
  execute at begin to={every node/.style={fill=red}},
    to path={-- ([xshift=50pt]\tikztostart.center)
    |- (#1) {}
    -| ([xshift=-50pt]\tikztotarget.center)
    -- (\tikztotarget)}
    }
}

\newsavebox{\pullbackright}
\sbox\pullbackright{%
\begin{tikzpicture}%
\draw (0,0) -- (1ex,0ex);%
\draw (1ex,0ex) -- (1ex,1ex);%
\end{tikzpicture}}

\newsavebox{\pullbackleft}
\sbox\pullbackleft{%
\begin{tikzpicture}%
\draw (0ex,0ex) -- (0ex,1ex);%
\draw (1ex,0ex) -- (0ex,0ex);%
\end{tikzpicture}}

\newcommand{\m}[1]{{\mathrm{#1}}}

\newcommand{\cat}[1]{\textbf{#1}}
\newcommand{\tot}[3]{\underset{#1}{\m{Tot}^{#2}}\left(#3\right)}
\newcommand{\und}[1]{\hspace{#1}\text{and}\hspace{#1}}

\makeatletter
\DeclareRobustCommand{\chemical}[1]{%
	{\(\m@th
		\edef\resetfontdimens{\noexpand\)%
			\fontdimen16\textfont2=\the\fontdimen16\textfont2
			\fontdimen17\textfont2=\the\fontdimen17\textfont2\relax}%
		\fontdimen16\textfont2=2.7pt \fontdimen17\textfont2=2.7pt
		\mathrm{#1}%
		\resetfontdimens}}
\makeatother

\title{Dirac Operators on Orbifold Resolutions:\\Uniform Elliptic Theory}

\author{V. F. Majewski}

\begin{document}

\maketitle

\begin{abstract}
We study Dirac operators on resolutions of Riemannian orbifolds by developing a uniform elliptic theory. The key idea is to view orbifolds as conically fibred singular (CFS) spaces and resolve them by gluing asymptotically conical fibrations (ACF) into the singular strata. This yields smooth Gromov–Hausdorff resolutions that preserve the large-scale structure of the orbifold while replacing its singularities with well-understood local models. \\

Dirac bundles are resolved compatibly with this construction, which allows us to analyse entire families of Dirac operators in a uniform way. Inspired by the linear gluing framework of Hutchings–Taubes, we build uniformly bounded right inverses and describe precisely how kernels and cokernels behave as the geometry degenerates. The analysis relies on weighted spaces adapted to the conically fibred, conically fibred singular and asymptotically conical fibred regions. These allow us to prove Fredholm properties, uniform bounds and to establish exactness of the linear gluing sequence. Consequently, we obtain an index formula decomposing the index into contributions from the orbifold and the ACF models. This framework provides a direct and flexible analytic approach to Dirac operators on orbifold resolutions, avoiding the full edge calculus, and setting the stage for applications to special holonomy metrics and gauge theory.
\end{abstract}

\tableofcontents

\section{Preface}
\label{Preface}

Dirac operators are central objects in differential geometry and global analysis. They arise as deformation operators in both gauge theory and the study of special holonomy metrics, governing the linearisation of moduli problems such as instantons \cite{donaldson1983application,uhlenbeck1986existence,donaldson1998gauge,tian2000gauge}, calibrated submanifolds \cite{Harvey1982,mclean1998deformations}, and special holonomy metrics \cite{joyce1996aI,Joyce1996b}. Their spectral and index-theoretic properties encode fundamental geometric and topological information.While their behaviour on smooth manifolds is well understood, Dirac operators on singular spaces pose subtle analytic challenges and have been studied extensively in index theory and spectral geometry \cite{bismut1989eta,dai1991adiabatic,mazzeo1990adiabatic,Schulze1991Pseudo,melrose1991Pseudo}. In this work we study families of Dirac operators that degenerate, with the aim of understanding the transition from the well-behaved theory on smooth manifolds to their analytic properties on singular spaces. Such families naturally appear in analytic gluing constructions \cite{walpuski2012g_2,walpuski2017spin,joyce2017new,platt,gutwein2023coassociative,majewski2025spin7orbifoldresolutions} and play a central role in the compactification and completions of moduli spaces of geometric structures.\\

Riemannian orbifolds form a natural class of mild singular spaces and frequently appear as boundary points in moduli spaces of smooth Riemannian metrics. In particular, they arise as singular limits of families of smooth manifolds with controlled geometry, and often carry residual structure compatible with the limiting geometry.\cite{cheeger1986collapsing,cheeger1990collapsing} To understand the behaviour of moduli spaces near such singular boundaries, it is essential to analyse how geometric and analytic structures behave in the degeneration to an orbifold. In particular, it is crucial to understand the following question:
\begin{center}
    \textbf{How does a family of Dirac operators behave as the underlying family of Riemannian spaces\footnote{Here we intentionally use the ambiguous class of Riemannian spaces which includes (singular) Riemannian manifolds as well as Riemannian orbifolds.} degenerate to a Riemannian orbifold?}
\end{center}
This question is of independent interest and requires developing elliptic theory on spaces with degenerating metrics and stratified singularities. To address it, we introduce a uniform Fredholm theory for Dirac operators on geometries whose singularities are modelled by conical fibrations (CF). This work develops analytic techniques that build upon and generalize existing methods in adiabatic limit theory, microlocal analysis, and the geometry of weighted function spaces.\\

Motivated by moduli problems in special holonomy, this paper develops an analytic framework to understand and uniformly control Dirac operators under degenerations in the orbifold limit. We will study solutions to
\begin{align*}
D_t \Phi = \Psi
\end{align*}
for a family of Dirac operators $D_t$ defined on the degenerating family of metrics $(X_t, g_t)$ with the goal of constructing a uniformly bounded right-inverse.\\

A key technical challenge arises from the collapsing geometry in the adiabatic limit. As the exceptional set shrinks, the underlying topology forces eigensections of $D_t$ to degenerate, obstructing uniform bounds for the classical right-inverse. To overcome this, we develop a uniform elliptic theory on carefully constructed weighted function spaces and introduce linear gluing techniques adapted to the geometry of orbifold resolutions.\\

In order to analyse the families of Dirac operators $D_t$ on the resolved orbifolds $(X_t,g_t)$ we split the latter into three topological regions: the ACF-part, the CF-part and the CFS-part. We will proceed by analysing all the different components individually in order to derive a uniform elliptic theory on the total space of the resolution.\\
\begin{figure}[!ht]
    \centering
    \scalebox{0.4}{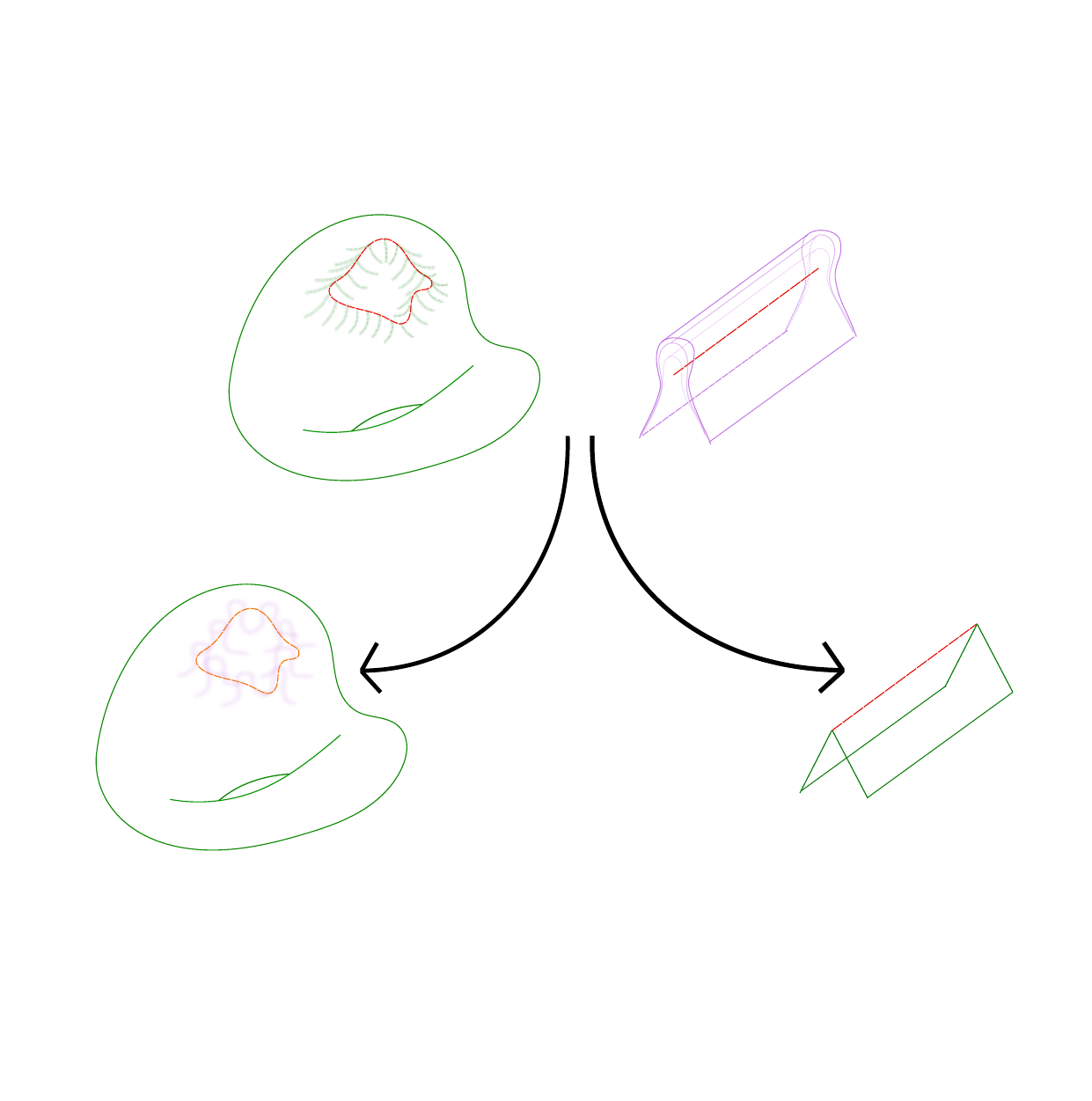}
    \caption{Gluing the CFS and ACF spaces along their common CF part produces both the compact gluing space (the orbifold resolution) and the noncompact antigluing space (the normal cone bundle of the singular stratum).}
\end{figure}
We begin by considering the CF-part of the orbifold resolution. This part, more generally referred to as the neck part of the gluing construction, can be realised as an annulus subbundle of the normal cone bundle of shrinking outer diameter. In this region, the analytic behaviour of $D_t$ is determined by the normal operator $\widehat{D}_0$ of $D$, which can be seen as the leading order operator in an asymptotic expansion of $D$ close to the singular stratum. Via a \say{separation of variables}-style argument we can explicitly expand the kernel of $\widehat{D}_0$. Using this expansion, we argue that $\widehat{D}_0$, $D$ and $\widehat{D}_{t^2\cdot\zeta}$, with the latter being a resolution of $\widehat{D}_0$, should be realised as bounded operators on weighted function spaces, subject to boundary conditions. By carefully choosing the right function spaces, the normal operator can be realised as an isomorphism of Banach spaces; $\widehat{D}_{t^2\cdot\zeta}$ and $D$ as Fredholm maps.\\

We continue by defining the so called approximate (co)kernel of $D_t$ by interpolating the (co)kernel of the operators $\widehat{D}_{t^2\cdot\zeta}$ and $D$. One of the objectives of this paper is to prove the exactness of linear gluing, by virtue of the exactness of the sequence
\begin{equation*}
    \begin{tikzcd}
        \m{ker}(D_t)\arrow[r,"i_{\beta;t}",hook]&
             \m{xker}_\beta(D_t)\arrow[r,"\m{ob}_{\beta;t}"]&\m{xcoker}_\beta(D_t)\arrow[r,"p_{\beta;t}",two heads]&\m{coker}(D_t).
    \end{tikzcd}
\end{equation*}
In order to prove the above, we will need to study the so called (anti-) gluing equations in the style of Hutchings-Taubes obstruction bundle gluing \cite[Sec. 2, Sec. 3]{hutchingstaubesI}\cite[Sec. 5, Sec. 9]{hutchingstaubesII}.\\

Further, we will prove the existence of right-inverses to the operator $D_t$ and construct adapted norms, for which these right-inverse are uniform bounded. A priori bounds on the right-inverse fail to be uniform due to the degenerating phenomena in the ACF-part of the linear gluing. We will continue by describing this phenomena in more detail. \\

The $\mathbb{R}_{>0}$-action on $N_0$ by scaling the fibres, lifts naturally to the CF-Hermitian Dirac bundle. We will further assume that this $\mathbb{R}_{> 0}$-action lifts to the ACF Hermitian Dirac bundle structures $(N_{t^2\zeta},g_{t^2\zeta})$ by means of the existence of a family of ACF structures $(N_\zeta,g^t_\zeta)$ resolving the dilated normal cone bundle $(N_0,g^t_0)$. The associated Dirac operator $\widehat{D}^t_\zeta$ has the form 
\begin{align*}
    \widehat{D}^t_\zeta=\widehat{D}^t_{\zeta;H}+t^{-1}\cdot \widehat{D}_{\zeta;V},
\end{align*}
where $\widehat{D}^t_{\zeta;H}=\mathcal{O}(1)$ as $t\to 0$. The vertical Dirac operator $\widehat{D}_{\zeta;V}$ scales by $t^{-1}$ and hence, becomes singular in the limit $t\to 0$. In particular, the norm of sections that are vertically harmonic, i.e. in the kernel of $\widehat{D}_{\zeta;V}$, collapses in the adiabatic limit and hence, the \say{classical} operator norm of $D_t$ is not $t$-independently bounded from below. This consequential unboundedness from above of a right-inverse is problematic for the application in gluing analysis and hence, we need to adapt the used norms to compensate for this collapsing phenomena. The construction of these norms is inspired by work of Walpuski \cite{walpuski2012g_2,walpuski2017spin} and Platt \cite{platt}.\\

This paper is a revised version of the original paper and draws heavily from the author's thesis \cite{majewskithesis}, with substantial overlap in many sections. The authors work \cite{majewski2025spin7orbifoldresolutions} on the construction of $\m{Spin}(7)$-orbifold resolutions, relies on the analytic results established in the present paper. We now outline the structure and main results of this paper.\\

Section \ref{Geometric Structures on Orbifolds} discusses the category of orbifolds via Lie groupoids and Morita equivalence. This categorical perspective allows  to rigorously define orbifolds as stacky quotients, encoding both the behaviour of singularities (via isotropy groups) and their global smooth structure. The notion of stratification by isotropy type is introduced to describe the singular part of an orbifold as a union of smooth open subspaces. Further, $G$-structures on orbifolds are discussed, with a focus on reductions of the frame bundle to subgroups $G \subset \mathrm{GL}(n)$ and the definition of Hermitian Dirac bundles.\\

Section \ref{Smooth Gromov-Hausdorff-Resolutions of Riemannian Orbifolds} is concerned with the resolution of singularities of Riemannian orbifolds. These are smooth families of Riemannian manifolds that degenerate (in the Gromov-Hausdorff sense) to the original orbifold. We emphasise the connection of such resolutions to paths in moduli spaces of Riemannian metric, reaching their boundary. The concept of conically fibred (CF) structures is introduced to model the local geometry near singular strata of orbifolds seen as conically fibred singular (CFS) spaces. Resolutions of such spaces are constructed via gluing asymptotically conically fibred (ACF) spaces into the singular locus along its conically fibred neighbourhood.\\

Section \ref{Hermitian Dirac Bundles on Orbifolds and their Resolutions} constructs Hermitian Dirac bundles over orbifolds and their resolutions. Since the analysis of Dirac operators is central to this paper, particular attention is given to the behaviour of Clifford modules and connections under degeneration and resolution. This includes the introduction of adiabatic families and the interpolation of analytic data between local models near the singular strata and the smooth part of the resolution.\\

In Section \ref{Analytic setup} we revisit linear gluing methods and study the model operator $\widehat{D}_0$. By separating fibre and base-wise dependences, i.e. \say{separating variables}, we obtain a polyhomogeneous expansion of the kernel and cokernel elements of the Dirac operators involved in this construction. Using the gained insights, we discuss the right definition of weighted function spaces to mimic the gluing method outlined at the beginning of the subject section.\\

In Section \ref{Uniform Elliptic Theory for Conically Fibred Dirac Operators} we discuss the analytic realisation of the operator $\widehat{D}_0$. We explicitly compute its polyhomogenous solutions and construct a $\mathbb{R}$-family of weighted function spaces realising this operator as an unbounded operator. By imposing APS-type boundary conditions and by completing the domain with respect to the graph norm, we realise this operator as an isomorphism of Banach spaces outside a discrete set $\mathcal{C}(\widehat{D}_0)$ of critical rates.
    \begin{manualproposition}{\ref{widehatD0iso}}
        Let $\beta\notin \mathcal{C}(\widehat{D}_0)$. The space $  C^{k+1,\alpha}_{\m{CF};\beta}(\widehat{D}_0;\m{APS})$ is complete and 
            \begin{align*}
                \widehat{D}_0:  C^{k+1,\alpha}_{\m{CF};\beta}(\widehat{D}_0;\m{APS})\rightarrow   C^{k,\alpha}_{\m{CF};\beta-1}(N_0,\widehat{E}_0)
            \end{align*}
            is an isomorphism of Banach spaces.
    \end{manualproposition}

In Section \ref{Elliptic Theory for Dirac Operators on Orbifolds as CFS Spaces} we use the derived results and establish an elliptic theory for CFS Dirac operators on orbifolds. We establish a uniform CFS-Schauder estimate in weighted function spaces.
    \begin{manualproposition}{\ref{SchauderD}}[CFS-Schauder Estimates]
    For all $\Phi\in  C^{k+1,\alpha}_{\m{CFS};\beta;\epsilon}(X,E;\m{APS})$
    \begin{align*}
        \left|\left|\Phi\right|\right|_{C^{k+1,\alpha}_{\m{CFS};\beta;\epsilon}} \lesssim& \left|\left|D\Phi\right|\right|_{C^{k,\alpha}_{\m{CFS};\beta-1;\epsilon}}+\left|\left|\Phi\right|\right|_{C^{0}_{\m{CFS};\beta;\epsilon}}
    \end{align*}
    holds.
    \end{manualproposition}
By improving on this Schauder estimate, we are able to prove that the chosen realisation of $D$ is Fredholm unless $\beta\in\mathcal{C}(\widehat{D}_0)$. Further, its index satisfies a wall-crossing formula.
    \begin{manualproposition}{\ref{FredholmCFS}}
        Let $\beta\notin \mathcal{C}(\widehat{D}_0)$. The map 
    \begin{align*}
        D:  C^{k+1,\alpha}_{\m{CFS};\beta;\epsilon}(X,E;\m{APS})\rightarrow  C^{k,\alpha}_{\m{CFS};\beta-1;\epsilon}(X,E)
    \end{align*}
    is Fredholm and for $\beta_2<\beta_1$
    \begin{align*}
        \m{ind}_{\m{CFS};\beta_2}(D)-\m{ind}_{\m{CFS};\beta_1}(D)=\sum_{\beta_2<\lambda<\beta_1}d_{\lambda+\frac{m-1}{2}-\delta(E)}(\widehat{D}_0).
    \end{align*}
    \end{manualproposition}
The value of $d_{\lambda + \frac{m-1}{2} - \delta(E)}(\widehat{D}_0)$ equals the dimension of the kernel of a Dirac operator on the singular stratum associated to a critical rates $\lambda\in\mathcal{C}(\widehat{D}_0)$ between $\beta_2$ and $\beta_1$.\\

In Section \ref{Uniform Elliptic Theory of Adiabatic Families of ACF Dirac Operators} we will study the family of ACF-Dirac operators $\widehat{D}^t_\zeta$. Such families of Hermitian Dirac bundles are the noncompact analogue of the adiabatic families of Dirac bundles studied by  Bismut \cite{bismut1986,bismut1989eta,bismut1991,bismut1995}, Goette \cite{goette2014adiabatic} and others. We will build upon their ideas and analyse the family of Dirac operators $\widehat{D}^t_\zeta$ close to their adiabatic limit. We identify the adiabatic kernel and cokernel of $\widehat{D}^t_\zeta$ with the kernel and cokernel of differential operators on the singular stratum. These operator are referred to as the adiabatic residues 
\begin{align*}
    \mathfrak{D}_{\mathcal{K};\beta}: \Gamma(S,\mathcal{K}_{\m{AC};\beta}(\pi_\zeta/\nu_\zeta))\rightarrow  \Gamma(S,\mathcal{K}_{\m{AC};\beta}(\pi_\zeta/\nu_\zeta))
\end{align*}
and 
\begin{align*}
    \mathfrak{D}_{\mathcal{C}o\mathcal{K};\beta-1}: \Gamma(S,\mathcal{C}o\mathcal{K}_{\m{AC};\beta-1}(\pi_\zeta/\nu_\zeta))\rightarrow  \Gamma(S,\mathcal{C}o\mathcal{K}_{\m{AC};\beta-1}(\pi_\zeta/\nu_\zeta))
\end{align*}
These operators are elliptic, $t$-independent Dirac type operators on the singular stratum acting on sections of the finite dimensional Hermitian Dirac bundle $\mathcal{K}_{\m{AC};\beta}(\pi_\zeta/\nu_\zeta)$ and $\mathcal{C}o\mathcal{K}_{\m{AC};\beta-1}(\pi_\zeta/\nu_\zeta)$ of the bundles of (co)kernels of $\widehat{D}_{\zeta;V}$.\\

We will further introduce the notion of isentropic families of Dirac operators on $N_\zeta$. An isentropic process in statistical mechanics is an adiabatic process that is reversible. In the same manner an isentropic operator can be modelled by its adiabatic limit. For these isentropic Dirac operators, the space of adiabatic harmonic spinors is isomorphic to the space of harmonic spinors close to the adiabatic limit. \\

By following ideas of Walpuski \cite{walpuski2012g_2,walpuski2017spin}, Platt \cite{platt}, Degeratu and Mazzeo \cite{mazzeo2018fredholm} we will construct adapted norms and show that uniform estimates for the operator bounds of the family of Dirac operators hold. These adapted norms are of the form 
    \begin{align*}
        \left|\left|\Phi\right|\right|_{\mathfrak{D}^{k+1,\alpha}_{\m{ACF};\beta;t}}\coloneqq &\left|\left|\pi_{\mathcal{C}o\mathcal{I};\beta}\Phi\right|\right|_{C^{k+1,\alpha}_{\m{ACF};\beta;t}}+t^{-\kappa}\left|\left|\varpi_{\mathcal{K};\beta}\Phi\right|\right|_{C^{k+1,\alpha}_t}\\
        \left|\left|\Phi\right|\right|_{\mathfrak{C}^{k,\alpha}_{\m{ACF};\beta-1;t}}\coloneqq&\left|\left|\pi_{\mathcal{I};\beta-1}\Phi\right|\right|_{C^{k,\alpha}_{\m{ACF};\beta-1;t}}+t^{-\kappa}\left|\left|\varpi_{\mathcal{C}o\mathcal{K};\beta-1}\Phi\right|\right|_{C^{k,\alpha}_t}
    \end{align*}
\noindent where $\pi_{\mathcal{C}o\mathcal{I};\beta}$, $\varpi_{\mathcal{K};\beta}$, $\pi_{\mathcal{I};\beta-1}$ and $\varpi_{\mathcal{C}o\mathcal{K};\beta-1}$ denote the projection onto the vertical coimage, kernel, image and cokernel of the vertical Dirac operator $\widehat{D}_{\zeta;V}$.

\begin{manualproposition}{\ref{FredholmACFDC}}
    The space $\mathfrak{D}^{k+1,\alpha}_{\m{ACF};\beta;t}(\widehat{D}^t_{\zeta};\m{APS})$ is a Banach space and the map 
    \begin{align*}
        \widehat{D}^t_{\zeta}:\mathfrak{D}^{k+1,\alpha}_{\m{ACF};\beta;t}(\widehat{D}^t_{\zeta};\m{APS})\rightarrow \mathfrak{C}^{k,\alpha}_{\m{ACF};\beta-1;t}(N_\zeta,\widehat{E}_\zeta)
    \end{align*}
    is Fredholm. Elements of the kernel and cokernel are smooth, i.e.  independent of $\alpha$ and $k$. Moreover, 
    \begin{align*}
        \m{ind}_{\m{ACF};\beta}(\widehat{D}^t_{\zeta})\cong\mathfrak{ind}_{\m{ACF};\beta}(\widehat{D}_{\zeta})\coloneqq\m{dim}(\m{ker}(\mathfrak{D}_{\mathcal{K};\beta}))-\m{dim}(\m{coker}(\mathfrak{D}_{\mathcal{C}o\mathcal{K};\beta-1}))
    \end{align*}
    and hence, if $\beta_2<\beta_1$, then 
    \begin{align*}
       \m{ind}_{\m{ACF};\beta_2}(\widehat{D}^t_{\zeta})-\m{ind}_{\m{ACF};\beta_1}(\widehat{D}^t_{\zeta})=\sum_{\beta_2<\lambda<\beta_1}d_{\lambda+\frac{m-1}{2}-\delta(\widehat{E}_0)}(\widehat{D}_0)
    \end{align*}
\end{manualproposition}
Furthermore, we are able to improve on the weighted Schauder estimate and show that the operator $\widehat{D}^t_\zeta$ is uniform bounded from below on the complement of the adiabatic kernel.\\

There exist constants $\iota$ and $\pi$ depending on the geometry of the resolution such that the following proposition holds true.

\begin{manualproposition}{\ref{uniformboundszetaprop}}
Let $-\iota<\kappa<\pi$. Then the operator
\begin{align*}
    \widehat{D}^t_{\zeta}:\mathfrak{D}^{k+1,\alpha}_{\m{ACF};\beta;t}(\widehat{D}^t_{\zeta};\m{APS})\rightarrow \mathfrak{C}^{k,\alpha}_{\m{ACF};\beta-1;t}(N_\zeta,\widehat{E}_\zeta)
\end{align*}
satisfies
\begin{align*}   \left|\left|\Phi\right|\right|_{\mathfrak{D}^{k+1,\alpha}_{\m{ACF};\beta;t}}\lesssim \left|\left|\widehat{D}^t_{\zeta}\Phi\right|\right|_{\mathfrak{C}^{k,\alpha}_{\m{ACF};\beta-1;t}}+t^{-\kappa}\cdot \left|\left|\pi_{\mathcal{K};\beta}\Phi\right|\right|_{C^{0}_{t}}
\end{align*}
Further, let $\Phi\in\mathfrak{Ker}_{\m{ACF};\beta}(\widehat{D}_\zeta)^{\perp}\coloneqq\m{ker}(\mathfrak{D}_{\mathcal{K};\beta})^\perp$. Then 
\begin{align*}
\left|\left|\Phi\right|\right|_{\mathfrak{D}^{k+1,\alpha}_{\m{ACF};\beta;t}}\lesssim&\left|\left|\widehat{D}^t_{\zeta}\Phi\right|\right|_{\mathfrak{C}^{k,\alpha}_{\m{ACF};\beta-1;t}}.
\end{align*}
\end{manualproposition}

In Section \ref{Uniform Elliptic Theory for Dirac Operators on Orbifold Resolutions of type A} we will state and prove one of the main theorem of this paper. We begin by defining the approximate kernel and cokernel of $D_t$ by
    \begin{align*}
        \m{xker}_{\beta}(D_t)\coloneqq\begin{array}{c}
             \m{ker}\left(\mathfrak{D}_{\mathcal{K};\beta}: C^{k+1,\alpha}_{t}\left(S,\mathcal{K}_{\m{AC};\beta}(\pi_\zeta/\nu_\zeta)\right)\rightarrow  C^{k,\alpha}_{t}\left(S, \mathcal{K}_{\m{AC};\beta}(\pi_\zeta/\nu_\zeta)\right)\right)\\
             \oplus\\
             \m{ker}\left(D:C^{\bullet,\alpha}_{\m{CFS};\beta;\epsilon}(X,E;\m{APS})\rightarrow C^{\bullet-1,\alpha}_{\m{CFS};\beta-1;\epsilon}(X,E)\right)
        \end{array}
    \end{align*}
    and\\ 
    \resizebox{0.9\linewidth}{!}{
\begin{minipage}{\linewidth}
        \begin{equation*}
        \m{xcoker}_{\beta-1}(D_t)\coloneqq\begin{array}{c}
             \m{coker}\left(\mathfrak{D}_{\mathcal{C}o\mathcal{K};\beta-1}: C^{k+1,\alpha}_{t}\left(S,\mathcal{C}o\mathcal{K}_{\m{AC};\beta-1}(\pi_\zeta/\nu_\zeta)\right)\rightarrow  C^{k,\alpha}_{t}\left(S,\mathcal{C}o \mathcal{K}_{\m{AC};\beta}(\pi_\zeta/\nu_\zeta)\right)\right)\\
             \oplus\\
             \m{coker}\left(D:C^{\bullet,\alpha}_{\m{CFS};\beta;\epsilon}(X,E;\m{APS})\rightarrow C^{\bullet-1,\alpha}_{\m{CFS};\beta-1;\epsilon}(X,E)\right).\\
        \end{array}
    \end{equation*}
    \end{minipage}}

    Further, we will introduce function spaces of sections of $E_t$ by interpolating the ACF and CFS-norms. Using the results of the prior sections we are able to state and to prove the main theorem:    
    \begin{manualtheorem}{\ref{lineargluingthm}}
        The operator $D_t$ satisfies 
                \begin{align*}
                \left|\left|\Phi\right|\right|_{\mathfrak{D}^{k+1,\alpha}_{\beta;t}}\lesssim \left|\left|D_t\Phi\right|\right|_{\mathfrak{C}^{k,\alpha}_{\beta-1;t}}
            \end{align*}
        for all $\Phi\in \m{xker}(D_t)^\perp$. Moreover, there exists an exact sequence 
        \begin{equation}
            \begin{tikzcd}
                \m{ker}(D_t)\arrow[r,"i_{\beta;t}",hook]&
                     \m{xker}_{\beta}(D_t)\arrow[r,"\m{ob}_{\beta;t}"]&\m{xcoker}_{\beta-1}(D_t)\arrow[r,"p_{\beta;t}",two heads]&\m{coker}(D_t)
            \end{tikzcd}
        \end{equation}
        which implies the existence of a uniform bounded right-inverse of $D_t$ if $\m{ob}_{\beta;t}=0$.
    \end{manualtheorem}
    By taking the Euler characteristic of the sequence, we deduce that 
    \begin{align*}
        \m{ind}(D_t)=\m{ind}_{\m{CFS};\beta}(D)+\mathfrak{ind}_{\m{ACF};\beta}(\widehat{D}_\zeta).\\
    \end{align*}

The Appendix \ref{The Spectrum of Dirac Operators on Spheres} includes a discussion of the spectral theory for Dirac operators on spheres, following the work of \cite{baer1996dirac} and \cite{ikeda1978spectra}. This material is included since the critical rates (or indicial roots) that arise in the analysis of the normal operator are directly determined by the eigenvalues of Dirac operators on spheres. The appendix therefore provides the necessary background to identify and compute these rates in the main text.

\section*{Acknowledgements}
I am grateful to my PhD supervisor Thomas Walpuski for his encouragement and support. Additionally, I extend my gratitude to Gorapada Bera, Dominik Gutwein, Jacek Rzemieniecki and Thibault Langlais for their valuable discussions, constructive tips and helpful insights. This work is funded by the Deutsche Forschungsgemeinschaft (DFG, German Research Foundation) under Germany's Excellence Strategy – The Berlin Mathematics Research Center MATH+ (EXC-2046/1, project ID: 390685689).

\section{Geometric Structures on Orbifolds}
\label{Geometric Structures on Orbifolds}

In the following section we will briefly discuss Riemannian orbifolds by following the work of Moerdijk in \cite{moerdijk2002orbifolds}. In this approach, the category of orbifolds is associated with the category of proper foliation groupoids up to Morita equivalence. By representing an orbifold by a Lie groupoid, the notions of tangent bundle and normal bundle can be understood in a more natural way than by the usual chart based approach.
\subsection{Orbifolds}
\label{Orbifolds}

In this section, we recall the groupoid-theoretic definition of orbifolds and their morphisms, following the modern approach via proper étale Lie groupoids and Morita equivalence.

\begin{defi}
Let $X^\bullet$ be a Lie groupoid, i.e. a groupoid object \footnote{A category in which every morphism is an isomorphism.} in $\cat{C}^\infty\cat{Mfd}$. We will denote by
\begin{align*}
    X^0\coloneqq \m{Obj}(X^\bullet)\und{1.0cm}X^1\coloneqq \m{Mor}(X^\bullet).
\end{align*}

\begin{itemize}
    \item $X^\bullet$ is called a \textbf{proper Lie groupoid} if $(s,t):X^1\rightarrow X^0\times X^0$ is a proper map. In particular, every isotropy group is a compact Lie group.
    \item $X^\bullet$ is called a \textbf{foliation groupoid} if every isotropy group $\m{Isot}_x$ is discrete.
    \item $X^\bullet$ is called \textbf{étale} if $s$ and $t$ are local diffeomorphisms.
\end{itemize}
Let $X^\bullet$ be an étale Lie groupoid. Then every arrow $\phi:x\to y$ in $X^\bullet$ induces a germ of a diffeomorphism $\tilde{\phi}:(U_x,x)\rightarrow (V_y,y)$ as $\tilde{\phi}=t\circ \hat{\phi}$, where $\hat{\phi}:U_X\rightarrow X^1$ is an $s:X^1\rightarrow X^0$ section, defined on a sufficiently small neighbourhood $U_x$ of $x$. 
\begin{itemize}
    \item $X^\bullet$ is called \textbf{faithfull}, if for each point $x\in X^0$ the homomorphism $\m{Isot}_x\rightarrow \m{Diff}_x(X^0)$ is injective.
\end{itemize}
\end{defi}

\begin{defi}
Let $X^\bullet$ be a Lie groupoid. A \textbf{left (right) $X^\bullet$-space} is a smooth manifold $E$ equipped with an action of $X^\bullet$. Such an action is given by the smooth morphisms  
\begin{align*}
    \pi:E\rightarrow X^0\hspace{1.0cm}\mu_L:X^1\times_{s,\pi}E\rightarrow E\hspace{1.0cm}(\mu_R:E\times_{\pi;t}X^1\rightarrow E)
\end{align*}
satisfying $\pi(\phi.e)=s(\phi)$, $1_x.e=e$ and $\eta.(\phi.e)=(\eta\phi).e$. For each $X^\bullet$ space $E$, we define the translation groupoid $( X\ltimes E)^\bullet$ whose objects are points in $E$ and whose morphisms $\phi:e\rightarrow e'$ are morphisms $\phi:\pi(e)\rightarrow\pi(e')$ in $X^1$ with $\phi.e=e'$. There exists a homomorphism $\pi_E:(X\ltimes E)^\bullet\rightarrow X^\bullet$ of orbifolds. 
\end{defi}

\begin{rem}
By replacing the left by right actions in the above definition we can define right $X^\bullet$-spaces.
\end{rem}

\begin{rem}
Notice that on the level of groupoids, the fibre over $x\in X^0$ is $\pi^{-1}(x)$; at the level of orbit spaces $|X\ltimes E|\rightarrow |X|$ the fibres are $\pi^{-1}(x)/\m{Isot}_x$.
\end{rem}

\begin{defi}
A \textbf{right/left-$X^\bullet$-$Y^\bullet$-bibundle} $Z$ is given by smooth maps $\alpha_X:Z\rightarrow X^0$ and $\alpha_Y:Z\rightarrow Y^0$ called the moment maps such that both $\alpha_X$ is a right/left $X^\bullet$-space and $\alpha_Y$ is a right/left $Y^\bullet$-space, i.e.
\begin{equation*}
\begin{tikzcd}
	{X^1} && Z && {Y^1} \\
	\\
	{X^0} &&&& {Y^0}
	\arrow["t"{description}, shift left=2, from=1-1, to=3-1]
	\arrow["s"{description}, shift right=2, from=1-1, to=3-1]
	\arrow["{\alpha_X}"{description}, from=1-3, to=3-1]
	\arrow["{\alpha_Y}"{description}, from=1-3, to=3-5]
	\arrow["t"{description}, shift left=2, from=1-5, to=3-5]
	\arrow["s"{description}, shift right=2, from=1-5, to=3-5]
\end{tikzcd}
\end{equation*}
\end{defi}

\begin{defi}
A \textbf{generalised morphism} 
\begin{align*}
    [Z]:X^\bullet\rightsquigarrow Y^\bullet
\end{align*}
is given by an isomorphism class of right $X^\bullet$-$Y^\bullet$-bibundles.
\end{defi}

\begin{defi}
A Lie groupoid $X^\bullet$ is \textbf{Morita equivalent} to Lie groupoid $Y^\bullet$ if there exists a generalised morphism 
\begin{align*}
    [Z]:X^\bullet\rightsquigarrow Y^\bullet
\end{align*}
which is also a left $X^\bullet$-$Y^\bullet$-bibundle.
\end{defi}

\begin{defi}
A \textbf{smooth orbifold} $X$ is the Morita equivalence class of a smooth, proper, foliation Lie groupoid $X^\bullet$. Given such an equivalence class, we will write  
\begin{align*}
    [X^0/X^1]:=[X^\bullet].
\end{align*}
A \textbf{morphism of orbifolds} is given by an equivalence class of generalised morphism 
\begin{align*}
    [[Z]]:[X^\bullet]\rightsquigarrow[Y^\bullet].
\end{align*}
\end{defi}

\begin{rem}
In every equivalence class of an orbifold $[X^\bullet]\subset\cat{LieGrpd}$ there exists a proper, étale groupoid representing the orbifold. 
\end{rem}

\begin{rem}
Notice that the orbit space/coarse moduli space $X^0/X^1$ just carries the data of a topological space. The stacky quotient $[X^0/X^1]$ \say{remembers} the isotropy.
\end{rem}

\begin{defi}
The \textbf{tangent bundle} of an orbifold $[X^\bullet]$ is defined by Morita equivalence class of the $X^\bullet$-space
\begin{align*}
    \pi_X:TX^0 \rightarrow X^0,\quad \mu_{TX}:X^1 \times_{s, \pi_{TX^0}} TX^0 \rightarrow TX^0,
\end{align*}
where $\mu_{TX}(g, V_x) = T g(V_x)$ for $g: x \to y$ in $X^1$. 
\end{defi}

If $S^\bullet\subset X^\bullet$ is subgroupoid such that all isotropy groups are conjugated to a finite group $\Gamma$, then the normal bundle $NS^\bullet \rightarrow S^\bullet$ is naturally an $S^\bullet$-vector bundle, whereas the normal bundle $N|S| \to |S|$ is a (quasi-)conical bundle on the coarse moduli space.

\begin{defi}
The \textbf{frame bundle} of an $n$-dimensional orbifold $[X^\bullet]$ is defined to be the Morita equivalence class of the $X^\bullet$-principal $\mathrm{GL}(n)$-bundle given by 
\begin{align*}
    f_X: Fr(X^0) \rightarrow X^0,\quad 
    \mu_{Fr(X)}: X^1 \times_{s, \pi_{X^0}} Fr(X^0) \rightarrow Fr(X^0),
\end{align*}
where $\mu_{Fr(X)}(\psi, f_x) = T\psi \circ f_x$ for $\psi: x \to y$ in $X^1$.
\end{defi}

The following corollary can be seen as the traditional definition of orbifolds via orbifold charts, i.e. $X$ being a topological space covered by charts of the form $(\m{Isot}_x\ltimes U_x)^\bullet$ subject to gluing conditions.

\begin{cor}\cite[Sec. 3.4]{moerdijk2002orbifolds}
Let $X$ be a smooth orbifold presented by a proper étale Lie groupoid $X^\bullet$. For every open subset $U\subset X^0$ we write $(X|_U)^\bullet$ for the full subgroupid whose objects are $U$. Since $X^1$ is assumed to be proper and étale, for $x\in X^0$ there exists an arbitrarily small neighbourhood $U_x$ such that $(X|_{U_x})^\bullet$ is isomorphic to the action groupoid $(\m{Isot}_x\ltimes U_x)^\bullet$. 
\end{cor}

\begin{defi}
    The \textbf{dimension of a smooth orbifold} is defined as $ \m{dim}(X)\coloneqq \m{dim}(U_x)$ which is independent of the point $x$.
\end{defi}

Orbifolds are examples of stratified spaces. The stratification of $X$ consists of the suborbifolds $S\subset X$ for which $\m{Isot}_s\cong\m{Isot}_{s'}$ for all $s,s'\in S$ and
\begin{align*}
    X^1/X^0=\bigcup_{[ \Gamma_\alpha]}X_\alpha\in\cat{Top}
\end{align*}
where the union runs over all conjugacy classes $[ \Gamma_\alpha]$ of finite subgroups of $\m{GL}(n)$. Notice, that every $X_\alpha$ is a smooth (open) manifold.\\ 

We define a partial order on the set $\{X_\alpha\}$ by 
\begin{align*}
    X_\alpha\leq  X_\beta
\end{align*}
if $ \Gamma_\beta\subset  \Gamma_\alpha$. Notice that since $\{X_\alpha\}$ defines a stratification, $X_\beta\leq X_\alpha$ implies $X_\beta\subset \overline{X_\alpha}$. Hence, we can define a prestratification of $X$ by $X_{\tilde{\alpha}}=\{\bigcup_{\beta\leq\alpha}X_\beta\}_\alpha$. In particular, we say that two connected singular strata $X_\alpha$ and $X_{\alpha'}$ intersect in a singular stratum $X_\beta$, if $X_\beta\leq X_\alpha$ and $X_\beta\leq X_{\alpha'}$ or equivalently 
\begin{align*}
    X_{\tilde{\alpha}}\cap X_{\tilde{\alpha}'}=X_{\beta}.
\end{align*}
Notice, that $\overline{X_\alpha}=X_\alpha\sqcup \bigsqcup_{\beta\leq\alpha}X_\beta$. 

\begin{rem}
    The top stratum of an $n$-dimensional orbifold $X$, i.e. the open dense subset consisting of points with trivial isotropy will be denoted by $X^{reg}$. If the isotropy group $\Gamma_{reg} = 1$, we call the orbifold effective. For the remainder of this paper we assume all orbifolds to be effective in order to avoid additional complications.
\end{rem}

The following definitions quantify how severe or complex a given singularity can be. The depth measures the position of a stratum within the nested hierarchy of singularities, while the singularity type describes the local model and isotropy representation near the stratum. Together, these notions capture both the local complexity and the global stratified structure of the singular set.

\begin{defi}
    We define the \textbf{depth} of $X_{\alpha}$, as the length of the longest chain of inclusions of singular strata, i.e.  
    \begin{align*}
        \m{depth}(X_\alpha)=\max\{i\in \mathbb{N}|X_{\alpha}= X_{\beta_0}\leq...\leq X_{\beta_i}= X^{reg}|X_{\beta_k}\neq X_{\beta_{k-1}}\}.
    \end{align*}
\end{defi}

\begin{defi}
Let $(X,g)$ be a Riemannian orbifold and let $X^{sing}$ denote the union of its singular strata. Let $S$ be a stratum whose isotropy group is conjugated to a finite group $\Gamma\subset\m{O}(n)$. We say
\begin{itemize}
    \item $X$ is of \textbf{singularity type A} at $S$, if $X$ at $S\subset X$ is locally modelled on
    \begin{align*}
    \mathbb{R}^{n-m}\times\mathbb{R}^m/\Gamma
    \end{align*}
    such that $\Gamma$ acts freely on $\mathbb{R}^m\backslash\{0\}$.
    \item $X$ is of \textbf{singularity type B} at $S$, if $X$ at $S\subset X$ is locally modelled on
    \begin{align*}
    \mathbb{R}^{n-\sum_im_i}\times\prod_i\mathbb{R}^{m_i}/ \Gamma_i
    \end{align*}
    such that $ \Gamma_i$ acts freely on $\mathbb{R}^{m_i}\backslash\{0\}$.
    \item $X$ is of \textbf{singularity type C} at $S$, if it is neither type A nor type B
\end{itemize}
\end{defi}

\begin{rem}
Let $S_1$ and $S_2$ be two singular strata of isotropy $ \Gamma_1$ and $ \Gamma_2$. Then the intersection $S_3=\overline{S_1}\cap\overline{S_2}$ is a singular strata of isotropy $ \Gamma_1 \Gamma_2\subset \Gamma_3$ and $X$ is singular of type B or C.
\end{rem}

\begin{nota}
    Throughout this paper we will denote by 
    \begin{align*}
        \m{N}_G(H)\und{1.0cm} \m{C}_G(H)
    \end{align*}
    the normaliser and the commutator subgroup of $H$ in $G$.
\end{nota}

\begin{prop}
Let $S\subset X$ be a singular strata of isotropy type $\Gamma$. Then there exists a $\m{N}_{\m{GL}(n)}(\Gamma)\hookrightarrow \m{GL}(n)$ reduction of the restriction of the frame bundle $Fr_{X^\bullet}$ to $S$.
\end{prop}

\begin{proof}
Let $X^\bullet$ be a Lie groupoid representing the orbifold $[X^0/X^1]$. Let $S^\bullet$ denote the sub Lie groupoid corresponding to $S\subset [X^0/X^1]$ in the quotient space. Notice that the $X^\bullet$-frame bundle restricts to a $\m{GL}(n)$-principal bundle of $S^\bullet$, i.e. 
\begin{align*}
    f_X|_S:Fr_{X^0}|_{S^0}\rightarrow S^0\hspace{1.0cm}\mu_{Fr_X|_S}:S^1\times_{s,\phi_{X^0}|_{S^0}}Fr_{X^0}|_{S^0}\rightarrow Fr_{X^0}|_{S^0}.
\end{align*}
There exists a short exact sequence of $S^\bullet$-vector bundles 
\begin{equation*}
    \begin{tikzcd}
        TS^\bullet\arrow[r,"T\iota",hook]&TX^\bullet|_S\arrow[r,"p",two heads]&NS^\bullet
    \end{tikzcd}
\end{equation*}
In particular, there exists a natural reduction to frames 
\begin{align*}
    f_x:H\oplus V\rightarrow TX^0
\end{align*}
such that $p\circ f_x|_V\in \m{Fr}(NS^\bullet)$. A moment's thought reveals that this reduction can be represented by smooth principal-$\m{N}_{\m{GL}(n)}({\Gamma})$-bundle $F_{\m{N}_{\m{GL}(n)}(\Gamma)}(S)\rightarrow S$. 
\end{proof}

\subsection{$G$-Structures}
\label{G Structures}

We begin by recalling the general framework of $G\subset\m{GL}(n)$-structures orbifolds. These are constructed via reductions of the frame bundle of the orbifold. 
 
\begin{defi}
Let $X$ be an orbifold and $Fr(X)$ its $\m{GL}(n)$-frame bundle. The \textbf{bundle of $G$-structures} on $X$ is the homogeneous bundle
    \begin{align*}
        \gamma\colon Str_G(X)\coloneqq\m{Fr}(X)/G\rightarrow X.
    \end{align*}
    A $G$-structure on an orbifold is given by a section $\Phi\in \Gamma(X,Str_G(X))$. Such a section induces a $G$-reduction of the frame bundle by 
    \begin{align*}
        Fr(X,\Phi)\coloneqq \Phi^*Fr(X)\hookrightarrow Fr(X)
    \end{align*}
    pulling back the $G$-bundle $\phi_{/G}:Fr(X)\rightarrow Str_G(X)$. 
\end{defi}

\begin{lem}
    Given a nested subgroup $G\subset H\subset \m{GL}(W)$, there exists a natural projection map 
    \begin{align*}
        Str_G(X)\twoheadrightarrow Str_H (X),
    \end{align*}
    whose fibres are isomorphic to the homogeneous space $H/G$.
\end{lem}

The existence of a $G$-structure on an orbifold can be obstructed by the nature of its isotropy groups. The following proposition clarifies the necessary condition.

\begin{prop}
    Let $X$ be an orbifold modelled on finite quotients of a finite dimensional vector space $W$ and let $\phi:Fr(X)\rightarrow X$ denote its frame bundle. If $Fr(X)$ admits a reduction to a $G$-structure, then all isotropy groups are conjugated to subgroups of $G$.
\end{prop}

\begin{proof}
    The isotropy group of a point $x\in X$ can be identified with a conjugacy class of a finite subgroup $\m{Isot}(x)\subset \m{GL}(W)$.  The frame bundle of $X$ at $x$ naturally reduces to frames  
    \begin{align*}
        f:W\rightarrow T_xX
    \end{align*}
    that normalizes the linearised action of the isotropy group, i.e. its representation in $\m{GL}(T_xX)$. These frames form a $\m{N}_{\m{GL}(T_xX)}(\m{Isot}(x))$ reduction of the frame bundle along the stratum containing $x$. If $X$ contains strata whose isotropy is not conjugated to a subgroup of $G$, the reduction $\m{N}_{\m{GL}(T_xX)}(\m{Isot}(x))$ can not be a $G$ reduction. Hence, $Fr(X)$ can not be reduced to a sub $G$-bundle.
\end{proof}

\begin{ex}
    Let $(W,\mathfrak{o})$ be an oriented vector space. An orientation on an orbifold, locally modelled by finite quotients of $W$, is oriented if it admits a section of the orientation bundle 
    \begin{align*}
        \mathfrak{or}(X)\coloneqq Fr(X)/\m{GL}_+(W)\cong (\wedge^{top}T^*X)/\mathbb{R}_{>0}.
    \end{align*}
\end{ex}

\begin{ex}
Let $(W,g_0)$ be an Euclidean vector space. The most important $G$-structure on an orbifold $X$ locally modelled on $W$ in Riemannian geometry is an $\m{O}(W,g_0)$-structure, i.e. a Riemannian orbifold structure. A Riemannian metric on $X$ corresponds to a section of the homogeneous bundle 
\begin{align*}
    Met(X)\coloneqq Fr(X)/\m{O}(W,g_0)\rightarrow X
\end{align*}
with fibres being 
\begin{align*}
    Met(W)\coloneqq \m{GL}(W)/\m{O}(W,g_0).
\end{align*}
Given a section $g\in \Gamma(X,Met(X))$ we obtain a reduction of the framebundle 
\begin{align*}
    Fr(X,g)\coloneqq g^*Fr(X)\hookrightarrow Fr(X)
\end{align*}
by pulling back the $\m{O}(W,g_0)$-bundle $\phi_{/\m{O}(W,g_0)}:Fr(X)\rightarrow Met(X)$.
\end{ex}

We will now define the Clifford bundle of $X$. This is a finite dimensional orbifold bundle, such that given a section $g:X\rightarrow Met(X)$, the pullback of the Clifford bundle corresponds to the Clifford bundle of $X$ with respect to $g$. This notion will be needed in the discussion of families of Dirac bundles later in this paper.

\begin{defi}
    We define the \textbf{Clifford bundle} of $X$ to be the associated bundle 
    \begin{align*}
        \m{Cl}(X)\coloneqq Fr(X)\times_{\m{O}(W,g_0)}\m{Cl}(W^*,g_0)\rightarrow Met(X).
    \end{align*}
    Given a metric on $X$, i.e. a section $g:X\rightarrow Met(X)$, we have
    \begin{align*}
        g^*\m{Cl}(X)=\m{Cl}(T^*X,g).
    \end{align*}
\end{defi}

\subsection{$G$-Structures on Fibrations}
\label{Gstructures on Fibrations}
Throughout this paper we will work with geometric structures on fibrations. The following section is devoted to the discussion of such structures from the general perspective of $G$-structures as introduced in the previous sections.\\

Let $M$ be an orbifold locally modelled on a vector space $W=H\oplus V$ and $\pi:M\rightarrow B$ a fibration. We define the group 
\begin{align*}
    \m{GL}(W\to H)\coloneqq \{A\in\m{GL}(H\oplus V)|A(V)=V,\m{pr}_H\circ A(H)=H\}.
\end{align*}
There exists a natural $\m{GL}(W\to H)$-reduction of the frame bundle 
\begin{align*}
    Fr(\pi:M\rightarrow B)\hookrightarrow Fr(X)
\end{align*}
of frames
\begin{align*}
    Fr(\pi\colon M\rightarrow B)\coloneqq\{f_1\colon H\oplus V \rightarrow T_{m}M|f_1|_{V}\colon V\cong V_m\pi,\, T\pi\circ f_1|_H\in Fr(B)_{\pi(m)}\}. 
\end{align*}

We will now be interested in reductions of such fibered structures and their torsions. In particular, this will identify the torsion of such structures with geometric quantities. In general, we will set 
\begin{align*}
    G(W\to H)\coloneqq G\cap \m{GL}(W\to H).
\end{align*}
Notice that this group depends on the embedding of $G\subset \m{GL}(W)$. We will now investigate fibred Riemannian structures. 

\begin{defi}
    A \textbf{fibred Riemannian structure} is given by a section of 
    \begin{align*}
        Fr(\pi\colon M\rightarrow B)/O(W\to H)\rightarrow M.
    \end{align*}
    Here $\m{O}(W\to H)=O(W)\cap \m{GL}(W\to H)\cong \m{O}(H)\times \m{O}(V)$. 
\end{defi}

The unique torsion-free Levi-Civita connection $\varphi_g$ on $Fr(M,g)$ pulls back to $Fr(\pi\colon M\to B,g)$ and decomposes into  
    \begin{align*}
        \varphi_g=\varphi_\oplus+T_{\oplus}.
    \end{align*}
We will refer to $T_{\oplus}$ as the torsion of the fibred Riemannian structure. This tensor is valued in
    \begin{align*}
        T_\oplus\in \Omega^0(Fr(\pi\colon M\rightarrow B),\wedge^2 H^*\otimes V\oplus \m{Sym}^2 V^*\otimes H)^{\m{O}(W\to H)}.
    \end{align*}

We will now identify the tensor $T_{\oplus}$ with geometric quantities of the Riemannian fibration. These results will be used throughout this paper to understand the behaviour of Dirac operators on collapsing fibrations and to construct resolutions of tubular neighbourhoods of singular strata of $\m{Spin}(7)$-orbifolds. The broader discussion is based on the work of Gromoll and Walshap in \cite{gromoll2009metric} and Berline, Getzler and Vergne in \cite{berline1992heat}.\\

The connection $\varphi_{\oplus}$ induces a covariant derivative that can be identified with
\begin{align*}
    \nabla^{\oplus}=\pi^*\nabla^{g_S}\oplus\nabla^{g_{V}},
\end{align*}
i.e. the direct sum connection of the vertical connection and the lift of the Levi-Civita connection on the base space through the Ehresmann connection $H=VM^\perp$. Let in the following 
\begin{align*}
    X,Y,Z\in\mathfrak{X}^{1,0}(M)\und{1.0cm}U,V,W\in\mathfrak{X}^{0,1}(M)
\end{align*}
and denote by $\nabla$ the Levi-Civita connection of $g$.

\begin{defi}
We define the two tensor $A\in\Omega^{1,0}(M,\m{Hom}(H,VM))$ by
\begin{align*}
    A(X)Y=\nabla^{0,1}_XY=\frac{1}{2}[X,Y]^{0,1}
\end{align*}
and
    \begin{align*}
    S\in\Omega^{1,0}(M,\m{End}(VM)) \quad \text{by} \quad S(X)U=-(\nabla_UX)^{0,1}.
\end{align*}

\end{defi}

The second fundamental form of the fibre can thus be identified with the tensor
\begin{align*}
    \m{II}(U,V)=\left<S U,V\right>^\flat.
\end{align*}
By taking the vertical trace we obtain the fibrewise mean curvature, i.e. the one form 
\begin{align*}
    k=\m{tr}_{g;V}\m{II}.
\end{align*}

\begin{rem}
For a Riemannian fibration, the fibrewise second fundamental form $\m{II}$ or equivalently $S$ vanishes if and only if the fibres are totally geodesic. Furthermore, the tensor $A$ vanishes if and only if the Ehresmann connection $H$ is flat.
\end{rem}

\begin{prop}\cite[Prop. 10.6]{berline1992heat}
Given $X,Y,Z\in\mathfrak{X}(M)$ the torsion tensor $T_{\oplus}\in\Omega^1(M,\wedge^2T^* M)$ can be expressed by 
\begin{align*}
    T_{\oplus}(X)(Y,Z)=&g(\m{II}(X,Z),Y)-g(\m{II}(X,Y),Z)\\
    &+\tfrac{1}{2}\left(g(F_{H}(X,Z),Y)-g(F_{H}(X,Y),Z)\right.\\
    &+\left.g(F_{H}(Y,Z),X)\right).
\end{align*}
\end{prop}

\begin{rem}\cite[Prop. 10.1]{berline1992heat}
The de Rham differential on $M$ decomposes with respect to $H$ and the Riemannian structure $g=p^*g_S\oplus g_{V}$ into 
\begin{align*}
    \m{d}=&\m{d}^{1,0}+\m{d}^{0,1}+\m{d}^{2,-1}\\
    =&\m{d}_{\nabla^{g_{V}}}-\hat{\iota}_{\m{II}}+\m{d}^{0,1}+\hat{\iota}_{F_H}.
\end{align*}
In particular, let $\m{vol}_{V}$ denote the vertical volume form with respect to $g_{V}$. Then 
\begin{align*}
    \m{d}k(X,Y)=-2\m{div}(A(X)Y)\und{1.0cm}\m{d}\m{vol}_{V}=k\wedge\m{vol}_{V}+\iota_{F_{H}}\m{vol}_{V}.
\end{align*}
\end{rem}

\subsection{Hermitian Dirac Bundles}
\label{Hermitian Dirac Bundles}

In the following section we will discuss the notion of Dirac bundles on Riemannian orbifolds. These structures play a central role in this paper.

\begin{defi}
Let $\nabla$ be a metric connection on $(X,g)$. A \textbf{Hermitian Dirac bundle} on $(X,g)$ is given by a tuple
\begin{align*}
    (E, \m{cl}_g, h, \nabla^h)
\end{align*}
consisting of a $\m{Cl}(T^*X,g)$-module 
\begin{align*}
    \pi:E\rightarrow X\hspace{1.0cm}\m{cl}_g:T^*X\rightarrow\m{End}(E)
\end{align*}
endowed with a Hermitian structure $h$ and a Hermitian connection $\nabla^h$ compatible with the $\m{Cl}(T^*X,g)$-module, i.e. 
\begin{align*}
    \nabla^h(\m{cl}_g(\xi)\Phi)=&\m{cl}_g(\nabla\xi)\Phi+\m{cl}_g(\xi)\nabla^h\Phi\\
    0=&h(\m{cl}_g(\xi)\Phi,\Phi')+h(\Phi,\m{cl}_g(\xi)\Phi')\\
    \m{d}h(\Phi,\Phi')=&h(\nabla^h\Phi,\Phi')+h(\Phi,\nabla^h\Phi').
\end{align*}
The \textbf{Dirac operator} associated to the Dirac bundle $(E,\m{cl}_g,h,\nabla)$ is given by 
\begin{equation*}
\begin{tikzcd}
    D: \Gamma(X,E)\arrow[r,"\nabla^h"]&\Omega^1(X,E)\arrow[r,equal]& \Gamma(X,T^*X\otimes E)\arrow[r,"\m{cl}_g"]& \Gamma(X,E)
\end{tikzcd}   
\end{equation*}
If $\nabla\neq \nabla^g$ we will refer to $\nabla^h$ as a Clifford module connection with torsion.
\end{defi}

\begin{rem}
    Notice that since $\m{cl}_g$ is skew-Hermitian
    \begin{align*}
        |\m{cl}_g(\xi)\Phi|^2=h(\m{cl}_g(\xi)\Phi,\m{cl}_g(\xi)\Phi)=g(\xi,\xi)h(\Phi,\Phi)=|\xi|^2|\Phi|^2.
    \end{align*}
\end{rem}

\begin{rem}
    The Clifford multiplication $\m{cl}_g:T^*X\rightarrow \m{End}(E)$ extends to an algebra morphism 
    \begin{align*}
        \m{cl}_g:\m{Cl}(T^*X,g)\rightarrow \m{End}(E).
    \end{align*}
    The Clifford multiplication can be seen as a parallel section of 
    \begin{align*}
        \m{Hom}(\m{Cl}(T^*X,g),\m{End}(E)),
    \end{align*}
    i.e. $\nabla^{\m{End}(E)}\m{cl}_g(\omega)=\m{cl}_g(\nabla^g\omega)$.
\end{rem}

In some cases we will require the Dirac bundle to carry an additional geometric structure.

\begin{defi}
A \textbf{grading} on a Hermitian Dirac bundle $(E,\m{cl}_g,h,\nabla^h)$ is given by an endomorphism $\epsilon\in \Gamma(X,\m{End}(E))$ such that 
\begin{align*}
    \epsilon^2=1,\hspace{1.0cm}[\m{cl}_g,\epsilon]=0,\hspace{1.0cm}\nabla^h\epsilon=0\und{1.0cm}\epsilon^*h=h.
\end{align*}
\end{defi}

Given a grading on a Dirac bundle, the Clifford module $E$ decomposes into $h$-orthogonal $\pm 1$-eigenspaces 
\begin{align*}
    E=E_+\oplus E_-.
\end{align*}
Notice that in this case the Dirac operator decomposes into 
\begin{align*}
    D=\left(\begin{array}{cc}
         &  D^-\\
         D^+& 
    \end{array}\right).
\end{align*}

The main object of interest of this paper is to study families of Hermitian Dirac bundles, whose underlying Riemannian structures form a smooth Gromov-Hausdorff resolution of a Riemannian orbifold $(X,g)$ equipped with a Hermitian Dirac orbifold bundle. We will define such $I$-families as certain geometric structures on $I\times X$. This definition has the benefit that we can compare elements of such families in a smooth way without choosing identifications.\\

Let in the following $I$ be a manifold and denote by 
\begin{equation*}
    \begin{tikzcd}
        I&\arrow[l,two heads,"\m{pr}_I",swap]I\times X\arrow[r, two heads, "\m{pr}_X"]&X.
    \end{tikzcd}
\end{equation*}

\begin{defi}
An \textbf{$I$-family of Hermitian Dirac bundles} is given by an $I$-family of Riemannian metrics, seen as a section
\begin{align*}
    g\in\Gamma(I\times X,\m{pr}_X^*Met(X))
\end{align*}
and a $\m{pr}_I$-vertical metric connection, i.e. 
\begin{align*}
    \nabla^{g}\in\Gamma(I\times X,g^*\m{pr}^*_X \m{Jet}^1_X(Str_G(X))
\end{align*}
a Hermitian vector bundle 
\begin{align*}
    \pi:E\rightarrow I\times X \und{1.0cm}
    h\in \Gamma(I\times X,Herm(E))\subset \Gamma(I\times X,\m{Sym}^2(E^*))
\end{align*}
with a vertical Hermitian connection 
\begin{align*}
    \nabla^{h}:\Gamma(I\times X,h^*\m{Jet}^1_X(Herm(E))
\end{align*}
and a Clifford module structure, i.e. a representation 
\begin{align*}
    \m{cl}_g:g^*\m{pr}^*_X\m{Cl}(X)\rightarrow \m{u}(E,h)
\end{align*}
such that $\nabla^h$ is a $(g^*\m{pr}^*_X\m{Cl}(X),\nabla^g)$-module connection.
\end{defi}

Locally $X$ is always spin and hence, locally the Dirac bundle uniquely decomposes into a twisted spinor bundles, i.e. 
\begin{align*}
    E\cong_{loc}(SX,b,\overline{\nabla}^g)\otimes (F,h_{tw.},\nabla^{tw.})
\end{align*}
In particular, the Hermitian Clifford module connection is the product of a spinor connection and a Hermitian connection on the local twisting bundle. This locally defined connection is called the twisting connection. Its curvature is given by a globally defined 
\begin{align*}
    F^{tw.}_{\nabla^h}\in\Omega^2(X,\m{End}(E)).
\end{align*}

\begin{defi}
Define the vector bundle $\pi_{\m{End}(E)}\colon\m{End}(E,\m{cl}_{g},h)\rightarrow X$ of \textbf{twisting morphisms} by 
\begin{align*}
    \m{End}(E,\m{cl}_{g},h)=\{T\in\m{End}(E),[\m{cl}_{g},T]=0\text{ and }h(T.,.)=h(.,T.)\}.
\end{align*}
Further, there exists a morphism 
\begin{align*}
    \varrho_{g}\colon \m{End}(TX,g)\rightarrow\m{End}(E,\m{cl}_{g},h),\,\varrho_{g}(A)=\tfrac{1}{2}\m{tr}([\m{cl}_{g},\m{cl}_{g}\circ A]).
\end{align*}
This map satisfies 
\begin{align*}
    [\varrho_{g}(A),\m{cl}_{g}(v)]=\m{cl}_{g}(Av).
\end{align*}
\end{defi}

\begin{defi}
Let $q_{g}\colon \wedge^\bullet T^*X\rightarrow\m{Cl}(T^*X,g)$ denote the canonical quantisation map. We define a morphism $\Lambda\colon \wedge^2 T^*X\rightarrow \m{End}(T^*X)$ by 
\begin{align*}
    q_{g}(\Lambda(\alpha)\xi)\coloneqq[q_{g}(\alpha),q_{g}(\xi)].
\end{align*}
Clearly, 
\begin{align*}
    \Lambda(\beta_1\wedge\beta_2)\xi=2(g(\beta_1,\xi)\beta_2-g(\beta_2,\xi)\beta_1).
\end{align*}
The family of maps $\Lambda$ equips $\wedge^2T^*X$ with a family of Lie-algebra structures via
\begin{align*}
    \Lambda([\alpha_1,\alpha_2])\coloneqq[\Lambda(\alpha_1),\Lambda(\alpha_2)].
\end{align*}
Further, denote let $\overline{\Lambda}(\alpha)=-(\Lambda(\alpha))^\dag$. Then 
\begin{align*}
    \frac{1}{2}g(\overline{\Lambda}(\alpha)X,Y)=\left<\alpha,X\wedge Y\right>
\end{align*}
holds.
\end{defi}

Using these algebraic structures, we can characterise the twisting connection globally and show how it appears in the Weitzenböck formula. This formula describes the precise geometric relation between the square of the Dirac operator and the connection Laplacian on the Hermitian Dirac bundle. 

\begin{lem}[Weitzenböck-Formula]
\label{WeitzenböckorbifoldLemma}
The twisting curvature of a Dirac bundle is given by 
\begin{align*}
    F^{tw.}_{\nabla^h}=F_{\nabla^h}-\varrho_g(R_g)\in\Omega^2(X,\m{End}(E,\m{cl}_{g},h))
\end{align*}
and coincides with the curvature of the local twisting connection. The twisting curvature contributes to the Weitzenböck-formula 
\begin{align}
\label{Weitzenböckonorbifolds}
    D^2 =(\nabla^h)^*\nabla^h+\frac{1}{4}\m{scal}_g+\m{cl}_g(F^{tw.}_{\nabla^h}).
\end{align}
\end{lem}

\begin{rem}[Bochner Technique]
Integrating both hand sides of \eqref{Weitzenböckonorbifolds} on a closed orbifold shows that 
\begin{align*}
    \left|\left|D\Phi\right|\right|_{L^2}^2=\left|\left|\nabla^h\Phi\right|\right|_{L^2}^2+\frac{1}{4}\int_X\m{scal}_g|\Phi|^2_{h}\m{vol}_g+\left<\Phi,\m{cl}_g(F^{tw.}_{\nabla^h})\Phi\right>_{L^2}.
\end{align*}
Notice that the parity of the second and third term of the right hand side heavily influences the existence of solutions. In particular, if mentioned terms are positive, no nontrivial solution to $D\Phi=0$ exists.
\end{rem}

\section{Smooth Gromov-Hausdorff-Resolutions of Riemannian\\ Orbifolds}
\label{Smooth Gromov-Hausdorff-Resolutions of Riemannian Orbifolds}

In the following section we will introduce and further develop the theory of smooth resolutions of Riemannian orbifolds using the formalism of Riemannian Lie groupoids, as presented in \cite{posthuma2017resolutionsproperriemannianlie}. We define resolutions via Morita morphisms between Lie groupoids and study the real blow-up construction along the singular strata. Building on this, we reinterpret Riemannian orbifolds as iterated edge spaces, introducing a classification of singularities by type and describing the associated geometry via normal cone bundles and conically fibred neighbourhoods. We then define (quasi-)conical fibrations and their asymptotic analogues (A(Q)CF), which serve as local models for tubular neighbourhoods near singular strata and their resolutions. This framework allows us to formulate a category of smooth Gromov–Hausdorff resolutions, which model the desingularisation of Riemannian orbifolds via families of smooth manifolds collapsing to the orbifold limit. Finally, we construct explicit resolutions of type A singularities, discuss the extension to intersecting strata, and compute the change in cohomology induced by such resolutions. These constructions provide the geometric and analytic setting for later applications to special holonomy metrics and Dirac operators.

\subsubsection{Resolutions of Riemannian Orbifolds}
\label{Resolutions of Riemannian Orbifolds}

In the following we will introduce the resolutions of Riemannian orbifolds and show that each Riemannian orbifold can be resolved by a manifold with corners. The latter is the so called real blow-up and prominently used in the Mazzeo's edge-calculus \cite{mazzeo1991elliptic,mazzeo2014elliptic}.

\begin{defi}
Let $(X,g)$ be a Riemannian orbifold presented by a Riemannian Lie groupoid $(X^\bullet,g^\bullet)$ Morita equivalent to a proper foliation Lie groupoid. A \textbf{resolution} of $(X^\bullet,g^\bullet)$ is given by Riemannian Lie groupoid $(\widetilde{X}^\bullet,\widetilde{g}^\bullet)$ Morita equivalent to a regular, foliation Lie groupoid, and a submersion of Lie groupoids 
\begin{align*}
    \rho\colon (\widetilde{X}^\bullet,\widetilde{g}^\bullet)\rightarrow (X^\bullet,g^\bullet),
\end{align*}
which is an isomorphism almost everywhere and an isometry on $X^{reg}\backslash U$ for an open proper subset $U\subset X$ containing $X^{sing}$. A morphisms of resolutions of Riemannian orbifolds
\begin{align*}
    (\widetilde{X}^\bullet,\widetilde{g}^\bullet)\rightsquigarrow (\widetilde{X}'^\bullet,\widetilde{g}'^\bullet)
\end{align*}
is a Riemannian orbifold morphism, such that 
\begin{equation*}
    \begin{tikzcd}
        (\widetilde{X}^\bullet,\widetilde{g}^\bullet)\arrow[dr]\arrow[rr,rightsquigarrow] &&(\widetilde{X}'^\bullet,\widetilde{g}'^\bullet)\arrow[dl]\\
        &(X^\bullet,g^\bullet)&
    \end{tikzcd}
\end{equation*}
We will write $\cat{Res}(X,g)$ for the category of Riemannian orbifold resolutions.
\end{defi}

We now construct the real blow-up of an orbifold, an example that plays a central role in the theory of incomplete edge calculus. This construction replaces singular strata by boundary hypersurfaces, producing a manifold with corners from a singular stratified space and providing a framework well-suited for refined analysis near singular loci.\\

Let $(X,g)$ be a Riemannian orbifold, presented by a Riemannian proper foliating Lie groupoid $(X^\bullet,g^\bullet)$ and let 
\begin{align*}
    \pi\colon (X^0,g^0)\rightarrow ([X^0/X^1],g)=(X,g)
\end{align*}
denote the orbit map. Let $\pi^{-1}X^{sing}\subset X^0$ denote the preimage of the singular strata of $X$ and let us denote by
\begin{align*}
    \beta^0_e\colon ([X^0\colon \pi^{-1}X^{sing}],g^0_e)\dashrightarrow (X^0,g^0)
\end{align*}
the real blow-up of $X^0$ along $\pi^{-1}X^{sing}$, i.e. the smooth manifold 
\begin{align*}
    X^0\backslash \pi^{-1}X^{sing}\cup_{exp_{g^0}} \mathbb{S}N\pi^{-1}X^{sing}
\end{align*}
equipped with the minimal smooth structure making the blow down map 
\begin{align*}
    \beta_e\colon [X^0\colon \pi^{-1}X^{sing}]\dashrightarrow X^0
\end{align*}
smooth. Further $g^0_e$ denotes the blow-up of $g_0$. As proven in \cite[Thm. 1.2]{posthuma2017resolutionsproperriemannianlie}, the action of the morphisms $X^1$ lifts to $([X^0\colon \pi^{-1}X^{sing}],g^0_e)$, endowing 
\begin{align*}
    ([X^\bullet\colon \pi^{-1}X^{sing}],g^\bullet_e)
\end{align*}
with the structure of a Riemannian, foliation Lie groupoid.

\begin{defi}
     We define the \textbf{real blow-up} of $(X,g)$
\begin{align*}
    \beta_e\colon ([X\colon X^{sing}],g_e)\dashrightarrow (X,g)
\end{align*}
as the orbit space of $ ([X^\bullet\colon \pi^{-1}X^{sing}],g^\bullet_e)$.
\end{defi}

\begin{lem}\cite[Thm. 1.2]{posthuma2017resolutionsproperriemannianlie}
   The real blow-up $\beta_e\colon ([X\colon X^{sing}],g_e)\dashrightarrow (X,g)$ is a smooth Riemannian manifold with edges and $\beta_e$ is a smooth submersion.
\end{lem}

\begin{defi}
        Let further $\iota_{\beta^{-1}_e(X^{sing})}\colon \beta^{-1}_e(X^{sing})\hookrightarrow [X\colon X^{sing}]$ denote the inclusion of the blow up of the singular strata and we will denote by
    \begin{align*}
        \m{res}_{\beta^{-1}_e(X^{sing})}(\omega)\coloneqq \iota_{\beta^{-1}_e(X^{sing})}^*\omega
    \end{align*}
    the \textbf{restriction map}.
\end{defi}

\subsection{Riemannian Orbifolds as Conically Fibred Singular Spaces}
\label{Riemannian Orbifolds as Conically Fibred Singular Spaces}

In this section, we introduce the notion of conically fibred (CF) spaces and conically fibred singular (CFS) spaces as a systematic way to describe the local geometry of Riemannian orbifolds and more general stratified spaces. This viewpoint refines the classical description of quasi-conical neighbourhoods near singular strata by emphasising the fibration structure of each stratum, which plays a central role in resolution constructions.\\

We relate this to the notion of iterated/incomplete edge spaces as developed by Mazzeo in \cite{mazzeo1991elliptic,mazzeo2014elliptic} and to the analytic treatment of orbifolds found in \cite{pflaum2001analytic}. Whereas Mazzeo’s framework resolves the singular strata to manifolds with corners equipped with a hierarchy of edge vector fields, our approach retains the explicit conical fibration near each stratum. This keeps the singular data visible and aligns more closely with the direct asymptotic analysis in the spirit of Lockhart–McOwen \cite{lockhart1985elliptic}. The analogy is as follows:
\begin{align*}
    \text{CS-spaces} &\leftrightsquigarrow \text{b-spaces} \\
    \text{CFS-spaces} &\leftrightsquigarrow \text{ie-spaces.}
\end{align*}
Just as Lockhart and McOwen analyse isolated conical singularities (CS) directly while Melrose \cite{melrose1992atiyah} reformulates this via $b$-spaces (manifolds with boundary), we describe stratified orbifolds as CFS-spaces instead of adopting Mazzeo’s iterated edge resolution. The following subsection makes this comparison precise and motivates the use of the CFS perspective in the analytic constructions that follows.\\

Let in the following $\vartheta\colon (Y,g_Y)\rightarrow (S,g_S)$ be a Riemannian fibration\footnote{Here we allow the Riemannian structure on the fibres to be singular.} of dimension $m-1+s=n-1$ (here $S$ is an $s$-dimensional Riemannian manifold) with compact fibres. The Riemannian structure on $Y$ induces an Ehresmann connection on the fibration, i.e. $TY\cong H_{Y}\oplus VY$.

\begin{defi}
A \textbf{conical fibration (CF)} over a link-fibration $(\vartheta\colon (Y,g_{Y})\rightarrow (S,g_S))$ is given by the (singular) warped product
\begin{align*}
    \left(\nu_0\colon (CF_S(Y),g_{\m{CF}}\right)\rightarrow (S,g_S))\coloneqq\left((\vartheta\circ\m{pr}_2\colon \mathbb{R}_{\geq 0}\times Y,\m{d}r^2+g^{2,0}_{Y}+r^2g^{0,2}_{Y})\rightarrow (S,g_S)\right)
\end{align*}
Usually, we will write $g_{\m{CF}}=\vartheta^*g_{S}+g_{\m{CF};V}$ where $g_{\m{CF};V}$ denotes the fibrewise conical structure. The connection $H_{Y}$ induces the connection $H_{\m{CF}}$ on $CF_S(Y)$ via
\begin{align*}
    (\m{id}\oplus A_{H_{Y}})\colon TCF_S(Y)\cong \underline{\mathbb{R}}\oplus TY\rightarrow \underline{\mathbb{R}}\oplus VY.
\end{align*}
A singular Riemannian manifold $(X_{\m{CFS}},g_{\m{CFS}})$ is called \textbf{conically fibred singular (CFS)} of rate $\sigma$ if there exists a compact subset $K_R$ and a diffeomorphism $\rho:X_{\m{CFS}}\backslash K_R\rightarrow [0,R)\times Y$ such that 
\begin{align*}
    g_{\m{CFS}}=g_{\m{CFS};0}+g_{\m{CFS};hot}\hspace{1.0cm}|g_{\m{CFS};hot}|_{g_{\m{CF}}}=\mathcal{O}(r^\sigma)
\end{align*}
and
\begin{align*}
    (\nabla^{g_{\m{CF}}})^k\left(\rho_*g_{\m{CFS};0}-g_{\m{CF}}\right)=&\mathcal{O}(r^{\sigma})\\
    (\nabla^{g_{\m{CF}}}_V)^k\left(\rho_*g_{\m{CFS};0}-g_{\m{CF}}\right)=&\mathcal{O}(r^{\sigma-k}).
\end{align*}
By pulling back the projection $\nu_0$ as well as the connection $H_{\m{CF}}$ we obtain a fibre bundle 
\begin{align*}
    \nu_0:X_{\m{CFS}}\backslash K_R\rightarrow S \und{1.0cm}T(X_{\m{CFS}}\backslash K_R)\cong H_{\m{CFS}}\oplus V(X_{\m{CFS}}\backslash K_R).
\end{align*}
\end{defi}

\begin{rem}
    Notice that the geometry of the conical fibration is the one of a warped product 
    \begin{align*}
        (\mathbb{R}_{\geq 0}\times Y,\m{d}r^2+g_{Y;r})
    \end{align*}
    of the real half-line and a collapsing fibration 
    \begin{align*}
        (Y,g_{Y;r}=g^{2,0}_{Y}+r^2g^{0,2}_{Y})\rightarrow (S,g_S).
    \end{align*}
\end{rem}

\begin{lem}
    Let $S\subset X$ be a singular stratum of a Riemannian orbifold $(X,g)$. Its normal cone bundle $\nu_0:(N_0,g_0)\rightarrow (S,g_S)$ is a Riemannian CF-space. A tubular neighbourhood of the singular stratum endows $(X,g)$ with the structure of a Riemannian CFS-space of rate $1$.
\end{lem}

\begin{proof}
Let $i\colon (S,g_S)\hookrightarrow (X,g)$ be a connected singular stratum of codimension $m$ whose isotropy group is $\Gamma\subset\m{O}(m)$. We define the normal bundle of $S$ by
\begin{equation*}
    \begin{tikzcd}
        TS\arrow[r,hook]&i^*TX\arrow[r, two heads]&NS.
    \end{tikzcd}
\end{equation*}
We will think of $NS$ as a vector bundle $\nu\colon NS\rightarrow S$ with a fibrewise $\Gamma$-action. We define the normal cone bundle to be the conically fibred space 
\begin{align*}
    \nu_0\colon N_0\coloneqq NS/\Gamma\rightarrow S.
\end{align*}
Let further $Fr(X/S,g)$ be the natural $\m{N}_{\m{O}(n)}(\Gamma)$-reduction of the  $\m{O}(n)$ frame bundle $Fr(X,g)$ of $(X,g)$ restricted to $S$. Define the associated (orbifold) bundle 
\begin{align*}
    Y\coloneqq Fr(X/S,g)\times_{\m{N}_{\m{O}(n)}(\Gamma)}\mathbb{S}^{m-1}/\Gamma\xrightarrow[]{\vartheta} S.
\end{align*}
Whenever $(S,g_S)$ is of type A, the bundle $Y$ is smooth.\\

Locally, the blow-down map $\beta_e\colon [X\colon X^{sing}]\dashrightarrow X$ is given 
\begin{equation*}
    \begin{tikzcd}
        \left[N_0\colon S\right]\arrow[d,equal]\arrow[r,"\beta_e"]& N_0\arrow[d,equal]\\
        \mathbb{R}_{\geq 0}\times Y\arrow[r,"\beta_e"]& \left[S/\m{Isot}(S)\right]\sqcup \mathbb{R}_{>0}\times Y
    \end{tikzcd}
\end{equation*}
and we identify $\partial[X:X^{sing}]\cong Y$. There exists an exact sequence of orbifold vector bundles on $N_0$
\begin{equation*}
    \begin{tikzcd}
        VN_0\arrow[r,hook]&TN_0\arrow[r, two heads]&\nu_0^*TS.
    \end{tikzcd}
\end{equation*}
or equivalently, a splitting of the bundle 
\begin{equation*}
    \begin{tikzcd}
        \prescript{ie}{}{}V[N_0\colon S]\arrow[r,hook]&\prescript{ie}{}{}T[N_0\colon S]\arrow[r, two heads]&\nu_0^*TS.
    \end{tikzcd}
\end{equation*}
The Riemannian orbifold structure $g$ induces a splitting 
\begin{align*}
    TN_0\cong H_0\oplus VN_0\cong\nu^*_0TS\oplus\nu^*_0NS
\end{align*}
and a Riemannian orbifold structure on the normal cone bundle $N_0$ such that
\begin{align*}
    g_0=\nu^*i^*g=\nu^*_0g_S+g_{0;V}\in  \Gamma\left(N_0,\m{Sym}^{2}H_0^*\oplus\m{Sym}^2V^*N_0\right).
\end{align*}
Equivalently, we can think of a Riemannian orbifold structure on $(N_0,g_0)$ as a real Hermitian Lie algebroid structure on $(\prescript{ie}{}{}T[N_0\colon S],g_0)$.\\

Let $\epsilon>0$ and let  
\begin{align*}
    \m{Tub}_{5\epsilon}(S)\coloneqq (Y\times (0,5\epsilon))\subset N_0
\end{align*}
and
\begin{align*}
    \m{exp}_{g}\colon \m{Tub}_{5\epsilon}(S)\rightarrow U_{5\epsilon}\subset X
\end{align*}
be a tubular neighbourhood of width $5\epsilon$, induced by the Riemannian structure.\footnote{Later we will assume that $\epsilon\sim t^\lambda$ for $0\leq\lambda<1$.} Let $r\colon N_0\rightarrow[0,\infty)$ denote the radius function on the normal cone bundle. Then there exists an asymptotic expansion
\begin{align*}
    \m{exp}_{g}^*g=g_0+g_{hot}
\end{align*}
and the higher order terms satisfy
\begin{align*}
    |g_{hot}|_{g_0}=\mathcal{O}(r).
\end{align*}
\end{proof}

\begin{rem}
    Given a singular stratum $S$ of arbitrary type, the link-fibration $Y\cong \partial[X:S]$ of its normal cone bundle is a Riemannian orbifold whose singular strata are of \say{less singular} than the type of $S$.
\end{rem}

Following \cite{mazzeo1991elliptic,melrose1991Pseudo} we define a vector bundle on $[X\colon X^{sing}]$ that realises the orbifold structure.
\begin{defi}
    We define the \textbf{incomplete edge-forms} on $[X\colon X^{sing}]$ to be the graded subalgebra 
    \begin{align*}
        \Omega^\bullet_{ie}([X\colon X^{sing}])=\{\omega\in \Omega^\bullet([X\colon X^{sing}])|\m{res}_{X^{sing}}(\omega)\in \vartheta^*\Omega^\bullet(X^{sing})\}\subset \Omega^\bullet([X\colon X^{sing}]),
    \end{align*}
    where $\vartheta:\partial [X\colon X^{sing}]\rightarrow X^{sing}$. 
    This subspace of forms on $[X\colon X^{sing}]$ realised as the exterior algebra of section of the incomplete edge-cotangent bundle 
    \begin{align*}
       \prescript{ie}{}{}T^*[X\colon X^{sing}]\rightarrow [X\colon X^{sing}].
    \end{align*}
\end{defi}

\begin{rem}
    The blow-down map $\beta_e$ induces a Lie algebroid structure on the dual to $\prescript{ie}{}{}T^*[X\colon X^{sing}]$ via the anchor map
    \begin{align*}
        \prescript{ie}{}{}\alpha\colon \prescript{ie}{}{}T[X\colon X^{sing}]\rightarrow T[X\colon X^{sing}].\qedhere
    \end{align*}
    realising the edge vector field, i.e. a section of $\prescript{ie}{}{}T[X\colon X^{sing}]$ as a vector field on $X$. 
\end{rem}

\begin{defi}
A tubular neighbourhood $\m{exp}_{g}\colon \m{Tub}_{5\epsilon}(S)\rightarrow X$ is called a flat tubular neighbourhood if $g_{hot}=0$. We will write $|g_{hot}|_{g_0}=\mathcal{O}(r^{\sigma})$ for $\sigma=1,\infty$.
\end{defi}

Given the Riemannian structure $g_0$ on $N_0$, the associated Levi-Civita connection will be denoted by $\nabla^{g_0}$. Moreover, there exists an additional metric connection on $N_0$ given by lift of the Levi-Civita of $g_S$ along the trivial connection on $VN_0$. 

\begin{lem}\cite[Sec. 2.1]{Albin2016Index}
    The connection $\nabla^g$ and $\nabla^{\oplus_0}$ define connections on $\prescript{ie}{}{T[X\colon X^{sing}]}$ and $\prescript{ie}{}{T[N_0\colon S]}$ respectively.
\end{lem}

This connection can be seen as the diagonal terms of 

\begin{align*}
    \m{res}_{Y;0}\nabla^g=\nabla^{\oplus_0}+\left(\begin{array}{cc}
         0& \m{II}^g_S \\
         -(\m{II}^g_S)^*& 0
    \end{array}\right),
\end{align*}
where $\m{II}^g_S\in\Gamma(N_0,\m{Hom}(\m{Sym}^2H_0,V N_0))$ denotes the second fundamental form of $S\subset (X,g)$.

\begin{lem}
\label{nablagexpansion}
    Let $(X,g)$ be a Riemannian orbifold. Then all the singular strata are (open) totally geodesic and 
    \begin{align*}
    \m{res}_{Y;0}\nabla^g=\nabla^{\oplus_0} \und{1.0cm}\nabla^g=\nabla^{\oplus_0}+\mathcal{O}(r).
\end{align*}
\end{lem}

\begin{proof}
    Notice, that the singular strata locally is given as the fixed-point set of an isometric action of the isotropy group. Hence, by a standard fact from Riemannian geometry, it is totally geodesic.
\end{proof}

\subsection{Smooth Gromov-Hausdorff-Resolutions}
\label{Smooth Gromov-Hausdorff-Resolutions}

In the following subsection we will define smooth Gromov-Hausdorff resolutions of a Riemannian orbifold and describe how these are constructed using geometric resolution data.

\begin{defi}
Let $(X,g)$ be a Riemannian orbifold. A \textbf{smooth Gromov-Hausdorff resolution} of $(X,g)$ is given by a smooth family of (singular) Riemannian orbifolds $(X_t,g_t)$ and morphisms
\begin{align*}
    \rho_t\colon X_t\dashrightarrow X
\end{align*}
such that $\rho_t$ restricts to an orbifold diffeomorphism $(\rho_t)^{-1}(X^{reg})\cong X^{reg}$, the exceptional set $\Upsilon_t=(\rho_t)^{-1}(X^{sing})$ is of codimension $> 0$ and 
\begin{align*}
    (X_t,g_t)\xrightarrow[GH]{t\to 0}(X,g)\und{1.0cm} (X_t\backslash \Upsilon_t,g_t)\xrightarrow[C^\infty]{t\to 0}(X^{reg},g).
\end{align*}
In particular,
\begin{align*}
    \m{vol}_{g_t}(\Upsilon_t)\xrightarrow[]{t\to 0} 0\und{1.0cm}(\rho_t)_*g_t\xrightarrow[C^\infty]{t\to 0}g.
\end{align*}
A complete smooth Gromov-Hausdorff-resolution of $(X,g)$ is given by a smooth family of Riemannian manifolds $(X_t,g_t)$ resolving $(X,g)$ at $X^{sing}$. Smooth Gromov-Hausdorff resolutions of Riemannian orbifolds form a category $\widetilde{\cat{GHRes}}(X,g)$. A smooth Gromov-Hausdorff resolution $(X_t,g_t,\rho_t)$ is a $(0,\epsilon)$-family of Riemannian orbifold resolutions. A morphism $(\rho_t,X_t,g_t)\rightsquigarrow(\rho'_t,X'_t,g'_t)$ of Riemannian orbifold resolutions is given by a $(0,\min\{\epsilon,\epsilon'\})$-family of commuting diagrams
\begin{equation*}
    \begin{tikzcd}
	{(X_t,g_t)}\arrow[rd,"{\rho_t}"', dashed] \arrow[rr,"{\phi_t}", dashed]&& {(X'_t,g'_t)}\arrow[dl,"{\rho'_t}", dashed] \\
	 & {(X,g)}
\end{tikzcd}
\end{equation*}
Two resolutions $(\rho_t,X_t,g_t)$ and $(\rho'_t,X_t,g'_t)$ are called equivalent, if the morphisms $\phi_t$ are isometries of Riemannian orbifolds. The category of smooth Gromov-Hausdorff-resolutions of Riemannian is the quotient category 
\begin{align*}
    \cat{GHRes}(X,g)=\widetilde{\cat{GHRes}}(X,g)/\sim.
\end{align*}
\end{defi}

\begin{rem}
With this definition, there is no need to repeatedly shrink the parameter range $(0, T)$ at each step. Some arguments may only apply to a smaller subfamily $(0, T')$, but this is irrelevant for our purposes. When discussing resolutions of Riemannian orbifolds, we are typically not concerned with the maximal interval of existence, but rather with the existence of a resolution at all. The above “germ-like” definition of the resolution category allows us to avoid unnecessary technicalities. Starting from a resolution $(X_t, g_t) \in \widetilde{\cat{GHRes}}(X, g)$, certain analytic steps may require restricting the family to $(0, \epsilon') \subset (0, \epsilon)$; however, in the category $\cat{GHRes}(X, g)$, these are identified as the same object, and no distinction is needed.
\end{rem}

\subsubsection{Local Model of Resolutions}
\label{Local Model of Resolutions}

In the following we will axiomatises the local model of resolutions of CFS spaces at the working example of Riemannian orbifolds. We will introduce the concept of asymptotically conical fibrations (ACF).

\begin{defi}
A Riemannian fibre bundle $\tilde{p}:(X_{\m{ACF}},g_{\m{ACF}})\rightarrow (S,g_S)$ is an \textbf{asymptotic conical fibration (ACF)} over a fibration $(q:(Y,g_{Y})\rightarrow (S,g_S))$ of rate $\gamma$, if there exists compact subset $K_R$ and a fibre bundle diffeomorphism $\rho:X_{\m{ACF}}\backslash K_R\rightarrow CF_S(E)\backslash B_R(E)$ such that
\begin{align*}
    (\nabla^{g_{\m{CF}}})^k\left(\rho_*g_{\m{ACF}}-g_{\m{CF}}\right)=&\mathcal{O}(r^{\gamma})\\
    (\nabla^{g_{\m{CF}}}_V)^k\left(\rho_*g_{\m{ACF};V}-g_{\m{CF};V}\right)=&\mathcal{O}(r^{\gamma-k}).
\end{align*}
We will further demand that the ACF metric is of the form $g_{\m{ACF}}=\tilde{p}^*g_S+g_{\m{ACF};V}$ and hence, we equip $p:(X_{\m{ACF}},g_{\m{ACF}})\rightarrow (S,g_S)$ with a connection $H_{\m{ACF}}$.
\end{defi}

Generalising the notion of these noncompact geometries leads to the notion of QACF spaces. These geometries include all local models of orbifold resolutions.

\begin{defi}
A Riemannian fibre bundle $\tilde{\nu}\colon (X,g_{QACF})\rightarrow S$ is \textbf{quasi asymptotic to a conical fibration (QACF)} over a fibration $(\overline{\nu}\colon Y\rightarrow S,g_{0;Y})$ of rate $\gamma$, if there exists subset $S_R$ and a fibre bundle diffeomorphism $\rho_{QACF}\colon X\backslash S_R\rightarrow QCF_S(Y)$ such that 
\begin{align*}
    (\nabla^{g_{QCF}})^k\left(\rho_{*_{QACF}}g-g_{QCF}\right)=\mathcal{O}(r^{\gamma}).
\end{align*}
where  
\begin{align*}
    g_{QCF}\coloneqq g^{2,0}_{Y}+\m{d}r^2+r^2(g_{1/r;Y})^{0,2}
\end{align*}
and where 
\begin{align*}
    (\widetilde{Y},g^{0,2}_{1/r;Y})\xrightarrow[GH]{r\to \infty}(Y,g^{0,2}_Y)
\end{align*}
is a family of (singular) Riemannian metrics converging in the Gromov-Hausdorff sense to the (singular) $(Y,g^{0,2}_{0,Y})$.
\end{defi}

A special class of spaces that are (Q)ACF spaces are the so called (Quasi) Asymptotically Locally Euclidean ((Q)ALE) spaces. They are formed by the subclass of the former that fibre over the point $S=\{pt.\}$ and were introduced by Joyce in \cite[Chap. 8 and 9]{joyce2000compact}. Moreover, given a $G$-structure on these noncompact spaces, we will specify their asymptotic behaviour compared to a flat model structure.\\ 

In \cite{majewski2025spin7orbifoldresolutions} the author constructs ACF spaces with $G$-structures that resolve the model cone $G$-structures by constructing $S$-families of such (Q)ALE spaces.\\

Let in the following $(V,\Phi_V)$ be a $m$-dimensional Euclidean space and $\Gamma\subset G\subset \m{SO}(V,g_V)$ a finite subgroup, with a flat, translational invariant $G$-structure $\Phi_V\in \Gamma(V,Str_G(V))^V$.

\begin{defi}
Assume that $\Gamma$ acts freely on $V\backslash \{0\}$. An orbifold $(M,\Phi)$ with a torsion-free $G$-structure $\Phi\in\Gamma(M,Str_G(M))$ is called an \textbf{asymptotically locally Euclidean (ALE)} $G$-structure of rate $\gamma$, if there exists a compact subset $B_R\subset M$ and a diffeomorphisms 
\begin{align*}
    \rho:M\backslash B_R\rightarrow (V\backslash B_R(0))/\Gamma
\end{align*}
and the $G$-structure satisfies
\begin{align*}
    (\nabla^{g_{V}})^k\left(\rho_*\Phi-\Phi_{V}\right)=\mathcal{O}(r^{\gamma-k}),
\end{align*}
where $r:V/\Gamma\rightarrow [0,\infty)$ denotes the radius function.
\end{defi}

\begin{rem}
By a classical result of Bando, Kasue and Nakajima \cite{bando1989construction} Euclidean volume growth and $L^2$-integrable curvature implies that a non-compact, Ricci-flat manifold $(M,g)$ is ALE of rate $\gamma=1-m$. If further $m=4$ or $g$ Kähler, then $\gamma=-m$.
\end{rem}

Let $\Gamma\subset G\subset \m{SO}(V,g_V)$ be a finite subgroup acting on $V$ in a possible non-free way. Previously we assumed that the stratum of points with nontrivial stabiliser was only the origin; now we allow nontrivial linear subspaces to be stabilised by nontrivial subgroups of $\Gamma$.\\

We define 
\begin{align*}
    Fix(\Gamma_i)=&\{v\in V|\m{Stab}(v)=\Gamma_i\subset\Gamma\}\\
    C(T)=&\{g\in\Gamma|g.t=t,\forall t\in T\}\\
    N(T)=&\{g\in\Gamma|g.T=T\}
\end{align*}
 and denote by $\mathcal{L}=\{T_i=Fix(\Gamma_i)|\Gamma_i\subset\Gamma\}$ the set of all fixpoint vector spaces and let $I$ be an indexing set for $\mathcal{L}$. Set $T_0=Fix(1)=V$ and $T_\infty=Fix(\Gamma)=0$ and define a partial ordering $i\succeq j$ if $T_i\subset T_j$. If $V=T_i\oplus T_i$, then $T_i=Fix(C(T_i))$ and $C(T_i)$ act on $V/C(T_i)\cong T_i\times S_i/C(V_i)$. In particular, $W/N(T_i)=(T_i\times S_i/C(T_i))/(N(T_i)/C(T_i))$.\\

Let $T_i,T_j\in\mathcal{L}$. Then 
\begin{align*}
    C(T_i)C(T_j)=C(T_i\cap T_j)
\end{align*}
and consequently, the set $\mathcal{L}$ is closed under intersections. Moreover, the group $\Gamma$ acts on the indexing set $I$, by $T_{g.i}=g.T_i$.\\

We can identify the singular set of $V/\Gamma$ with the set 
\begin{align*}
    Sing(V/\Gamma)=\bigcup_{i\in I\backslash\{0\}}T_i/\Gamma.
\end{align*}
For a generic point $[v]\in Sing(V/\Gamma)$ the singularity of $V/\Gamma$ at $[v]$ is modelled on $T_i\times S_i/C(T_i)$.\\

Let $Sing(V/\Gamma)$ be the singular set of $V/\Gamma$. The singular set is given by the union 
\begin{align*}
    Sing(V/\Gamma)=\bigcup_{i\in I}T_i/\Gamma.
\end{align*}
Define the distance functions for $0\neq i\in I$
\begin{align*}
    d_i:V/\Gamma\rightarrow[0,\infty),\hspace{1.0cm}v\mapsto \m{dist}_{g_V}(v,T_i).
\end{align*}
Notice that $d_\infty=r$ is the standard radius function on $V/\Gamma$ and that $d_j\leq d_i$ for $i\succeq i$.

\begin{defi}
A noncompact, connected $m$-dimensional, orbifold $(M,\Phi)$ with a torsion-free $G$-structure $\Phi\in\Gamma(M,Str_G(M))$ is called \textbf{quasi asymptotically locally Euclidean (QALE) $G$-structure} of rate $\gamma$, if there exists a asymptotically cylindrical suborbifold\footnote{Here asymptotic cylindrical is meant in the topological sense.} $S_R$ and a diffeomorphism 
\begin{align*}
    \rho:M\backslash S_R\rightarrow (V/\Gamma)\backslash \m{Tub}_R(Sing(V/\Gamma))
\end{align*}
such that 
\begin{align*}
  (\nabla^{g_{V}})^k(\rho_*\Phi-\Phi_{V})=\sum_{0\neq i\in I}\mathcal{O}(d_i^{\gamma-k})
\end{align*}
for all $k\geq0$.
\end{defi}

\begin{defi}
Let $(M,\Phi)$ be a QALE $G$-structure. A set of QALE distance functions $\{\delta_i\}_{0\neq i\in I}$ on $M$ are functions $\delta_i:M\rightarrow [1,\infty)$ satisfying 
\begin{align*}
        (\nabla^{g_{V}})^k(\rho_*\delta_i-d_i)=\mathcal{O}(d_i^{1+\gamma-k})
\end{align*}
for all $k\geq 0$.
\end{defi}

\subsubsection{Interpolating and $k$-Tame smooth Gromov-Hausdorff Resolutions}
\label{Interpolating and k-Tame smooth Gromov-Hausdorff Resolutions}

In this section we develop a systematic framework for analysing smooth Gromov-Hausdorff resolutions of Riemannian orbifolds. The guiding principle is that such resolutions should be modelled, at small scales near the singular locus, by asymptotically conical fibrations (ACF) glued into the normal cone of the singular stratum. This motivates the notion of interpolating resolutions and serves as reference models for more general resolutions. To quantify how closely an arbitrary smooth Gromov-Hausdorff resolution approximates a interpolating one, we introduce the notion of $k$-tameness, measured by $C^k_{loc}$-convergence of the associated families of metrics under suitable identifications. This formalism provides a convenient language for comparing analytic data across families, and will be the main tool for studying geometric operators, such as Dirac operators, in later sections.\\

The group $(\mathbb{R}_{>0},\cdot)$ acts on the normal cone bundle $\nu_0\colon N_0\rightarrow S$ via the dilatation of the fibres 
    \begin{align*}
        \delta_t\colon N_0\rightarrow N_0,\,[n_s]\mapsto [t\cdot n_s].
    \end{align*}

\begin{nota}

    Given a tensor $\eta$ on $N_0$, we will denote by 
    \begin{align*}
        \eta^t\coloneqq \delta^*_t\eta 
    \end{align*}
    the $\mathbb{R}_{>0}$-family induced by the action.
\end{nota} 

In particular, the CF-normal cone structure satisfies 
\begin{align*}
    g^t_0=g_{0;H}+t^2\cdot g_{0;V}.
\end{align*}

Let us assume that there exists a family of ACF spaces of rate $\gamma$, forming a smooth Gromov-Hausdorff resolution 
\begin{align*}
    (N_{t^2\cdot\zeta},g_{t^2\cdot \zeta})\in\cat{GHRes}(N_0,g_0),
\end{align*}
such that
\begin{equation*}
\begin{tikzcd}
        (N_\zeta,g^t_\zeta)\arrow[d,dashed,"\rho_\zeta",swap]\arrow[r,"\delta_t"]&(N_{t^2\zeta},g_{t^2\zeta})\arrow[d,dashed,"\rho_{t^2\cdot\zeta}"]\\
        (N_0,g^t_0)\arrow[r,"\delta_t"]&(N_0,g_0)
\end{tikzcd}
\end{equation*}
and the Riemannian structure decomposes into 
\begin{align*}
    \delta_t^*g_{t^2\cdot \zeta}=g_{\zeta;H}+t^2\cdot g_{\zeta;V}.
\end{align*}

Using the splitting $H_\zeta\oplus VN_\zeta$ we define the connection $ \nabla^{\oplus_\zeta}$ and its torsion tensor $T_{\oplus_\zeta}$. This connection is $t$-independent and we identify
\begin{align*}
    \nabla^{g^t_\zeta}=\nabla^{\oplus_\zeta}+\frac{1}{2}\overline{\Lambda^t_\zeta}(T_{\oplus_\zeta}).
\end{align*}
Following \cite[(10.3)]{berline1992heat} the limit 
\begin{align*}
    \nabla^{g^0_\zeta}=\nabla^{\oplus_\zeta}+\frac{1}{2}\overline{\Lambda}^0(T_{\oplus_\zeta})
\end{align*}
is well-defined. 

\begin{defi}
We define the \textbf{difference torsion} 
    \begin{align*}
        \Omega_\zeta=T_{\oplus_\zeta}-\rho_\zeta^*T_{\oplus_0}.
    \end{align*}
and an adapted metric connection 
    \begin{align}
        \label{hatnablaoplustzeta}
    \widehat{\nabla}^{g^t_\zeta}\coloneqq&\nabla^{\otimes_\zeta}+\frac{1}{2}\overline{\Lambda}^t_\zeta(\Omega_\zeta).
    \end{align}
\end{defi}

\begin{rem}
    As $g_{t^2\cdot\zeta}$ is ACF with respect to $g_0$,
\begin{align*}
    \nabla^k\Omega_{t^2\cdot\zeta}=\mathcal{O}(r^{\gamma})&\und{1.0cm}(\nabla^{0,1})^k\Omega_{t^2\cdot\zeta}=\mathcal{O}(r^{\gamma-k}).
\end{align*}
\end{rem}

\begin{ex}
\label{NzetaasinResolutionsofSpin7Orbifolds}
In \cite{majewski2025spin7orbifoldresolutions} such families of ACF spaces are constructed to resolve Ricci-flat-orbifolds (e.g. $\m{Spin}(7)$-orbifolds). We will briefly sketch this construction and refer to \cite{majewski2025spin7orbifoldresolutions} for an in-depth discussion.\\

Let $\kappa\colon \mathbb{M}\rightarrow \Theta_{\m{Im}(\mathbb{K})}$ be the universal family of Calabi-Yau ALE spaces  $\rho_\zeta\colon (M_\zeta,g_{\zeta;V})\dashrightarrow(\mathbb{C}^{m/2}/\Gamma,g_{st})$ for $\Gamma\subset\m{SU}(m/2)$. The normalizer subgroup $\m{N}_{\m{Spin}(7)}(\Gamma)$ of the isotropy group $\Gamma$ acts on this universal family in an equivariant way. Then there exists a twisted vector bundle $\mathfrak{P}\coloneqq Fr(X/S,g)\times_{\m{N}_{\m{Spin}(7)}(\Gamma)}\Theta_{\m{Im}(\mathbb{K})}\rightarrow S$  and a moduli bundle $ \mathfrak{M}\coloneqq Fr(X/S,g)\times_{\m{N}_{\m{Spin}(7)}(\Gamma)}\mathbb{M}\rightarrow \mathfrak{P}$ such that $(N_\zeta,g_\zeta)$ can be constructed via a section 
\begin{align*}
    \zeta\in\Gamma(S,\mathfrak{P})\hspace{0.5cm}\text{as}\hspace{0.5cm}(N_\zeta,g_\zeta)\coloneqq(\zeta^*\mathfrak{M},\zeta^*g_\mathfrak{M}).
\end{align*}
The lift of the dilation action $\delta:(N_\zeta,g^t_\zeta)\rightarrow (N_{t^2\cdot\zeta},g_{t^2\cdot\zeta})$ is the restricition of a universal scaling action
\begin{align*}
    \delta_t\colon\mathfrak{M}\rightarrow \mathfrak{M}
\end{align*}
covering the map $\delta_{t^2}^{\mathfrak{P}}\colon\mathfrak{P}\rightarrow \mathfrak{P}$.
\end{ex}

\begin{rem}
    By abuse of notation we will denote the lift of the dilatation action by $\delta_t$.
\end{rem}

\begin{nota}
    Let $r\colon N_\zeta\rightarrow[0,\infty)$ be the radius function on $N_\zeta$ defined by $r=r\circ \rho_\zeta$. We will denote by
    \begin{align*}
        B_{R}(N_\zeta)\coloneqq r^{-1}(0,R)\und{1.0cm}A_{(R_1,R_2)}(S)\coloneqq\m{Tub}_{R_2}(S)\backslash \m{Tub}_{R_1}(S).
    \end{align*}
\end{nota}
\begin{defi}
We define the \textbf{gluing morphism} $\Gamma_{t^2\cdot\zeta}\coloneqq \m{exp}_g\circ \rho_{t^2\cdot\zeta}$ and denote the family of spaces
\begin{align*}
    X^{pre}_t\coloneqq(B_{5\epsilon}(N_{t^2\cdot\zeta})\sqcup X\backslash S)/_{\sim_ {\Gamma_{t^2\cdot\zeta}}}
\end{align*}
Furthermore, these spaces are equipped with a resolution morphism 
\begin{align*}
    \rho^{pre}_t\colon X^{pre}_t\dashrightarrow X.
\end{align*}
\end{defi}

\begin{rem}
The dilation map
\begin{align*}
    \delta_t \colon (N_\zeta, g^t_\zeta) \rightarrow (N_{t^2 \cdot \zeta}, g_{t^2 \cdot \zeta})
\end{align*}
identifies the resolution data of the normal cone bundle $(N_0, g_0)$ with the fixed fibration $\nu_\zeta \colon N_\zeta \to S$, equipped with a family of metrics $g^t_\zeta$. This allows us to define a family of spaces by gluing:
\begin{align*}
    \rho^{pre;t} \colon X^{pre;t} \coloneqq \left( B_{5t^{-1}\epsilon}(N_\zeta) \sqcup X \setminus S \right) / \sim_{\Gamma^t_\zeta} \dashrightarrow X,
\end{align*}
where the gluing map is $\Gamma^t_\zeta = \delta_t^* \Gamma_{t^2 \cdot \zeta}$. The spaces $X^t$ and $X_t$ are canonically diffeomorphic via the rescaling:
\begin{align*}
    \delta_t \colon X^{pre;t} \cong X^{pre}_t \colon \delta_{t^{-1}}.
\end{align*}

We emphasize that $X^{pre;t}$ and $X^{pre}_t$ describe the same space, but the notation reflects different perspectives: we use the superscript $X^{pre;t}$ to indicate the viewpoint relative to $N_\zeta$, and the subscript $X^{pre}_t$ when viewing the space relative to $X$. This distinction is particularly useful since the gluing map $\Gamma^t_\zeta$ incorporates the rescaling map $\delta_t$, and it is often convenient to switch between the two pictures depending on the geometric or analytic context.

\end{rem}

Let ${\chi}\colon [0,1]\rightarrow[0,1]$ be a smooth positive, function such that $\m{supp}({\chi})\subset(0,1]$ and $\m{supp}(1-{\chi})\subset[0,1)$ fixed throughout this work. Define the smooth functions
\begin{align*}
    {\chi_i}(x)\coloneqq&\left\{\begin{array}{cl}
     0&x\in B_{(i-1)\epsilon}(N_{t^2\cdot\zeta})\\
     {\chi}((i-1)\epsilon^{-1}r-i+1)&x\in B_{i\epsilon}(N_{t^2\cdot\zeta})\backslash B_{(i-1)\epsilon}(N_{t^2\cdot\zeta}) \\
     1& \text{else} 
    \end{array}\right.
\end{align*}
and set $\delta^*_t\chi_i=\chi^t_i$.
\begin{nota}
Given two sections $\Psi$ and $\Phi$ of mutually vector bundles on $N_\zeta$ and $X$ that \say{glue} together by a lift $\hat{\Gamma}_{t^2\cdot\zeta}$ of $ \Gamma_{t^2\cdot\zeta}$, we define a new section on $X_t$ by
\begin{align*}
    \Psi\cup^t\Phi\coloneqq&(1-\chi^t_3)\Psi+\chi^t_3(\hat{\Gamma}^t_\zeta)^*( \Gamma^t_\zeta)_*\Phi\\
    =&(1-\chi_3)(\hat{\Gamma}^{t^{-1}}_\zeta)^*( \Gamma^{t^{-1}}_\zeta)_*\Psi+\chi_3\Phi\eqqcolon (\hat{\Gamma}^{t^{-1}}_\zeta)^*( \Gamma^{t^{-1}}_\zeta)_*\Psi\cup_t\Phi
\end{align*}
Instead of $(\hat{\Gamma}^t_\zeta)^*( \Gamma^t_\zeta)_*\Phi$ we will abuse the notation and write $\delta_t^*\Phi$.
\end{nota}

\begin{defi}
    We interpolate between the Riemannian orbifold structure on $X$ and the ACF structure on $B_{5t^{-1}\epsilon}(N_\zeta)$ using the cut-off function ${\chi^t_3}$ in a less naive way by
\begin{align*}
    g^{pre;t}\coloneqq g^t_{\zeta}\cup^t g_0 +\delta^*_t g_{hot}=g_{t^2\cdot\zeta}\cup_t g_0 +g_{hot}\eqqcolon g^{pre}_t.
\end{align*}
\end{defi}

Using this adapted connection can be seen as a leading order term in the expansion of the Levi-Civita connection $\nabla^{g_t}$. The following lemma makes this statement more precise.

\begin{lem}
    The Levi-Civita connection with respect to $g^{pre}_t$ satisfies 
    \begin{align*}
        g^{pre}_t(\nabla^{g^{pre}_t}_XY,Z)=g_{t^2\cdot\zeta}(\widehat{\nabla}^{g_{t^2\cdot\zeta}}_XY,Z)\cup_tg_0(\nabla^{\oplus}_XY,Z)+\frac{1}{2}\m{d}\chi_3\wedge((\rho_{t^2\cdot\zeta})_*g_{t^2\cdot\zeta}-g_0)(X,Y,Z)+\mathcal{O}(r).
    \end{align*}
\end{lem}

\begin{proof}
    Let $X,Y,Z\in\mathfrak{X}(X_t)$. Then using Koszul's formula
    \begin{align*}
        2g^{pre}_t(\nabla^{g_t}_XY,Z)=&Xg^{pre}_t(Y,Z)+Yg^{pre}_t(X,Z)-Zg^{pre}_t(X,Y)\\
        &+g^{pre}_t([X,Y],Z)-g^{pre}_t([Y,Z],X)+g^{pre}_t([Z,X],Y),
    \end{align*}
    $\nabla^g=\nabla^{\oplus}+\mathcal{O}(r)$ we can deduce the first statement.
\end{proof}

\begin{cor}
    The family 
    \begin{align*}
        \rho^{pre}_t\colon(X^{pre}_t,g^{pre}_t)\dashrightarrow (X,g)
    \end{align*}
    defines a smooth Gromov-Hausdorff resolution. We will refer to such Gromov-Hausdorff resolutions as being \textbf{interpolating}. 
\end{cor}

In the following paper, we will analyse smooth Gromov-Hausdorff resolutions that are constructed through the described construction. In many cases, smooth Gromov-Hausdorff resolutions of interest are \say{close} to such resolutions constructed form this limiting data. We will informally define these $k$-tame resolutions and specify the precise definition ones we established the uniform elliptic theory for Dirac operators on interpolating smooth Gromov-Hausdorff resolutions.

\begin{defi}[informal]
\label{ktame}
    A smooth Gromov-Hausdorff resolution 
    \begin{align*}
        [\rho_t\colon(X_t,g_t)\dashrightarrow (X,g)]\in\cat{GHRes}(X,g)
    \end{align*}
    is called \textbf{$k$-tame}, if there exists a interpolating smooth Gromov-Hausdorff resolution 
    \begin{align*}
        [\rho^{pre}_t\colon(X^{pre}_t,g^{pre}_t)\dashrightarrow (X,g)]\in\cat{GHRes}(X,g)
    \end{align*}
    and a family of diffeomorphisms 
    \begin{align*}
        f_t\colon X_t\cong X^{pre}_{t}
    \end{align*}
    such that 
    \begin{align*}
       g^{pre}_{\zeta;t}\xrightarrow[C^{k}_{-1;t}]{}(f_t)_*g_t
    \end{align*}
    where $C^{k}_{-1;t}$ is a weighted $C^{k}$-norm defined in Definition \ref{ktameDiracbundle}.    
\end{defi}

In order to study the analytic properties of families of geometric operators, such as Dirac operators, on $k$-tame smooth Gromov-Hausdorff resolution $[\rho_t\colon(X_t,g_{t})\dashrightarrow (X,g)]\in \cat{GHRes}(X,g)$ it suffices to study them on the interpolating smooth Gromov-Hausdorff resolutions.\\

To avoid cumbersome booking work and improve the readability, we will restrict ourselves through out this paper to orbifolds with singular strata of depth one and type A. However, the results presented in this paper extend $k$-tame orbifold resolutions of general strata. For the rest of this paper we will assume following.

\begin{manualassumption}{1}
\label{ass1}    
We will assume the following 
\begin{itemize}
    \item The Riemannian orbifold $(X,g)$ contains one singular stratum of depth one and type A.
    \item 
The orbifold resolution 
\begin{align*}
[\rho_t\colon(X_t,g_{t})\dashrightarrow (X,g)]=[\rho^{pre}_t\colon(X^{pre}_t,g^{pre}_{t})\dashrightarrow (X,g)].
\end{align*}
is interpolating.
\end{itemize}
\end{manualassumption}

\section{Hermitian Dirac Bundles on Orbifolds and their Resolutions}
\label{Hermitian Dirac Bundles on Orbifolds and their Resolutions}

In this section, we study Dirac operators associated to Hermitian Dirac bundles on Riemannian orbifolds. Near each singular stratum, we expand the Dirac structure in terms of a model operator that captures the local geometry. This leading-order term is a Dirac operator on an associated conically fibred (CF) space.\\ 

Building on this expansion, we then introduce a resolution procedure for these singular Dirac bundles and construct smooth families of Dirac bundles on the resolved orbifold. This lays out the geometric setting for the detailed analytic constructions that follows.\\

Throughout, this paper we emphasise a direct approach in the spirit of CFS-analysis. Instead of relying fully on the heavy machinery of the iterated edge calculus, we exploit the explicit fibration structure and asymptotic expansions of the singular models. One might view this as avoiding the use of heavy artillery, i.e. the full edge pseudodifferential framework, when a careful singular expansion suffices for the problems at hand.

\subsection{Dirac Bundles on Orbifold as Conically Fibred Singular Spaces}
\label{Dirac Bundles on Orbifold as Conically Fibred Singular Spaces}

As we have already established in Section \ref{Smooth Gromov-Hausdorff-Resolutions of Riemannian Orbifolds}, every Riemannian orbifold can be viewed as an incomplete edge space. Consequently, Dirac operators on orbifolds belong to the class  
\begin{align*}
    D \in \m{Ell}^1_{ie}(E)
\end{align*}
of first-order elliptic (incomplete) edge operators, as developed in \cite{mazzeo1991elliptic,mazzeo2014elliptic,Albin2016Index,albin2023index}.\\ 

While we will adopt our own terminology based on conically fibred (CF) and conically fibred singular (CFS) spaces throughout this work, the reader should keep in mind that this is essentially a translation of the standard language of the edge calculus. In particular, many of the geometric and analytic notions have precise counterparts in the CF framework we develop below.\\

To analyse the behaviour near singular strata, we expand the Dirac operator $D$ in terms of a model operator $\widehat{D}_0$, which is a Dirac operator on the associated conical fibration (CF). This model operator will be studied in detail later and its mapping properties form the backbone of the elliptic theory we develop in subsequent sections. In the language of the edge calculus, $\widehat{D}_0$ corresponds to the \textbf{normal operator}
\begin{align*}
    \widehat{D}_0 \leftrightsquigarrow N(D).
\end{align*}
Moreover, $\widehat{D}_0$ naturally decomposes into a horizontal and a vertical part. The vertical component, denoted $\widehat{D}_{0;V}$, corresponds to the \textbf{indicial operator}
\begin{align*}
    \widehat{D}_{0;V} \leftrightsquigarrow I(D).
\end{align*}
This vertical operator controls the indicial roots, that is, the characteristic rates at which the Fredholm properties of the Dirac operator may fail.\\

To make these concepts precise near a singular stratum, we now describe how a Dirac bundle restricts to the local conical fibration. Let in the following $S\subset X^{sing}$ be a compact singular stratum of codimension $m$ and type A and let 
\begin{align*}
    \pi\colon(E,\m{cl}_g,h,\nabla^h)\rightarrow (X,g)
\end{align*}
be a Hermitian Dirac bundle. Locally near $S$, the orbifold can be described using its CF tubular neighbourhood. We lift the local geometry and the Dirac bundle accordingly.\\

Restricting the Hermitian Clifford module $(E,\m{cl}_{g},h)$ to the singular stratum $S$, naturally induces a CF-Hermitian Clifford module 
\begin{align*}
    \pi_0\colon (\widehat{E}_0,\m{cl}_{g_0},h_0)\rightarrow (N_0,g_0)
\end{align*}
on the normal cone bundle. Let $\epsilon>0$ and let 
\begin{equation*}
    \begin{tikzcd}
        \widehat{E}_{0}|_{\m{Tub}_{5\epsilon}(S)}\arrow[d,"\pi_{\widehat{E}_{0}}"]\arrow[r, hook,"\widehat{\m{exp}}_{g}"]&E\arrow[d,"\pi_E"]\\
        \m{Tub}_{5\epsilon}(S)\arrow[r,hook,"j"]&X
    \end{tikzcd}
\end{equation*}
denote a lift of the tubular neighbourhood. Notice that
\begin{align*}
\widehat{\m{exp}}_{g}^*h=h_0+h_{hot}\und{1.0cm}|h_{hot}|_{h_0}=\mathcal{O}(r).
\end{align*}
Locally we are able to identify
\begin{align*}
    \widehat{E}_{0}=_{loc} SN_0\otimes F_0
\end{align*}
where $F_0\rightarrow N_0\cong\mathbb{R}_{\geq 0}\times Y$ is a translational invariant Hermitian vector bundle. Using the connection $\nabla^{\oplus_0}$ we further define the Hermitian connection 
\begin{align*}
    \nabla^{\otimes_0}=_{loc}\overline{\nabla}^{\oplus_0}\otimes 1+1\otimes \m{res}_{Y;0}\nabla^{tw.}.
\end{align*}
This connection is a Clifford module connection with respect to $\nabla^{\oplus_0}$ and hence, 
\begin{align*}
    \pi_{0}\colon (\widehat{E}_{0},\m{cl}_{g_0},h_0,\nabla^{\otimes_0})\rightarrow (N_0,g_0)
\end{align*}
is a CF-Hermitian Dirac bundle.\\

This means that there exists a family of Hermitian Dirac bundle 
\begin{align*}
    (\widehat{E}_Y,\m{cl}_{g_{Y;r}},h_{Y;r},\nabla^{\otimes_Y})\rightarrow (Y,g_{Y;r})
\end{align*}
over the collapsing family of compact Riemannian fibration $\vartheta\colon (Y,g_{Y;r})\rightarrow (S,g_S)$, and a skew-Hermitian complex structure 
\begin{align*}
    \m{cl}_0:\widehat{E}_Y\rightarrow \widehat{E}_Y
\end{align*}
for which $\nabla^{\otimes_Y}$ is holomorphic, i.e. 
\begin{align*}
    [\m{cl}_0,\nabla^{\otimes_Y}]=0
\end{align*}
such that $\m{cl}_0$ can be identified with $\m{cl}_{g_0}(\m{d}r)$ and $\nabla^{\otimes_0}$ with $\partial_r+\frac{m-1}{2}+\nabla^{\otimes_Y}$.

\begin{rem}
    The connection $\nabla^{\otimes_Y}$ is dilation invariant.
\end{rem}

Using the fibred structure of the family $(\widehat{E}_Y,\m{cl}_{g_{Y;r}},h_{Y;r},\nabla^{\otimes_Y})$ the Dirac operator $\widehat{D}_{Y;r}$ decomposes in the following way.

\begin{lem}\cite[Eq. (4.26)]{bismut1989eta}
\label{DYdecomposition}
The Dirac operator of the Dirac bundle 
\begin{align*}
    (\widehat{E}_Y,\m{cl}_{g_{Y;r}}, h_{Y;r}, \nabla^{\otimes_Y})
\end{align*}
decomposes into 
\begin{align*}
    \widehat{D}_{Y;r}=\widehat{D}_{Y;H;r}+\widehat{D}_{Y;V;r}=\widehat{D}_{Y;H}+r^{-1}\cdot \widehat{D}_{Y;V}
\end{align*}
where $\widehat{D}_{Y;V}$ is a vertical $r$-independent Dirac operator. The operator $\widehat{D}_{Y;H;r}$ will be referred to as the horizontal Dirac operator. We can identify the horizontal Dirac operator with 
\begin{align*}
    \widehat{D}_{Y;H}=&\m{cl}_{g_{Y}}\circ\m{pr}_{H_Y}\circ\nabla^{\otimes_Y}.
\end{align*}
\end{lem}

As $(Y,g_{Y;r})$ is an associated bundle and the Riemannian structure is induced by the connection $H_Y$ on $Fr_{\m{SO}(W)}(X/S)$, and the standard Riemannian structure on $\mathbb{S}^{m-1}/\Gamma$, the fibres of $\vartheta\colon (Y,g_{Y;r})\rightarrow (S,g_S)$ are totally geodesic. Consequently, we are able to deduce the following crucial observation.

\begin{cor}
\label{anticommutatorY}
The anti commutator 
    \begin{align*}
        \{\widehat{D}_{Y;H},\widehat{D}_{Y;V}\}=0
    \end{align*}
    vanishes.
\end{cor}

\begin{proof}
The fibres of $\vartheta\colon (Y,g_{Y;r})\rightarrow (S,g_S)$ are totally geodesic. By the construction of $\nabla^{\otimes_Y}$ and Assumption \ref{ass2}, we conclude that
    \begin{align*}
        \{\widehat{D}_{Y;H},\widehat{D}_{Y;V}\}=&\m{cl}_{g_Y;H}\circ \m{cl}_{g_Y;V}\circ [\nabla^{\otimes_Y}_H,\nabla^{\otimes_Y}_V]\\
        =&0.
    \end{align*}
\end{proof}

Moreover, we are able to extend both the decomposition, as well as the observation to the Dirac operator on the CF-Hermitian Dirac bundle.
\begin{cor}
\label{anticommutatorN0}
    The Dirac operator $\widehat{D}_0$ associated to $(\widehat{E}_{0},\m{cl}_{g_0},h_0,\nabla^{\otimes_0})$ decomposes into   
\begin{align*}
    \widehat{D}_0=&\widehat{D}_{0;H}+\widehat{D}_{0;V}
\end{align*}
where
\begin{align*}
    \widehat{D}_{0;H}=\widehat{D}_{Y;H}\und{1.0cm}\widehat{D}_{0;V}=\m{cl}_0\circ \partial_r+r^{-1}\cdot \left(\widehat{D}_{Y;V}+\frac{m-1}{2}\m{cl}_0\right).
\end{align*}
Further, the anticommutator 
\begin{align*}
    \{\widehat{D}_{0;H},\widehat{D}_{0;V}\}=0
\end{align*}
vanishes.
\end{cor}

\begin{rem}
    The Dirac operators $\widehat{D}_0$ and $\widehat{D}_{Y;r}$ are self-adjoint.
\end{rem}

\begin{ex}
If the Dirac bundle $(\widehat{E}_{0},\m{cl}_{g_0}, h_0, \nabla^{\otimes_0})$ is given by 
\begin{align*}
    (\wedge^\bullet T^* N_0,\epsilon^{g_0}-\iota,g_0,\nabla^{\oplus_0})
\end{align*}
the corresponding Dirac operator is given by 
    \begin{align*}
        \widehat{D}_0 =\m{d}^{1,0}\oplus\m{d}^{-1,0}+\m{d}^{0,1}\oplus\m{d}^{0,-1}.
    \end{align*}
\end{ex}

For technical reasons, we will need to define and assume the following.

\begin{defi}
    The connection $\nabla^h$ is called \textbf{admissible}, if the twisting curvature of $\nabla^{\otimes_0}$ is of type $(2,0)$ with respect $H_0$ and 
    \begin{align*}
        \nabla^h=&\nabla^{\otimes_0}+\mathcal{O}(r).
    \end{align*}
\end{defi}

\begin{rem}
    If $\nabla^{h}$ is admissible, then $F_{\nabla^{\otimes_0}}=F_{\nabla^{\otimes_0}}^{2,0}$.
\end{rem}

\begin{manualassumption}{2}
\label{ass2}
     We will assume that $\nabla^h$ is an admissible Dirac module connection.
\end{manualassumption}

As previously noted, the model Dirac operator $\widehat{D}_0$ should be understood as the leading-order term in an asymptotic expansion of the full Dirac operator $D$. The following lemma makes this statement precise.

\begin{lem}
    The Dirac operator $D$ expands as 
    \begin{align*}
        D=&\widehat{D}_0+\left(\m{cl}_{g}-\m{cl}_{g_0}\right)\circ\nabla^h+\mathcal{O}(r).
    \end{align*}
\end{lem}

\subsection{Resolutions of Dirac Bundles}
\label{Resolutions of Dirac Bundles}
In this section, we construct a family of resolutions of a CF-Hermitian Dirac bundle and its associated Dirac operator. These resolutions provide the geometric framework necessary for constructing smooth Gromov–Hausdorff resolutions of the underlying Hermitian Dirac orbifold bundles.\\

A central point is that the Dirac operator $\widehat{D}_{t^2 \cdot \zeta}$ associated to these ACF-Hermitian Dirac bundles should be regarded as the leading-order model in an asymptotic expansion for the family of Dirac operators $D_t$ on the resolved space $(X_t,g_t)$.\\

The family of ACF spaces converges in the  smooth Gromov–Hausdorff sense to the normal cone bundle. To control this limit, we introduce an adiabatic rescaling and study the behaviour of the rescaled operator $\widehat{D}^t_\zeta$ as $t \to 0$.\\

The resolution construction is organized around the natural dilation action of $(\mathbb{R}_{>0}, \cdot)$ on the normal cone bundle, which preserves the horizontal distribution and scales the vertical geometry in a controlled way. This action lifts to the Dirac bundle, producing a rescaled family of Hermitian Dirac bundles whose module structures and Clifford actions vary with the adiabatic parameter.\\

The following material develops this structure explicitly and provides the precise assumptions on the resolution data.\\

The group $(\mathbb{R}_{>0},\cdot)$ acts on the normal cone bundle $N_0$ via dilation
\begin{align*}
    \delta_t\colon N_0\rightarrow N_0,[n_s]\mapsto[t\cdot n_s].
\end{align*}

Notice that the Ehresmann connection $H_0$ is invariant under this action and the Riemannian structure pulls back to 
\begin{align*}
     \delta_t^*g_0=g^t_0=\nu^*_0g_S+t^{2}g_{0;V}.
\end{align*}

A lift of the dilation action to the vector bundle $\widehat{E}_{0}$
\begin{equation*}
    \begin{tikzcd}
        \widehat{E}_{0}\arrow[r,"\hat{\delta}_t"]\arrow[d,"\pi_0",swap]&\widehat{E}_{0}\arrow[d,"\pi_0"]\\
        N_0\arrow[r,"\delta_t"]&N_0
    \end{tikzcd}
\end{equation*}
induces a family of Hermitian Dirac bundles
\begin{align*}
    (\widehat{E}_{0},\m{cl}_{g^t_0},h^t_0,\nabla^{\otimes^t_0}).
\end{align*}
In particular, the module structure is given by
\begin{align}
\label{clt(vt)=cl(v)}
    \m{cl}_{g^t_0}(V)=\hat{\delta}_t\circ \m{cl}_{g_0}((\delta_t)_* V)\circ \hat{\delta}_{t^{-1}}.
\end{align}
and 
\begin{align*}
    \pi_0\colon (\widehat{E}_{0}, h^t_0)\cong\bigoplus_i(E^i,h^i_{0;H}+\hat{\delta}_t^*h^i_{0;V})\rightarrow N_0.
\end{align*}

\begin{ex}
The family of Dirac bundles typically arises naturally on $N_0$. For instance, if the Dirac bundle is given by 
\begin{align*}
    (\widehat{E}_{0},h^t_0)=(\wedge^\bullet T^*N_0,g^t_0)=\left(\bigoplus_i \wedge^{\bullet,i}T^*N_0,t^{-i}g_0\right).
\end{align*}
\end{ex}

The Dirac operator associated to the dilated Hermitian Dirac bundle satisfies
\begin{equation*}
\begin{tikzcd}
	{\Gamma(N_0,\widehat{E}_{0})} && {\Gamma(N_0,\widehat{E}_{0})} \\
	\\
	{\Gamma(N_0,\widehat{E}_{0})} && {\Gamma(N_0,\widehat{E}_{0})}
	\arrow["{\widehat{D}_0}"{description}, from=1-1, to=1-3]
	\arrow["{\delta_t^*}"{description}, from=1-1, to=3-1]
	\arrow["{\delta_t^*}"{description}, from=1-3, to=3-3]
	\arrow["{\widehat{D}^t_0}"{description}, from=3-1, to=3-3]
\end{tikzcd}
\end{equation*}
where 
\begin{align*}
    \widehat{D}^t_0=\widehat{D}_{0;H}+t^{-1}\cdot \widehat{D}_{0;V}.
\end{align*}

With this setup, we now specify the precise asymptotic conditions that a smooth resolution must satisfy. These ensure that we can perform an adiabatic rescaling such that the resolved Dirac bundle and connection converge in a controlled way to the singular model as $t \to 0$.\\

Let in the following $\rho_{t^2\cdot\zeta}\colon(N_{t^2\cdot\zeta},g_{t^2\cdot\zeta})\dashrightarrow (N_0,g_0)$ be a family of ACF-spaces satisfying Assumption \ref{ass1}.

\begin{manualassumption}{3}
\label{ass3}
Let us assume that there exists a family
\begin{align*}
    (\widehat{E}_{t^2\cdot\zeta},\m{cl}_{g_{t^2\cdot\zeta}},h_{t^2\cdot\zeta},\widehat{\nabla}^{h_{t^2\cdot\zeta}})\rightarrow (N_{t^2\cdot\zeta},g_{t^2\cdot\zeta})
\end{align*}
of ACF Hermitian Dirac bundles asymptotic to $(\widehat{E}_{0},\m{cl}_{g_0},h_0,\nabla^{\otimes_0})$ that fit into a commuting diagram  
\begin{equation*}
\begin{tikzcd}
	{(\widehat{E}_\zeta,h^t_\zeta)} &&& {(\widehat{E}_{t^2\cdot\zeta},h_{t^2\cdot\zeta})} \\
	& {(\widehat{E}_{0},h^t_0)} &&& {(\widehat{E}_{0},h_0)} \\
	{(N_\zeta,g^t_\zeta)} &&& {(N_{t^2\cdot\zeta},g_{t^2\cdot\zeta})} \\
	& {(N_0,g^t_0)} &&& {(N_0,g_0)}
	\arrow["{\hat{\delta}_t}"{description}, from=1-1, to=1-4]
	\arrow["{\hat{\rho}_\zeta}"{description}, dashed, from=1-1, to=2-2]
	\arrow["{\pi_\zeta}"{description, pos=0.7}, from=1-1, to=3-1]
	\arrow["{\hat{\rho}_{t^2\cdot\zeta}}"{description}, dashed, from=1-4, to=2-5]
	\arrow["{\pi_{t^2\zeta}}"{description, pos=0.7}, from=1-4, to=3-4]
	\arrow["{\hat{\delta}_t}"{description}, from=2-2, to=2-5]
	\arrow["{\pi_0}"{description, pos=0.7}, from=2-2, to=4-2]
	\arrow["{\pi_0}"{description, pos=0.7}, from=2-5, to=4-5]
	\arrow["{\delta_t}"{description}, from=3-1, to=3-4]
	\arrow["{\rho_{\zeta}}"{description}, dashed, from=3-1, to=4-2]
	\arrow["{\rho_{t^2\cdot\zeta}}"{description}, dashed, from=3-4, to=4-5]
	\arrow["{\delta_t}"{description}, from=4-2, to=4-5]
\end{tikzcd}
\end{equation*}
whereby
\begin{align*}
    (\nabla^{\otimes_0;0,1})^k((\hat{\rho}_{t^2\cdot\zeta})_*h_{t^2\cdot\zeta;V}-h_{0;V})=&\mathcal{O}(r^{\eta-k})\\
    \nabla^k((\hat{\rho}_{t^2\cdot\zeta})_*h_{t^2\cdot\zeta;V}-h_{0;V})=&\mathcal{O}(r^{\eta}).
\end{align*}
We assume that the connections $\widehat{\nabla}^{h^t_\zeta}$ can be written as 
    \begin{align}
    \label{hatnablaotimestzeta}
        \widehat{\nabla}^{h^t_\zeta}=&\nabla^{\otimes_\zeta}+\frac{1}{2}\m{cl}_{g^t_\zeta}(\Omega_\zeta)+\tau^t_\zeta\und{1.0cm}|\tau^t_\zeta|=\mathcal{O}(t),
    \end{align}
where $\nabla^{\otimes_\zeta}+\tau^t_\zeta$ is the Clifford module connection with respect to $\nabla^{\oplus_\zeta}$ and $\nabla^{\otimes_\zeta}$ is $t$-independent. Further, we assume that the curvature of the latter satisfies
\begin{align*}
    (\nabla^{\otimes_{0};0,1})^k\left((\hat{\rho}_{t^2\cdot\zeta})_*F^{0,2}_{\nabla^{\otimes_{\zeta}}}\right)=&\mathcal{O}(r^{\eta-k})\\
    (\nabla^{\otimes_{0}})^k\left((\hat{\rho}_{t^2\cdot\zeta})_*F_{\nabla^{\otimes_{\zeta}}}-F_{\nabla^{\otimes_0}}\right)=&\mathcal{O}(r^{\eta}).
\end{align*}
\end{manualassumption}

\begin{rem}
    We will refer to $\tau^t_\zeta$ as the \textbf{twisting one-form}.
\end{rem}

\begin{rem}
    The limit $\widehat{\nabla}^{\otimes_\zeta}\coloneqq \lim_{t\to 0}\widehat{\nabla}^{h^t_{\zeta}}$ formally exists and equals
    \begin{align*}
        \widehat{\nabla}^{\otimes_\zeta}\coloneqq \nabla^{\otimes_\zeta}-\frac{1}{2}k_\zeta.
    \end{align*}
\end{rem}

\begin{rem}
    The connections $\widehat{\nabla}^{\otimes_\zeta}$ and $\widehat{\nabla}^{h^t_{\zeta}}$ are $h^t_{\zeta}$-Hermitian, and are $\left(\m{Cl}(T^*N_\zeta,g_\zeta),\nabla^{\oplus_\zeta}\right)$- and $\left(\m{Cl}(T^*N_\zeta,g^t_\zeta),\widehat{\nabla}^{g^t_\zeta}\right)$-module connection respectively.
\end{rem}

\begin{manualassumption}{4}
\label{ass4}
Assume that 
\begin{align*}
    (\widehat{E}_{\zeta;V},h^t_{\zeta;V})=\left(\bigoplus_{n\in\mathbb{Z}}t^{-n}\cdot \widehat{E}_{\zeta;V;n},h_{\zeta;V}\right)
\end{align*}
and that the decomposition is preserved by the connection $\nabla^{h_\zeta}$, i.e. it restricts to a Hermitian connection on the Hermitian subbundle $(t^{-n}\cdot \widehat{E}_{\zeta;V;n},h_{\zeta;V})$.  
\end{manualassumption}

We glue the Hermitian Dirac bundle $(\widehat{E}_\zeta,\m{cl}_{g^t_\zeta},h^t_\zeta,\nabla^{\otimes_\zeta})$ and $(E,\m{cl}_g,h,\nabla^h)$ along the collar
\begin{align*}
    B_{5t^{-1}\epsilon}(N_\zeta)\backslash \Upsilon_{\zeta}\cong \m{Tub}_{5\epsilon}(S)\backslash S
\end{align*}
using the lift of the gluing morphism,
\begin{equation*}
\begin{tikzcd}
	{ \widehat{E}_{\zeta}|_{B_{5t^{-1}\epsilon}(N_\zeta)}} && {\widehat{E}_{0}|_{\m{Tub}_{5t^{-1}\epsilon}(S)}} && {\widehat{E}_{0}|_{\m{Tub}_{5\epsilon}(S)}} && E \\
	{B_{5t^{-1}\epsilon}(N_\zeta)} && {\m{Tub}_{5t^{-1}\epsilon}(S)} && {\m{Tub}_{5\epsilon}(S)} && X
	\arrow["{\hat{\rho}_\zeta}"{description}, dashed, from=1-1, to=1-3]
	\arrow["{\hat{\Gamma}^t_\zeta}"{description},bend left, dashed, from=1-1, to=1-7]
	\arrow[from=1-1, to=2-1]
	\arrow["{\hat{\delta}_t}"{description}, from=1-3, to=1-5]
	\arrow[from=1-3, to=2-3]
	\arrow["{\widehat{\m{exp}}_{g}}"{description}, from=1-5, to=1-7]
	\arrow[from=1-5, to=2-5]
	\arrow[from=1-7, to=2-7]
	\arrow["{\rho_\zeta}"{description}, dashed, from=2-1, to=2-3]
	\arrow["{ \Gamma^t_\zeta}"{description}, bend right, dashed, from=2-1, to=2-7]
	\arrow["{\delta_t}"{description}, from=2-3, to=2-5]
	\arrow["j"{description}, from=2-5, to=2-7]
\end{tikzcd}
\end{equation*}

and denote the family of resolved spaces
\begin{align*}
    \hat{\delta}_t\colon E^t\coloneqq \widehat{E}_{\zeta}|_{B_{5t^{-1}\epsilon(N_\zeta)}}\cup_{ \hat{\Gamma}^t_\zeta}E_{X\backslash S}\cong \widehat{E}_{t^2\cdot\zeta}|_{B_{5\epsilon(N_{t^2\zeta})}}\cup_{ \hat{\Gamma}^t_\zeta}E_{X\backslash S}\eqqcolon E_t\colon \hat{\delta}_{t^{-1}}.
\end{align*}

Further, this space is equipped with a resolution morphism 
\begin{equation*}
\begin{tikzcd}
	{E^t} && {E_t} \\
	& E \\
	{X^t} && {X_t} \\
	& X
	\arrow["{\hat{\delta}_t}"{description}, shift left=2, from=1-1, to=1-3]
	\arrow[dashed, from=1-1, to=2-2]
	\arrow["{\pi^t}"{description}, from=1-1, to=3-1]
	\arrow["{\hat{\delta}_{1/t}}"{description}, shift left=2, from=1-3, to=1-1]
	\arrow[dashed, from=1-3, to=2-2]
	\arrow["{\pi_t}"{description}, from=1-3, to=3-3]
	\arrow["\pi"{description}, from=2-2, to=4-2]
	\arrow["{\delta_t}"{description, pos=0.3}, shift left=2, from=3-1, to=3-3]
	\arrow[dashed, from=3-1, to=4-2]
	\arrow["{\delta_{1/t}}"{description, pos=0.3}, shift left=2, from=3-3, to=3-1]
	\arrow[dashed, from=3-3, to=4-2]
\end{tikzcd}
\end{equation*}

On $B_{5t^{-1}\epsilon}(N_\zeta)\cong B_{5t^{-1}\epsilon}(N_\zeta)\cong U_{5 t^{-1}\epsilon}$ we have 
\begin{align*}
    h^t\coloneqq&h^t_{\zeta}\cup^t h_0+\delta^*_t h_{hot}=\hat{\delta}_t^*(h_{t^2\cdot\zeta}\cup_t h_0+h_{hot})\eqqcolon\hat{\delta}_t^*h_t.
\end{align*}

We will continue by defining a family of metric connections on $(N_\zeta,g^t_\zeta)$ and Hermitian module connections on $(E_\zeta,\m{cl}_{g^t_\zeta},h^t_\zeta)$. These should be understood as corrections to the adiabatic model, i.e. $(\nabla^{\oplus_\zeta},\nabla^{\otimes_\zeta})$ that correspond to ambient geometry in the resolved space $(X_t,g_t)$.\\

We define the Dirac operator associated to $(\widehat{E}_\zeta,\m{cl}_{g^t_\zeta},h^t,\widehat{\nabla}^{h^t_\zeta})$ to be the (uniform) elliptic operator $\widehat{D}^t_\zeta\in\m{Diff}^1(\widehat{E}_\zeta)$
\begin{align*}
    \widehat{D}^t_\zeta=\m{cl}_{g^t_\zeta}\circ\widehat{\nabla}^{h^t_\zeta}&\colon  \Gamma(N_\zeta,\widehat{E}_\zeta)\rightarrow \Gamma(N_\zeta,\widehat{E}_\zeta).
\end{align*}

Following the lines of Lemma \ref{DYdecomposition}, the Dirac operator $\widehat{D}^t_\zeta$ on $N_\zeta$ can be decomposed in the following way.

\begin{lem}\cite[Eq. (2.3)]{goette2014adiabatic}
\label{Diracdecomposition}
The Dirac operator decomposes into  
\begin{align*}
    \widehat{D}^t_\zeta=&\widehat{D}^t_{\zeta;H}+\widehat{D}^t_{\zeta;V}
\end{align*}
Here $\widehat{D}^t_{\zeta;V}=t^{-1}\cdot \widehat{D}_{\zeta;V}$ is a vertical Dirac operator and the horizontal Dirac operator can be identified with
\begin{align*}
    \widehat{D}^t_{\zeta;H}=\m{cl}_{g_\zeta}\circ\m{pr}_{H_\zeta}\circ\widehat{\nabla}^{\otimes_\zeta}+\m{cl}_{g_\zeta}(\tau_{\zeta;H}^t)+t^{-1}\cdot\m{cl}_{g_\zeta}(\tau^t_{\zeta;V})-\frac{t}{4}\m{cl}_{g_\zeta}(F_{H_\zeta}-\rho^*_\zeta F_{H_0}).
\end{align*}
\end{lem}

\section{Analytic setup}
\label{Analytic setup}

To construct a uniform elliptic theory for Dirac operators on spaces with singularities resolved by gluing, it is necessary to analyse how local solutions patch together across the regions where the geometry degenerates. This motivates the study of the linear gluing problem, which seeks to understand the analytic kernel and cokernel, and construct uniform bounded families of right-inverses of the family of Dirac operators $D_t$ arising from the resolution process.\\

In this subsection, we present a schematic version of the linear gluing method for Dirac operators on orbifold resolutions, inspired by the framework developed by Hutchings and Taubes \cite[Sec. 2, Sec. 3]{hutchingstaubesI} and \cite[Sec. 5, Sec. 9]{hutchingstaubesII}. The goal is to illustrate the core ideas behind the construction of a uniformly bounded right inverse for $D_t$ and to explain how local analytic data on the asymptotically conical fibre (ACF) region and the conically fibred singular (CFS) region can be glued to solve the global equation.

\subsection{Linear Gluing}
\label{Linear Gluing}

In the following subsection we will schematically describe the linear gluing problem of constructing the kernel and cokernel, as well as uniform bounded right-inverses of the resolved operators $D_t$. We will follow the discussion of Taubes and Hutchings \cite[Sec. 2, Sec. 3]{hutchingstaubesI} and \cite[Sec. 5, Sec. 9]{hutchingstaubesII} to discuss the basic ideas of linear gluing. The following section should be understood as a fairytale version of linear gluing results on orbifold resolutions.\\

Using the gluing morphisms $\Gamma^t_\zeta$ and $\hat{\Gamma}^t_\zeta$ we define the compact gluing space $(X_t,g_t)$ and the CF anti-gluing space $(\overline{X}_t,\overline{g}_t)$ by identifying the complementary parts of the ACF space $(N_\zeta,g_\zeta)$ and the CFS space $(X,g)$, i.e.   
\begin{equation*}
\begin{tikzcd}
	& {(N_\zeta,g^t_\zeta)\sqcup(X,g)} \\
	{(X^t,g^t)} && {(\overline{X}^t,\overline{g}^t)} & {(N_0,g^t_0)} \\
	{} \\
	{(X_t,g_t)} && {(\overline{X}^t,\overline{g}^t)} & {(N_0,g^t_0)}
	\arrow[squiggly, from=1-2, to=2-1]
	\arrow[squiggly, from=1-2, to=2-3]
	\arrow["{\delta_t}", shift left, from=2-1, to=4-1]
	\arrow[no head, equal, from=2-3, to=2-4]
	\arrow["{\delta_t}", shift left, from=2-3, to=4-3]
	\arrow["{\delta_t}", shift left, from=2-4, to=4-4]
	\arrow["{\delta_{1/t}}", shift left, from=4-1, to=2-1]
	\arrow["{\delta_{1/t}}", shift left, from=4-3, to=2-3]
	\arrow[equal, no head, from=4-3, to=4-4]
	\arrow["{\delta_{1/t}}", shift left, from=4-4, to=2-4]
\end{tikzcd}
\end{equation*}
and identify the anti-gluing space with the normal cone bundle $(N_0,g_0)$. In a similar manner, we construct the Hermitian Dirac bundles 
\begin{equation*}
\begin{tikzcd}
	{(E^t,\m{cl}_{g^t},h^t,\nabla^{h^t})} & {} & {(E_t,\m{cl}_{g_t},h_t,\nabla^{h_t})} \\
	&&& {} \\
	{(X^t,g^t)} && {(X_t,g_t)}
	\arrow["{\hat{\delta}_t}", shift left, from=1-1, to=1-3]
	\arrow["{\pi^t}"{description}, from=1-1, to=3-1]
	\arrow["{\hat{\delta}_{1/t}}", shift left, from=1-3, to=1-1]
	\arrow["{\pi_t}"{description}, from=1-3, to=3-3]
	\arrow["{\delta_t}", shift left, from=3-1, to=3-3]
	\arrow["{\delta_{1/t}}", shift left, from=3-3, to=3-1]
\end{tikzcd}
\end{equation*}
and
\begin{align*}
    (\overline{E}_t,\m{cl}_{\overline{g}_t},\overline{h}_0,\nabla^{\overline{h}_0})\coloneqq (\widehat{E}_0,\m{cl}_{g_0},h_0,\nabla^{\otimes_0}).
\end{align*}

In this section, we are interested in solving the Dirac equation
\begin{align}
\label{Diracequation}
    D_t \Sigma = \Xi,
\end{align}
for sections $\Sigma, \Xi \in \Gamma(X_t, E_t)$, where $D_t$ is a Dirac operator defined on the family of Dirac bundles $E_t$ over the family $(X_t,g_t)$. A solution to \eqref{Diracequation} exists if and only if $\Xi$ lies in the image of $D_t$, i.e. $\Xi \in \mathrm{im}(D_t)$, and in that case the solution space is the affine space
\begin{align*}
    \mathrm{Sol}_{D_t}(\Xi) \coloneqq \Sigma + \m{ker}(D_t).
\end{align*}
Thus, in order to analyse solvability, it is essential to understand both $\m{ker}(D_t)$ and $\m{coker}(D_t)$.\\

Our goal is to solve the Dirac equation \eqref{Diracequation} in a uniform way as $t \to 0$. Therefore we seek to construct a family of right inverses for $D_t$ whose operator norms remain uniformly bounded in $t$. Such uniform estimates are crucial for applications to nonlinear problems in geometric analysis, where $D_t$ arises as the linearisation of a nonlinear elliptic operator. This is the case, e.g. in the deformation theory of instantons, calibrated submanifolds, or special holonomy metrics. The existence of a uniform right inverse ensures control over solutions to the full nonlinear problem via the implicit function theorem or Newton iteration methods. \\

The geometry of the space $(X_t,g_t)$ naturally leads us to attempt a local-to-global strategy, i.e.  solving \eqref{Diracequation} by constructing solutions locally on the singular region $N_\zeta$ and on the regular stratum $X^{\mathrm{reg}}$, and then gluing these solutions together. Consequently, given $\Phi \in \Gamma(N_\zeta, E_\zeta)$ and $\Psi \in \Gamma(X^{\mathrm{reg}}, E)$, we can construct a global solution $\Sigma \in \Gamma(X_t, E_t)$.\\

To approach this, we will study the analytic properties of the operator
\begin{align*}
    D_t \colon V^{k+1}_t \rightarrow W^k_t
\end{align*}
where $D_t$ arises from a geometric gluing construction, and the function spaces $V^{k+1}_t$, $W^k_t$ reflect the analytic structure of the corresponding glued bundle. Our analysis will rely on an isomorphism of Banach spaces
\begin{align*}
    \widehat{D}_0 \colon V^{k+1}_{\m{CF}} \rightarrow W^k_{\m{CF}},
\end{align*}
and the existence of two uniform Fredholm realisations, 
\begin{align*}
    \widehat{D}^t_{\zeta} \colon V^{k+1}_{\m{ACF};t} \rightarrow W^k_{\m{ACF};t}\und{1.0cm}
    D \colon V^{k+1}_{\m{CFS}} \rightarrow W^k_{\m{CFS}},
\end{align*}
associated to the ACF-and CFS-region, respectively. A detailed analysis of the uniform elliptic estimates and the Fredholm properties of these operators, particularly near the adiabatic limit $t\to 0$, will be carried out in Sections \ref{Uniform Elliptic Estimates for Conically Fibred Dirac Operators}–\ref{Uniform Elliptic Theory for Dirac Operators on Orbifold Resolutions of type A}.\\
 
As we have mentioned above, a key ingredient to the construction of a uniform bounded right inverse of $D_t$ is to control the kernel and cokernel of the operator $D_t$. Consequently, we are looking for solutions of
\begin{align}
\label{Dtgammat=0}
    D_t(\Psi \cup_t \Phi) = 0.
\end{align}
Equation \eqref{Dtgammat=0} can be equivalently rewritten as the coupled system
\begin{align*}
    \Theta^t_{\mathrm{ACF}}(\Psi \oplus \Phi) \cup^t \Theta_{\mathrm{CFS};t}(\Psi \oplus \Phi) = 0,
\end{align*}
where
\begin{align*}
    \Theta^t_{\mathrm{ACF}}(\Psi \oplus \Phi) &\coloneqq \widehat{D}^t_{\zeta} \Psi + (D^t - \widehat{D}^t_{\zeta})(1 - \chi^t_5) \Psi + \mathrm{cl}_{g^t}(\mathrm{d}\chi^t_3)(\delta_t^* \Phi - \Psi), \\
    \Theta_{\mathrm{CFS};t}(\Psi \oplus \Phi) &\coloneqq D \Phi + (D_t - D) \chi_1 \Phi + \mathrm{cl}_{g_t}(\mathrm{d}\chi_3)(\Phi - \delta_{t^{-1}}^* \Psi).
\end{align*}
Here, $\chi^t_3$, $\chi^t_5$, and $\chi_1$ are smooth cut-off functions introduced to mediate between local models, and we have suppressed the bundle identifications $\hat{\Gamma}^t_\zeta$ for clarity.\\

The expressions $\Theta^t_{\mathrm{ACF}}(\Psi \oplus \Phi)$ and $\Theta_{\mathrm{CFS};t}(\Psi \oplus \Phi)$ naturally define a coupled system of equations on the disjoint union $N_\zeta \sqcup X$, valued in appropriate function spaces:
\begin{align*}
    \Theta^t_{\mathrm{ACF}}(\Psi \oplus \Phi) \oplus \Theta_{\mathrm{CFS};t}(\Psi \oplus \Phi) \in W^k_{\mathrm{ACF};t} \oplus W^k_{\mathrm{CFS}}.
\end{align*}
The goal of the following analysis is to study the solvability of this coupled system by exploiting the Fredholm theory of the model operators and the geometry of the gluing construction.\\

We emphasize that in this section, the operator $D_t$ is not assumed to be essentially self-adjoint. Rather, we focus on its realisations as bounded operators between suitable Banach spaces allowing us to formulate and study the gluing problem in an analytic framework.\\

We now formulate the key ingredients of the linear gluing theory for the family of Dirac operators $D_t$ discussed above. These steps culminate in an exact sequence relating the analytic kernel and cokernel of $D_t$ to the model operators on the ACF and CFS pieces. Our goal is to identify an obstruction map whose vanishing ensures the existence of a uniformly bounded right inverse of $D_t$.\\

The following list of meta-statements serves as a conceptual roadmap for establishing a uniform linear gluing theory for Dirac operators on orbifold resolutions. While stated here in schematic form, each statement will be rigorously proved in the appropriate analytic setting in subsequent sections. Together, they highlight the geometric structure underlying the gluing problem and guide the construction of a uniformly bounded right inverse.\\

We begin by formulating a decomposition result that expresses sections on the glued manifold in terms of compatible local data from the ACF and CFS regions.

\begin{lem}[Meta-Decomposition]
    Given $\Sigma\in V^{k+1}_t$ and $\Xi\oplus \m{H}\in W^{k}_{\m{ACF};t}\oplus W^{k}_{\m{CFS}}$ such that 
    \begin{align*}
        D_t\Sigma=\Xi\cup_t \m{H}
    \end{align*}   
    there exists a unique $\Psi\oplus\Phi\in V^{k+1}_{\m{ACF};t}\oplus V^{k+1}_{\m{CFS}}$ such that 
    \begin{align}
    \label{itVW}
        \Sigma=\Psi\cup_t\Phi\und{1.0cm}\Theta^t_{\m{ACF}}(\Psi\oplus\Phi)=\Xi\und{0.5cm}\Theta_{\m{CFS};t}(\Psi\oplus\Phi)=\m{H}.
    \end{align}
\end{lem}

As a next step, we consider the inverse problem. Given data in the kernel of the model operators, when does it extend to an approximate solution of the glued equation? This is captured by the following obstruction lemma.

\begin{lem}[Meta-Obstruction]
Let $\Xi\oplus \m{H}\in\m{ker}_{\m{ACF}}(\widehat{D}^t_\zeta)\oplus\m{ker}_{\m{CFS}}(D)$. There exists unique 
\begin{align*}
    \Psi\oplus\Phi\in\left(\m{ker}_{\m{ACF}}(\widehat{D}^t_{\zeta})\oplus\m{ker}_{\m{CFS}}(D)\right)^\perp
\end{align*}
such that 
\begin{align}
\label{obtVW}
    (\Theta^t_{\m{ACF}}\oplus \Theta_{\m{CFS};t})((\Xi+\Psi)\oplus(H+\Phi))\in\m{coker}_{\m{ACF}}(\widehat{D}^t_{\zeta})\oplus\m{coker}_{\m{CFS}}(D).
\end{align}
\end{lem}

\begin{rem}
A theorem of this form can be proven by realising it as a fixed-point problem. In order to use the contraction mapping principle to solve this gluing problem we need that both $\Phi$ and $\Psi$ satisfy the matching condition
\begin{align*}
    \left|\left|\Phi-\delta_{t^{-1}}^*\Psi\right|\right|_{W^{k+1}_{\m{CF}}(\m{A}_{2\epsilon,3\epsilon}(S))},\left|\left|\delta_t^*\Phi-\Psi\right|\right|_{V^{k+1}_{\m{CF}}(\m{A}_{(2t^{-1}\epsilon,3t^{-1}\epsilon)}(S))}\xrightarrow[t\to 0]{}0 .
\end{align*}
\end{rem}

To formalise the relation between both the analytic kernel and cokernel of $D_t$, and the model operators, we introduce the notion of approximate kernel and cokernel.

\begin{defi}
We define the \textbf{approximate kernel and cokernel} as the pullbacks 
\begin{equation*}
\adjustbox{scale=0.9,center}{
\begin{tikzcd}
	{\m{xker}(D_t)} && {\m{ker}_{\m{ACF}}(\widehat{D}^t_\zeta)} && {\m{xcoker}(D_t)} && {\m{coker}_{\m{ACF}}(\widehat{D}^t_\zeta)} \\
	&&& {\m{and}} \\
	{\m{ker}_{\m{CFS}}(D)} && {\m{ker}_{\m{CF}}(\widehat{D}_0)} && {\m{coker}_{\m{CFS}}(D)} && {\m{coker}_{\m{CF}}(\widehat{D}_0)}
	\arrow[from=1-1, to=1-3]
	\arrow[from=1-1, to=3-1]
	\arrow["\lrcorner"{anchor=center, pos=0.125}, draw=none, from=1-1, to=3-3]
	\arrow["{\m{res}_{\partial_\infty N_\zeta}}"{description}, from=1-3, to=3-3]
	\arrow[from=1-5, to=1-7]
	\arrow[from=1-5, to=3-5]
	\arrow["\lrcorner"{anchor=center, pos=0.125}, draw=none, from=1-5, to=3-7]
	\arrow["{\m{res}_{\partial_\infty N_\zeta}}"{description}, from=1-7, to=3-7]
	\arrow["{\m{res}_{X^{sing}}}"{description}, from=3-1, to=3-3]
	\arrow["{\m{res}_{X^{sing}}}"{description}, from=3-5, to=3-7]
\end{tikzcd}}
\end{equation*}
\end{defi}

\begin{rem}
    If the CF-normal operator $\widehat{D}_0 : V^{k+1}_{\mathrm{CF}} \rightarrow W^k_{\mathrm{CF}}$ is an isomorphism, then the fibre products appearing above are taken over the trivial vector space and hence, reduce to the direct sum of the corresponding fibres.
\end{rem}

The following theorem encapsulates the structure of the gluing theory in terms of these approximate spaces.

\begin{thm}[Meta-Linear Gluing Sequence]
There exists an exact sequence 
\begin{equation}
\label{lineargluingexactsequenceVW}
    \begin{tikzcd}
        \m{ker}(D_t)\arrow[r,"i_t",hook]&
             \m{xker}(D_t)\arrow[r,"\m{ob}_t"]&\m{xcoker}(D_t)\arrow[r,"p_t",two heads]&\m{coker}(D_t)
    \end{tikzcd}
\end{equation}
where we define the map
\begin{align*}
   i_t\colon \m{ker}(D_t)\hookrightarrow\m{xker}(D_t)
\end{align*}
by sending $\Sigma$ to $\Pi_{\m{ker}(\widehat{D}^t_{\zeta})}(\Psi)\oplus\Pi_{\m{ker}(D)}(\Phi)$ satisfying \eqref{itVW} for $\Xi\oplus \m{H}=0$. Further, we define the map 
\begin{align*}
    \m{ob}_t\colon \m{xker}(D_t) \rightarrow\m{xcoker}(D_t) 
\end{align*}
by sending an element $\Xi\oplus \m{H}$ to $\Theta^t_{\m{ACF}}(\Xi+\Psi,H+\Phi)\oplus\Theta_{\m{CFS};t}(\Xi+\Psi,H+\Phi)$ where $\Psi\oplus\Phi$ is the unique element in $\m{xker}(D_t)^\perp$ solving \eqref{obtVW}. Finally, the map 
\begin{align*}
    p_t\colon \m{xcoker}(D_t) \twoheadrightarrow \m{coker}(D_t)
\end{align*}
sends an element $\Upsilon\oplus \Sigma\in\m{coker}_{\m{ACF}}(\widehat{D}^t_{\zeta}) \oplus \m{coker}_{\m{CFS}}(D) $ to $\Pi_{\m{coker}(D_t)}(\Upsilon\cup_t\Sigma)$.
\end{thm}

As a final consequence, we observe that a vanishing obstruction yields a uniform analytic control over the Dirac operator on the glued manifold.

\begin{thm}[Meta-Existence of Uniform Right Inverse]
\label{thm:uniforminversemeta}
If the linear obstruction map $\mathrm{ob}_t = 0$ vanishes, and consequently the exact sequence \eqref{lineargluingexactsequenceVW} splits, the operator $D_t$ admits a parametrix
\begin{align*}
    R_t:W^k_t\rightarrow V^{k+1}_t 
\end{align*}
whose norm is uniformly bounded as $t \to 0$.
\end{thm}

\subsection{Polyhomogenous Solutions of the Normal Operator}
\label{Polyhomogenous Solutions of the Normal Operator}
In this section, we explicitly construct the polyhomogeneous kernel elements of the CF-Dirac operator $ \widehat{D}_0 $ by separating variables along the warped product structure and by reducing the problem to a parametrised family of ordinary differential equations. These solutions, which encode the asymptotic behaviour of sections near the boundary or singular stratum, play a central role in the gluing analysis and functional-analytic framework developed later.\\ 

This separation of variables provides a powerful framework to describe the asymptotic behaviour and polyhomogeneous structure of solutions. Building on methods developed by Bismut \cite{bismut1989eta} and Ammann and Bär \cite{ammann1998dirac}, we explicitly construct these solutions, adapting their techniques to treat the boundary conditions that arise in the functional-analytic realisations of the normal operator.\\

Understanding the precise asymptotic behaviour of spinors near the singular stratum is essential for constructing a right-inverse and proving regularity theorems in the later gluing analysis. The polyhomogeneous expansions capture both the decay and oscillatory behaviour induced by the geometry, and ensure compatibility with the weighted function spaces used in the global theory.\\

We begin by analysing the eigenspace decomposition induced by the vertical Dirac operator and identifying the associated Hermitian Dirac bundles over the base $S$.

\subsubsection{The Spectrum of the Dirac Operator on the Link Fibration}
\label{The Spectrum of the Dirac Operator on the Link Fibration}
Computing the kernel of the normal operator $\widehat{D}_0$ requires a detailed description of the spectrum of the family $\widehat{D}_{Y;r}$ of Dirac operators on the link fibration
\begin{align*}
    \vartheta\colon (Y,g^{2,0}_Y+r^2g^{0,2}_Y) \rightarrow (S,g_S).
\end{align*}
Since the total geometry is a compact, collapsing, totally geodesic fibration, we exploit the separation of variables to decouple the Dirac operator. The key idea is to reduce the study of $\widehat{D}_0$ to a family of ODEs indexed by spectral data arising from the horizontal and vertical components of the geometry. This techniques are build on adiabatic analysis of Bismut in \cite{bismut1989eta} and the techniques of Ammann and Bär \cite{ammann1998dirac} on Dirac operators on totally geodesic fibrations.\\

This analysis naturally begins with a study of the fibrewise spectrum of the vertical Dirac operator. Following the approach of Bismut \cite{bismut1989eta}, we define the $L^2$-pushforward of the vector bundle $\widehat{E}_Y$, which yields an infinite-dimensional Hermitian Dirac bundle over the base $S$. Since the fibres of the fibration $\vartheta$ are compact, the fibrewise vertical Dirac operator $\widehat{D}_{Y;V}$ has discrete spectrum and induces a decomposition into its eigenbundles. By an observation of Ammann and Bär \cite{ammann1998dirac}, and using that the fibres of $\vartheta$ are totally geodesic, the horizontal Dirac operator preserves the splitting induced by the vertical spectrum and restricts to finite-dimensional symplectic subbundles. A further spectral decomposition of these restricted horizontal Dirac operators yields precise control over the full spectrum of the family $\widehat{D}_{Y;r}$.\\

Let $\pi_{Y}\colon (\widehat{E}_Y,\m{cl}_{g_{Y;r}},h_{Y;r},\nabla^{h_{Y;r}})\rightarrow (Y,g_{Y;r})$ denote the Hermitian Dirac bundle introduced in Section \ref{Dirac Bundles on Orbifold as Conically Fibred Singular Spaces}. Further, denote $\m{cl}_{g_0}(\m{d}r)=\m{cl}_0$ the radial Clifford multiplication.\\

The Hilbert spaces $L^2_r(Y,\widehat{E}_Y)$ are symplectic Hilbert spaces whose symplectic form is given by 
    \begin{align*}
        \omega_r(\Phi,\Psi)=\int_Y h_r(\m{cl}_0 \Phi,\Psi)\m{vol}_{Y;r}
    \end{align*}
and $\m{cl}_0^2=-1$ the corresponding complex structure. The operator $\widehat{D}_{Y;r}$ anticommutes with this complex structure.\\

We define the infinite-dimensional vector bundle 
\begin{align*}
    \vartheta_{*_{L^2}}\widehat{E}_Y\rightarrow S
\end{align*}
of $\vartheta$-fibrewise sections, i.e. for open subsets $U\subset S$ we set
\begin{align*}
    L^2(U,\vartheta_{*_{L^2}}\widehat{E}_Y)=L^2(\vartheta^{-1}(U),E_{Y}).
\end{align*}

Let $Fr_X(E)$ denote the $\m{SO}(\mathbb{E})$ frame bundle of $E$. We can identify the pushforward of $E_Y$ with the associate bundle 
\begin{align*}
    \vartheta_{*_{L^2}}E_Y\cong (Fr_X(E)|_S\times Fr_{\m{SO}(W)}(X/S))\times_{\m{N}_{\m{SO}(\mathbb{E})\times\m{SO}(W)}(\m{Isot}(S))} L^2(\mathbb{S}^{m-1}/\Gamma,\mathbb{E}).
\end{align*}
Notice that this is naturally a $\vartheta_{*_{Cyl}}\underline{\mathbb{R}}$-module. The $\m{Cl}(T^*S,g_S)$-module structures is given by 
\begin{align*}
    \m{cl}_{g_S}(\xi)\Phi=\m{cl}_{g_{Y;r}}(\vartheta^*\xi)\Phi=\m{cl}_{g_{Y_1}}(\vartheta^*\xi)\Phi
\end{align*}
for a one form $\xi\in \Omega^1(S)$ and $\Phi\in \Gamma(S,\vartheta_{*_{L^2}}\widehat{E}_Y)$. Furthermore, the connection $\nabla^{\otimes_0}$ induces a connection on $\vartheta_{*_{L^2}}\widehat{E}_Y$ by 
\begin{align*}
    \nabla^{\otimes_{Y}}_V\Phi=\nabla^{\otimes_{Y}}_{V^{H_0}}\Phi,
\end{align*}
whose curvature is given by 
\begin{align*}
    F_{\nabla^{\otimes_{Y}}}(U,V)=F_{\nabla^{\otimes_{Y}}}(U^{H_0},V^{H_0})-\nabla^{\otimes_{Y}}_{F_{H_0}(U,V)}.
\end{align*}
 Naturally, the pushforward is equipped with a Hermitian structure 
\begin{align*}
    h_{Y;r}(\Phi,\Phi')_s=\int_{\vartheta^{-1}(s)}\left<\Phi,\Phi'\right>_{h_{Y;r}}\m{vol}_{r^2g_{Y;V}}
\end{align*}
given by the fibrewise $L^2_r$-product. As $\vartheta$ have totally geodesic fibres, and hence the fibrewise mean curvature $k_Y=0$ vanishes, the connection $\nabla^{\otimes_{Y}}$ is Hermitian connection. This connection is a $\m{Cl}(T^*S,g^S)$-module connection and thus
\begin{align*}
    (\vartheta_{*_{L^2}}\widehat{E}_Y,\m{cl}_{g_S},h_{Y;r},\nabla^{\otimes_{Y}})
\end{align*}
is an infinite dimensional Hermitian Dirac bundle.

\begin{lem}
\label{Elambda}
The Hermitian Dirac bundle $\vartheta_{*_{L^2}}\widehat{E}_Y$ decomposes into the eigenbundles of the endomorphism $\widehat{D}_{Y;V}$, i.e. 
\begin{align*}
    \vartheta_{*_{L^2}}\widehat{E}_Y=\bigoplus_{\lambda\in\sigma(\widehat{D}_{Y;V})}\mathcal{E}_{\lambda}.
\end{align*}
Moreover, spaces $\mathcal{E}_{\lambda}$ define smooth vector bundles on $S$ of rank $\m{m}(\lambda)=\m{m}(-\lambda)$\footnote{Here $\m{m}(\lambda)$ denote the multiplicity of the eigenvalue $\lambda$.}.
\end{lem}

\begin{proof}
Notice that we can write 
\begin{align*}
    \mathcal{E}_{\lambda}\coloneqq (Fr_X(E)|_S\times_S Fr_{\m{SO}(W)}(X/S))\times_{\m{N}_{\m{SO}(\mathbb{E})\times \m{SO}(W)}(\m{Isot}(S))}\mathcal{E}_\lambda(D_{\mathbb{S}^{m-1}/\m{Isot}(S)}).
\end{align*}
Further, $\m{N}_{\m{SO}(\mathbb{E})\times \m{SO}(W)}(\m{Isot}(S))$ acts via isometric transformations of $\mathbb{E}\rightarrow \mathbb{S}^{m-1}/\Gamma$ and hence, preserves the spectra and acts diagonally on the eigenspace decomposition. 
\end{proof}

\begin{cor}
    The $\lambda\in\sigma(\widehat{D}_{Y;V})$ seen as $\lambda\in C^\infty(S)$ are constant along $S$.
\end{cor}

\begin{rem}
    In the Appendix \ref{The Spectrum of Dirac Operators on Spheres} we investigate the spectra of Dirac operators on spheres.
\end{rem}

\begin{defi}
Let us define the projections
\begin{align*}
    \pi_{\Lambda}\colon  L^2(Y,\widehat{E}_{Y})\rightarrow  L^2(S,\mathcal{E}_{\Lambda})
\end{align*}
by the $L^2_r$-projection onto the $\Lambda=\lambda\oplus -\lambda$-eigenbundles. As $\m{cl}_0$ anticommutes with $\widehat{D}_{Y;V}$ it defines a map 
    \begin{align*}
        \m{cl}_0\colon \mathcal{E}_{\lambda}\rightarrow\mathcal{E}_{-\lambda}.
    \end{align*}
In particular, $L^2_r(S,\mathcal{E}_\Lambda)\subset L^2_r(Y,\widehat{E}_Y)$ is a symplectic subspace with respect to $\omega_r$ and 
    \begin{align*}
        \m{cl}_0|_{\mathcal{E}_\Lambda}\coloneqq \left(\begin{array}{cc}
             0&\m{cl}_{0;\Lambda;+}  \\
             \m{cl}_{0;\Lambda;-}&0 
        \end{array}\right).
    \end{align*}
\end{defi}

Using that $\vartheta:(Y,g_Y)\rightarrow (S,g_S)$ has totally geodesic fibres, we know that $\nabla^{\otimes_Y}_H$ and $\widehat{D}_{0;V}$ anticommute. Consequently, we are able to deduce the following.

\begin{lem}
\label{DHLambdalemma}
The Hermitian connection $\nabla^{\otimes_{Y}}$ restricts to a Hermitian connection 
\begin{align*}
    \nabla^{h_{Y;\Lambda}}=\pi_{\Lambda}\nabla^{\otimes_{Y}}\pi_{\Lambda}
\end{align*}
on $\mathcal{E}_{\Lambda}$ which makes
\begin{align*}
    (\mathcal{E}_{\Lambda},\m{cl}_{g_S;\Lambda},h_{Y;r;\Lambda},\nabla^{h_{Y;\Lambda}})
\end{align*}
a family of Hermitian Dirac bundles. Moreover, the associated Dirac operator $\m{cl}_{g_S;\Lambda}\circ \nabla^{h_{Y;\Lambda}}$
and 
\begin{align*}
    \widehat{D}_{Y;H;\Lambda}\coloneqq\pi_{\Lambda}\widehat{D}_{Y;H}\pi_{\Lambda}
\end{align*}
coincide. 
\end{lem}

\begin{rem}
    As $S$ is assumed to be compact, the analytic realisations 
    \begin{align*}
        \widehat{D}_{Y;H;\Lambda}\colon &W^{k,p}_r(S,\mathcal{E}_{\Lambda})\rightarrow W^{k-1,p}_r(S,\mathcal{E}_{\Lambda})\\
        \widehat{D}_{Y;H;\Lambda}\colon &  C^{k,\alpha}_r(S,\mathcal{E}_{\Lambda})\rightarrow   C^{k-1,\alpha}_r(S,\mathcal{E}_{\Lambda})
    \end{align*}
    are Fredholm and satisfy the elliptic estimates
    \begin{align*}
        \left|\left|\Phi_{\Lambda}\right|\right|_{W^{k+1,p}_r}\lesssim& \left|\left|\widehat{D}_{Y;H;\Lambda}\Phi_{\Lambda}\right|\right|_{W^{k,p}_r}+\left|\left|\Phi_{\Lambda}\right|\right|_{L^2_r}.\\
        \left|\left|\Phi_{\Lambda}\right|\right|_{C^{k+1,\alpha}_r}\lesssim& \left|\left|\widehat{D}_{Y;H;\Lambda}\Phi_{\Lambda}\right|\right|_{C^{k,\alpha}_r}+\left|\left|\Phi_{\Lambda}\right|\right|_{C^0_r}.
    \end{align*}
\end{rem}

\begin{cor}
    Let $\lambda\in\sigma(\widehat{D}_{Y;V})$ and $\mu_{\Lambda}\in\sigma(\widehat{D}_{Y;H;\Lambda})$. There exists an $\widehat{D}_{Y;H;\Lambda}$-eigenbasis of tuples $\{\Psi_{\lambda,\mu_{\Lambda}},\Psi_{-\lambda,\mu_{\Lambda}}\}$ of orthonormal vector such that 
    \begin{align*}
    (\widehat{D}_{Y;H}+r^{-1}\widehat{D}_{Y;V})|_{\{\Psi_{\lambda,\mu_{\Lambda}},\Psi_{-\lambda,\mu_{\Lambda}}\}}=\left(\begin{array}{cc}
             \frac{\lambda}{r}&\mu_{\Lambda}  \\
            \mu_{\Lambda}&-\frac{\lambda}{r}
        \end{array}\right)
    \end{align*}
\end{cor}

\begin{proof}
The operator $\widehat{D}_{Y,H;\Lambda}$ is a self-adjoint and anticommutes with $\widehat{D}_{Y;V}$. Thus, it maps $L^2_r(S,\mathcal{E}_\lambda)$ to $L^2_r(S,\mathcal{E}_{-\lambda})$. Consequently, its square is a self-adjoint operator
\begin{align*}
    \widehat{D}^2_{Y,H;\Lambda}\colon L^2_r(S,\mathcal{E}_\lambda)\rightarrow L^2_r(S,\mathcal{E}_\lambda).
\end{align*}
Let $\phi_{\lambda,\mu^2_{\Lambda}}$ be a $\mu^2_{\Lambda}$-eigenbasis of $\widehat{D}_{Y;H;\Lambda}^2$. Restricting the operator $\widehat{D}_{Y}$ to   
\begin{align*}
    \left\{\phi_{\lambda,\mu^2_{\Lambda}},
          \frac{1}{\mu}\widehat{D}_{Y;H;\Lambda}\phi_{\lambda,\mu^2_{\Lambda}}\right\}
\end{align*}
yields the desired representation.
\end{proof}

The following corollary can be seen as a generalisation of  \cite[Thm. 4.1]{ammann1998dirac}.
\begin{cor}
    The eigenvalues of $\widehat{D}_{Y;r}$ are indexed by $\lambda$ and $\mu_{\Lambda}$ given by
    \begin{align*}
        \pm\frac{1}{r}\sqrt{r^2\mu_{\Lambda}^2+\lambda^2}
    \end{align*}
\end{cor}

\begin{figure}[!h]
    \centering
    \begin{tikzpicture}
\begin{axis}[
  width=12cm, height=8cm,
    xmin=0, xmax=0.1,
    ymin=-30, ymax=30,
    axis lines=middle,
    domain=0.001:0.1,
    samples=500,
    xtick=\empty,
    ytick=\empty,
    xlabel={$r$},
]
\addplot[blue] {sqrt(2^2 + ((0.1)^2) / (x^2))};
\addplot[blue, dashed] {-sqrt(2^2 + ((0.1)^2) / (x^2))};
\addplot[purple] {sqrt(1^2 + ((0.5)^2) / (x^2))};
\addplot[purple, dashed] {-sqrt(1^2 + ((0.5)^2) / (x^2))};
\addplot[pink] {sqrt((0.5)^2 + (1^2) / (x^2))};
\addplot[pink, dashed] {-sqrt((0.5)^2 + (1^2) / (x^2))};
\addplot[green] {sqrt(0^2 + (1^2) / (x^2))};
\addplot[green, dashed] {-sqrt(0^2 + (1^2) / (x^2))};
\addplot[orange] {sqrt(2^2)};
\addplot[orange, dashed] {-sqrt(2^2)};
\addplot[red] {sqrt(4^2)};
\addplot[red, dashed] {-sqrt(4^2)};
\end{axis}
\end{tikzpicture}
    \caption{The Eigenvalues of $\widehat{D}_{Y;r}$.}
\end{figure}

\subsubsection{The Kernel of the Spectral Component $\widehat{D}_{0;\Lambda,M_\Lambda}$}
\label{The Kernel of the Spectral Component D0lambdamu}

We will now proceed by using these results to compute the polyhomogeneous kernel elements of $\widehat{D}_0$.\\

We will begin with explicitly computing the kernel elements of the normal operator. For each pair $(\lambda, \mu_\Lambda)$ we reduce $\widehat{D}_0$ to an explicit matrix-valued ODE acting on a 4-dimensional space generated by the corresponding eigensection. The following lemma identifies this basis and the induced form of the operator.

\begin{lem}
Let $m=\m{codim}(S)$. There exists an orthonormal eigenbasis for all $|\mu_{\Lambda}|\in|\sigma(\widehat{D}_{Y;H;\Lambda})|$
\begin{align*}
    \mathbb{E}_{\Lambda,M_{\Lambda}}\coloneqq\left\{\phi_{\lambda,\mu^2_{\Lambda}},\frac{1}{\mu_{\Lambda}}\widehat{D}_{Y;H;\Lambda}\phi_{\lambda,\mu^2_{\Lambda}},\m{cl}_0\phi_{\lambda,\mu^2_{\Lambda}},\frac{1}{\mu_\Lambda}\m{cl}_0\widehat{D}_{Y;H;\Lambda}\phi_{\lambda,\mu_{\Lambda}^2}\right\}.
\end{align*}
of $L^2_r(S,\mathcal{E}_{\Lambda})$ and for all $|\mu_{0}|\in|\sigma(\widehat{D}_{Y;H;0})|$
\begin{align*}
    \mathbb{E}_{0,M_{0}}\coloneqq\left\{\phi_{0,\mu_{0}},\m{cl}_0\phi_{0,\mu_{0}}\right\}
\end{align*}

of $L^2_r(S,\mathcal{E}_{0})$ such that on $L^2_{loc}(\mathbb{R}_{\geq 0},\mathbb{E}_{\Lambda,M_{\Lambda}})$ the operator $\widehat{D}_0$ acts via 
\begin{align*}
    \widehat{D}_{0;\Lambda,M_{\Lambda}}=& \left(\begin{array}{cc}
         0& 1 \\
         -1& 0
    \end{array}\right)\otimes\left(\begin{array}{cc}
         \left(\partial_r+\frac{m-1}{r}\right)&0  \\
         0&\left(\partial_r+\frac{m-1}{r}\right)
    \end{array}\right)+\left(\begin{array}{cc}
         1& 0 \\
         0& -1
    \end{array}\right)\otimes\left(\begin{array}{cc}
         0&\mu_{\Lambda} \\
          \mu_{\Lambda}&0
    \end{array}\right)\\
    &+\left(\begin{array}{cc}
         1& 0 \\
         0& -1
    \end{array}\right)\otimes\left(\begin{array}{cc}
         \frac{\lambda}{r}&0  \\
         0&-\frac{\lambda}{r}
    \end{array}\right)  
\end{align*}  
and on $L^2_{loc}(\mathbb{R}_{\geq 0},\mathbb{E}_{0,M_{0}})$ via
\begin{align*}
    \left[\left(\begin{array}{cc}
         0&\partial_r+\frac{m-1}{2r}  \\
         -\partial_r-\frac{m-1}{2r}&0
    \end{array}\right)+\left(\begin{array}{cc}
         \mu_{\Lambda}&0 \\
          0&-\mu_{\Lambda}
    \end{array}\right)\right]\left(\begin{array}{c}
         f_+  \\
         f_- 
    \end{array}\right)&=0
\end{align*}
respectively.
\end{lem}

\begin{proof}
Notice, that 
\begin{align*}
    \phi_{\lambda,\mu^2_{\Lambda}}\in \Gamma(S,\mathcal{E}_\lambda)\Rightarrow \m{cl}_0\phi_{\lambda,\mu^2_{\Lambda}}\in \Gamma(S,\mathcal{E}_{-\lambda}),\, \widehat{D}_{Y;H;\Lambda}\phi_{\lambda,\mu^2_{\Lambda}}\in \Gamma(S,\mathcal{E}_{-\lambda})
\end{align*}
and 
\begin{align*}
        \left<\phi_{\lambda,\mu^2_{\Lambda}},\m{cl}_0\widehat{D}_{Y;H;\Lambda}\phi_{\lambda,\mu^2_{\Lambda}}\right>=0
    \end{align*}
vanishes as $\m{cl}_0\widehat{D}_{Y;H;\Lambda}$ is skew-adjoint.
\end{proof}

In the following we will construct polyhomogenous solutions of the normal operator. This is achieved by solving for 
\begin{align}
\label{hatD0Phi=0}
    \widehat{D}_0\Psi=0
\end{align}
in $L^2_{loc}(\mathbb{R}_{\geq 0},L^2_r(S,\vartheta_{*_{L^2}}\widehat{E}_Y))$. By previous section, we can use the spectral decomposition and instead solve for 
\begin{align}
\label{Dkernel}
    \widehat{D}_0\Psi_{\Lambda,M_\Lambda}=0\und{1.0cm}\Psi_{\Lambda,\mu_\Lambda}\in L^2(\mathbb{R}_{\geq 0},\mathbb{E}_{\Lambda,M_\Lambda}).
\end{align}

To analyse the behaviour of kernel elements, we now solve the model ODE obtained from the spectral reduction. Different asymptotic regimes arise depending on whether $ \lambda$ or $\mu_\Lambda$ vanish, reflecting cylindrical ($\lambda=0$), conical ($\mu_\Lambda=0$) or mixed features of the geometry.

\begin{itemize}
    \item[$\boxed{\lambda=0:}$] 
Hence, $\mathcal{E}_{\Lambda}=\mathcal{E}_0$ and we take the basis $\left\{\phi_{0,\mu_0},\m{cl}_0\phi_{0,\mu_0}\right\}_{\mu_0\in\sigma(\widehat{D}_{Y;H;0})}$ such that 
\begin{align}
\label{ODElambda0}
    \left[\left(\begin{array}{cc}
         0&\partial_r+\frac{m-1}{2r}  \\
         -\partial_r-\frac{m-1}{2r}&0
    \end{array}\right)+\left(\begin{array}{cc}
         \mu_{\Lambda}&0 \\
          0&-\mu_{\Lambda}
    \end{array}\right)\right]\left(\begin{array}{c}
         f_1  \\
         f_2 
    \end{array}\right)&=0
\end{align}
which has the solution 
\begin{align*}
    f_1(r)=&r^{\frac{1-m}{2}}\left(c_-e^{-\mu_0 r}+\frac{c_+}{\mu_0}e^{\mu_0 r}\right)\\
    f_2(r)=&r^{\frac{1-m}{2}}\left(c_-e^{-\mu_0 r}-\frac{c_+}{\mu_0}e^{\mu_0 r}\right).
\end{align*}
\end{itemize}

In order to solve the ODE's in for $\lambda\neq 0$ we define a new basis 

\begin{align*}
    \widetilde{\mathbb{E}}_{\Lambda,M_{\Lambda}}\coloneqq&\left\{\phi_{\lambda,\mu^2_{\Lambda}}+\m{cl}_0\phi_{\lambda,\mu^2_{\Lambda}},\frac{1}{\mu_{\Lambda}}\widehat{D}_{Y;H;\Lambda}\phi_{\lambda,\mu^2_{\Lambda}}+\frac{1}{\mu_\Lambda}\m{cl}_0\widehat{D}_{Y;H;\Lambda}\phi_{\lambda,\mu_{\Lambda}^2},\right.\\
    &\left.\phi_{\lambda,\mu^2_{\Lambda}}-\m{cl}_0\phi_{\lambda,\mu^2_{\Lambda}},\frac{1}{\mu_{\Lambda}}\widehat{D}_{Y;H;\Lambda}\phi_{\lambda,\mu^2_{\Lambda}}-\frac{1}{\mu_\Lambda}\m{cl}_0\widehat{D}_{Y;H;\Lambda}\phi_{\lambda,\mu_{\Lambda}^2}\right\}
\end{align*}
in which $\widehat{D}_0$ reads 
\begin{align}
\label{ODE}
    \widehat{D}_{0;\Lambda,M_{\Lambda}}=& \left(\begin{array}{cc}
         1& 0 \\
         0& -1
    \end{array}\right)\otimes\left(\begin{array}{cc}
         \partial_r+\frac{m-1}{2r}&0  \\
         0&\partial_r+\frac{m-1}{2r}
    \end{array}\right)+\left(\begin{array}{cc}
         1& 0 \\
         0& 1
    \end{array}\right)\otimes\left(\begin{array}{cc}
         \frac{\lambda}{r}&\mu_{\Lambda} \\
          \mu_{\Lambda}&-\frac{\lambda}{r}
    \end{array}\right)
\end{align} 

and hence, we are looking for solutions of 

\begin{align*}
    \left[\left(\begin{array}{cc}
         \partial_r+\frac{m-1}{2r}&0  \\
         0&\partial_r+\frac{m-1}{2r}
    \end{array}\right)+\left(\begin{array}{cc}
         \frac{\lambda}{r}&\mu_{\Lambda} \\
          \mu_{\Lambda}&-\frac{\lambda}{r}
    \end{array}\right)\right]\left(\begin{array}{c}
         f_+  \\
         f_- 
    \end{array}\right)&=0\\
        \left[\left(\begin{array}{cc}
         \partial_r+\frac{m-1}{2r}&0  \\
         0&\partial_r+\frac{m-1}{2r}
    \end{array}\right)-\left(\begin{array}{cc}
         \frac{\lambda}{r}&\mu_{\Lambda} \\
          \mu_{\Lambda}&-\frac{\lambda}{r}
    \end{array}\right)\right]\left(\begin{array}{c}
         f^-  \\
         f^+ 
    \end{array}\right)&=0
\end{align*}
for $f_-,f_+,f^-,f^+\in L^2_{loc}(\mathbb{R}_{\geq 0},\mathbb{E}_{\Lambda,M_\Lambda}
)$.

\begin{itemize}
    \item[$\boxed{\mu_{\Lambda}=0}$:] In this case the ODE decouples 
\begin{align}
\label{ODEmu=0+}
    \left[\left(\begin{array}{cc}
         \partial_r+\frac{m-1}{2r}&0  \\
         0&\partial_r+\frac{m-1}{2r}
    \end{array}\right)+\left(\begin{array}{cc}
         \frac{\lambda}{r}&0 \\
          0&-\frac{\lambda}{r}
    \end{array}\right)\right]\left(\begin{array}{c}
         f_+  \\
         f_- 
    \end{array}\right)&=0
\end{align}
with solutions 
\begin{align*}
         f_\pm(r)=f^{\pm}(r)=c_\pm\cdot r^{\frac{1-m}{2}\mp\lambda}.
\end{align*}
\item[$\boxed{\lambda,\mu_{\Lambda}\neq 0}$:] In this case, the decoupled equations coincide for $f_\pm$ and $f^\pm$. Hence, we will write $f^\pm_\pm$ to indicate both. We decouple this system to
\begin{align}
\label{ODEdecoupled+}
    r^2\partial_r^2f^+_++(m-1)r\partial_rf^+_+-\left(\frac{m-1}{2}-\lambda-\left(\frac{m-1}{2}\right)^2+\lambda^2+\mu_{\Lambda}^2r^2\right)f^+_+=&0\\
    \label{ODEdecoupled-}
    r^2\partial_r^2f^-_-+(m-1)r\partial_rf^-_--\left(\frac{m-1}{2}+\lambda-\left(\frac{m-1}{2}\right)^2+\lambda^2+\mu_{\Lambda}^2r^2\right)f^-_-=&0.
\end{align}
Notice, that $f_\pm$ and $f^\pm$ satisfy the same differential equation.\\

The solutions $f^\pm_\pm(r)$ can be expressed as real-valued functions using modified Bessel functions $I_\nu$ and $K_\nu$ as follows:
\begin{align*}
    f_+(r) &= r^{1 - \frac{m}{2}} \left( c^I_- \cdot I_{|\lambda + \frac{1}{2}|}(|\mu_\Lambda| r) + c^K_- \cdot K_{|\lambda + \frac{1}{2}|}(|\mu_\Lambda| r) \right), \\
    f_-(r) &= r^{1 - \frac{m}{2}} \left( -c^I_- \cdot I_{|\lambda-\frac{1}{2}|}(|\mu_\Lambda| r) + c^K_- \cdot K_{|\lambda-\frac{1}{2}|}(|\mu_\Lambda| r) \right), \\
    f^+(r) &= r^{1-\frac{m}{2}} \left( c^+_I \cdot I_{|\lambda + \frac{1}{2}|}(|\mu_\Lambda| r) + c_K^+ \cdot K_{|\lambda + \frac{1}{2}|}(|\mu_\Lambda| r) \right),\\
    f^-(r) &= r^{1-\frac{m}{2}} \left( -c_I^+ \cdot I_{|\lambda-\frac{1}{2}|}(|\mu_\Lambda| r) + c_K^+ \cdot K_{|\lambda-\frac{1}{2}|}(|\mu_\Lambda| r) \right),
\end{align*}
where $c_I^\pm, c_K^\pm \in \mathbb{R}$ are constants determined by boundary or matching conditions.
\end{itemize}

\begin{rem}
    Here $I_\nu(z)$ and $K_\nu(z)$ are solutions to the modified Bessel equation
    \begin{align*}
        z^2\frac{\m{d}^2w}{\m{d}z^2}+z\frac{\m{d}w}{\m{d}z}-(z^2+\nu^2)w=0,
    \end{align*}
    which are real-valued for $z\in \mathbb{R}_{>0}$ and satisfy 
    \begin{align*}
        I_\nu(z)\sim_{z\to 0}&\frac{\left(\frac{1}{2}z\right)^\nu}{\Gamma(\nu+1)}&\und{1.0cm}&I_\nu(z)\sim_{z\to \infty}\frac{e ^z}{\sqrt{2\pi z}}\\
        K_\nu(z)\sim_{z\to 0}&\frac{\Gamma(\nu)\left(\frac{1}{2}z\right)^{-\nu}}{2}&\und{1.0cm}&K_\nu(z)\sim_{z\to \infty}\frac{e ^{-z}\sqrt{\pi }}{\sqrt{2z}}.
    \end{align*}
\end{rem}

Let $\delta(\widehat{E}_{0})$ denote the CF-degree\footnote{The degree is given by an automorphism of $\widehat{E}_{0}$ that restricts to scalar multiplication by $r^{\delta(\widehat{E}_{0;i})}$ on each $r$-homogenous summand of $\widehat{E}_{0}$.} of the bundle $\widehat{E}_{0}$, i.e. the power
\begin{align*}
    |\Psi|_{h_{r;Y}}=r^{\delta(\widehat{E}_{0})}\cdot|\Psi|_{h_{1;Y}}.
\end{align*}

Given homogeneous elements $\Phi_{\Lambda,M_{\Lambda}}(r)$, we know that 
\begin{align*}
    |\Psi_{\Lambda,M_{\Lambda}}|_{h_{0}}(r)=r^{1+\delta(\widehat{E}_{0})-\frac{m}{2}}\left|c_I^\mp\cdot I_{|\lambda\mp\frac{1}{2}|}(|\mu_{\Lambda}| r)+c_K^\mp \cdot K_{|\lambda\mp\frac{1}{2}|}(|\mu_{\Lambda}|r)\right|.
\end{align*}

\subsubsection{Right-Inverse of $B_{\Lambda,\mu_{\Lambda}}$}
\label{Right-Inverse of BLambdamu}

We construct a right-inverse for the normal operator $\widehat{D}_0$ explicitly using variation of parameters associated to the Bessel-type system. Hereby we follow the work of \cite[Sec. 2.2.2]{yang2007dirac} and \cite[(3.21),(3.22)]{Albin2016Index}.\\

Let $\Sigma=\sum_{\lambda,\mu_\Lambda} g_{\lambda,\mu_{\Lambda}}(r)\Phi_{\lambda,\mu_{\Lambda}}\in \Gamma_{cc}(N_0,\widehat{E}_0)$. Further let $\vec{g}\in C^\infty_{cc}(\mathbb{R}_{\geq 0},\mathbb{E}_{\Lambda,M_\Lambda})$ denote the $\mathbb{E}_{\Lambda,M_\Lambda}$-component of $\Sigma$.\\

We will continue to define an operator $\widehat{R}_0$ on each $C^\infty_{cc}(\mathbb{R}_{\geq 0},\mathbb{E}_{\Lambda,M_\Lambda})$ such that 
\begin{align*}
    \widehat{D}_0\widehat{R}_0\Sigma=\Sigma.
\end{align*}

\begin{defi}
 We define the right inverse on $C^\infty_{cc}(\mathbb{R}_{\geq 0},\mathbb{E}_{0,\mu_0})$ by
\begin{align}
\label{R0mu}
    \widehat{R}_{0;0;\mu_0}(\vec{g})(r)\coloneqq &\left(\begin{array}{c}
         r^{\frac{1-m}{2}}e^{-\mu_0 r}\int_0^r\xi^{\frac{m-1}{2}}e^{\mu_0 r}(g_1(\xi)+g_2(\xi))\m{d}\xi  \\
         -r^{\frac{1-m}{2}}e^{\mu_0 r}\int^\infty_r \xi^{\frac{m-1}{2}}e^{\mu_0 r}(g_1(\xi)-g_2(\xi))\m{d}\xi.
    \end{array}\right)
\end{align}
On $C^\infty_{cc}(\mathbb{R}_{\geq 0},\mathbb{E}_{\Lambda,0})$ we define it by 
\begin{align}
\label{Rlambda0}
    \widehat{R}_{0;\Lambda;0}(\vec{g})(r)\coloneqq &\left(\begin{array}{c}
         r^{\frac{1-m}{2}-\lambda}\int_0^r\xi^{\frac{m-1}{2}+\lambda}g_+(\xi)\m{d}\xi  \\
         -r^{\frac{1-m}{2}+\lambda}\int^\infty_r \xi^{\frac{m-1}{2}-\lambda}g_-(\xi)\m{d}\xi
    \end{array}\right)
\end{align}
In case $\lambda,\mu_{\Lambda}\neq 0$ let 
\begin{align*}
    \m{F}(r)=&r^{1-m/2}\left(\begin{array}{cc}
         I_{|\lambda+\frac{1}{2}|}(|\mu_\Lambda|r)&K_{|\lambda+\frac{1}{2}|}(|\mu_\Lambda|r)  \\
         -I_{|\lambda-\frac{1}{2}|}(|\mu_\Lambda|r)& K_{|\lambda-\frac{1}{2}|}(|\mu_\Lambda|r)
    \end{array}\right)
\end{align*}
denote the fundamental solution to \eqref{ODEdecoupled+} and \eqref{ODEdecoupled-}. Further let $\vec{h}(r)$ denote a solution to
\begin{align*}
    \m{F}(r)\partial_r\vec{h}(r)=\vec{g}(r)\Leftrightarrow \partial_r\vec{h}(r)=\m{F}^{-1}(r)\vec{g}(r).
\end{align*}
We can compute $\vec{h}(r)$ to be 
\begin{align*}
    \vec{h}(r)=-\left(\begin{array}{c}
           \int^\infty_r\left(K_{|\lambda+\frac{1}{2}|}(|\mu_{\Lambda}|\xi)g_-(\xi)-K_{|\lambda-\frac{1}{2}|}(|\mu_{\Lambda}|\xi)g_+(\xi)\right)\xi^{m/2}\m{d}\xi\\
          \int^r_0\left(I_{|\lambda+\frac{1}{2}|}(|\mu_{\Lambda}|\xi)g_-(\xi)-I_{|\lambda-\frac{1}{2}|}(|\mu_{\Lambda}|\xi)g_+(\xi)\right)\xi^{m/2}\m{d}\xi\\
    \end{array}\right)
\end{align*}
The right-inverse to $\widehat{D}_{0;\Lambda,M_{\Lambda}}$ is given by  
   \begin{align}    
    \label{Rlambdamu}
    \widehat{R}_{0;\Lambda,M_{\Lambda}}(\vec{g})(r)\coloneqq&\m{F}(r)\vec{h}(r).
\end{align}
\end{defi}

\begin{rem}
\label{rightinversesatisfiesboubdary}
    Notice, that the $\vec{h}(0)=-\left(\begin{array}{c}
     *  \\
     0 
\end{array}\right)$.
\end{rem}

Summarising the above, the following holds.

\begin{prop}\cite[(3.23)]{albin2023index}
\label{rightinverse}
   Let $\Sigma\in C^{k,\alpha}_{cc}(N_0,\widehat{E}_0)$. Then $\widehat{R}_0\Sigma\in C^{k+1,\alpha}_{cc}(N_0,\widehat{E}_0)$ and
    \begin{align*}
        \widehat{D}_{0;\Lambda,M_\Lambda}\widehat{R}_{\Lambda,M_{\Lambda}}\Sigma=\Sigma.
    \end{align*}
\end{prop}

\subsection{e- vs. CF-Calculus \& Hölder vs. Sobolev}  
\label{e- vs CF-Calculus and Hölder vs. Sobolev}  

In this section, we will discuss the challenges and significance of selecting an appropriate functional analytic framework for establishing a uniform elliptic theory for Dirac operators on orbifold resolutions. As outlined in Section \ref{Linear Gluing}, the strategy involves decomposing the orbifold resolutions into fundamental components, i.e. the orbifold itself and an asymptotically conical fibration resolving the normal cone bundle. These components are interconnected with the normal cone bundle serving as the anti-gluing space. Consequently, it is essential to construct functional analytic realisations of the operators $\widehat{D}_0$, $D$ and $\widehat{D}^t_{\zeta}$, such that the realisation of $\widehat{D}_0$ defines an isomorphism, ensuring that the induced realisations of $\widehat{D}^t_{\zeta}$ and $D$ are Fredholm.\\  

Previous works by Schulze, Mazzeo, and Melrose \cite{Schulze1991Pseudo,Schulze1998Boundary,mazzeo1991elliptic,mazzeo2014elliptic,mazzeo2018fredholm} introduced Sobolev spaces realising extensions of the minimal $L^2$-domain of uniform elliptic iterated edge differential operators, leading to the development of the edge-calculus. However, in the context of Dirac operators on orbifold resolutions, the usage of edge-calculus resembles \textit{shooting sparrows with cannons}.\\

As shown in the previous section, the kernel of the normal operator admits polyhomogeneous expansions of the form  
\begin{align*}
    \Phi_{\lambda,0}(r) \sim& r^{\pm\lambda+\frac{1-m}{2}+\delta(\widehat{E}_{0})}\\
    \Phi_{0,\mu_0}(r) \sim& e^{\pm\mu_{\Lambda} r}r^{\frac{1-m}{2}+\delta(\widehat{E}_{0})}\\ \Phi_{\lambda,\mu_\Lambda}(r) \sim& r^{1-\frac{m}{2}\delta(\widehat{E}_{0})}Z_{|\lambda\pm \frac{1}{2}|}(|\mu_\Lambda|r),\hspace{1.0cm}\text{for }Z=I,K.
\end{align*}
This explicit control over the asymptotic behaviour of kernel and cokernel elements will enable us to define \say{new function spaces} which appear more suitable for the study of Dirac operators on orbifold resolutions.\\

The main challenges in defining suitable functional analytic realisations of the Dirac operators $\widehat{D}_0$, $D$ and $\widehat{D}^t_\zeta$ can be summarised as follows:  

\begin{itemize}
    \item[1.] \textbf{Local Control and Rescaling Sensitivity:}  
    Gluing problems require precise local control via $C^{k,\alpha}$-estimates of pre-glued solutions which must also be sensitive to rescaling/blow-up arguments.  

    \item[2.] \textbf{Infinite-Dimensional Kernel of $\widehat{D}_0$:}  
    The normal operator $\widehat{D}_0$ has an infinite-dimensional kernel comprising solutions to \eqref{ODEdecoupled-} and \eqref{ODEdecoupled+} for $\lambda, \mu_{\Lambda} \in \mathbb{R}^2$. Therefore, the choice of function spaces must effectively \say{exclude} these kernel elements, up to a controllable, finite dimensional subspace.  

    \item[3.] \textbf{Fibred Structure Sensitivity:}  
    Given the asymptotic behaviour:  
    \begin{align*}
            \nabla^{1,0} \Phi_{\Lambda, M_\Lambda}(r) \sim& \mathcal{O}(\Phi_{\Lambda,M_\Lambda}(r)),\\
            \nabla^{0,1} \Phi_{\Lambda, M_\Lambda}(r) \sim &\mathcal{O}(\Phi_{\Lambda,M_\Lambda}(r))(1+r^{-1}),
    \end{align*}
    the function spaces must be adapted to the fibred structures of $N_0$ and $N_\zeta$ to facilitate a meaningful adiabatic limit analysis.  

    \item[4.] \textbf{Unbounded Operators on Weighted Spaces:}  
    For functions $|\Phi| \sim r^\beta$, we have:  
    \begin{align*}
        \widehat{D}_0 \Phi \sim r^\beta + r^{\beta-1}.
    \end{align*}
    This behaviour implies that the analytic realisations act as unbounded operators on the corresponding weighted function spaces.  
\end{itemize}

To address these challenges, we propose the following strategy:  

\begin{itemize}
    \item[$\Rightarrow$1.] \textbf{Hölder vs. Sobolev Spaces:}  
    Due to the need for fine local control, we will use local Hölder norms.
    \item[$\Rightarrow$2.] \textbf{Weighted Function Spaces and Boundary Conditions:}  
    We define weighted function spaces based on solutions to $\widehat{D}_0$ and impose boundary conditions to restrict the infinite-dimensional kernel. Both the choice of weights and the boundary conditions will ensure that a \say{half-space} of kernel elements of $\widehat{D}_0$ is excluded. Combing both ensures that the space of kernel elements of $\widehat{D}_0$ is finite dimensional. Specifically, we focus on weights derived from special solutions of \eqref{ODEdecoupled-} and \eqref{ODEdecoupled+} for $\mu_\Lambda = 0$, corresponding to fibrewise conical weight functions. The remaining kernel will be controlled by imposing APS-type boundary conditions.
    \item[$\Rightarrow$3.] \textbf{Transition from e- to CF-Norms:}  
    Instead of edge-norms, we adopt CF-norms ensuring that smooth homogeneous elements are appropriately incorporated. Hence, only vertical derivative are weighted by an increased weight.
    \item[$\Rightarrow$4.] \textbf{Weighted Hölder to Graph Norms:}  
    The weighted Hölder norms are replaced by graph norms to guarantee that the operators $D, \widehat{D}^t_\zeta, \widehat{D}_0$ are left semi-Fredholm.  
\end{itemize}

\section{Uniform Elliptic Theory for Conically Fibred Dirac Operators}
\label{Uniform Elliptic Theory for Conically Fibred Dirac Operators}

In this subsection, we introduce weighted function spaces on conically fibred (CF) manifolds. These spaces are tailored to control the homogeneous kernel elements of the model operator $\widehat{D}_0$ that exhibit conical behaviour. The weights are defined by solutions of the ODE \eqref{ODElambda0} for real parameters $\beta \in \mathbb{R}$. Consequently, a solution to \eqref{ODElambda0} lies in the function space if and only if its $\lambda$ is related to the rate $\beta$. A choice of weight function effectively cuts the infinite-dimensional kernel in half, by excluding exponential and polynomial growth at $r\to \infty$. In order to eliminate the remaining half, we impose additional APS-type boundary conditions.\\

Combining these two ingredients yields a realisation of the normal operator $\widehat{D}_0$ as an isomorphism between Banach spaces. The constructed function spaces, together with the boundary conditions, dictate the ACF- and CFS-realisations of $\widehat{D}^t_\zeta$ and $D$ and ultimately ensure their Fredholm property.

\subsection{Weighted Function Spaces on Conically Fibrations}
\label{Weighted Function Spaces on Conically Fibrations}

In the following section we will introduce weighted function spaces of sections of Dirac bundles on $(N_0,g_0)$ adapted to the CF geometry.

\begin{nota}
We will denote by $w_{\m{CF}}\coloneqq r:N_0\rightarrow [0,\infty)$ and set
\begin{align*}
    w^{-\beta}_{\m{CF}}=w_{\m{CF};\beta}\und{1.0cm}w_{\m{CF}}(x,y)=\m{min}(w_{\m{CF}}(x),w_{\m{CF}}(y)).
\end{align*}
where $\beta\in C^\infty(S)$.
\end{nota}

\begin{nota}
Let $\m{word}^+_{i,j}(A,B)$ denote the sum of all words of length $j$ in $A$ and $B$ counting $i$-times the letter $A$. In particular, $\m{word}^+_{i,j}(A,B)=\m{word}^+_{j-i,j}(B,A)$. 
\end{nota}

\begin{defi}
 We define the space $ \Gamma^k_{\m{CF};\beta}(N_0,\widehat{E}_{0})$ to be the space of all $C^k$-sections $\Phi$ such that 
\begin{align*}
    \left|\left|w_{\m{CF};\beta}\left\{(\nabla^{\otimes_0})^j\right\}\Phi\right|\right|_{C^0}<\infty
\end{align*}
where 
\begin{align*}
    w_{\m{CF};\beta}\left\{(\nabla^{\otimes_0})^j\right\}:=\bigoplus_{i=0}^jw_{\m{CF};(\beta-j+i)}\m{word}^+_{i,j}(\nabla^{\otimes_0}_H,\nabla^{\otimes_0}_V) 
\end{align*}
and equip the spaces $ \Gamma^k_{\m{CF};\beta}(N_0,\widehat{E}_{0})$ with the metric
\begin{align*}
    \left|\left|\Phi\right|\right|_{C^{k}_{\m{CF};\beta}}=&\sum_{j=0}^k\left|\left|w_{\m{CF};\beta}\left\{(\nabla^{\otimes_0})^j\right\}\Phi\right|\right|_{C^0}
\end{align*}
The space $ \Gamma_{\m{CF};\beta}(N_0,\widehat{E}_{0})$ is defined to be the intersection of all $ \Gamma^k_{\m{CF};\beta}(N_0,\widehat{E}_{0})$ for all $k$. Furthermore, we define the $w_{\m{CF};\beta}$-weighted Sobolev $W^{k,p}_{\m{CF};\beta}(N_0,\widehat{E}_{0})$ and Hölder spaces $  C^{k,\alpha}_{\m{CF};\beta}(N_0,\widehat{E}_{0})$ as the completions with respect to the norms
\begin{align*}
    \left|\left|\Phi\right|\right|^p_{W^{k,p}_{\m{CF};\beta}}=&\sum_{j=0}^k\int_{X}|w_{\m{CF};\beta}\left\{(\nabla^{\otimes_0})^j\right\}\Phi|^p_{h_0,g_0} w_{\m{CF};m}\m{vol}_{g_0}\\
    \left|\left|\Phi\right|\right|_{C^{k,\alpha}_{\m{CF};\beta}}=&\sum_{j=0}^k\left|\left|w_{\m{CF};\beta}\left\{(\nabla^{\otimes_0})^j\right\}\Phi\right|\right|_{C^0}+[(\nabla^{\otimes_0})^j\Phi]_{C^{0,\alpha}_{\beta-j;CF}}\\
    [(\nabla^{\otimes_0})^j\Phi]_{C^{0,\alpha}_{\beta-j;CF}}=&\underset{B_{g_0}}{\m{sup}}\left\{\bigoplus_{i=0}^jw_{\m{CF};(\beta-j+i-\alpha)}(x,y)\cdot\right.\\
    &\left.\frac{\left|\m{word}^+_{i,j}(\nabla^{\otimes_0}_H,\nabla^{\otimes_0}_V) \Phi(x)-\Pi^{x,y}\m{word}^+_{i,j}(\nabla^{\otimes_0}_H,\nabla^{\otimes_0}_V) \Phi(y)\right|_{h_0,g_0}}{d_{g_0}(x,y)^\alpha}\right\}
\end{align*} 
where $B_{g_0}(U)\coloneqq\{x,y\in U|\,x\neq y,\,\m{dist}_{g_0}(x,y)\leq w_{\m{CF};(\beta-j+i-\alpha)}(x,y)\}\subset U\times U$ and $\Pi^{x,y}$ denotes the parallel transport along the geodesic connecting $x$ and $y$.
\end{defi}

\begin{rem}
All $W^{k,p}_{\m{CF};\beta}$, $ \Gamma^k_{\m{CF};\beta}$ and $  C^{l,\alpha}_{\m{CF};\beta}$ are Banach spaces, and $H^{k}_{\m{CF};\beta}=W^{k,2}_{\m{CF};\beta}$ are Hilbert spaces. However, $\Gamma_{\m{CF};\beta}$ are not Banach spaces.
\end{rem}

\begin{lem}
    Let $\frac{1}{p}+\frac{1}{q}=1$ and $m=\m{codim}(S)$. The $L^2$-product induces a perfect paring 
    \begin{align*}
        L^{p}_{\m{CF};\beta}\cong(L^{q}_{\m{CF};-\beta-m})^*.
    \end{align*}
    
\end{lem}

\begin{rem}
Given a weight function $w_{0}$ for $(N_0, g_{0})$ we define new weight functions on $(N_0,  g^t_{0})$ by 
\begin{align*}
    w_{\m{CF};t}:=\delta_t^*w_{0}=t\cdot w_{0}=t\cdot w_{\m{CF}}.
\end{align*} 
\end{rem}

\begin{lem}[Technical Lemma]
\label{technicallemmachi}
Let $\chi_i$ be defined as in Section \ref{Smooth Gromov-Hausdorff-Resolutions}. Then  
    \begin{align*}
        \left|\left|\chi_i\Phi\right|\right|_{C^{k,\alpha}_{\m{CF};\beta+\omega}}\lesssim&\epsilon^{-\omega}\left|\left|\Phi\right|\right|_{C^{k,\alpha}_{\m{CF};\beta}(N_0\backslash \m{Tub}_{(i-1)\epsilon}(S))}\\
        \left|\left|(1-\chi_i)\Phi\right|\right|_{C^{k,\alpha}_{\m{CF};\beta+\omega}}\lesssim&(1+\epsilon^{-\omega})\left|\left|\Phi\right|\right|_{C^{k,\alpha}_{\m{CF};\beta}( \m{Tub}_{i\epsilon}(S))}\\
        \left|\left|\m{cl}_{g_0}(\m{d}\chi_i)\Phi\right|\right|_{C^{k,\alpha}_{\m{CF};\beta+\omega}}\lesssim&\epsilon^{-1-\omega}\left|\left|\Phi\right|\right|_{C^{k,\alpha}_{\m{CF};\beta}(A_{(i-1)\epsilon,i\epsilon}(S))}.
    \end{align*}
\end{lem}

\begin{proof}
Notice that $\nabla^{\oplus_0}\chi_i\Phi=\chi_i\nabla^{\oplus_0}\Phi+(\partial_r\chi_i)\m{d}r\otimes \Phi$ and
\begin{align*}
    \left|\left|\partial_r^{j}{\chi_i}\right|\right|_{C^{0,\alpha}_{\omega-j}}=& \sup_{r\in[\epsilon (i-1),\epsilon i]}r^{j-\omega}\epsilon^{-j}\left|(\partial_r^{j}{\chi})(\epsilon^{-1}r-i)\right|\\
    &+\sup_{|x-y|\in[0,\epsilon ]}\m{min}\{x,y\}^{-\omega+j+\alpha}\epsilon^{-j}\frac{\left|(\partial_r^{j}{\chi})(\epsilon^{-1}x-i)-(\partial_r^{j}{\chi})(\epsilon^{-1}y-i)\right|}{|x-y|^\alpha}\\
    \leq&\epsilon^{-\omega}\left|\left|\partial^j_r\chi\right|\right|_{C^{0,\alpha}}
\end{align*}
suffice to deduce the statements.
\end{proof}

We need to point out two key observations. Let $\Phi\in C^{k,\alpha}_{loc}(N_0,\widehat{E}_{0})$. Then 
    \begin{align}
        \left|\left|\widehat{D}_0(1-\chi_i)\Phi\right|\right|_{C^{k,\alpha}_{\m{CF};\beta-1}}\lesssim& \left|\left|(1-\chi_i)\Phi\right|\right|_{C^{k+1,\alpha}_{\m{CF};\beta}}\\
        \left|\left|\widehat{D}_0\chi_i\Phi\right|\right|_{C^{k,\alpha}_{\m{CF};\beta-1}}\lesssim& \left|\left|\chi_i\Phi\right|\right|_{C^{k+1,\alpha}_{\m{CF};\beta-1}}.
    \end{align}

This observation manifests itself in the following lemma.

\begin{lem}
\label{boundedonsmallandbig}
The analytic realisation of the family of operators $\widehat{D}_0$ seen as a map 
\begin{align*}
    \widehat{D}_0:&\m{dom}_{  C^{k+1,\alpha}_{\m{CF};\beta}}(\widehat{D}_0)\subset   C^{k+1,\alpha}_{\m{CF};\beta}(N_0,\widehat{E}_{0})\rightarrow  C^{k,\alpha}_{\m{CF};\beta-1}(N_0,\widehat{E}_{0})
\end{align*}
is unbounded map of Banach spaces.
\end{lem}

In particular, if we define 
\begin{align*}
    C^{k+1,\alpha}_{\beta;\delta}\coloneqq \left\{\Phi\left|(1-\chi_i)\Phi\in C^{k+1,\alpha}_{\m{CF};\beta},\chi_i\Phi\in C^{k+1,\alpha}_{\m{CF};-\delta}\right.\right\}
\end{align*}
as in \cite[Lem. 5.5]{mazzeo1991elliptic}, then 
\begin{align*}
    \widehat{D}_0:C^{k+1,\alpha}_{\beta;1-\beta}\rightarrow C^{k,\alpha}_{\beta-1;1-\beta}.
\end{align*}
is a bounded map of Banach spaces.\\

The second observation goes as follows. Let us assume that we are in a conical setting. By a prior comment we know that all solutions $\widehat{D}_0\Phi=0$ are homogeneous $\sim r^\lambda$. If 
\begin{align*}
    \chi_i\cdot \Phi \in \Gamma_{\beta}(N_0,\widehat{E}_{0}),\,\,\text{if  }\,\lambda<\beta\und{1.0cm} (1-\chi_i)\cdot \Phi \in \Gamma_{\beta}(N_0,\widehat{E}_{0}),\,\,\text{if  }\,\beta<\lambda
\end{align*}
then
\begin{align*}
    (1-\chi_i)\cdot \Phi \notin \Gamma_{\beta}(N_0,\widehat{E}_{0}),\,\,\text{if  }\,\lambda<\beta\und{1.0cm} \chi_i\cdot \Phi \notin \Gamma_{\beta}(N_0,\widehat{E}_{0}),\,\,\text{if  }\,\beta<\lambda.
\end{align*}
In the conically fibred setup this no longer holds, as solutions to $\widehat{D}_0\Phi=0$ are homogeneous of $\Phi\sim_{r\to \infty} e^{\pm|\mu_{\Lambda}| r}$. Hence, imposing a polynomial decay only \say{kills} half of the infinite dimensional kernel. As imposing conical weights implies that there are no solutions involving
\begin{align*}
   r^{\frac{1-m}{2}}e^{|\mu_\Lambda|r}\und{1.0cm } r^{1-m/2}I_{|\lambda\pm1/2|}(|\mu_\Lambda|r).
\end{align*}

\subsection{The Restriction and Extension Map}
\label{The Restriction and Extension Map}
In the following we will impose boundary conditions on the function spaces realising the normal operator. This should be seen as an analogy to APS boundary condition in the case of Dirac operators on manifolds with boundaries (i.e. $m=1$). These boundary conditions ensure that the operator $\widehat{D}_0$ can be realised as an isomorphism, which is crucial in linear gluing.\\

In Section \ref{Polyhomogenous Solutions of the Normal Operator} we have solved $\widehat{D}_0\Phi=0$ using the spectral decomposition of $L^2(Y,\widehat{E}_Y)$ into 
\begin{align*}
    L^2(Y,\widehat{E}_Y)=\bigoplus_{|\lambda|\in|\sigma(D_Y)|}L^2(S,\mathcal{E}_\Lambda)=\bigoplus_{|\lambda|\in|\sigma(D_Y)|}\bigoplus_{|\mu_{\Lambda}|\in|\sigma(B_{\Lambda;H})|}\mathbb{E}_{\Lambda,M_{\Lambda}}.
\end{align*}

\begin{defi}
\label{CFresandCFext}
    We define the \textbf{restriction map}
    \begin{align*}
        \mathrm{res}^\epsilon_{\m{CF}} : C^{k,\alpha}_{\mathrm{loc}}(N_0, \widehat{E}_0) \longrightarrow C^{k,\alpha}_{r=5\epsilon/2}(Y, \widehat{E}_Y), \qquad 
        \Phi \mapsto \Phi|_{r = \tfrac{5\epsilon}{2}}.
    \end{align*}
    Its right inverse, called the \textbf{extension map}, is given by
    \begin{align*}
        \mathrm{ext}_{\m{CF}} : C^{k,\alpha}(Y, \widehat{E}_Y) \longrightarrow C^{k,\alpha}_{\mathrm{loc}}(N_0, \widehat{E}_0), \qquad
        \Phi_{\Lambda, M_\Lambda} \mapsto f_{\Lambda, M_\Lambda}(r) \Phi_{\Lambda, M_\Lambda},
    \end{align*}
    where $f_{\Lambda, M_\Lambda}(r)$ denotes the appropriate solution to the model ODE \eqref{ODE} corresponding to the spectral mode $(\Lambda, M_\Lambda)$.
\end{defi}

\begin{rem}
    By construction, the extension map satisfies
    \begin{align*}
        \widehat{D}_0 \circ \mathrm{ext}_{\m{CF}} \, \Phi = 0
    \end{align*}
    for all $\Phi \in C^{k,\alpha}(Y, \widehat{E}_Y)$.
\end{rem}

\begin{defi}
\label{APSCFdefi}
    Let 
    \begin{align*}
        F_{\leq 0}=\left\{\left.\sum_{fin.}\Phi_{\lambda,\mu_{\Lambda}}\right|\mu_{\Lambda}\leq 0\right\}\subset \bigoplus_{|\lambda|\in|\sigma(D_Y)|} L^2_{(-\infty,0]}(S,\mathcal{E}_{\Lambda})
    \end{align*}
    be the sum of finite Fourier modes of nonpositive $\widehat{D}_{Y;H;\Lambda}$-eigenvalues, i.e. $L^2_{(-\infty,0]}(S,\mathcal{E}_{\Lambda})$. We define the space
    \begin{align*}
          C^{k,\alpha}(Y,\widehat{E}_Y)_{\leq 0}=\overline{F_{\leq 0}}^{||.||_{  C^{k,\alpha}}}
    \end{align*}
    as the completion of $F_{\leq 0}$ and the Banach subspace 
    \begin{align*}
         C^{k+1,\alpha}_{\m{CF};\beta}(N_0,\widehat{E}_{0};\m{APS})\coloneqq C^{k+1,\alpha}_{\m{CF};\beta}(N_0,\widehat{E}_{0})\cap(\m{res}^\varepsilon)^{-1}C^{k+1,\alpha}_{r=5\epsilon/2}(Y,\widehat{E}_Y)_{\leq 0}.
    \end{align*}
\end{defi}

\begin{lem}
\label{differentchoicesofrestriciton}
    For all choices of $\varepsilon$ the spaces $C^{k+1,\alpha}_{\m{CF};\beta}(N_0,\widehat{E}_{0};\m{APS})$ are mutually equivalent.
\end{lem}

\begin{proof}
The scaled dilation $\left(\frac{\varepsilon}{\varepsilon'}\right)^{\delta(\widehat{E}_{0})}\delta_{\frac{\varepsilon'}{\varepsilon}}$ defines an isometry of Banach spaces.
\end{proof}

\begin{defi}
    We define the \textbf{set of critical rates} of $\widehat{D}_0$, $D$ and $\widehat{D}^t_{\zeta}$ to be 
    \begin{align*}
        \mathcal{C}(\widehat{D}_0)\coloneq  \left\{\lambda+\frac{m-1}{2}-\delta(\widehat{E}_{0})\in\sigma(\widehat{D}_{Y;V})\right\}.
    \end{align*}
\end{defi}

\begin{lem}
\label{extensionlemma}
    The intersection of the preimage of the extension map and the image of the restriction map
    \begin{align*}
    (\mathrm{ext}_{\m{CF}})^{-1}\left( C^{k,\alpha}_{\m{CF};\beta}(N_0,\widehat{E}_0) \right) \cap C^{k,\alpha}(Y, \widehat{E}_Y)_{\leq 0}
    \end{align*}
    is nontrivial if and only if $\beta \in \mathcal{C}(\widehat{D}_0)$. In that case, it is given by
    \begin{align*}
    \m{ker}\left(\widehat{D}_{Y;H;\beta+ \frac{m-1}{2} - \delta(\widehat{E}_0)}\right)=\left\{ \Phi_{\beta + \frac{m-1}{2} - \delta(\widehat{E}_0), 0} \right\}.
    \end{align*}
\end{lem}

\begin{proof}
    Solutions to the model ODE \eqref{ODE} of the form
    \begin{align*}
        r^{1 - \tfrac{m}{2}} K_{|\lambda \pm \frac{1}{2}|}(|\mu_{\Lambda}| r)
    \end{align*}
    decay exponentially like $e^{-|\mu_{\Lambda}| r}$ as $r \to \infty$ and corresponds to the positive spectral component of the horizontal Dirac operator
    \begin{align*}
        \widehat{D}_{Y;H} : L^2(Y, \widehat{E}_Y) \to L^2(Y, \widehat{E}_Y).
    \end{align*}
    The CF boundary condition enforces that all sections in $C^{k+1,\alpha}_{\m{CF};\beta}(\widehat{D}_0;\m{APS})$ restrict to the nonpositive spectral subspace of $\widehat{D}_{Y;H}$. Extensions of nonpositive spectral modes correspond to solutions of the form
    \begin{align*}
        r^{\frac{1 - m}{2} \pm |\lambda|}, \qquad 
        r^{\frac{1 - m}{2}} e^{|\mu_0| r}, \qquad 
        r^{1 - \frac{m}{2}} I_{|\lambda \pm \frac{1}{2}|}(|\mu_{\Lambda}| r).
    \end{align*}
    Among these, only the first type can potentially lie in the weighted space $C^{k+1,\alpha}_{\m{CF};\beta}$ and this occurs if and only if
    \begin{align*}
        \beta = \pm |\lambda| + \tfrac{1 - m}{2} - \delta(\widehat{E}_0),
    \end{align*}
    i.e. precisely when $\beta \in \mathcal{C}(\widehat{D}_0)$. In this case, the only admissible extension is the one generated by the eigensections $\Phi_{\beta + \tfrac{m-1}{2} - \delta(\widehat{E}_0), 0}$, i.e. kernel of $\widehat{D}_{Y;H;\beta+ \tfrac{m-1}{2} - \delta(\widehat{E}_0)}$.
\end{proof}

\begin{rem}
\label{rem:extensionlemma_boundary}
    Lemma \ref{extensionlemma} plays a central role in the CF analysis, as it precisely identifies when boundary data on the link $Y$ can be extended to solutions on the total space. It captures the subtle relationship between decay rates and admissible boundary conditions. In particular, it explains why critical rates correspond to obstructions in solving the Dirac equation across the conically fibred end, and thereby it governs the structure of the kernel and the necessity of boundary conditions in the analytic setup.
\end{rem}

\subsection{Invertibility of Conically Fibred Dirac Operators}
\label{Invertibility of Conically Fibred Dirac Operators}

In this subsection, we develop an elliptic theory for the normal operator $\widehat{D}_0$ acting on sections of the conically fibred Dirac bundle 
$(\widehat{E}_0, \mathrm{cl}_{g_0}, h_0, \nabla^{\otimes_0})$. The main goal is to construct function spaces such that the realisation of the normal operator is invertible.

\begin{defi}
\label{APSCFDomaindefi}
We will denote the domain of $\widehat{D}_0$ by
\begin{align*}
    C^{k+1,\alpha}_{\m{CF};\beta}(\widehat{D}_0;\m{APS})\coloneqq\left\{\left.\Phi\in  C^{k+1,\alpha}_{\m{CF};\beta}(N_0,\widehat{E}_{0};\m{APS})\right|\widehat{D}_0\Phi\in   C^{k,\alpha}_{\m{CF};\beta-1}(N_0,\widehat{E}_{0})\right\}
\end{align*}
and equip it with the graph norm 
\begin{align*}
    \left|\left|\Phi\right|\right|_{  C^{k+1,\alpha}_{\m{CF};\beta}(\widehat{D}_0)}=\left|\left|\Phi\right|\right|_{C^{k+1,\alpha}_{\m{CF};\beta}}+\left|\left|\widehat{D}_0\Phi\right|\right|_{C^{k,\alpha}_{\m{CF};\beta-1}}.
\end{align*}
\end{defi}

\begin{cor}
    The analytic realisation
    \begin{align*}
        \widehat{D}_0 : C^{k+1,\alpha}_{\m{CF};\beta}(\widehat{D}_0;\m{APS}) \hookrightarrow C^{k,\alpha}(N_0, \widehat{E}_0)
    \end{align*}
    is injective for all $\beta \notin \mathcal{C}(\widehat{D}_0)$. If $\beta \in \mathcal{C}(\widehat{D}_0)$, then the operator $\widehat{D}_0$ has finite-dimensional kernel given by
    \begin{align*}
        \ker_{\m{CF};\beta}(\widehat{D}_0) \cong \ker\left( \widehat{D}_{Y;H;\beta + \tfrac{m-1}{2}-\delta(\widehat{E}_0)} \right).
    \end{align*}
\end{cor}

\begin{proof}
    By the above reasoning, we know that the kernel elements of $\widehat{D}_0$ corresponding to the solutions  
    \begin{align*}
    r^{\frac{1-m}{2}+\delta(\widehat{E}_0)} e^{|\mu_{\Lambda}| r} \quad \text{and} \quad r^{1 - \frac{m}{2}+\delta(\widehat{E}_0)} I_{|\lambda \pm \tfrac{1}{2}|}(|\mu_{\Lambda}| r)
    \end{align*}
    of \eqref{ODE} do not lie in $C^{k,\alpha}_{\m{CF};\beta}$. Using Lemma \ref{extensionlemma}, we deduce that the only possible solutions that may lie in $C^{k,\alpha}_{\m{CF};\beta}$ are those of the form
    \begin{align*}
    r^{\frac{1-m}{2} \pm |\lambda|+\delta(\widehat{E}_0)}. 
    \end{align*}
    These solutions belong to $C^{k,\alpha}_{\m{CF};\beta}$ if and only if $\beta = \frac{1-m}{2} \pm |\lambda|+\delta(\widehat{E}_0)$.
\end{proof}

\begin{lem}
    The space $C^{k+1,\alpha}_{\m{CF};\beta}(\widehat{D}_0;\m{APS})$ defined in Definition \ref{APSCFDomaindefi} is a Banach space and for $\delta>\beta$ there exists a compact embedding 
    \begin{align}
    \label{compactembeddingCF}
        C^{k+1,\alpha}_{\m{CF};\beta}(\widehat{D}_0;\m{APS})\hookrightarrow C^{0}_{\m{CF};\delta}(N_0,\widehat{E}_{0}).
    \end{align}
\end{lem}

\begin{proof}
We only need to show that the operator is closed. This follows from the fact that 
\begin{align*}
    \widehat{D}_{0;H}:\m{dom}_{  C^{k+1,\alpha}_{\m{CF};\beta}}(\widehat{D}_{0;H};\m{APS})\rightarrow  C^{k-1,\alpha}_{\m{CF};\beta-1}(N_0,\widehat{E}_{0})
\end{align*}
is closed and $ \m{dom}_{  C^{k+1,\alpha}_{\m{CF};\beta}}(\widehat{D}_{0;H})\subset\m{dom}(\widehat{D}_{0;V})$ with $\widehat{D}_{0;V}$ being bounded. Consequently, $\widehat{D}_0$ is closed.\\

By the Hölder embedding we obtain a compact embedding 
    \begin{align*}
        C^{k,\alpha}_{\m{CF};\beta}\hookrightarrow C^0_{\m{CF};\delta}.
    \end{align*}
    As $C^{k+1,\alpha}_{\m{CF};\beta}(\widehat{D}_0;\m{APS})\hookrightarrow C^{k+1,\alpha}_{\m{CF};\beta}\times C^{k,\alpha}_{\m{CF};\beta-1} $ is a sub-Banach space and 
    \begin{align*}
        C^{k+1,\alpha}_{\m{CF};\beta}\times C^{k,\alpha}_{\m{CF};\beta-1}\hookrightarrow C^0_{\m{CF};\delta}\times C^0_{\m{CF};\delta}\twoheadrightarrow C^0_{\m{CF};\delta}
    \end{align*}
    is a compact embedding, we conclude the statement.
\end{proof}

\begin{prop}[CF-Schauder Estimates]
\label{SchauderDotimes0}
For all $\Phi\in C^{k+1,\alpha}_{\m{CF};\beta}(\widehat{D}_0;\m{APS})$, the Schauder-estimate 
\begin{align}
   \left|\left|\Phi\right|\right|_{C^{k+1,\alpha}_{\m{CF};\beta}(\widehat{D}_0)} &\lesssim \left|\left| \widehat{D}_0\Phi\right|\right|_{C^{k,\alpha}_{\m{CF};\beta-1}}+\left|\left| \Phi\right|\right|_{C^{0}_{\m{CF};\beta}}
\end{align}
holds.
\end{prop}

\begin{proof}
Schauder estimates are local, and as such proving the statement on an open cover of $N_0$ suffices to prove the statement. We start with proving the statement for $k=0$.\\

Let $\mathcal{U}_{\varepsilon}$ denote a good cover of $S$ by geodesic balls of radius $\varepsilon$ and let $\nu_0^{-1}\mathcal{U}_{\varepsilon}$ denote the cover of $\m{Tub}_{2\epsilon}(S)$ by open subsets of the form 
\begin{align*}
    \nu_0^{-1}U_s\cong B_{\varepsilon}(s) \times B_{3\epsilon}(0)/\Gamma.
\end{align*}
We can do this in a way such that for each $0<c$ there exists an $0<\varepsilon$ such that 
\begin{align*}
    \left|\left|g_0-g_S(s)\oplus g_{0;V;s}\right|\right|_{C^{0,\alpha}}<c
\end{align*}
holds on $\mathcal{U}_{\varepsilon}$. As $g_0$ is uniform close to the product structures on $\nu_0^{-1}\mathcal{U}_{\varepsilon}$, so is the Hermitian Dirac bundle structure. Let $r=r(s,m)$ denote the radius of the second component of $B_{\varepsilon}(s)\times \mathbb{R}^m/\Gamma$. Take the sets $U_s=B_{\varepsilon/2}(s)\times B_{3\epsilon}(0)$ and $V_s=B_{\varepsilon}(s)\times B_{3\epsilon}(0)$. The local model of the unweighted CF Hölder norm is given by
\begin{align*}
    LHS=&\left|\left|\Phi\right|\right|_{C^{0}(U_s)}+r ^{\alpha}[\Phi]_{C^{0,\alpha}(U_s)}+\varepsilon\left|\left|\nabla^{H}\Phi\right|\right|_{C^{0}(U_s)}\\
    &+r\left|\left|\nabla^{V}\Phi\right|\right|_{C^{0}(U_s)}
    +\epsilon ^{\alpha+1}[\nabla^H\Phi]_{C^{0,\alpha}(U_s)}
    +r ^{\alpha+1}[\nabla^V\Phi]_{C^{0,\alpha}(U_s)}
\end{align*}
This can be bounded from the above by 
\begin{align*}
    LHS\lesssim&\left|\left|\Phi\right|\right|_{C^{0}(U_s)}+(\varepsilon^2+r^2 )^{\alpha/2}[\Phi]_{C^{0,\alpha}(U_s)}+\varepsilon\left|\left|\nabla^{H}\Phi\right|\right|_{C^{0}(U_s)}\\
    &+r\left|\left|\nabla^{V}\Phi\right|\right|_{C^{0}(U_s)}+(\varepsilon^2+r^2 )^{(\alpha+1)/2}
    [\nabla\Phi]_{C^{0,\alpha}(U_s)}\\
    \lesssim&(\varepsilon^2+r^2)^{1/2}\left|\left|\widehat{D}_0\Phi\right|\right|_{C^{0}(V_s)}+(\varepsilon^2+r^2)^{(\alpha+1)/2}[\widehat{D}_0\Phi]_{C^{0,\alpha}(V_s)}\\
    &+\left|\left|\Phi\right|\right|_{C^{0}(V_s)}\\
    \lesssim&RHS
\end{align*}
Here we use \cite[Prop. 7.7]{walpuski2012g_2} and \cite[Prop. 7.11]{walpuski2012g_2} where $RHS$ is given by 
\begin{align*}
    RHS= &r\left|\left|\widehat{D}_0\Phi\right|\right|_{C^{0}(V_s)}+r^{\alpha+1}[\widehat{D}_0\Phi]_{C^{0,\alpha}(V_s)}+\left|\left|\Phi\right|\right|_{C^{0}(V_s)}
\end{align*}
and hence, we deduce the local unweighted CF Schauder estimates. By now, multiplying both sides with $r^{-\beta}$ and by using that $S$ is compact, we can deduce the result. \\

By Lemma \ref{boundedonsmallandbig} the graph norm on $N_0\backslash \m{Tub}_{i\epsilon}(S)$ is uniform bounded above by 
\begin{align*}
    \left|\left|\chi_i\Phi\right|\right|_{C^{k+1,\alpha}_{\m{CF};\beta}(\widehat{D}_0)}\lesssim \left|\left|\chi_i\Phi\right|\right|_{C^{k+1,\alpha}_{\m{CF};\beta-1}}.
\end{align*}
Let $\nu_0^{-1}\mathcal{U}'_{\varepsilon}$ denote the cover of $N_0\backslash \m{Tub}_{i\epsilon}(S)$ by open subsets of the form 
\begin{align*}
    \nu_0^{-1}U_s\cong B_{\varepsilon}(s) \times \left(\mathbb{R}^m\backslash B_{\epsilon}(0)\right)/\Gamma.
\end{align*}
We can do this in a way such that for each $0<c$ there exists an $0<\varepsilon$ such that 
\begin{align*}
    \left|\left|g_0-g_S(s)\oplus g_{0;V;s}\right|\right|_{C^{0,\alpha}}<c
\end{align*}
holds on $\mathcal{U}_{\varepsilon}$. As $g_0$ is uniform close to the product structures on $\nu_0^{-1}\mathcal{U}'_{\varepsilon}$, so is the Hermitian Dirac bundle structure. Let $r=r(s,m)$ denote the radius of the second component of $B_{\varepsilon}(s)\times \mathbb{R}^m/\Gamma$. Take the sets $U_s=B_{\varepsilon/2}(s)\times \left(\mathbb{R}^m\backslash B_{1\epsilon}(0)\right)$ and $V_s=B_{\varepsilon}(s)\times \left(\mathbb{R}^m\backslash B_{1\epsilon}(0)\right)$. The local model of the unweighted CF Hölder norm is given by
\begin{align*}
    LHS=&\left|\left|\Phi\right|\right|_{C^{0}(U_s)}+r ^{\alpha}[\Phi]_{C^{0,\alpha}(U_s)}+\varepsilon\left|\left|\nabla^{H}\Phi\right|\right|_{C^{0}(U_s)}\\
    &+r\left|\left|\nabla^{V}\Phi\right|\right|_{C^{0}(U_s)}
    +\epsilon ^{\alpha+1}[\nabla^H\Phi]_{C^{0,\alpha}(U_s)}
    +r ^{\alpha+1}[\nabla^V\Phi]_{C^{0,\alpha}(U_s)}
\end{align*}
This can be bounded from the above by 
\begin{align*}
    LHS\lesssim&\left|\left|\Phi\right|\right|_{C^{0}(U_s)}+(\varepsilon^2+r^2 )^{\alpha/2}[\Phi]_{C^{0,\alpha}(U_s)}+\varepsilon\left|\left|\nabla^{H}\Phi\right|\right|_{C^{0}(U_s)}\\
    &+r\left|\left|\nabla^{V}\Phi\right|\right|_{C^{0}(U_s)}+(\varepsilon^2+r^2 )^{(\alpha+1)/2}
    [\nabla\Phi]_{C^{0,\alpha}(U_s)}\\
    \lesssim&(\varepsilon^2+r^2)^{1/2}\left|\left|\widehat{D}_0\Phi\right|\right|_{C^{0}(V_s)}+(\varepsilon^2+r^2)^{(\alpha+1)/2}[\widehat{D}_0\Phi]_{C^{0,\alpha}(V_s)}\\
    &+\left|\left|\Phi\right|\right|_{C^{0}(V_s)}\\
    \lesssim&RHS
\end{align*}
Here we use \cite[Prop. 7.7]{walpuski2012g_2} and \cite[Prop. 7.11]{walpuski2012g_2} where $RHS$ is given by 
\begin{align*}
    RHS= &\varepsilon\left|\left|\widehat{D}_0\Phi\right|\right|_{C^{0}(V_s)}+\varepsilon^{\alpha+1}[\widehat{D}_0\Phi]_{C^{0,\alpha}(V_s)}+\left|\left|\Phi\right|\right|_{C^{0}(V_s)}
\end{align*}
and hence, we deduce the local unweighted CF Schauder estimates. By now multiplying with $r^{1-\beta}$ and as $S$ is compact, we can deduce the result. \\

Higher regularity follows directly from the following commutator estimates.\\

In order to bootstrap the result to $k>1$, we use that 
\begin{align*}
    [\nabla^{\otimes_0},\widehat{D}_0]=-\m{cl}_{g_0}([\nabla^{\otimes_0},\nabla^{\otimes_0}])
\end{align*}
\begin{align*}
    [\nabla^{\otimes_0;H},\widehat{D}_0]=&-\m{cl}_{g_0}([\nabla^{\otimes_0;H},\nabla^{\otimes_0;H}])-\m{cl}_{g_0}([\nabla^{\otimes_0;H},\nabla^{\otimes_0;V}])\\
    =&-\m{cl}_{g_0}([\nabla^{\otimes_0;H},\nabla^{\otimes_0;H}])\\
    [r\nabla^{\otimes_0}_U,\widehat{D}_0]=&-r\m{cl}_{g_0}([\nabla^{\otimes_0;V},\nabla^{\otimes_0;H}])-r\m{cl}_{g_0}([\nabla^{\otimes_0;V},\nabla^{\otimes_0;V}])-\widehat{D}_{0;V}.\\
    =&-\widehat{D}_{0;V}
\end{align*}
Then 
\begin{align*}
    \left|\left|\nabla^{1,0}\Phi\right|\right|_{C^{k,\alpha}_{\m{CF};\beta}}\lesssim& \left|\left|\widehat{D}_0\nabla^{1,0}\Phi\right|\right|_{C^{k-1,\alpha}_{\m{CF};\beta-1}}+\left|\left|\nabla^{1,0}\Phi\right|\right|_{C^{0}_{\m{CF};\beta}}\\
    \lesssim& \left|\left|\nabla^{1,0}\widehat{D}_0\Phi\right|\right|_{C^{k-1,\alpha}_{\m{CF};\beta-1}}+\left|\left|\m{cl}_{g_0}([\nabla^{\otimes_0;H},\nabla^{\otimes_0;H}])\Phi\right|\right|_{C^{k-1,\alpha}_{\m{CF};\beta-1}}+\left|\left|\nabla^{1,0}\Phi\right|\right|_{C^{0}_{\m{CF};\beta}}
\end{align*}
and 
\begin{align*}
    \left|\left|r\nabla^{0,1}\Phi\right|\right|_{C^{k,\alpha}_{\m{CF};\beta}}\lesssim& \left|\left|\widehat{D}_0r\nabla^{0,1}\Phi\right|\right|_{C^{k-1,\alpha}_{\m{CF};\beta-1}}+\left|\left|r\nabla^{0,1}\Phi\right|\right|_{C^{0}_{\m{CF};\beta}}\\
    \lesssim& \left|\left|r\nabla^{0,1}\widehat{D}_0\Phi\right|\right|_{C^{k-1,\alpha}_{\m{CF};\beta-1}}+\left|\left|\widehat{D}_{0;V}\Phi\right|\right|_{C^{k-1,\alpha}_{\m{CF};\beta-1}}+\left|\left|r\nabla^{0,1}\Phi\right|\right|_{C^{0}_{\m{CF};\beta}}.
\end{align*}
Now, as $\m{supp}(\Phi)\subset\m{Tub}_{3\epsilon}(S)$, $\left|\left|\m{cl}_{g_0}([\nabla^{\otimes_0;H},\nabla^{\otimes_0;H}])\Phi\right|\right|_{C^{k,\alpha}_{\m{CF};\beta}}\lesssim \epsilon \left|\left|\Phi\right|\right|_{C^{k-1,\alpha}_{\m{CF};\beta}}$ . Further $\widehat{D}_{0;V}:C^{k+1,\alpha}_{\m{CF};\beta}\rightarrow C^{k,\alpha}_{\m{CF};\beta-1}$ is bounded by the induction assumption, we conclude the statement. 
\end{proof}

\begin{prop}
\label{widehatD0iso}
Let us assume that $\beta\notin \mathcal{C}(\widehat{D}_0)$. The space $  C^{k+1,\alpha}_{\m{CF};\beta}(\widehat{D}_0;\m{APS})$ is complete and 
\begin{align*}
    \widehat{D}_0:  C^{k+1,\alpha}_{\m{CF};\beta}(\widehat{D}_0;\m{APS})\cong  C^{k,\alpha}_{\m{CF};\beta-1}(N_0,\widehat{E}_{0}):\widehat{D}_0^{-1}
\end{align*}
is an isomorphism of Banach spaces and satisfies 
\begin{align*}
    \left|\left|\Phi\right|\right|_{C^{k+1,\alpha}_{\m{CF};\beta}(\widehat{D}_0)}\lesssim \left|\left|\widehat{D}_0\Phi\right|\right|_{C^{k,\alpha}_{\m{CF};\beta-1}}.
\end{align*}
\end{prop}

\begin{proof}
We can construct the inverse explicitly using the right-inverse $\widehat{R}_{0;\Lambda,M_\Lambda}$ in \eqref{Rlambdamu},
\begin{equation*}
    \begin{tikzcd}
           \widehat{D}_0^{-1}\colon C^{k,\alpha}_{\m{CF};\beta-1}(N_0,\widehat{E}_{0})\arrow[r]&  C^{k+1,\alpha}_{\m{CF};\beta}(\widehat{D}_0;\m{APS})\\
        \Psi_{\Lambda,M_{\Lambda}}\arrow[r,mapsto]&\widehat{R}_{0;\Lambda,M_\Lambda}(\Psi_{\Lambda,M_{\Lambda}}).
    \end{tikzcd}
\end{equation*}
By the definition of the right-inverse $\widehat{R}_{0;\Lambda,M_\Lambda}$ and Remark \ref{rightinversesatisfiesboubdary} we know that 
\begin{align*}
    \m{res}^0_{\m{CF}}\widehat{R}_{0}(\Phi)\subset C^{k+1,\alpha}(Y;\widehat{E}_Y)_{\leq 0}.
\end{align*}
\end{proof}

In the theory of operator pencils (see \cite{agmon1961properties} for the study of cylindrical operators) the space of logarithmic solutions is important in asymptotic expansions of kernel and cokernels of asymptotically conical or conically singular operators. These solutions can not exist in the present setup as seen in Section \ref{Polyhomogenous Solutions of the Normal Operator}. Nevertheless, we will explicitly show in the following line of arguments that these solutions can not arise.

\begin{defi}
\label{logCFdefi}
    We define the space of $\beta$-\textbf{logarithmic solutions} 
    \begin{align*}
        \mathcal{L}_{\m{CF},\beta}(\widehat{D}_0;\m{APS})\coloneqq\left\{\Phi=\left.\sum_{i=0}^n\log^i(r)\Phi_i\in \Gamma_{\m{CF};\beta}(N_0,\widehat{E}_{0};\m{APS})\right| \widehat{D}_0\Phi=0,\Phi_{i}\,\text{is }\beta\text{-homogen.}\right\}
    \end{align*}
    and set 
    \begin{align*}
        d_\lambda(\widehat{D}_0)\coloneqq \m{dim}\left(\mathcal{L}_{\m{CF};\beta}(\widehat{D}_0;\m{APS})\right).
    \end{align*}
\end{defi}

\begin{lem}
\label{logarithmicdecaybundle}
Let $\beta<0$. The space of $\beta$-logarithmic solutions of $\widehat{D}_0$ can be identified with the kernel 
    \begin{align*}
        \mathcal{L}_{\m{CF},\beta}(\widehat{D}_0;\m{APS})\cong\m{ker}\left(\widehat{D}_{0;H;\beta+\frac{m-1}{2}-\delta(\widehat{E}_0)}:\Gamma\left(S,\mathcal{E}_{\beta+\frac{m-1}{2}-\delta(\widehat{E}_0)}\right)\rightarrow\Gamma\left(S,\mathcal{E}_{-\beta-\frac{m-1}{2}+\delta(\widehat{E}_0)}\right)\right).
    \end{align*}
of the chiral component of the Dirac operator $\widehat{D}_{0;H;\Lambda}$ for $\lambda=\beta+\frac{m-1}{2}-\delta(\widehat{E}_0)$ constructed in Lemma \ref{DHLambdalemma}.
\end{lem}

\begin{proof}
The proof follows the ideas of Bera in \cite[Lem. 3.70]{bera2023deformCS}. Let $\Phi=\sum_{i=0}^l\log^i(r)\Phi_i\in\mathcal{L}_\beta(\widehat{D}_0;\m{APS})$. Then $\widehat{D}_0\Phi=0$ is equivalent to 
    \begin{align*}
        \widehat{D}_0\Phi_l=0\und{1.0cm}\widehat{D}_0\Phi_i=-\frac{i+1}{r}\Phi_{i+1}.
    \end{align*}
    The first equation reveals that $\Phi_n\sim r^{\beta}$. Hence,  
    \begin{align*}
        \widehat{D}_0\Phi_{n-1}=-\frac{n}{r}\Phi_{n}
    \end{align*}
    implies that $\widehat{D}_{0;H}\Phi_{l-1}=0$ as the right hand side is homogeneous of degree $\beta-1$, but 
    \begin{align*}
        \widehat{D}_{0;H}\Phi_{l-1}\sim r^{\beta}.
    \end{align*}
    Hence,  
    \begin{align*}
        \widehat{D}_{0;V}\Phi_l=0\und{1.0cm}\widehat{D}_{0;V}\Phi_{l-1}=-\frac{l}{r}\Phi_{l}.
    \end{align*}
    Taking the vertical $L^2$-product
    \begin{align*}
        0\leq l\left<\Phi_l,\Phi_l\right>_{L^2_V}=&\int_{\mathbb{R}_{\geq0}}\left<\Phi_l, r\widehat{D}_{0;V}\Phi_{l-1}\right>_{L^2_{r;Y}}r^{-1}\m{d}r\\
        =&\left< \widehat{D}_{0;V}\Phi_l,\Phi_{l-1}\right>_{L^2_{Cyl(Y)}}=0.
    \end{align*}
\end{proof}

\subsection{Uniform Elliptic Estimates for Conically Fibred Dirac Operators}
\label{Uniform Elliptic Estimates for Conically Fibred Dirac Operators}

As we have already discussed in Section \ref{Resolutions of Dirac Bundles}, the dilation action of $\mathbb{R}_{\geq0}$ on the CF Dirac bundle induces a family of Dirac bundle structures. In Section \ref{Uniform Elliptic Theory of Adiabatic Families of ACF Dirac Operators} we will compare a family of ACF-Dirac operators to the family of CF-Dirac operators, rather than to the fixed CF-Dirac operator $\widehat{D}_0$. This procedure is needed to obtain a uniform understanding of the adiabatic limit $t\to 0$.\\

The first lemma will clarify how the weighted Hölder norms behave under the dilation action.

\begin{lem}
    Let $\Phi\in C^{k,\alpha}_{loc}(N_0,\widehat{E}_{0})$. Then 
    \begin{align*}
        \left|\left|\Phi|\right|\right|_{C^{k,\alpha}_{\m{CF};\beta}}=t^{-\beta-\delta(\widehat{E}_{0})}\left|\left|\Phi|\right|\right|_{C^{k,\alpha}_{\m{CF};\beta;t}}
    \end{align*}
\end{lem}

\begin{proof}
    This directly follows from the definition of the weighted Hölder norms.
\end{proof}

The following theorem states the uniformity of the family of operators $\widehat{D}^t_0$ in $t$.

\begin{thm}
    \label{Dotimes0uniform}
    Let $\Phi\in C^{k+1,\alpha}_{\m{CF};\beta}(\widehat{D}_0;\m{APS})$
    \begin{align}
        \left|\left|\Phi|\right|\right|_{C^{k+1,\alpha}_{\m{CF};\beta;t}(\widehat{D}_0)}\lesssim \left|\left|\widehat{D}^t_0\Phi\right|\right|_{C^{k+1,\alpha}_{\m{CF};\beta;t}}.
    \end{align}
\end{thm}

\begin{proof}
    This follows directly from the definition of the weighted Hölder norms.
\end{proof}

\section{Elliptic Theory for Dirac Operators on Orbifolds as CFS Spaces}
\label{Elliptic Theory for Dirac Operators on Orbifolds as CFS Spaces}

In the following we will analyse the Dirac operator on the orbifold $(X,g)$. As we have established in Section \ref{Riemannian Orbifolds as Conically Fibred Singular Spaces}, orbifolds are special edge spaces. In contrast to the standard weighted function spaces for iterated edge spaces we will use more refined version by introducing function spaces on conically fibred singular spaces. We again establish improved Schauder estimates in CFS-Hölder norms and prove the (left) semi-Fredholm property of the analytic realisation of the CFS-Dirac operator. 

\subsubsection{Unweighted Elliptic Theory for Dirac Operators on Orbifolds}
\label{Unweighted Elliptic Theory for Dirac Operators on Orbifolds}

In this subsection, we briefly recall the classical elliptic theory for Dirac operators on compact Riemannian orbifolds. While many analytic results for elliptic operators on smooth manifolds extend naturally to the orbifold setting, certain subtleties arise due to the presence of singularities modelled on finite group quotients. Nonetheless, the local structure of orbifolds allows one to reduce analytical questions to the equivariant setting on smooth domains, where standard elliptic theory applies. In particular, Dirac-type operators on compact orbifolds enjoy the same elliptic regularity and Fredholm properties as on smooth manifolds with control of solutions in both Sobolev and Hölder spaces. We summarize these facts below for later use in our analysis.

\begin{rem}
     As $(X,g)$ is compact, there exists a good cover of $(X,g)$ by action Lie groupoids, i.e. by Lie groupoids of the form 
\begin{align*}
    \Gamma\ltimes B_{\epsilon}(0)\rightrightarrows B_{\epsilon}(0).
\end{align*}
\end{rem}

We briefly justify the above elliptic estimates and Fredholm properties. As $(X,g)$ is a compact, oriented orbifold, it admits a finite good cover by charts of the form $(\tilde{U}_i, \Gamma_i)$, where each $\tilde{U}_i \subset \mathbb{R}^n$ is an open ball and $\Gamma_i$ is a finite subgroup of $\m{SO}(n)$ acting effectively and smoothly. On each such chart, the Dirac operator $D$ lifts to a $\Gamma_i$-equivariant Dirac-type operator $\tilde{D}_i$ on $\tilde{U}_i$, acting on $\Gamma_i$-invariant sections of a lifted bundle. Since elliptic regularity for Dirac operators holds in the smooth category, one obtains classical Sobolev and Hölder estimates for $\tilde{D}_i$ on $\tilde{U}_i$, with constants independent of the $\Gamma_i$-action. These estimates descend to the quotient orbifold charts and patch together globally using a partition of unity subordinate to the orbifold atlas. This yields the following elliptic estimates.

\begin{prop}
\label{ellipticestimatesorbifold}
Suppose $D\Phi=\Psi$ is a weak solution. Then $\Phi\in W^{0,p}(X,E)$ and $\Psi\in W^{l,p}(X,E)$ implies that $\Phi\in W^{l+1,p}(X,E)$. Moreover,
\begin{align*}
    \left|\left|\Phi\right|\right|_{W^{l+1,p}}\lesssim\left|\left|D\Phi\right|\right|_{W^{l,p}}+\left|\left|\Phi\right|\right|_{W^{0,p}}.
\end{align*}
If $\Phi\in   C^{0}(X,E)$ and $\Psi\in  C^{k,\alpha}(X,E)$ then $\Phi\in  C^{k+1,\alpha}(X,E) $ and 
\begin{align*}
    \left|\left|\Phi\right|\right|_{C^{k+1,\alpha}}\lesssim\left|\left|D\Phi\right|\right|_{C^{l,\alpha}}+\left|\left|\Phi\right|\right|_{C^{0}}.
\end{align*}
\end{prop}

Since the Rellich–Kondrakov theorem holds for compact orbifolds, the embeddings
\begin{align*}
    W^{k+1,p}(X,E) \hookrightarrow W^{k,p}(X,E)\und{1.0cm}C^{k+1,\alpha}(X,E) \hookrightarrow C^{k,\alpha}(X,E)
\end{align*}
are compact. Together with the elliptic estimates, this implies that the Sobolev and Hölder realisations of $D$ are Fredholm. Moreover, any weak solution in $W^{k,p}$ or $C^{k,\alpha}$ is in fact smooth by elliptic bootstrapping, so the kernels of all analytic realisations coincide and consist of smooth sections.

\begin{thm}
\label{FredholmOrbifold}
The analytic realisations 
    \begin{align}
    \label{Dunweighted}
        D:&W^{k+1,p}(X,E)\rightarrow W^{k,p}(X,E)\\
        D:&   C^{k+1,\alpha}(X,E)\rightarrow   C^{k,\alpha}(X,E)
    \end{align}
    are Fredholm. Moreover, their kernels coincide and do not depend on $k$, $p$ or $\alpha$  
\end{thm}

In particular, we can apply this to the Hodge-de Rham operator $\m{d}+\m{d}^*$ on $(X,g)$. On compact manifolds, Hodge theory links the kernel of the $\m{d}+\m{d}^*$ to the de Rham cohomology groups of $X$. The de Rham theorem proves that the de Rham cohomology and the singular cohomology coincide. The following result ensures that this still holds true on Riemannian orbifolds.

\begin{thm}\cite[Thm. 2.13]{adem2007orbifolds}
    There exists an isomorphism of graded vector spaces 
    \begin{align*}
        \m{ker}\left(\m{d}+\m{d}^*\right)\cong \m{H}^\bullet_{dR}(X,\mathbb{R})\cong\m{H}^\bullet_{orb}(X,\mathbb{R})\coloneqq\m{H}^\bullet_{sing}(\m{B}X^\bullet).
    \end{align*}
    Here $\m{B}X^\bullet$ is the classifying space of a Lie groupoid $X^\bullet$ presenting the orbifold $X$.\footnote{For a discussion of the classifying space $\m{B}X^\bullet$ we refer to \cite[Sec. 4]{moerdijk2002orbifolds}.}
\end{thm}

\subsubsection{Weighted Function Spaces on CFS Spaces}
\label{Weighted Function Spaces on CFS Spaces}

In the following section we will introduce weighted function spaces of sections of Dirac bundles on Riemannian orbifolds, adapted to the CFS geometry.

\begin{nota}
A CFS weight function is a smooth function coinciding with the distance function $w_{\m{CFS}}=\m{dist}_{g}(S)$ in a tubular neighbourhood of $S$. Set 
\begin{align*}
    w^{-\beta}_{\m{CFS}}=w_{\m{CFS};\beta;\epsilon}\und{1.0cm}w_{\m{CFS}}(x,y)=\m{min}(w_{\m{CFS}}(x),w_{\m{CFS}}(y)).
\end{align*}
\end{nota}

Let in the following $(X,g)$ be a Riemannian orbifold and 
\begin{align*}
    \m{exp}_{g}\colon\m{Tub}_{5\epsilon}(X^{sing})\rightarrow X
\end{align*}
be the exponential tubular neighbourhood of the singular stratum of width $\epsilon$. For simplicity reasons we will assume that $X^{sing}=S$ is compact and connected. In the presence of the Ehresmann connection $H$, the connection $\nabla^h$ decomposes into 
\begin{align*}
    \nabla^h=\nabla^{h;\epsilon}_H+\nabla^{h;\epsilon}_V=(\nabla-\chi_3\nabla^{h}_H)+(1-\chi_3)\nabla^{h}_V,
\end{align*}
Notice that $\chi_3$ depends on $\epsilon$.

\begin{defi}
We define the space $ \Gamma^k_{\m{CFS};\beta;\epsilon}(X,E)$ to be the space of all $C^k$-sections $\Phi$ such that 
\begin{align*}
    \left|\left|w_{\m{CFS};\beta;\epsilon}\left\{(\nabla^h)^j\right\}\Phi\right|\right|_{C^0}<\infty
\end{align*}
We equip the spaces $ \Gamma^k_{\m{CFS};\beta;\epsilon}(X,E)$ with the metric
\begin{align*}
    \left|\left|\Phi\right|\right|_{C^{k}_{\m{CFS};\beta;\epsilon}}=&\sum_{j=0}^k\left|\left|w_{\m{CFS};\beta;\epsilon}\left\{(\nabla^h)^j\right\}\Phi\right|\right|_{C^0}
\end{align*}
The space $ \Gamma_{\m{CFS};\beta;\epsilon}(X,E)$ is defined to be the intersection of all $ \Gamma^k_{\m{CFS};\beta;\epsilon}(X,E)$ for all $k$. Furthermore, we define the $w_{\m{CFS};\beta;\epsilon}$-weighted Sobolev $W^{k,p}_{\m{CFS};\beta;\epsilon}(X,E)$ and Hölder spaces $  C^{k,\alpha}_{\m{CFS};\beta;\epsilon}(X,E)$ by replacing the usual metrics by the following ones 
\begin{align*}
    \left|\left|\Phi\right|\right|^p_{W^{k,p}_{\m{CFS};\beta;\epsilon}}=&\sum_{j=0}^k\int_{X}|w_{\m{CFS};\beta;\epsilon}\left\{(\nabla^h)^j\right\}\Phi|^p_{h,g} w_{\m{CFS};m}\m{vol}_{g}\\
    \left|\left|\Phi\right|\right|_{C^{l,\alpha}_{\m{CFS};\beta;\epsilon}}=&\sum_{j=0}^l\left|\left|w_{\m{ACF};\beta}\left\{(\nabla^h)^j\right\}\Phi\right|\right|_{C^0}+[(\nabla^h)^j\Phi]_{C^{0,\alpha}_{\m{CFS};\beta-j;\epsilon}}\\
    [(\nabla^h)^j\Phi]_{C^{0,\alpha}_{\m{CFS};\beta-j;\epsilon}}=&\underset{B_{g}}{\m{sup}}\left\{\bigoplus_{i=0}^jw_{\m{CFS};(\beta-j+i-\alpha)}(x,y)\cdot\right.\\
    &\left.\frac{\left|\m{word}^+_{i,j}(\nabla^{h,\epsilon}_H,\nabla^{h,\epsilon}_V) \Phi(x)-\Pi^{x,y}\m{word}^+_{i,j}(\nabla^{h,\epsilon}_H,\nabla^{h,\epsilon}_V) \Phi(y)\right|_{h,g}}{d_{g}(x,y)^\alpha}\right\}
\end{align*} 
where $B_{g}(U)=\{x,y\in U|\,x\neq y,\,\m{dist}_{g}(x,y)\leq w_{\m{ACF};(\beta-j+i-\alpha)}(x,y)\}\subset U\times U$ and $\Pi^{x,y}$ denotes the parallel transport along the geodesic connecting $x$ and $y$.
\end{defi}

\begin{rem}
All $W^{k,p}_{\m{CFS};\beta;\epsilon}$, $ C^k_{\m{CFS};\beta;\epsilon}$ and $  C^{l,\alpha}_{\m{CFS};\beta;\epsilon}$ are Banach spaces, and $H^{k}_{\m{CFS};\beta}=W^{k,2}_{\m{CFS};\beta;\epsilon}$ is a Hilbert space. However, $ \Gamma_{\m{CFS};\beta;\epsilon}$ is not a Banach space.
\end{rem}

\begin{thm}
\label{weightedsobolovembeddingCFS}
The following holds:
\begin{itemize}
    \item Let $\frac{1}{p}+\frac{1}{q}=1$. The $L^2_{\m{CFS}}$-pairing induces a perfect pairing 
    \begin{align*}
        L^p_{\m{CFS};\beta}\cong(L^q_{\m{CFS};-\beta-m})^*
    \end{align*}
    \item If $\beta>\beta'$ then $W^{l,p}_{\m{CFS};\beta}\subset W^{l,p}_{\m{CFS};\beta'}$ and $  C^{l,\alpha}_{\m{CFS};\beta}\subset  C^{l,\alpha}_{\m{CFS};\beta'}$
    \item If $1\leq p,\infty$ then $ \Gamma_{\m{CFS};\beta}$ is dense in $W^{l,p}_{\m{CF};\beta}$
    \item (Sobolev, Kondrakov) If $(p_i,l_i)\in\mathbb{N}\times[1,\infty)$ with $l_0\leq l_1$ such that
    \begin{align*}
        l_0-\frac{m}{p_0}\geq l_1-\frac{m}{p_1}
    \end{align*}
    and one of the two conditions
    \begin{itemize}
        \item[(i)] $p_0\leq p_1$ and $\beta_0\geq\beta_1$
        \item[(ii)] $p_0>p_1$ and $\beta_0>\beta_1$
    \end{itemize}
    then there exists a continuous embedding 
    \begin{align*}
        W^{l_0,p_0}_{\m{CFS};\beta_0}\hookrightarrow W^{l_1,p_1}_{\m{CFS};\beta_1}.
    \end{align*}
    If moreover $l_0>l_1$, $l_0-\frac{m}{p_0}> l_1-\frac{m}{p_1}$ and $\beta_0>\beta_1$ then the above embedding is compact, i.e. any bounded sequence in the $(l_0,p_0)$ norm has a convergent subsequence in the $(l_1,p_1)$ norm.
    \item (Hölder) If $(l_i,\alpha)\in\mathbb{N}\times[0,\infty)$ and $\beta_0<\beta_1$ with 
    \begin{align*}
        l_0+\alpha_0\geq l_1+\alpha_1
    \end{align*}
    then there exist continuous embeddings
    \begin{align*}
          C^{l_0+1}_{\m{CFS};\beta_0}\hookrightarrow   C^{l_0,\alpha_0}_{\m{CFS};\beta_0}\hookrightarrow  C^{l_1,\alpha_1}_{\m{CFS};\beta_1}\hookrightarrow  C^{l_1}_{\m{CFS};\beta_1}.
    \end{align*}
    If $\alpha_1=0$ and $l_0=l_1$ the inclusion $  C^{l_0,\alpha_0}_{\m{CFS};\beta_0}\hookrightarrow  C^{l_1,\alpha_1}_{\m{CFS};\beta_1}$ is compact.
    \item (Morrey, Kondrakov) If $(p_i,l_i)\in\mathbb{N}\times[1,\infty)$ and $(l,\alpha)\in\mathbb{N}\times[1,\infty)$, $\beta>\delta$, and 
    \begin{align*}
        l_0-\frac{m}{p_0}\geq l+\alpha\geq l_1-\frac{m}{p_1}
    \end{align*}
    then there exists a continuous embedding 
    \begin{align*}
        W^{l_0,p_0}_{\m{CFS};\beta}\hookrightarrow   C^{l,\alpha}_{\m{CFS};\beta}\hookrightarrow W^{l_1,p_1}_{\m{CFS};\delta}.
    \end{align*}
    \item Let $(\Phi,\Psi)\mapsto(\phi\cdot \Phi)$ be a bilinear form satisfying $|\Phi\cdot \Psi|\lesssim|\Phi||\Psi|$, then 
\begin{align*}
    \left|\left|\Phi\cdot \Psi\right|\right|_{C^{l,\alpha}_{\m{CFS};\beta_1+\beta_2}}\lesssim&\left|\left|\Phi\right|\right|_{C^{l,\alpha}_{\m{CFS};\beta_1}}\left|\left| \Psi\right|\right|_{C^{l,\alpha}_{\m{CFS};\beta_2}}.
\end{align*}
\end{itemize}
\end{thm}

\begin{proof}
These statements follow immediately from the standard embedding theorems on compact and on conical spaces \cite[Thm. 4.18]{marshal2002deformations}.\footnote{It suffice to prove the statements on $B_{1}(0)\times\mathbb{R}^m/\Gamma$.}
\end{proof}

\subsubsection{Fredholm Theory for Dirac Operators on Orbifolds}
\label{Fredholm Theory for Dirac Operators on Orbifolds}

In the following we will establish a Fredholm theory for Dirac operators on weighted function spaces on orbifolds. The Fredholm property will fail at critical rates and wall-crossing formulae are established.\\

In Section \ref{Invertibility of Conically Fibred Dirac Operators} we established that the operator $\widehat{D}_0$ can be realised as an isomorphism of weighted function spaces, subject to boundary conditions. In order to show that the orbifold Dirac operator can be realised as a Fredholm map, we will have to compare it to the operator $\widehat{D}_0$. The following lemma is proving that close to the singular stratum $D$ is approximated by $\widehat{D}_0$.\\

Recall from \ref{Dirac Bundles on Orbifold as Conically Fibred Singular Spaces} that $D=\widehat{D}_0+(\m{cl}_g-\m{cl}_{g_0})\circ\nabla^h+\mathcal{O}(r)$.\\

To compare solutions of the orbifold Dirac operator $D$ with those of the model operator $\widehat{D}_0$, we need a way to transfer sections between $X$ and its CF-model near the singular set. For this purpose, we introduce the CFS-restriction and CFS-extension maps that encode the asymptotic data near the singular locus.

\begin{defi}
\label{CFSresandCFSext}
     Recall Definition \ref{CFresandCFext} for the definition of the CF- restriction and CF- extension map. We define the \textbf{restriction map }
    \begin{align*}
        \m{res}_{\m{CFS};\beta;\epsilon}\colon  C^{k,\alpha}_{\m{CFS};\beta;\epsilon}(X,E)\rightarrow 
             C^{k,\alpha}_{r=5\epsilon/2}(Y,\widehat{E}_Y)
    \end{align*}
    \begin{align*}
        \Phi\mapsto \m{res}^\epsilon_{\m{CF};\beta}(1-\chi_4)\Phi
    \end{align*}
    and a section 
    \begin{align*}
        \m{ext}_{\m{CFS}}\colon    C^{k,\alpha}(Y,\widehat{E}_Y)\rightarrow   C^{k,\alpha}_{loc}(X^{reg},E)
    \end{align*}
    \begin{align*}
        \Phi_{\Lambda,M_\Lambda}\mapsto  (1-\chi_4) \m{ext}_{\m{CF}}\Phi_{\Lambda,M_\Lambda}.
    \end{align*}
\end{defi}

We now define the natural domain of the orbifold Dirac operator adapted to the conically fibred singularities. This domain enforces a spectral boundary condition of APS-type, modelled on the behaviour of the horizontal operator $\widehat{D}_{Y;H}$ along the link.

\begin{defi}
\label{APSCFSdefi}
Recall Definition \ref{CFSresandCFSext} for the definition of $\m{res}_{CFS;\beta;\epsilon}$. Define the space 
\begin{align*}
     C^{k+1,\alpha}_{\m{CFS};\beta;\epsilon}(X,E;\m{APS})\coloneqq \m{res}^{-1}_{\m{CFS};\beta;\epsilon}(  C^{k+1,\alpha}(Y,E)_{\leq 0}).
\end{align*}
\end{defi}

The resulting space of admissible sections is a Banach space under the graph norm. Moreover, weighted regularity implies compactness of embeddings between spaces with different weights and regularity, which is essential for the Fredholm theory.

\begin{lem}
    The space $C^{k+1,\alpha}_{\m{CFS};\beta;\epsilon}(X,E;\m{APS})$ as defined in Definition \ref{APSCFSdefi} is a Banach space and for $\delta>\beta$ and 
    \begin{align*}
        k+\alpha>k'+\alpha'
    \end{align*}
    there exists a compact embedding 
    \begin{align}
     \label{compactembeddingCFS}
        C^{k,\alpha}_{\m{CFS};\beta;\epsilon}(X,E;\m{APS})\hookrightarrow C^{k',\alpha'}_{\m{CFS};\delta}(X,E).
    \end{align}
\end{lem}

We now show that, locally near the singular stratum, the orbifold operator $D$ is well-approximated by the model operator $\widehat{D}_0$ in the weighted Hölder topology. This approximation is controlled by the tubular neighbourhood size $\epsilon\sim t^\lambda$.

\begin{lem}
\label{D-D0lemma}
    Let $\Phi\in \Gamma^{k+1,\alpha}_{loc}(X,E)$ with $\m{supp}(\Phi)\subset\m{Tub}_{2\epsilon}(S)$. Then by Assumption \ref{ass2} 
    \begin{align}
        \label{D-D0}
        \left|\left|(D-\widehat{D}_0)\Phi\right|\right|_{C^{k,\alpha}_{\m{CFS};\beta-1;\epsilon}}\lesssim \epsilon \left|\left|\Phi\right|\right|_{C^{k+1,\alpha}_{\m{CFS};\beta;\epsilon}}.
    \end{align}
\end{lem}

\begin{proof}
We use the expansion 
\begin{align*}
    D-\widehat{D}_0=(\m{cl}_g-\m{cl}_{g_0})\circ\nabla^h+\m{cl}_g\circ\left(\nabla^{h}-\nabla^{\otimes_0}\right).
\end{align*}
Notice, that $\m{cl}_g\circ\left(\nabla^{h}-\nabla^{\otimes_0}\right)=\mathcal{O}(r)$ is a zeroth order differential operator. The bound follows from the definition of the $C^{k,\alpha}_{\m{CFS};\beta}$-norms.
\end{proof}

To control solutions of the Dirac equation globally, we derive Schauder-type estimates in the weighted CFS setting. These combine interior elliptic estimates on the regular region with the model Schauder theory near the singular locus.

\begin{prop}[CFS-Schauder Estimates]
\label{SchauderD}
Let $\Phi\in  C^{k+1,\alpha}_{\m{CFS};\beta;\epsilon}(X,E;\m{APS})$ as in Definition \ref{APSCFSdefi}
\begin{align}
    \label{SchauderDeq}
    \left|\left|\Phi\right|\right|_{C^{k+1,\alpha}_{\m{CFS};\beta;\epsilon}} \lesssim& \left|\left|D\Phi\right|\right|_{C^{k,\alpha}_{\m{CFS};\beta-1;\epsilon}}+\left|\left|\Phi\right|\right|_{C^{0}_{\m{CFS};\beta;\epsilon}}
\end{align}
holds.
\end{prop}

\begin{proof}
As Schauder estimates are local estimates, the statement holds for all regularity in the \say{compact region} of the orbifold. Close to the singular stratum we relate the CFS Schauder estimates to the CF-Schauder estimates. As $\left|\left|.\right|\right|_{C^{k+1,\alpha}_{\m{CF};\beta}}$ and $\left|\left|.\right|\right|_{C^{k+1,\alpha}_{\m{CFS};\beta;\epsilon}}$ are equivalent on $\m{Tub}_{2\epsilon}(S)$,
\begin{align*}
    \left|\left|\Phi\right|\right|_{C^{k+1,\alpha}_{\m{CFS};\beta;\epsilon}}\lesssim& \left|\left|\widehat{D}_0\Phi\right|\right|_{C^{k,\alpha}_{\m{CFS};\beta-1;\epsilon}}+\left|\left|\Phi\right|\right|_{C^{0}_{\m{CFS};\beta;\epsilon}}\\
    \lesssim &\left|\left|D\Phi\right|\right|_{C^{k,\alpha}_{\m{CFS};\beta-1;\epsilon}}+\left|\left|(D-\widehat{D}_0)\Phi\right|\right|_{C^{k,\alpha}_{\m{CFS};\beta-1;\epsilon}}+\left|\left|\Phi\right|\right|_{C^{0}_{\m{CFS};\beta;\epsilon}}\\
     \lesssim &\left|\left|D\Phi\right|\right|_{C^{k,\alpha}_{\m{CFS};\beta-1;\epsilon}}+\epsilon\left|\left|\Phi\right|\right|_{C^{k+1,\alpha}_{\m{CFS};\beta;\epsilon}}+\left|\left|\Phi\right|\right|_{C^{0}_{\m{CFS};\beta;\epsilon}}
\end{align*}
Thus, by absorbing into the left-hand side we conclude the statement.
\end{proof}

We now arrive at the central result of this section, the realisation of the Dirac operator as a Fredholm map between weighted function spaces. This holds away from the critical set $\mathcal{C}(\widehat{D}_0)$ and captures the wall crossing phenomenon of the index as the weight crosses critical rates.

\begin{thm}
\label{FredholmCFS}
Let $0>\beta\notin \mathcal{C}(\widehat{D}_0)$. The map 
    \begin{align*}
        D\colon  C^{k+1,\alpha}_{\m{CFS};\beta;\epsilon}(X,E;\m{APS})\rightarrow  C^{k,\alpha}_{\m{CFS};\beta-1;\epsilon}(X,E)
    \end{align*}
    is Fredholm and its kernel and cokernel agree for all $\alpha$ and $k$. Furthermore, if $\beta_2<\beta_1$ we have
\begin{align*}
    \m{ind}_{\m{CFS};\beta_2}(D)-\m{ind}_{\m{CFS};\beta_1}(D)=\sum_{\beta_2<\lambda<\beta_1}d_{\lambda+\frac{m-1}{2}-\delta(\widehat{E}_{0})}(\widehat{D}_0),
\end{align*}
where $d_{\beta}(\widehat{D}_0)$ as defined in Definition \ref{logCFdefi}.
\end{thm}

\begin{proof}
    The first statement follows from Proposition \ref{widehatD0iso} and \ref{SchauderD}, and goes analogously to the proof of Fredholmness for ACyl, AC or CS operators. Notice, that the Schauder estimate \eqref{SchauderD} does not suffice to prove left-semi Fredholmness as the inclusion of 
    \begin{align*}
        C^{k+1,\alpha}_{\m{CFS};\beta;\epsilon}(X,E;\m{APS})\hookrightarrow C^0_{\m{CFS};\beta;\epsilon}(X,E)
    \end{align*}   
    is not compact. Thus, we need to improve the Schauder estimate to
\begin{align}
    \label{improvedSchauderDeq}
    \left|\left|\Phi\right|\right|_{C^{k+1,\alpha}_{\m{CFS};\beta;\epsilon}} \lesssim& \left|\left|D\Phi\right|\right|_{C^{k,\alpha}_{\m{CFS};\beta-1;\epsilon}}+\left|\left|\Phi\right|\right|_{C^{0}_{\m{CFS};\delta;\epsilon}}
\end{align}
    for $\delta>\beta$. Set $\chi_\varepsilon =\chi(\varepsilon^{-1} r-1)$. Let us define the operator  
\begin{align*}
    \widetilde{D}_{\varepsilon}=D+\chi_{\varepsilon}(\widehat{D}_0-D).
\end{align*}
This operator is conjugated to $\widehat{D}_0$ on  $\m{Tub}_{\varepsilon-e}(S)$ and coincides with $D$ on the compact part of $X$. 
\begin{align*}
    \left|\left|\Phi\right|\right|_{C^{k+1,\alpha}_{\m{CFS};\beta;\epsilon}}&\lesssim\left|\left|(1-\chi_{\varepsilon})\Phi\right|\right|_{C^{k+1,\alpha}_{\m{CFS};\beta;\epsilon}}+\left|\left|\chi_{\varepsilon}\Phi\right|\right|_{C^{k+1,\alpha}_{\m{CFS};\beta;\epsilon}}\\
    &\lesssim\left|\left|\widetilde{D}_{\varepsilon+e}(1-\chi_{\varepsilon})\Phi\right|\right|_{C^{k,\alpha}_{\m{CFS};\beta-1;\epsilon}}+\left|\left|\widetilde{D}_{\varepsilon+e}\chi_{\varepsilon}\Phi\right|\right|_{C^{k,\alpha}_{\m{CFS};\beta-1;\epsilon}}\\
    &+\left|\left|\chi_{\varepsilon}\Phi\right|\right|_{C^{0}_{\m{CFS};\beta;\epsilon}}   
\end{align*}
Here we used that $\widetilde{D}_{\varepsilon+e}$ is conjugate to $\widehat{D}_0$ on $\m{Tub}_{\varepsilon}(S)$ and by Proposition \ref{widehatD0iso} 
\begin{align*}
        \widehat{D}_0\colon C^{k+1,\alpha}_{\m{CF};\beta}(\widehat{D}_0;\m{APS})\rightarrow C^{k,\alpha}_{\m{CF};\beta-1}(N_0,\widehat{E}_{0})
\end{align*}
defines an isomorphism for $\beta\notin\mathcal{C}(\widehat{D}_0)$. Further we have

\begin{align*}
   \left|\left|\widetilde{D}_{\varepsilon+e}(1-\chi_{\varepsilon})\Phi\right|\right|_{C^{k,\alpha}_{\m{CFS};\beta-1;\epsilon}}&\lesssim\left|\left|(1-\chi_{\varepsilon})\widetilde{D}_{\varepsilon+e}\Phi\right|\right|_{C^{k,\alpha}_{\m{CFS};\beta-1;\epsilon}}+\left|\left|[(1-\chi_{\varepsilon}),\widetilde{D}_{\varepsilon+e}]\Phi\right|\right|_{C^{k,\alpha}_{\m{CFS};\beta-1;\epsilon}}
\end{align*}
and
\begin{align*}
   \left|\left|\widetilde{D}_{\varepsilon+e}\chi_{\varepsilon}\Phi\right|\right|_{C^{k,\alpha}_{\m{CFS};\beta-1;\epsilon}}&\lesssim\left|\left|\chi_{\varepsilon}\widetilde{D}_{\varepsilon+e}\Phi\right|\right|_{C^{k,\alpha}_{\m{CFS};\beta-1;\epsilon}}+\left|\left|[\chi_{\varepsilon},\widetilde{D}_{\varepsilon+e}]\Phi\right|\right|_{C^{k,\alpha}_{\m{CFS};\beta-1;\epsilon}}
\end{align*}

The operator $[\widetilde{D}_{\varepsilon+e},(1-\chi_{\varepsilon})]=-[\widetilde{D}_{\varepsilon+e},\chi_{\varepsilon}]=\sigma_{\widetilde{D}_{\varepsilon+e}}(\m{d}\chi_{\varepsilon})=K_{\varepsilon}$ is compact. Further, 
\begin{align*}
    \left|\left|\chi_{\varepsilon}\Phi\right|\right|_{C^{0}_{\m{CFS};\beta;\epsilon}}\lesssim \varepsilon^{\delta-\beta}\left|\left|\chi_{\varepsilon}\Phi\right|\right|_{C^{0}_{\m{CFS};\delta;\epsilon}}
\end{align*}
which leads to the following estimate
\begin{align*}
    \left|\left|\Phi\right|\right|_{C^{k+1,\alpha}_{\m{CFS};\beta}}&\lesssim\left|\left|\widetilde{D}_{\varepsilon+e}\Phi\right|\right|_{C^{k,\alpha}_{\m{CFS};\beta-1;\epsilon}}+\left|\left|K_{\varepsilon}\Phi\right|\right|_{C^{k,\alpha}_{\m{CFS};\beta-1;\epsilon}}+\varepsilon^{\delta-\beta}\left|\left|\chi_{\varepsilon}\Phi\right|\right|_{C^{0}_{\m{CFS};\delta;\epsilon}}.
\end{align*}

We conclude that the family of operators $\widetilde{D}_{\varepsilon}$ are semi-Fredholm, i.e. have finite dimensional kernel and a closed image. Further, let $P_{\varepsilon}$ be a parametrix for the operator $\widetilde{D}_{\varepsilon}$ on $X\backslash\m{Tub}_{\varepsilon}(S)$. Define the operator 
\begin{align*}
    Q_{\varepsilon}=\chi_{\varepsilon}P_{\varepsilon}\chi_{\varepsilon+e}+(1-\chi_{\varepsilon})\left(\widehat{D}_0\right)^{-1}(1-\chi_{\varepsilon+e}).
\end{align*}
Then $\widetilde{D}_{\varepsilon}Q_{\varepsilon}=1+\widetilde{K}_{\varepsilon}$, where $\widetilde{K}_{\varepsilon}$ is compact and hence, $\widetilde{D}_{\varepsilon}$ is Fredholm. Since $D_{\varepsilon}$ converges in the norm-topology and by taking the limit $\varepsilon\to 0$ we deduce the Fredholmness of $D$.\\

    For the second statement, we follow the proof of \cite[Lem. 3.70]{bera2023deformCS} and it is sufficient to consider the case where $\lambda$ is the only critical rate between $\beta_1$ and $\beta_2$ with $|\beta_2-\beta_1|<1$. Consider the space 
    \begin{align*}
        V_\lambda\coloneqq \m{ker}(\widehat{D}_{H;\lambda+\frac{m-1}{2}-\delta(\widehat{E}_{0})}),
    \end{align*}
    the projection $\Pi_\lambda\colon\Gamma(Y,\widehat{E}_Y)_{\leq 0}\rightarrow V_\lambda$, and the space 
    \begin{align*}
        S_{\beta_2}\coloneqq\left\{\Phi\in C^{k+1,\alpha}_{\m{CFS};\beta_2;\epsilon}(X,E;\m{APS})| D\Phi\in C^{k,\alpha}_{\m{CFS};\beta_1-1;\epsilon}(X,E)\right\}.
    \end{align*}
    Then for $\Phi\in S_{\beta_2}$
    \begin{align*}
        \Phi-\m{ext}_{\m{CFS};\lambda}\circ \Pi_{\lambda}\circ \m{res}_{\m{CFS};\lambda}(\Phi)\in C^{k+1,\alpha}_{\m{CFS};\beta_1;\epsilon}(X,E;\m{APS})
    \end{align*}
    Moreover, we have $S_{\beta_2}\subset C^{k+1,\alpha}_{\m{CFS};\lambda;\epsilon}(X,E;\m{APS})$, $\m{ker}_{\m{CFS};\beta_2}(D)=\m{ker}_{\m{CFS};\lambda}(D)$ and
    \begin{align*}
        C^{k+1,\alpha}_{\m{CFS};\beta_1;\epsilon}(X,E;\m{APS})=\m{ker}(\Pi_{\lambda}\circ\m{res}_{\m{CFS};\lambda}\colon S_{\beta_2}\rightarrow V_\lambda)
    \end{align*}
    as well as
    \begin{align*}
        \m{ker}_{\m{CFS};\beta_1}(D)=\m{ker}\left(\Pi_{\lambda}\circ\m{res}_{\m{CFS};\lambda}\colon\m{ker}_{\m{CFS};\beta_2}(D)\rightarrow V_{\lambda}\right).
    \end{align*}
    Furthermore, we have 
    \begin{align*}
        S_{\beta_2}=C^{k+1,\alpha}_{\m{CFS};\beta_1;\epsilon}(X,E;\m{APS})+\m{im}\left(\m{ext}_{\m{CFS};\lambda}|_{V_\lambda}\right)
    \end{align*}
    and 
    \begin{align*}
        \m{ker}_{\m{CFS};\beta_2}(D)=&\m{ker}\left(D|_{S_{\beta_2}}\colon S_{\beta_2}\rightarrow C^{k,\alpha}_{\m{CFS};\beta_1-1;\epsilon}(X,E)\right)\\\m{coker}_{\m{CFS};\beta_2}(D)=&\m{coker}\left(D|_{S_{\beta_2}}\colon S_{\beta_2}\rightarrow C^{k,\alpha}_{\m{CFS};\beta_1-1;\epsilon}(X,E)\right).
    \end{align*}
    The wall-crossing formula follows immediately.
\end{proof}

\begin{rem}
Notice, that for $\beta$ small enough, the kernel of this map coincides with the kernel of the maps \eqref{Dunweighted}.
\end{rem}

\begin{cor}
    Let $0>\beta\in\mathcal{C}(\widehat{D}_0)$ be the largest rate in $(-\infty,0)$. Then for $\beta<\beta'<0$ 
    \begin{align*}
        \m{ker}(D)=\m{ker}_{\m{CFS};\beta'}(D).
    \end{align*}
\end{cor}

\section{Uniform Elliptic Theory of Adiabatic Families of ACF Dirac Operators}
\label{Uniform Elliptic Theory of Adiabatic Families of ACF Dirac Operators}

In the following section we will discuss the ACF analysis associated to the family of Dirac operators with torsion on $N_\zeta$. These Dirac operators are asymptotic to the operator $\widehat{D}_0$ and a uniform elliptic theory is essential for linear gluing. 

\subsection{Weighted Function Spaces on Asymptotically Conical Fibrations}
\label{Weighted Function Spaces on Asymptotically Conical Fibrations}

In the following section we will introduce weighted function spaces of sections of Dirac bundles on $(N_\zeta,g_\zeta)$ that are adapted to the ACF geometry.

\begin{nota}
An ACF weight function is a smooth function $w_{\m{ACF}}$ on $N_\zeta$ such that 
\begin{align*}
    (\nabla^{g_{\m{CF};V}})^{k}(\rho_*w_{\m{ACF}}-r)=&\mathcal{O}(r^{1+\gamma-k})\\
    (\nabla^{g_{\m{CF}}})^{k}(\rho_*w_{\m{ACF}}-r)=&\mathcal{O}(r^{1+\gamma})
\end{align*}
Set $w^{-\beta}_{\m{ACF}}=w_{\m{ACF};\beta}$  and $w_{\m{ACF}}(x,y)=\m{min}(w_{\m{ACF}}(x),w_{\m{ACF}}(y))$. 
\end{nota}

\begin{rem}
Given a weight function $w_{\zeta}$ for $(N_{ t^2\cdot\zeta}, g_{t^2\cdot\zeta})$ we define new weight functions on $(N_\zeta,g^t_\zeta)$ by 
\begin{align*}
   w_{\m{ACF};t}\coloneqq\delta_t^*w_{t^2\cdot\zeta}=t\cdot w_{\zeta}=t\cdot w_{\m{ACF}}.
\end{align*}
\end{rem}

\begin{defi}
\label{ACFHolderspaces}
We define the space $ \Gamma^k_{\m{ACF};\beta;t}(N_\zeta,\widehat{E}_\zeta)$ to be the space of all $C^k$-sections $\Phi$ such that 
\begin{align*}
    \left|\left|w_{\m{ACF};\beta;t}\left\{(\widehat{\nabla}^{h^t_\zeta})^j\right\}\Phi\right|\right|_{C^0_{t}}<\infty.
\end{align*}
We equip the spaces $ \Gamma^k_{\m{ACF};\beta}(N_\zeta,\widehat{E}_\zeta)$ with the metric
\begin{align*}
    \left|\left|\Phi\right|\right|_{C^{k}_{\m{ACF};\beta;t}}=&\sum_{j=0}^k\left|\left|w_{\m{ACF};\beta;t}\left\{(\widehat{\nabla}^{h^t_\zeta})^j\right\}\Phi\right|\right|_{C^0_t}
\end{align*}
The space $ \Gamma_{\m{ACF};\beta;t}(N_\zeta,\widehat{E}_\zeta)$ is defined to be the intersection of all $ \Gamma^k_{\m{ACF};\beta;t}(N_\zeta,\widehat{E}_\zeta)$ . Furthermore, we define the $w_{\m{ACF};\beta;t}$-weighted Sobolev $W^{k,p}_{\m{ACF};\beta;t}(N_\zeta,\widehat{E}_\zeta)$ and Hölder spaces $  C^{k,\alpha}_{\m{ACF};\beta;t}(N_\zeta,\widehat{E}_\zeta)$ by replacing the usual metrics by the ones 
\begin{align*}
    \left|\left|\Phi\right|\right|^p_{W^{k,p}_{\m{ACF};\beta;t}}=&\sum_{j=0}^k\int_{X}|w_{\m{ACF};\beta;t}\left\{(\widehat{\nabla}^{h^t_\zeta})^j\right\}\Phi|^p_{h^t_\zeta,g^t_\zeta} w_{\m{ACF};m;t}\m{vol}_{g^t_\zeta}\\
    \left|\left|\Phi\right|\right|_{C^{l,\alpha}_{\m{ACF};\beta;t}}=&\sum_{j=0}^l\left|\left|w_{\m{ACF};\beta;t}\left\{(\widehat{\nabla}^{h^t_\zeta})^j\right\}\Phi\right|\right|_{C^0_t}+[(\widehat{\nabla}^{h^t_\zeta})^j\Phi]_{C^{0,\alpha}_{\m{ACF};\beta-j;t}}\\
    [\nabla^j\Phi]_{C^{0,\alpha}_{\m{ACF};\beta-j;t}}=&\underset{B_{g_\zeta}}{\m{sup}}\left\{\bigoplus_{i=0}^jw_{\m{ACF};(\beta-j+i-\alpha);t}(x,y)\cdot\right.\\
    &\left.\frac{\left|\m{word}^+_{i,j}(\widehat{\nabla}^{h^t_\zeta}_H,\widehat{\nabla}^{h^t_\zeta}_V) \Phi(x)-\Pi^{x,y}\m{word}^+_{i,j}(\widehat{\nabla}^{h^t_\zeta}_H,\widehat{\nabla}^{h^t_\zeta}_V) \Phi(y)\right|_{h^t_\zeta,g^t_\zeta}}{d_{g^t_\zeta}(x,y)^\alpha}\right\}
\end{align*} 
where $B_{g^t_\zeta}(U)=\{x,y\in U|\,x\neq y,\,\m{dist}_{g^t_\zeta}(x,y)\leq w_{\m{ACF};(\beta-j+i-\alpha);t}(x,y)\}\subset U\times U$ and $\Pi^{x,y}$ denotes the parallel transport along the geodesic connecting $x$ and $y$.
\end{defi}

\begin{rem}
All $W^{k,p}_{\m{ACF};\beta}$, $ \Gamma^k_{\m{ACF};\beta}$ and $  C^{l,\alpha}_{\m{ACF};\beta}$ are Banach spaces, and $H^{k}_{\m{ACF};\beta}=W^{k,2}_{\m{ACF};\beta}$ are Hilbert spaces. However, $ \Gamma_{\m{ACF};\beta}$ are not Banach spaces.
\end{rem}

\begin{thm}
\label{weightedsobolovembeddingACF}
The following holds:
\begin{itemize}
    \item Let $\frac{1}{p}+\frac{1}{q}=1$. The $L^2_{\m{ACF}}$-pairing induces a perfect pairing 
    \begin{align*}
        L^p_{\m{ACF};\beta;t}\cong(L^q_{\m{ACF};-\beta-m;t})^*
    \end{align*}
    \item If $\beta>\beta'$ then $W^{l,p}_{\m{ACF};\beta;t}\subset W^{l,p}_{\m{ACF};\beta';t}$ and $  C^{l,\alpha}_{\m{ACF};\beta;t}\subset  C^{l,\alpha}_{\m{ACF};\beta';t}$
    \item If $1\leq p,\infty$ then $ \Gamma_{\m{CFS};\beta;t}$ is dense in $W^{l,p}_{\m{CF};\beta;t}$
    \item (Sobolev, Kondrakov) If $(p_i,l_i)\in\mathbb{N}\times[1,\infty)$ with $l_0\leq l_1$ such that
    \begin{align*}
        l_0-\frac{m}{p_0}\geq l_1-\frac{m}{p_1}
    \end{align*}
    and one of the two conditions
    \begin{itemize}
        \item[(i)] $p_0\leq p_1$ and $\beta_0\geq\beta_1$
        \item[(ii)] $p_0>p_1$ and $\beta_0>\beta_1$
    \end{itemize}
    then there exists a continuous embedding 
    \begin{align*}
        W^{l_0,p_0}_{\m{ACF};\beta_0;t}\hookrightarrow W^{l_1,p_1}_{\m{ACF};\beta_1;t}.
    \end{align*}
    If moreover $l_0>l_1$, $l_0-\frac{m}{p_0}> l_1-\frac{m}{p_1}$ and $\beta_0>\beta_1$ then the above embedding is compact, i.e. any bounded sequence in the $(l_0,p_0)$ norm has a convergent subsequence in the $(l_1,p_1)$ norm.
    \item (Hölder) If $(l_i,\alpha)\in\mathbb{N}\times[0,\infty)$ and $\beta_0<\beta_1$ with 
    \begin{align*}
        l_0+\alpha_0\geq l_1+\alpha_1
    \end{align*}
    then there exist continuous embeddings
    \begin{align*}
          C^{l_0+1}_{\m{ACF};\beta_0;t}\hookrightarrow   C^{l_0,\alpha_0}_{\m{ACF};\beta_0;t}\hookrightarrow  C^{l_1,\alpha_1}_{\m{ACF};\beta_1;t}\hookrightarrow  C^{l_1}_{\m{ACF};\beta_1;t}.
    \end{align*}
    If $\alpha_1=0$ and $l_0=l_1$ the inclusion $  C^{l_0,\alpha_0}_{\m{ACF};\beta_0;t}\hookrightarrow  C^{l_1,\alpha_1}_{\m{ACF};\beta_1;t}$ is compact.
    \item (Morrey, Kondrakov) If $(p_i,l_i)\in\mathbb{N}\times[1,\infty)$ and $(l,\alpha)\in\mathbb{N}\times[1,\infty)$, $\beta>\delta$, and 
    \begin{align*}
        l_0-\frac{m}{p_0}\geq l+\alpha\geq l_1-\frac{m}{p_1}
    \end{align*}
    then there exists a continuous embedding 
    \begin{align*}
        W^{l_0,p_0}_{\m{ACF};\beta;t}\hookrightarrow   C^{l,\alpha}_{\m{ACF};\beta;t}\hookrightarrow W^{l_1,p_1}_{\m{ACF};\delta;t}.
    \end{align*}
    \item Let $(\Phi,\Psi)\mapsto(\phi\cdot \Phi)$ be a bilinear form satisfying $|\Phi\cdot \Psi|\lesssim|\Phi||\Psi|$, then 
\begin{align*}
    \left|\left|\Phi\cdot \Psi\right|\right|_{C^{l,\alpha}_{\m{ACF};\beta_1+\beta_2;t}}\lesssim&\left|\left|\Phi\right|\right|_{C^{l,\alpha}_{\m{ACF};\beta_1;t}}\left|\left| \Psi\right|\right|_{C^{l,\alpha}_{\m{ACF};\beta;t}}.
\end{align*}
\end{itemize}
\end{thm}

\begin{proof}
These statements follow immediately from the standard Sobolev-Kondrakov-, Hölder- and Morrey-Kondrakov-embedding theorems on compact and on conical spaces \cite[Thm. 4.18]{marshal2002deformations}.
\end{proof}

\subsection{Uniform Fredholm Theory for Adiabatic Families of Asymptotically Conically Fibred Dirac Operators}
\label{Uniform Fredholm Theory for Adiabatic Families of Asymptotically Conically Fibred Dirac Operators}

In the following section we will address the Fredholm analysis of the non-compact ACF-family of Dirac bundles, modelling the resolution locally. By the results of the prior section we know that the following is true.\\

In Section \ref{Invertibility of Conically Fibred Dirac Operators} we established that the operator $\widehat{D}_0$ can be realised as an isomorphism of weighted function spaces, subject to boundary conditions. In order to show that the orbifold Dirac operator can be realised as a Fredholm map, we will have to compare it to the operator $\widehat{D}_0$. The following lemma shows that close to the singular stratum, $D$ is approximated by $\widehat{D}_0$. Recall from \ref{Resolutions of Dirac Bundles} that 
\begin{align*}
   (\rho_\zeta)_*\widehat{D}^t_\zeta=\widehat{D}^t_0+\mathcal{O}(r^{1+\eta}). 
\end{align*}

To transfer data between the orbifold and the CF model, we define restriction and extension maps adapted to the geometry and weight structure.

\begin{defi}
\label{ACFresandACFext}
     Recall Definition \ref{CFresandCFext}. We define the \textbf{restriction map} 
    \begin{align*}
        \m{res}_{\m{ACF};\beta;t}:  C^{k,\alpha}_{\m{ACF};\beta;t}(N_\zeta,\widehat{E}_\zeta)\rightarrow 
             C^{k,\alpha}(Y,\widehat{E}_Y)
    \end{align*}
    \begin{align*}
        \Phi\mapsto \m{res}^\epsilon_{\m{CF};\beta}\delta^*_{t^{-1}}\chi^t_2\Phi
    \end{align*}
    and a right-inverse
    \begin{align*}
        \m{ext}_{\m{ACF}}:    C^{k,\alpha}(Y,\widehat{E}_Y)\rightarrow   C^{k,\alpha}_{loc}(N_\zeta,\widehat{E}_\zeta)
    \end{align*}
    \begin{align*}
        \Phi_{\Lambda,M_\Lambda}\mapsto  \chi^t_2 \m{ext}_{\m{CF}}\Phi_{\Lambda,M_\Lambda}.
    \end{align*}
\end{defi}

We now define a space of admissible sections for the Dirac operator that incorporates spectral boundary conditions of APS-type, adapted to the model geometry near the singular stratum.

\begin{defi}
\label{APSACFdefi}
Recall Definitions \ref{APSCFdefi} and \ref{ACFresandACFext}. Define the space 
\begin{align*}
    C^{k+1,\alpha}_{\m{ACF};\beta;t}(N_\zeta,\widehat{E}_\zeta;\m{APS})\coloneqq \m{res}^{-1}_{\m{ACF};\beta;t}(  C^{k+1,\alpha}(Y,E)_{\leq 0}).
\end{align*}
We define 
\begin{align*}
    C^{k+1,\alpha}_{\m{ACF};\beta;t}(\widehat{D}^t_\zeta;\m{APS})=\left\{\Phi\in C^{k+1,\alpha}_{\m{ACF};\beta;t}(N_\zeta,\widehat{E}_\zeta;\m{APS})|\widehat{D}^t_\zeta\Phi\in C^{k,\alpha}_{\m{ACF};\beta-1;t}(N_\zeta,\widehat{E}_\zeta)\right\}
\end{align*}
as the domain, equipped with the graph norm.
\end{defi}

The following lemma ensures that the newly defined space of admissible sections is a Banach space and behaves well analytically. In particular, it admits compact embeddings into higher weight spaces, which is crucial for establishing Fredholm properties.

\begin{lem}
\label{Dzeta-D0lemma}
    Let $\Phi\in C^{k+1,\alpha}_{\m{CF};\beta;t}(N_\zeta,\widehat{E}_\zeta;\m{APS})$ with $\m{supp}(\Phi)\subset N_\zeta\backslash B_{2t^{-1}\epsilon}(N_\zeta)$. Then 
    \begin{align}
        \label{Dzeta-D0}
        \left|\left|(\widehat{D}^t_\zeta-\widehat{D}^t_0)\Phi\right|\right|_{C^{k,\alpha}_{\m{ACF};\beta-1;t}}\lesssim t^{-\eta-1} \left|\left|\Phi\right|\right|_{C^{k+1,\alpha}_{\m{ACF};\beta;t}}.
    \end{align}
\end{lem}

\begin{proof}
We use the expansion 
\begin{align*}
    \widehat{D}^t_\zeta-\widehat{D}_0=(\m{cl}_{g^t_\zeta}-\m{cl}_{g^t_0})\circ\widehat{\nabla}^{h^t_\zeta}+\m{cl}_{g^t_\zeta}\circ\left(\widehat{\nabla}^{h^t_\zeta}-\nabla^{\otimes_0}\right).
\end{align*}
Notice that $\m{cl}_{g^t_\zeta}\circ\left(\widehat{\nabla}^{h^t_\zeta}-\nabla^{\otimes_0}\right)=\mathcal{O}(r^{1+\eta})$ is a zeroth order differential operator. The bound follows from the definition of the $C^{k,\alpha}_{\m{ACF};\beta}$-norms.
\end{proof}

As a next step, we establish weighted Schauder estimates for the orbifold Dirac operator. These estimates rely on the close approximation by the model operator and are key to proving regularity and compactness properties.

\begin{prop}[ACF Schauder Estimates]
\label{SchauderDtzeta}
Let $\Phi\in  C^{k+1,\alpha}_{\m{ACF};\beta;t}(\widehat{D}^t_{\zeta};\m{APS})$ as defined in Definition \ref{APSACFdefi}. Then the estimate
\begin{align}
    \left|\left|\Phi\right|\right|_{C^{k+1,\alpha}_{\m{ACF};\beta;t}(\widehat{D}^t_{\zeta})}\lesssim& \left|\left|\widehat{D}^t_{\zeta}\Phi\right|\right|_{C^{k,\alpha}_{\m{ACF};\beta-1;t}}+\left|\left|\Phi\right|\right|_{C^{0}_{\m{ACF};\beta;t}}
\end{align}
holds.
\end{prop}

\begin{proof}
We will prove \eqref{SchauderDtzeta} analogously to the proof of \eqref{SchauderDotimes0}. Let $\mathcal{U}_{\epsilon}$ denote a good cover of $N_\zeta$ by geodesic balls of radius $\epsilon$ and let $\nu_\zeta^{-1}\mathcal{U}_{\epsilon}$ denote the cover of $N_\zeta$ by open subsets of the form 
\begin{align*}
    \nu_\zeta^{-1}U\cong B_{\epsilon}(s) \times M_{\zeta(s)},
\end{align*}
such that for each $0<c$ there exists an $0<\epsilon$ such that 
\begin{align*}
    \left|\left|g^t_\zeta-g_S\oplus t^2\cdot g_{\zeta(s)}\right|\right|_{C^{0,\alpha}_{t}}<c
\end{align*}
holds on $\mathcal{U}_{\epsilon}$. As $g^t_\zeta$ is uniform close to the product structures on $\nu_\zeta^{-1}\mathcal{U}_{\epsilon}$, so is the Hermitian Dirac bundle structure. By now rescaling the balls and by using \cite[Prop. 7.7]{walpuski2012g_2} as well as \cite[Prop. 7.11]{walpuski2012g_2} we deduce the standard Schauder estimates.\\

Higher regularity automatically follows from the CF-Schauder estimate and the fact that $\widehat{D}^t_\zeta\sim \widehat{D}^t_0$
\end{proof}

With the analytic estimates and approximation results in place, we now conclude that the Dirac operator is Fredholm between weighted function spaces, provided the weight avoids critical rates. We also describe how the index jumps as the weight crosses these critical thresholds.

\begin{thm}
\label{FredholmACF}
Let $\beta\notin \mathcal{C}(\widehat{D}_0)$. The map 
    \begin{align*}
        \widehat{D}^t_\zeta:  C^{k+1,\alpha}_{\m{ACF};\beta;t}(\widehat{D}^t_\zeta;\m{APS})\rightarrow  C^{k,\alpha}_{\m{ACF};\beta-1;t}(N_\zeta,\widehat{E}_\zeta)
    \end{align*}
    is Fredholm and its kernel and cokernel are smooth, i.e. coincide for all $\alpha$ and $k$. Furthermore, if $\beta_2>\beta_1$ we have
\begin{align*}
    \m{ind}_{\m{ACF};\beta_1}(\widehat{D}^t_\zeta)-\m{ind}_{\m{ACF};\beta_2}(\widehat{D}^t_\zeta)=\sum_{\beta_2<\lambda<\beta_1}d_{\lambda+\frac{m-1}{2}-\delta(\widehat{E}_{0})}(\widehat{D}_0),
\end{align*}
where $d_{\lambda}(\widehat{D}_0)$ as defined in Definition \ref{logCFdefi}.
\end{thm}

\begin{proof}
    The first statement follows from Proposition \ref{widehatD0iso} and \ref{SchauderD} and goes analogously to the proof of Fredholmness for ACyl, AC or CS operators. Notice, that the Schauder estimate \eqref{SchauderD} does not suffice to prove left-semi Fredholmness as the inclusion of 
    \begin{align*}
        C^{k+1,\alpha}_{\m{ACF};\beta;t}(\widehat{D}^t_\zeta;\m{APS})\hookrightarrow C^0_{\m{ACF};\beta;t}(N_\zeta,\widehat{E}_\zeta)
    \end{align*}   
    is not compact. Thus, we need to improve the Schauder estimate to
\begin{align}
    \label{improvedSchauderDteq}
    \left|\left|\Phi\right|\right|_{C^{k+1,\alpha}_{\m{ACF};\beta;t}} \lesssim& \left|\left|\widehat{D}^t_\zeta\Phi\right|\right|_{C^{k,\alpha}_{\m{ACF};\beta-1;t}}+\left|\left|\Phi\right|\right|_{C^{0}_{\m{ACF};\delta;t}}
\end{align}
    for $\delta>\beta$. Let us define the operator  
\begin{align*}
    \widetilde{D}^t_{3\epsilon}=\widehat{D}^t_\zeta+(1-\chi^t_{3})(\widehat{D}^t_0-\widehat{D}^t_\zeta).
\end{align*}
This operator is conjugated to $\widehat{D}^t_0$ on  $N_\zeta\backslash\m{Tub}_{3\epsilon}(S)$ and coincides with $\widehat{D}^t_\zeta$ on the \say{compact} part of $N_\zeta$. 
\begin{align*}
    \left|\left|\Phi\right|\right|_{C^{k+1,\alpha}_{\m{ACF};\beta;t}}&\lesssim\left|\left|(1-\chi^t_{3})\Phi\right|\right|_{C^{k+1,\alpha}_{\m{ACF};\beta;t}}+\left|\left|\chi^t_{3}\Phi\right|\right|_{C^{k+1,\alpha}_{\m{ACF};\beta;t}}\\
    &\lesssim\left|\left|\widetilde{D}^t_{3\epsilon+\varepsilon}(1-\chi^t_{3})\Phi\right|\right|_{C^{k,\alpha}_{\m{ACF};\beta-1;t}}+\left|\left|(1-\chi^t_{3})\Phi\right|\right|_{C^{0}_{\m{ACF};\beta;t}}+\left|\left|\widetilde{D}^t_{3\epsilon}\chi^t_{3}\Phi\right|\right|_{C^{k,\alpha}_{\m{ACF};\beta-1;t}} 
\end{align*}
Here we used that $\widetilde{D}_{3\epsilon+\varepsilon}$ is conjugate to $\widehat{D}_0$ on $N_\zeta\backslash\m{Tub}_{3\epsilon+\varepsilon}(S)$ and by Proposition \ref{widehatD0iso} 
\begin{align*}
        \widehat{D}_0:C^{k+1,\alpha}_{\m{CF};\beta}(\widehat{D}_0;\m{APS})\rightarrow C^{k,\alpha}_{\m{CF};\beta-1}(N_0,\widehat{E}_{0})
\end{align*}
defines an isomorphism for $\beta\notin\mathcal{C}(\widehat{D}_0)$. Further we have
\begin{align*}
   \left|\left|\widetilde{D}^t_{3\epsilon+\varepsilon}(1-\chi^t_{3})\Phi\right|\right|_{C^{k,\alpha}_{\m{ACF};\beta-1;t}}\lesssim&\left|\left|(1-\chi^t_{3})\widetilde{D}^t_{3\epsilon+\varepsilon}\Phi\right|\right|_{C^{k,\alpha}_{\m{ACF};\beta-1;t}}\\
   &+\left|\left|[(1-\chi^t_{3}),\widetilde{D}^t_{3\epsilon+\varepsilon}]\Phi\right|\right|_{C^{k,\alpha}_{\m{ACF};\beta-1;t}}
\end{align*}
and
\begin{align*}
    \left|\left|\widetilde{D}^t_{3\epsilon+\varepsilon}\chi^t_{3}\Phi\right|\right|_{C^{k,\alpha}_{\m{ACF};\beta-1;t}}&\lesssim\left|\left|\chi^t_{3}\widetilde{D}^t_{3\epsilon+\varepsilon}\Phi\right|\right|_{C^{k,\alpha}_{\m{ACF};\beta-1;t}}+\left|\left|[\chi^t_{3},\widetilde{D}^t_{3\epsilon+\varepsilon}]\Phi\right|\right|_{C^{k,\alpha}_{\m{ACF};\beta-1;t}}
\end{align*}
Now $[\widetilde{D}^t_{3\epsilon+\varepsilon},(1-\chi^t_{3})]=-[\widetilde{D}^t_{3\epsilon+\varepsilon},\chi^t_{3}]=\sigma_{\widetilde{D}^t_{3\epsilon+\varepsilon}}(\m{d}\chi^t_{3})=K^t_{\varepsilon}$ is a compact operator. Further, 
\begin{align*}
    \left|\left|(1-\chi^t_{3})\Phi\right|\right|_{C^{0}_{\m{ACF};\beta;t}}\lesssim t^{\delta-\beta}\left|\left|(1-\chi^t_{3})\Phi\right|\right|_{C^{0}_{\m{ACF};\delta;t}}
\end{align*}
which leads to the following estimate
\begin{align*}
    \left|\left|\Phi\right|\right|_{C^{k+1,\alpha}_{\m{ACF};\beta}}&\lesssim\left|\left|\widetilde{D}^t_{3\epsilon+\varepsilon}\Phi\right|\right|_{C^{k,\alpha}_{\m{ACF};\beta-1;t}}+\left|\left|K^t_{\varepsilon}\Phi\right|\right|_{C^{k,\alpha}_{\m{ACF};\beta-1;t}}+t^{\delta-\beta}\left|\left|\chi^t_{3}\Phi\right|\right|_{C^{0}_{\m{CFS};\delta;\epsilon}}.
\end{align*}

We conclude that the family of operators $\widetilde{D}^t_{3\epsilon}$ are semi-Fredholm, i.e. have finite dimensional kernel and a closed image. Further, let $P^t_{2\epsilon}$ be a parametrix for the operator $\widetilde{D}^t_{3\epsilon}$ on $N_\zeta\backslash\m{Tub}_{3\epsilon}(S)$. Define the operator 
\begin{align*}
    Q^t_{\epsilon}=(1-\chi^t_{3})P^t_{2\epsilon}(1-\chi^t_{3\epsilon-\varepsilon})+\chi^t_{3}\left(\widehat{D}_0\right)^{-1}\chi_{3\epsilon-\varepsilon}.
\end{align*}
Then $\widetilde{D}^t_{3\epsilon}Q^t_{\epsilon}=1+\widetilde{K}^t_{\varepsilon}$, where $\widetilde{K}_{\varepsilon}$ is compact and hence, $\widetilde{D}^t_{3\epsilon}$ is Fredholm. Taking the limit $\varepsilon\to \infty$ proves the statement.\\

For the second statement, we follow \cite[Lem. 3.70]{bera2023deformCS} and it suffices to consider the case where $\lambda$ is the only critical rate between $\beta_1$ and $\beta_2$. Consider the space 
    \begin{align*}
        V_\lambda=\m{ker}\left(\widehat{D}_{Y;H;\lambda+\frac{m-1}{2}-\delta(\widehat{E}_{0})}\right),
    \end{align*}
    the projection $\Pi_\lambda:\Gamma(Y,\widehat{E}_Y)_{\leq 0}\rightarrow V_\lambda$, and the space 
    \begin{align*}
        S_{\beta_2}:=\left\{\Phi\in C^{k+1,\alpha}_{\m{ACF};\beta_2;t}(N_\zeta,\widehat{E}_\zeta;\m{APS})| \widehat{D}^t_\zeta\Phi\in C^{k,\alpha}_{\m{ACF};\beta_1-1;t}(N_\zeta,\widehat{E}_\zeta)\right\}.
    \end{align*}
    Then for $\Phi\in S_{\beta_2}$
    \begin{align*}
        \Phi-\m{ext}_{\m{ACF};\lambda;t}\circ \Pi_{\lambda}\circ \m{res}_{\m{ACF};\lambda;t}(\Phi)\in C^{k+1,\alpha}_{\m{ACF};\beta_1;t}(N_\zeta,\widehat{E}_\zeta;\m{APS})
    \end{align*}
    Moreover, we have $S_{\beta_2}\subset C^{k+1,\alpha}_{\m{ACF};\lambda;t}(\widehat{D}^t_\zeta;\m{APS})$, $\m{ker}_{\m{ACF};\beta_2}(\widehat{D}^t_\zeta)=\m{ker}_{\m{ACF};\lambda}(\widehat{D}^t_\zeta;\m{APS})$, and $C^{k+1,\alpha}_{\m{ACF};\beta_1;t}(\widehat{D}^t_\zeta)=\m{ker}(\Pi_{\lambda}\circ\m{res}_{\m{ACF};\lambda;t}:S_{\beta_2}\rightarrow V_\lambda)$ as well as
    \begin{align*}
        \m{ker}_{\m{ACF};\beta_2}(\widehat{D}^t_\zeta)=\m{ker}\left(\Pi_{\lambda}\circ\m{res}_{\m{ACF};\lambda;t}:\m{ker}_{\m{ACF};\beta_2}(\widehat{D}^t_\zeta)\rightarrow V_{\lambda}\right).
    \end{align*}
    Furthermore, we have 
    \begin{align*}
        S_{\beta_2}=C^{k+1,\alpha}_{\m{ACF};\beta_1;t}(\widehat{D}^t_\zeta;\m{APS})+\m{im}\left(\m{ext}_{\m{ACF};\lambda;t}|_{V_\lambda}\right)
    \end{align*}
    and 
    \begin{align*}
        \m{ker}_{\m{ACF};\beta_2}(\widehat{D}^t_\zeta)=&\m{ker}\left(\widehat{D}^t_\zeta|_{S_{\beta_2}}:S_{\beta_2}\rightarrow C^{k,\alpha}_{\m{ACF};\beta_1-1;t}(N_\zeta,\widehat{E}_\zeta)\right)\\\m{coker}_{\m{ACF};\beta_2}(\widehat{D}^t_\zeta)=&\m{coker}\left(\widehat{D}^t_\zeta|_{S_{\beta_2}}:S_{\beta_2}\rightarrow C^{k,\alpha}_{\m{ACF};\beta_1-1;t}(N_\zeta,\widehat{E}_\zeta)\right).
    \end{align*}
    The wall-crossing formula follows immediately.
\end{proof}

Notice, that the parametrices are not uniform bounded in $t$ and hence, not suitable for further use in gluing analysis. The following section is devoted to carefully analysing the family $\widehat{D}^t_\zeta$ and to construct adapted norms that allow uniform bounds on a right-inverse of $\widehat{D}^t_\zeta$.

\subsubsection{Pushforwards of ACF Dirac Bundles and Concentrating Dirac Operators}
\label{Pushforwards of ACF Dirac Bundles and Concentrating Dirac Operators}

Using the projection $\nu_\zeta:N_\zeta\rightarrow S$ and the Ehresmann connection $H_\zeta$ we define the $C^{\infty}_{\m{ACF};\beta}$-pushforward of the Clifford module $(\widehat{E}_\zeta,\m{cl}_{g_\zeta})$ on $S$.\\

Notice that even though $C^{\infty}_{loc}$-sections of $\widehat{E}_\zeta$ define a sheaf $\Gamma_{loc}(\widehat{E}_{\zeta})$ of Fréchet spaces on $N_\zeta$; the weighted $C^{\infty}_{\m{ACF};\beta}$-sections only form a presheaf by the lack of gluing. Nevertheless, the $\nu_\zeta$-pushforward of $C^{\infty}_{\m{ACF};\beta;t}$ defines a coherent, locally free sheave of $(\nu_\zeta)_*C^\infty_{ACyl}$-modules
\begin{align*}
    (\nu_{\zeta})_*\Gamma_{\m{ACF};\beta;t}(\widehat{E}_\zeta)
\end{align*}
on $S$. This sheave can be identified with the sheave of $C^{\infty}$-sections of the infinite dimensional Frechét vector bundle 
\begin{align*}
    (\nu_\zeta)_{*_{\m{ACF};\beta}}\widehat{E}_\zeta.
\end{align*}
In a similar way we define the $H^\infty_{\m{ACF}}$-pushforward 
\begin{align*}
    (\nu_\zeta)_{*_{H^\infty_{\m{ACF}}}}\widehat{E}_\zeta.
\end{align*}

\begin{lem}
    The $C^\infty_{\m{ACF};\beta}$-pushforward of $\widehat{E}_\zeta$ can be endowed with the structure of an infinite-dimensional Fréchet Dirac bundle. The $H^\infty_{\m{ACF}}$-pushforward can be endowed with the structure of an infinite-dimensional Hermitian Fréchet Dirac bundle.
\end{lem}

\begin{proof}

We define a $\m{Cl}(TS,g_S)$-module structures on $(\nu_\zeta)_{*_{\m{ACF};\beta}}\widehat{E}_\zeta$ by 
\begin{align*}
    \m{cl}_{g_S}(\xi)\Phi=\m{cl}_{g^t_\zeta}(\nu_\zeta^*\xi)\Phi=\m{cl}_{g_\zeta}(\nu_\zeta^*\xi)\Phi
\end{align*}
for vector fields $\xi\in \Omega^1(S)$. Furthermore, the connection $\nabla^{\otimes_\zeta}$ induces a connection on $(\nu_\zeta)_{*_{\m{ACF};\beta}}\widehat{E}_\zeta$ by 
\begin{align*}
    \nabla^{\otimes_\zeta}_V\Phi=\nabla^{\otimes_\zeta}_{V^{H_\zeta}}\Phi,
\end{align*}
whose curvature is given by 
\begin{align*}
    F_{\nabla^{\otimes_\zeta}}(U,V)=F_{\nabla^{\otimes_\zeta}}(U^{H_\zeta},V^{H_\zeta})-\nabla^{\otimes_\zeta}_{F_{H_\zeta}(U,V)}.
\end{align*}

In a similar way, we define the vector bundle 
\begin{align*}
    (\nu_{\zeta})_{*_{H^\infty_{\m{ACF}}}}\widehat{E}_{\zeta}
\end{align*}
and equip this bundle with a Dirac bundle structure analogously to the above. The $H^\infty_{\m{ACF}}$-pushforward is equipped with Hermitian structure 
\begin{align*}
    h^t_\zeta(\Phi,\Phi')_s=\int_{\nu_\zeta^{-1}(s)}\left<\Phi,\Phi'\right>_{h^t_\zeta}\m{vol}_{g^t_{\zeta;V}}
\end{align*}
given by the fibrewise $L^2_t$-product. Even though the connection $\nabla^{\otimes^0_\zeta}$ fails to be a Hermitian connection, the connection 
\begin{align*}
    \widehat{\nabla}^{\otimes_\zeta}=\nabla^{\otimes_\zeta}-\tfrac{1}{2}k_\zeta.
\end{align*}
is Hermitian.

\begin{rem}
    In Example \ref{NzetaasinResolutionsofSpin7Orbifolds} the fibres of $(N_\zeta,g_\zeta)\rightarrow (S,g_S)$ are minimal, hence $k_\zeta=0$.
\end{rem}

In particular, the latter is a $\m{Cl}(T^*S,g_S)$-module connection and 
thus
\begin{align*}
    ((\nu_{\zeta})_{*_{\m{ACF};\beta}}\widehat{E}_{\zeta},\m{cl}_{g_S},h^t_\zeta,\widehat{\nabla}^{h^t_\zeta})
\end{align*}
is an infinite dimensional Fréchet Dirac bundle and
\begin{align*}
    ((\nu_{\zeta})_{*_{H^\infty_{\m{ACF}}}}\widehat{E}_{\zeta},\m{cl}_{g_S},h^t_\zeta,\widehat{\nabla}^{h^t_\zeta})
\end{align*}
is an infinite dimensional Hermitian Dirac bundle.
    
\end{proof}

Notice that a grading $\epsilon_\zeta$ on $(\widehat{E}_\zeta,\m{cl}_{g^t_\zeta},h^t_\zeta,\widehat{\nabla}^{h^t_\zeta})$ induces a grading on both the $C^{\infty}_{\m{ACF};\beta}$ and the $L^2_{\m{ACF}}$-pushforward.\\

Instead of viewing the Dirac Operator $\widehat{D}^t_{\zeta}$ as an operator
\begin{align*}
    \widehat{D}^t_{\zeta}:\Gamma_{\m{ACF};\beta;t}(\widehat{D}^t_{\zeta};\m{APS})\rightarrow \Gamma_{\m{ACF};\beta-1;t}(N_\zeta,\widehat{E}_\zeta)
\end{align*}
we will now consider it as an operator 
\begin{align*}
    \widehat{D}^t_{\zeta}:\m{dom}(\widehat{D}^t_{\zeta};\m{APS})\subset \Gamma(S,(\nu_\zeta)_{*_{\m{ACF};\beta}}\widehat{E}_\zeta)\rightarrow  \Gamma(S,(\nu_\zeta)_{*_{\m{ACF};\beta-1}}\widehat{E}_\zeta)
\end{align*}

\begin{lem}
    The horizontal Dirac operator defines a Dirac operator
    \begin{align*}
        \widehat{D}^t_{\zeta;H}:\Gamma(S,(\nu_\zeta)_{*_{\m{ACF};\beta}}\widehat{E}_\zeta)\rightarrow \Gamma(S,(\nu_\zeta)_{*_{\m{ACF};\beta}}\widehat{E}_\zeta).
    \end{align*}
\end{lem}

\begin{proof}
The horizontal Dirac operator $\widehat{D}^t_{\zeta;H}$ as established in Section \ref{Resolutions of Dirac Bundles} is given by 
    \begin{align*}
    \widehat{D}^t_{\zeta;H}=\m{cl}_{g_\zeta}\circ\m{pr}_{H_\zeta}\circ\widehat{\nabla}^{\otimes_\zeta}+\m{cl}_{g_\zeta}(\tau_{\zeta;H}^t)+t^{-1}\cdot\m{cl}_{g_\zeta}(\tau^t_{\zeta;V})-\frac{t}{4}\m{cl}_{g_\zeta}(F_{H_\zeta}-\rho^*_\zeta F_{H_0}).
\end{align*}
Its $(\nu_\zeta)_{*_{\m{ACF};\beta}}$-pushforward is a first order differential operator on sections of $(\nu_\zeta)_{*_{\m{ACF};\beta}}$ whose symbol is given by 
\begin{align*}
    \sigma_{\widehat{D}^t_{\zeta;H;\beta}}(\xi)=\m{cl}_{g_\zeta}(\xi^H)=\m{cl}_{g_S}(\xi)
\end{align*}
and hence, is a Dirac operator.
\end{proof}

By the previous section, we know that the operator $\widehat{D}^t_{\zeta}$ is of the form 
\begin{align*}
    \widehat{D}^t_{\zeta}=\widehat{D}^t_{\zeta;H}+t^{-1}\cdot \widehat{D}_{\zeta;V}
\end{align*}
and we notice that the vertical Dirac operator $\widehat{D}_{\zeta;V}\in  \Gamma(S,\m{End}((\nu_\zeta)_{*_{L^2}}\widehat{E}_\zeta))$ and hence, $\widehat{D}^t_\zeta$ is family of concentrating Dirac operators\footnote{A family of Dirac operators is called concentrating if $D^t=D+t^{-1}\cdot \Phi$ for $\Phi\in \Gamma(S,\m{End}(E_S)$ (see \cite{parker2023deformations}}. The idea of Bismut in \cite{bismut1986} was that the kernel of $D^t_\zeta$ is determined by the operator that concentrates on the kernel of the vertical Dirac operator. In the following we will make this idea concrete in the present case. Notice that a major difference between Bismut's construction and ours is that the fibres of $N_\zeta$ are noncompact.

\begin{defi}
    Define the vector bundle $\mathfrak{l}_{\m{C},\lambda}\colon\mathcal{L}_{\m{C};\beta}(\widehat{D}_{0;V})\rightarrow S$ of conically $\beta$\textbf{-logarithmic solutions} of $\widehat{D}_{0;V}$ by 
    \begin{align*}
        \mathcal{L}_{\m{C};\beta}(\widehat{D}_{0;V})\coloneqq\left\{\Phi=\left.\sum^n_{i=0}\m{log}^i(r)\Phi_i\in \Gamma_{C;\beta}(\nu_0^{-1}(s),\widehat{E}_0)\right|\widehat{D}_{0;V}\Phi=0,\Phi_i\text{ is }\beta\text{-homogen.}\right\}.
    \end{align*}
\end{defi}

\begin{lem}\cite[Thm. 4.22, Thm. 4.24]{marshal2002deformations}
The vertical Dirac operator defines bounded linear maps
\begin{align*}
    \widehat{D}_{\zeta;V}\colon&W^{k+1,p}_{\m{AC};\beta}(\nu_\zeta^{-1}(s),\widehat{E}_\zeta)\rightarrow W^{k,p}_{\m{AC};\beta-1}(\nu_\zeta^{-1}(s),\widehat{E}_\zeta)\\
    \widehat{D}_{\zeta;V}\colon&  C^{k+1,\alpha}_{\m{AC};\beta}(\nu_\zeta^{-1}(s),\widehat{E}_\zeta)\rightarrow   C^{k,\alpha}_{\m{AC};\beta-1}(\nu_\zeta^{-1}(s),\widehat{E}_\zeta).
\end{align*}
These maps are Fredholm as long as $\beta\notin\mathcal{C}(\widehat{D}_0)$ and by elliptic regularity their kernels are independent of $k$. Furthermore, the indices of these operators satisfy the wall-crossing formula 
\begin{align*}
    \m{ind}_{\m{AC};\beta_1}\left(\widehat{D}_{\zeta;V}\right)-\m{ind}_{\m{AC};\beta_2}\left(\widehat{D}_{\zeta;V}\right)=\sum_{\beta_1<\beta<\beta_2}d_{\beta}(\widehat{D}_{0;V}),
\end{align*}
where $d_{\lambda}(\widehat{D}_{0;V})\coloneqq \m{rank}\left(\mathcal{L}_{\m{C};\lambda}(\widehat{D}_{0;V})\right)$.
\end{lem}

\begin{defi}
\label{splittingdefi}
    We define the splitting 
    \begin{align*}
        (\nu_\zeta)_{*_{\m{ACF};\beta}}\widehat{E}_\zeta\cong \mathcal{K}_{\m{AC};\beta}(\pi_\zeta/\nu_\zeta)\oplus\mathcal{C}o\mathcal{I}_{\m{AC};\beta}(\pi_\zeta/\nu_\zeta)
    \end{align*}
    into the kernel and coimage of $\widehat{D}_{\zeta;V}$ as well as the splitting of
    \begin{align*}
        (\nu_\zeta)_{*_{\m{ACF};\beta-1}}\widehat{E}_\zeta\cong \mathcal{C}o\mathcal{K}_{\m{AC};\beta-1}(\pi_\zeta/\nu_\zeta)\oplus\mathcal{I}_{\m{AC};\beta-1}(\pi_\zeta/\nu_\zeta)
    \end{align*}
    into cokernel and image of $\widehat{D}_{0;V}$.
\end{defi}

As long as $\beta\notin\mathcal{C}(\widehat{D}_0)$, both $\mathcal{K}_{\m{AC};\beta}(\pi_\zeta/\nu_\zeta)$ and $\mathcal{C}o\mathcal{K}_{\m{AC};\beta-1}(\pi_\zeta/\nu_\zeta)$ are finite dimensional, the former remains finite dimensional even when $\beta\in\mathcal{C}(\widehat{D}_0)$. In particular, the vertical Dirac operator defines an isomorphism 
\begin{align*}
    \widehat{D}_{\zeta;V}:\mathcal{C}o\mathcal{I}_{\m{AC};\beta}(\pi_\zeta/\nu_\zeta)\xrightarrow[]{\cong}\mathcal{I}_{\m{AC};\beta-1}(\pi_\zeta/\nu_\zeta).
\end{align*}

\begin{manualassumption}{5}
\label{ass5}
Assume that $\mathcal{K}_{\m{AC};\beta}(\pi_\zeta/\nu_\zeta)$ and $\mathcal{C}o\mathcal{K}_{\m{AC};\beta-1}(\pi_\zeta/\nu_\zeta)$ form vector bundles of the kernel of the vertical Dirac Operator, i.e. the fibres at $s\in S$ are given by
\begin{align*}
    \mathcal{K}_{\m{AC};\beta}(\pi_\zeta/\nu_\zeta)_s=\m{ker}\left(\widehat{D}_{\zeta;V}(s):C^{k+1,\alpha}_{\m{ACF};\beta}(\nu_\zeta^{-1}(s),\widehat{E}_\zeta)\rightarrow C^{k,\alpha}_{\m{ACF};\beta-1}(\nu_\zeta^{-1}(s),\widehat{E}_\zeta)\right)
\end{align*}
\begin{align*}
    \mathcal{C}o\mathcal{K}_{\m{AC};\beta-1}(\pi_\zeta/\nu_\zeta)_s=\m{coker}\left(\widehat{D}_{\zeta;V}(s):C^{k+1,\alpha}_{\m{ACF};\beta}(\nu_\zeta^{-1},\widehat{E}_\zeta)\rightarrow C^{k,\alpha}_{\m{ACF};\beta-1}(\nu_\zeta^{-1},\widehat{E}_\zeta)\right)
\end{align*}
\end{manualassumption}

\begin{lem}
\label{exactsequenceverticalbundles}
Recall Lemma \ref{Elambda}. There exists and isomorphism of vector bundles 
\begin{align*}
    \mathcal{L}_{\m{C};\beta}(\widehat{D}_{0;V})\cong\mathcal{E}_{\beta+\frac{m-1}{2}-\delta(\widehat{E}_0)}.
\end{align*}
In addition, assume that Assumption \ref{ass5} holds. Then for $\beta_1>\beta_2$ there exists an exact sequence  
\begin{equation*}
    \begin{tikzcd}
	\begin{array}{c}
        \mathcal{K}_{\m{AC};\beta_2}(\pi_\zeta/\nu_\zeta)\\
        \ominus\\
        \mathcal{C}o\mathcal{K}_{\m{AC};\beta_2}(\pi_\zeta/\nu_\zeta)\\
	\end{array} & \begin{array}{c}
	      \mathcal{K}_{\m{AC};\beta_1}(\pi_\zeta/\nu_\zeta) \\
	     \ominus \\
      \mathcal{C}o\mathcal{K}_{\m{AC};\beta_1}(\pi_\zeta/\nu_\zeta)
	\end{array}  & {\bigoplus_{\beta_2<\beta<\beta_1}\mathcal{L}_{\m{C};\beta}(\widehat{D}_{0;V})}  	
	\arrow[from=1-1, hook, to=1-2]
	\arrow[from=1-2,two heads, to=1-3]
\end{tikzcd}
\end{equation*}
where $\ominus$ denotes the formal difference.
\end{lem}

\begin{proof}
For the first part follows analogous to proof of Lemma \ref{logarithmicdecaybundle}. \\

This statement follows the lines of the proof of Theorem \ref{FredholmCFS}. It suffices to consider the case where $\lambda$ is the only critical rate between $\beta_1$ and $\beta_2$. Consider the projection $\Pi_\lambda:\vartheta_*\widehat{E}_Y\rightarrow \mathcal{L}_\lambda(\widehat{D}_{0;V})$ and the subbundle 
    \begin{align*}
        S_{\beta_2}:=\widehat{D}_{\zeta;V}^{-1}\left((\nu_\zeta)_{*_{\m{ACF};\beta_1-1}}\widehat{E}_\zeta\right)\subset (\nu_\zeta)_{*_{\m{ACF};\beta_2}}\widehat{E}_\zeta
    \end{align*}
    Then for $\Phi\in S_{\beta_2}$
    \begin{align*}
        \Phi-\m{ext}_{\m{ACF};\lambda}\circ \Pi_{\lambda}\circ \m{res}_{\m{ACF};\lambda}(\Phi)\in (\nu_\zeta)_{*_{\m{ACF};\beta_1;\epsilon}}\widehat{E}_\zeta.
    \end{align*}
    Moreover, we have $S_{\beta_2}\subset (\nu_\zeta)_{*_{\m{ACF};\lambda}}\widehat{E}_\zeta$, $\m{ker}_{\m{CF};\beta_2}(\widehat{D}_{0;V})=\m{ker}_{\m{CF};\lambda}(\widehat{D}_{0;V})$, and 
    \begin{align*}
        (\nu_\zeta)_{*_{\m{CFS};\beta_1;\epsilon}}\widehat{E}_\zeta=\m{ker}(\Pi_{\lambda}\circ\m{res}_{\m{ACF};\lambda}:S_{\beta_2}\rightarrow \mathcal{L}_\lambda(\widehat{D}_{0;V}))
    \end{align*}
    as well as
    \begin{align*}
        \m{ker}_{\beta_1}(\widehat{D}_{0;V})=\m{ker}\left(\Pi_{\lambda}\circ\m{res}_{\m{ACF};\lambda}:\m{ker}_{\beta_2}(\widehat{D}_{0;V})\rightarrow \mathcal{L}_{\lambda}(\widehat{D}_{0;V})\right).
    \end{align*}
    Furthermore, we have 
    \begin{align*}
        S_{\beta_2}=(\nu_\zeta)_{*_{\m{ACF};\beta_1}}\widehat{E}_\zeta+\m{im}\left(\m{ext}_{\m{ACF};\lambda}|_{\mathcal{L}_\lambda(\widehat{D}_{0;V})}\right)
    \end{align*}
    and 
    \begin{align*}
        \m{ker}_{\beta_2}(\widehat{D}_{0;V})=&\m{ker}\left(\widehat{D}_{0;V}|_{S_{\beta_2}}:S_{\beta_2}\rightarrow (\nu_\zeta)_{*_{\m{ACF};\beta_1-1}}\widehat{E}_\zeta\right)\\\m{coker}_{\beta_2}(\widehat{D}_{0;V})=&\m{coker}\left(\widehat{D}_{0;V}|_{S_{\beta_2}}:S_{\beta_2}\rightarrow (\nu_\zeta)_{*_{\m{ACF};\beta_1-1}}\widehat{E}_\zeta\right).
    \end{align*}
    The wall-crossing formula follows immediately.
\end{proof}

\begin{rem}
    If $\beta<-\frac{m}{2}$, then there exists a compact embedding 
    \begin{align*}
        (\nu_\zeta)_{*_{\m{ACF};\beta}}\widehat{E}_\zeta\hookrightarrow (\nu_\zeta)_{*_{L^2_{\m{ACF}}}}\widehat{E}_\zeta
    \end{align*}
    and hence, $\mathcal{K}_{\m{AC};\beta}(\pi_\zeta/\nu_\zeta)$ and $\mathcal{C}o\mathcal{K}_{\m{AC};\beta-1}(\pi_\zeta/\nu_\zeta)$ carry  Hermitian structures given by the $L^2_{\m{ACF}}$-product. Further, there exists an identification 
    \begin{align*}
        \mathcal{C}o\mathcal{K}_{\m{AC};\beta-1}(\pi_\zeta/\nu_\zeta)\cong \mathcal{K}_{\m{AC};-m-\beta-1}(\pi_\zeta/\nu_\zeta)
    \end{align*}
    and hence, both $\mathcal{K}_{\m{AC};\beta}(\pi_\zeta/\nu_\zeta)$ and $\mathcal{C}o\mathcal{K}_{\m{AC};\beta-1}(\pi_\zeta/\nu_\zeta)$ are Hermitian for all $\beta\notin\mathcal{C}(\widehat{D}_0)$.
\end{rem}

\begin{rem}
    There exists a family of Hermitian structures on $\mathcal{K}_{\m{AC};\beta}(\pi_\zeta/\nu_\zeta)$ given by the fibrewise $L^2_t$-product. The vector bundle $\mathcal{K}_{\m{AC};\beta}(\pi_\zeta/\nu_\zeta)$ will be referred to as the vector bundle of vertical $L^2$-kernels. Notice, that if $\beta<-\frac{m}{2}$ then 
    \begin{align*}
        C^{k+1,\alpha}_{\m{AC};\beta}(\nu_\zeta^{-1},\widehat{E}_\zeta)\subset L^2_{\m{AC}}(\nu_\zeta^{-1},\widehat{E}_\zeta)
    \end{align*}
    and thus the bundles 
    $\mathcal{K}_{\m{AC};\beta}(\pi_\zeta/\nu_\zeta)$ and $\mathcal{C}o\mathcal{K}_{\m{AC};\beta-1}(\pi_\zeta/\nu_\zeta)$
    are finite dimensional Hermitian Dirac bundles.
\end{rem}

\begin{rem}    
Let us briefly imagine that the fibres of $N_\zeta$ were compact. If $m$ is even, the vertical Dirac operator $\widehat{D}_{\zeta;V}$ is of the form
\begin{align*}
    \widehat{D}_{\zeta;V}=\left(\begin{array}{cc}
         0&D^+_{\zeta;V}  \\
         D^-_{\zeta;V}& 0
    \end{array}\right)
\end{align*}
In this case the index of the family of Dirac operators $D^\pm_{\zeta;V}$ yields a well-defined class 
\begin{align*}
    [\m{Ind}(D^\pm_{\zeta;V})]&\in K^0(S)\hspace{1.0cm}\text{If }m\equiv 0\, \m{mod} 2\\
    [\m{Ind}(\widehat{D}_{\zeta;V})]&\in K^1(S)\hspace{1.0cm}\text{If }m\equiv 1\, \m{mod} 2\\
\end{align*}
Under a small generic perturbation $D^{(\pm)}_{\zeta;V}+\hbar$ of $D^{(\pm)}_{\zeta;V}$ the bundle 
\begin{align*}
    \m{Ind}(D^{(\pm)}_{\zeta;V}+\hbar)=\m{ker}(D^{(\pm)}_{\zeta;V}+\hbar)\in\cat{Vect}_{k}(S)
\end{align*}
and 
\begin{align*}
    [\m{Ind}(D^{(\pm)}_{\zeta;V})]=[\m{Ind}(D^{(\pm)}_{\zeta;V}+\hbar)]\in K^\bullet(S).
\end{align*}
As $N_\zeta$ is a noncompact fibration, the construction of a family index bundle becomes more intricate. The Fredholm property of these operators depends on the realisation of these operators on weighted function spaces. In Section \ref{Uniform Elliptic Theory of Adiabatic Families of ACF Dirac Operators} we will discuss this issue in more detail.
\end{rem}

\begin{lem}
\label{bundelsexists}
Let $D=\m{d}+\m{d}^*$. Then 
\begin{align*}
    \mathcal{K}_{\m{AC};\beta}(\pi_\zeta/\nu_\zeta)\cong&\wedge^\bullet T^*S\otimes\mathcal{H}^\bullet_{\m{AC};\beta}(N_\zeta/S)\\
    \mathcal{C}o\mathcal{K}_{\m{AC};\beta-1}(\pi_\zeta/\nu_\zeta)\cong&\wedge^\bullet T^*S\otimes\mathcal{C}o\mathcal{H}^\bullet_{\m{AC};\beta}(N_\zeta/S)
\end{align*}
Moreover, in Example \ref{NzetaasinResolutionsofSpin7Orbifolds} the bundles $\mathcal{H}_{\m{AC};\beta}(N_\zeta/S)$ and $\mathcal{C}o\mathcal{H}_{\m{AC};\beta}(N_\zeta/S)$ always exist and coincide with 
    \begin{align*}
        \mathcal{H}_{\m{AC};\beta}(N_\zeta/S)=\zeta^*\mathcal{H}_{\m{AC};\beta}(\mathfrak{M}/\mathfrak{P})
    \end{align*}
    and
    \begin{align*}
        \mathcal{C}o\mathcal{H}_{\m{AC};\beta}(N_\zeta/S)=\zeta^*\mathcal{C}o\mathcal{H}_{\m{AC};\beta}(\mathfrak{M}/\mathfrak{P}).
    \end{align*}
\end{lem}

\begin{proof}
    The existence of the bundles of vertically AC-harmonic and AC-coharmonic forms on $\mathfrak{M}$ follows directly from the algebraic construction of $\mathfrak{M}$ and the family version of \cite{lockhart1985elliptic,lockhart1987fredholm}. 
\end{proof}

\begin{defi}
\label{projectionbundledefi}
Let us define the projections
\begin{align*}
    \pi_{\mathcal{K};\beta}:& \Gamma(S,(\nu_\zeta)_{*_{\m{ACF};\beta}}\widehat{E}_\zeta)\rightarrow  \Gamma\left(S,\mathcal{K}_{\m{AC};\beta}(\pi_\zeta/\nu_\zeta)\right)\\
    \pi_{\mathcal{C}o\mathcal{I};\beta}:& \Gamma(S,(\nu_\zeta)_{*_{\m{ACF};\beta}}\widehat{E}_\zeta)\rightarrow  \Gamma\left(S,\mathcal{C}o\mathcal{I}_{\m{AC};\beta}(\pi_\zeta/\nu_\zeta)\right)\\
    \pi_{\mathcal{C}o\mathcal{K};\beta-1}:& \Gamma(S,(\nu_\zeta)_{*_{\m{ACF};\beta-1}}\widehat{E}_\zeta)\rightarrow  \Gamma\left(S,\mathcal{C}o\mathcal{K}_{\m{AC};\beta-1}(\pi_\zeta/\nu_\zeta)\right)\\
    \pi_{\mathcal{I};\beta-1}:& \Gamma(S,(\nu_\zeta)_{*_{\m{ACF};\beta-1}}\widehat{E}_\zeta)\rightarrow  \Gamma\left(S,\mathcal{I}_{\m{AC};\beta-1}(\pi_\zeta/\nu_\zeta)\right).
\end{align*}

Let $R_{V;\beta}$ be a right-inverse of $\widehat{D}_{\zeta;V}$ and let $L_{V;-m-\beta-1}$ be its adjoint. Notice, that both $R_{V;\beta}$ and $L_{V;-m-\beta-1}$ are bounded Fredholm maps of Banach bundles. We can write the projections $\pi_{\mathcal{C}o\mathcal{I};\beta}$ and $\pi_{\mathcal{I};\beta-1}$ as 
\begin{align}
\label{coimbeta}
   \pi_{\mathcal{C}o\mathcal{I};\beta}=& R_{V;\beta}\widehat{D}_{\zeta;V}=\widehat{D}_{\zeta;V}^*L_{V;-\beta-m-1}\\
\label{imbeta}
    \pi_{\mathcal{I};\beta-1}=& \widehat{D}_{\zeta;V}R_{V;\beta}=L_{V;-\beta-m-1}\widehat{D}_{\zeta;V}^*.
\end{align}
\end{defi}

\begin{rem}
\label{projectioncommuteswithcliffordhorizontal}
The operator $\pi_{\mathcal{K};\beta}$ and $\pi_{\mathcal{C}o\mathcal{K};\beta-1}$ are fibrewise finite rank smoothing operators that commutes with the Clifford multiplication by horizontal forms.
\end{rem}

\begin{lem}
The family of Dirac bundle structures of $(\nu_\zeta)_{*_{\m{ACF};\beta}}\widehat{E}_\zeta$ and $(\nu_\zeta)_{*_{\m{ACF};\beta-1}}\widehat{E}_\zeta$ restricts to Dirac bundle structures on 
\begin{align*}
    \mathcal{K}_{\m{AC};\beta}(\pi_\zeta/\nu_\zeta),\hspace{0.5cm}\mathcal{C}o\mathcal{I}_{\m{AC};\beta}(\pi_\zeta/\nu_\zeta),\hspace{0.5cm}\mathcal{C}o\mathcal{K}_{\m{AC};\beta-1}(\pi_\zeta/\nu_\zeta)\und{1.0cm}\mathcal{I}_{\m{AC};\beta-1}(\pi_\zeta/\nu_\zeta).
\end{align*}
Moreover, the induced Hermitian structure is Hermitian Dirac.
\end{lem}

We decompose the Dirac operator into
\begin{align*}
    \widehat{D}^t_{\zeta}=\left(\begin{array}{cc}
         \widehat{D}^t_{\zeta;\beta;\mathcal{C}o\mathcal{II}}&\widehat{D}^t_{\zeta;\beta;\mathcal{KI}}\\
         \widehat{D}^t_{\zeta;\beta;\mathcal{C}o\mathcal{I }\mathcal{C}o\mathcal{K}}&\widehat{D}^t_{\zeta;\beta;\mathcal{K}\mathcal{C}o\mathcal{K}}\\
    \end{array}\right)
\end{align*}
according to the decomposition of the pushforward of the Dirac bundle.

\begin{lem}
    \label{offdiagonalvanishing}
    If $\widehat{D}^t_{\zeta;H}$ and $\widehat{D}^t_{\zeta;V}$ anti-commute the off-diagonal terms 
    \begin{align*}
        \widehat{D}^t_{\zeta;\beta;\mathcal{KI}}:&\Gamma(S,\mathcal{K}_{\m{AC};\beta}(\pi_\zeta/\nu_\zeta))\rightarrow \Gamma(S,\mathcal{I}_{\m{AC};\beta-1}(\pi_\zeta/\nu_\zeta))\\
        \widehat{D}^t_{\zeta;\beta;\mathcal{C}o\mathcal{I }\mathcal{C}o\mathcal{K}}:&\Gamma(S,\mathcal{C}o\mathcal{I}_{\m{AC};\beta}(\pi_\zeta/\nu_\zeta))\rightarrow \Gamma(S,\mathcal{C}o\mathcal{K}_{\m{AC};\beta-1}(\pi_\zeta/\nu_\zeta))
    \end{align*}
    vanish.
\end{lem}

\begin{proof}
We can rewrite the off-diagonal terms as 
\begin{align*}
    \widehat{D}^t_{\zeta;\beta;\mathcal{KI}}=&\pi_{\mathcal{I};\beta-1}\widehat{D}^t_{\zeta;H}\pi_{\mathcal{K};\beta}\\
    =&L_{V;-\beta-m-1}\widehat{D}_{\zeta;V}\widehat{D}^t_{\zeta;H}\pi_{\mathcal{K};\beta}\\
    =&L_{V;-\beta-m-1}\{\widehat{D}_{\zeta;V},\widehat{D}^t_{\zeta;H}\}\pi_{\mathcal{K};\beta}\\
    \widehat{D}^t_{\zeta;\beta;\mathcal{C}o\mathcal{I }\mathcal{C}o\mathcal{K}}=&\pi_{\mathcal{C}o\mathcal{K};\beta-1}\widehat{D}^t_{\zeta;H}\widehat{D}_{\zeta;V}L_{V;-\beta-m-1}\\
    =&\pi_{\mathcal{C}o\mathcal{K};\beta-1}\{\widehat{D}^t_{\zeta;H},\widehat{D}_{\zeta;V}\}L_{V;-\beta-m-1}
\end{align*}
\end{proof}

\subsubsection{Vertically Harmonic Sections}
\label{Vertically Harmonic Sections}

Let in the following assume that $\nu_\zeta:(N_\zeta,g_\zeta)\rightarrow (S,g_S)$ is ACF of rate $\gamma=\eta=-m$.\\

The following Lemma proves that if the vertical kernel bundles exists, the fibres are given by smooth elements that decay at least by a power of $1-m$, i.e. 
\begin{align*}
    \mathcal{K}_{\m{AC};\beta}(\pi_\zeta/\nu_\zeta)=\mathcal{K}_{L^2}(\pi_\zeta/\nu_\zeta)
\end{align*}
for $\beta\in(1-m,0)$.

\begin{lem}
\label{adiabaticmircokernels}
The fibres of the vertical kernels coincide with the kernels of 
\begin{align*}
    \widehat{D}_{\zeta;V}(s):  C^{k+1,\alpha}_{\m{AC};\beta}(\nu_\zeta^{-1}(s),\widehat{E}_{\zeta(s)})\rightarrow  C^{k,\alpha}_{\m{AC};\beta-1}(\nu_\zeta^{-1}(s),\widehat{E}_{\zeta(s)})
\end{align*}
independent of $0<\alpha<1$, $k$ and $\beta\in[1-m,0)$. When $\beta<-m/2$ then $\mathcal{K}_{\m{AC};\beta}(\pi_\zeta/\nu_\zeta)\cong\mathcal{K}_{L^2}(\pi_\zeta/\nu_\zeta)$ by Theorem \ref{weightedsobolovembeddingACF}.
\end{lem}

By the elliptic regularity of uniform elliptic asymptotically conical operator, the first part of the statement holds.\cite[Thm. 4.21]{marshal2002deformations} To prove the independence in $\beta$ we need the following result.

\begin{prop}
\label{diracspinorsthatdecaydecayfast}
Let $\Phi\in \Gamma(\nu^{-1}_\zeta(s),\widehat{E}_\zeta)$. If $\widehat{D}_{\zeta;V}\Phi=0$ and $\Phi=o(1)$ then $\nabla^k\Phi=\mathcal{O}(r^{1-m-k})$.
\end{prop}

\begin{proof}
We will prove this statement similar to the proof of \cite[Prop. 5.10]{walpuski2012g_2}. Notice that $\widehat{D}_{\zeta;V}$ is a Dirac operator and hence, the Kato inequality
\begin{align}
\label{Katoinequalty}
    |\m{d}|\Phi||\leq\sqrt{\frac{m-1}{m}}|\nabla^{h_{\zeta;V}}\Phi|
\end{align}
holds on the complement of the vanishing locus of $\Phi$. Now, using the Weitzenböck formula 
\begin{align*}
    (\widehat{D}_{\zeta;V})^2=(\nabla^{h_{\zeta;V}})^*\nabla^{h_{\zeta;V}}+\frac{1}{4}\m{scal}_{g_{\zeta;V}}+\m{cl}_{VV}(F^{tw.}_{\nabla^{h_{\zeta;V}}}),
\end{align*}
$|F^{tw.}_{\nabla^{h_{\zeta;V}}}|=\mathcal{O}(r^{-m})$, and the identity
\begin{align*}
    \Delta^{g_{\zeta;V}}|\Phi|^2+2|\nabla^{h_{\zeta;V}}\Phi|^2=2\left<(\nabla^{h_{\zeta;V}})^*\nabla^{h_{\zeta;V}}\Phi,\Phi\right>
\end{align*}
we obtain the following estimate on the complement of the vanishing locus of $\Phi$
\begin{align*}
    \frac{2(m-1)}{m-2} \Delta^{g_{\zeta;V}}|\Phi|^{\frac{m-2}{m-1}}\leq&|\Phi|^{-\frac{m+1}{m}}\left(\Delta^{g_{\zeta;V}}|\Phi|^{2}+\frac{2m}{m-1}|\m{d}|\Phi||^2\right)\\
    \leq&|\Phi|^{-\frac{m}{m-1}}\left(\Delta^{g_{\zeta;V}}|\Phi|^{2}+2|\nabla^{h_{\zeta;V}}\Phi|^2\right)\\
    =&2|\Phi|^{-\frac{m}{m-1}}\left<(\nabla^{h_{\zeta;V}})^*\nabla^{h_{\zeta;V}}\Phi,\Phi\right>\\
    \leq&2|\Phi|^{-\frac{m}{m-1}}\left(\left<(\widehat{D}_{\zeta;V}^2\Phi,\Phi\right>-\frac{1}{4}\m{scal}_{g_{\zeta;V}}|\Phi|^2-\left<\m{cl}_{VV}(F_{\nabla^{h_{\zeta;V}}})\Phi,\Phi\right>\right)\\
    \leq&\mathcal{O}(r^{-m})|\Phi|^{\frac{m-2}{m-1}}
\end{align*}
Now, let $f=|\Phi|^{\frac{m-2}{m-1}}$ on $(V_\Phi)^c=\{m\in M_\zeta|\Psi(m)\neq0\}$. By the above
\begin{align*}
    \Delta^{g_{\zeta;V}}f\lesssim f\cdot w_{\m{AC};m}.
\end{align*}
Since $f$ is bounded, by \cite[Theorem 8.3.6(a)]{joyce2000compact}\footnote{Here we used a Sobolev embedding of $W^{2,p}\hookrightarrow C^{0,\alpha}$ for big enough $p>1$.} there exists a $g=\mathcal{O}(r^{2-m})$ such that 
\begin{align*}
    \Delta g=\left\{\begin{array}{ll}
         (\Delta^{g_{\zeta;V}} f)^+& \text{on } (V_\Phi)^c\\
         0& V_\Phi
    \end{array}\right.
\end{align*}
where $(.)^+$ denotes the positive part. As $g$ is superharmonic and decays at infinity, the maximum principle implies its nonnegativity. Further, the function $f-g$ is subharmonic on $(V_\Phi)^c$, decays at infinity and is nonnegative on the boundary of $(V_\Phi)^c$. Again, using the maximum principle $f\leq g=\mathcal{O}(r^{2-m})$ we deduce that
\begin{align*}
    |\Phi|=\mathcal{O}(r^{1-m}).
\end{align*}

\end{proof}

\begin{rem}
From now on we assume that there exists a critical rate $\mu_{\mathcal{K}}\in(-\infty,1-m]$ such that the kernels 
\begin{align*}
    \widehat{D}_{\zeta;V}(s):  C^{k+1,\alpha}_{\m{AC};\beta}(\nu_\zeta^{-1}(s),\widehat{E}_{\zeta})\rightarrow  C^{k,\alpha}_{\m{AC};\beta-1}(\nu_\zeta^{-1}(s),\widehat{E}_{\zeta})
\end{align*}
are isomorphic for all $\beta\in[\mu_{\mathcal{K}},0)$. Lemma \ref{exactsequenceverticalbundles} implies that for $\beta_2>\beta_1>\mu_{\mathcal{K}}$ there exists an exact sequence  
\begin{equation*}
    \begin{tikzcd}
	\mathcal{C}o\mathcal{K}_{\m{AC};\beta_1}(N_\zeta/S) &  \mathcal{C}o\mathcal{K}_{\m{AC};\beta_2}(N_\zeta/S) & {\bigoplus_{\beta_1<\beta<\beta_2}\mathcal{L}_{\m{C},\beta}(\widehat{D}_{0;V})}  
	\arrow[from=1-1, hook, to=1-2]
	\arrow[from=1-2,two heads, to=1-3]
\end{tikzcd}
\end{equation*}
\end{rem}

\begin{ex}
Let $D_\zeta=\m{d}\oplus\m{d}^{*_{g_\zeta}}$. The vector bundle of vertical kernels will always exist and will be referred to as the vector bundle of vertical harmonic forms. In particular, we identify
\begin{align*}
    \mathcal{K}_{\m{AC};\beta}(\pi_{\wedge^\bullet T^*N_\zeta}/\nu_\zeta)\cong\wedge^\bullet T^* S\otimes \mathcal{H}^\bullet(N_\zeta/S)
\end{align*}
In the case $(N_\zeta,g^t_\zeta)$ is given by Example \ref{NzetaasinResolutionsofSpin7Orbifolds}, \cite[Ex. (0.16.1)]{lockhart1987fredholm} and \cite[Thm. 8.4.1]{joyce2000compact} ensure the existence of an isomorphism of flat vector bundles
\begin{align*}
    \underline{\mathbb{R}}\oplus\mathcal{H}^\bullet(N_\zeta/S)\cong\m{H}^\bullet(N_\zeta/S).
\end{align*}
\end{ex}

\begin{lem}
\label{imporveddecay}
Let $(N_\zeta,g_\zeta)$ be given by Example \ref{NzetaasinResolutionsofSpin7Orbifolds} with $m=4$, $\Gamma\subset\m{SU}(2)$ and 
\begin{align*}
    (E,\m{cl}_g,h,\nabla^h)=(\wedge^\bullet T^* X,\m{cl}_g,g,\nabla^g).
\end{align*}
Then $\mathcal{K}_{\m{AC};\beta}(\pi_\zeta/\nu_\zeta)\cong \wedge^\bullet T^* S\otimes\mathcal{H}^2_{-}(N_\zeta/S)$ and $\mu_{\mathcal{K}}=-4$.
\end{lem}

\begin{proof}

In some cases we can even further improve the Kato inequality \eqref{Katoinequalty}. The only vertical harmonic forms on $(N_\zeta,g_\zeta)$ are anti-selfdual two forms. Following the results of Walter Seaman \cite[Thm. 1]{seaman1991} we improve the Kato inequality to 
\begin{align}
\label{improvedKatoinequalty}
    |\m{d}|\Phi||\leq\sqrt{\frac{2}{3}}|\nabla^{g_{\zeta;V}}\Phi|.
\end{align}
This is achieved by a closer inspection of the Weitzenböck-curvature. Then by following the arguments of the proof of Proposition \ref{diracspinorsthatdecaydecayfast} with $|\Phi|^{2/3}$ being replaced by $|\Phi|^{1/2}$ we deduce the statement.    
\end{proof}

\begin{ex}
    Let $N_\zeta$ be as in Example \ref{NzetaasinResolutionsofSpin7Orbifolds} with $m=4$ and $\Gamma=\mathbb{Z}_2$. Then in \cite[Thm. 3.26]{platt} it was shown that $\mu_{\mathcal{K}}=-4$ and $\mu_{\mathcal{C}o\mathcal{K}}=-2$ and
    \begin{align*}
        rk(\mathcal{C}o\mathcal{K}_{\m{AC};\beta-1})=\left\{\begin{array}{ll}
             0,&\text{if } \beta\in(0,-2) \\
             6,&\text{if } \beta\in(-2,-4)
        \end{array}\right.
    \end{align*}
\end{ex}

\subsubsection{The Adiabatic Residues, Adiabatic Kernels and Cokernels}
\label{The Adiabatic Residues, Adiabatic Kernels and Cokernels}

In the following we will construct an elliptic first order differential operator on sections of $\mathcal{K}_{\m{AC};\beta}(\pi_\zeta/\nu_\zeta)$ and $\mathcal{C}o\mathcal{K}_{\m{AC};\beta-1}(\pi_\zeta/\nu_\zeta)$ which we will refer to as the adiabatic residue of $\widehat{D}^t_{\zeta}$. This operator will be used as a model operator for $\widehat{D}^t_{\zeta}$. The idea of considering this operator as a model operator builds up on the work of Bismut in \cite{bismut1986,bismut1995,bismut1989eta}, generalising the work of Walpuski in \cite{walpuski2012g_2,walpuski2017spin} and Platt in \cite{platt}. This operator will be the right model operator to consider for understanding uniform elliptic theory on $(X^t,g^t)$ as it encaptures the geometric structure of the singular strata in contrast to a more standard approach to construct a model operator by using $\widehat{D}^t_{\zeta}$ for the flat metric $g_{st}$ on $\mathbb{R}^n$.

\begin{defi}
Recall Definition \ref{projectionbundledefi}. We define the effective kernel operator and the effective cokernel operators
\begin{align*}
    \mathfrak{D}_{\mathcal{K};\beta}^t\coloneqq&\pi_{\mathcal{K};\beta}\widehat{D}^t_{\zeta;H}\pi_{\mathcal{K};\beta}\\
    \mathfrak{D}_{\mathcal{C}o\mathcal{K};\beta-1}^t\coloneqq&\pi_{\mathcal{C}o\mathcal{K};\beta-1}\widehat{D}^t_{\zeta;H}\pi_{\mathcal{C}o\mathcal{K};\beta-1}
\end{align*}
These operators are formally self-adjoint, elliptic, first order differential operators that are given by 
\begin{align*}
    \mathfrak{D}_{\mathcal{K};\beta}^t=&D_{\mathcal{K};\beta}+\m{cl}_{g_S;\mathcal{K}}((\tau^t_{\zeta;H})_{\mathcal{K};\beta})+t^{-1}\cdot\m{cl}_{g_\zeta}(\tau^t_{\zeta;V})_{\mathcal{K};\beta}\\
    &-\frac{t}{4}\m{cl}_{g_S;\mathcal{K};\beta}(\m{cl}_{g_\zeta}(F_{H_\zeta}-\rho_\zeta^*F_{H_0})_{\mathcal{K};\beta}),\\
    \mathfrak{D}_{\mathcal{C}o\mathcal{K};\beta-1}^t=&D_{\mathcal{C}o\mathcal{K};\beta-1}+\m{cl}_{g_S;\mathcal{C}o\mathcal{K};\beta-1}((\tau^t_{\zeta;H})_{\mathcal{C}o\mathcal{K};\beta-1})+t^{-1}\cdot\m{cl}_{g_\zeta}(\tau^t_{\zeta;V})_{\mathcal{C}o\mathcal{K};\beta-1}\\
    &-\frac{t}{4}\m{cl}_{g_S;\mathcal{C}o\mathcal{K};\beta-1}(\m{cl}_{g_\zeta}(F_{H_\zeta}-\rho_\zeta^*F_{H_0})_{\mathcal{C}o\mathcal{K};\beta-1}),
\end{align*}
whereby 
\begin{align*}
    D_{\mathcal{K};\beta}:&  \Gamma\left(S,\mathcal{K}_{\m{AC};\beta}(\pi_\zeta/\nu_\zeta)\right)\rightarrow  \Gamma\left(S,\mathcal{K}_{\m{AC};\beta}(\pi_\zeta/\nu_\zeta)\right)\\
    D_{\mathcal{C}o\mathcal{K};\beta-1}:&  \Gamma\left(S,\mathcal{C}o\mathcal{K}_{\m{AC};\beta-1}(\pi_\zeta/\nu_\zeta)\right)\rightarrow  \Gamma\left(S,\mathcal{C}o\mathcal{K}_{\m{AC};\beta-1}(\pi_\zeta/\nu_\zeta)\right)
\end{align*}
are the Dirac-operators associated to the Dirac bundles $\mathcal{K}_{\m{AC};\beta}(\pi_\zeta/\nu_\zeta)$ and $\mathcal{C}o\mathcal{K}_{\m{AC};\beta-1}(\pi_\zeta/\nu_\zeta)$. The limits 
\begin{align*}
    \lim_{t\to0}\mathfrak{D}^t_{\mathcal{K};\beta}=&\mathfrak{D}_{\mathcal{K};\beta}=D_{\mathcal{K};\beta}+\m{cl}_{g_\zeta}(\dot{\tau}^0_{\zeta;V})_{\mathcal{K};\beta}\\
    \lim_{t\to0}\mathfrak{D}^t_{\mathcal{C}o\mathcal{K};\beta-1}=&\mathfrak{D}_{\mathcal{C}o\mathcal{K};\beta-1}=D_{\mathcal{C}o\mathcal{K};\beta-1}+\m{cl}_{g_\zeta}(\dot{\tau}^0_{\zeta;V})_{\mathcal{C}o\mathcal{K};\beta-1}
\end{align*}
will be referred to as the adiabatic residues of $\widehat{D}^t_{\zeta}$. Here $\dot{\tau}^0_{\zeta;V}=\left.\frac{\m{d}}{\m{d}t}\tau^t_{\zeta;V}\right|_{t=0}$.
\end{defi}

\begin{cor}
If $(\widehat{E}_\zeta,\m{cl}_{g^t_\zeta},h^t_\zeta,\nabla^{h^t_\zeta})$ is graded, so are all the constructed structures.
\end{cor}

\begin{rem}
     In \cite{goette2014adiabatic} the adiabatic residue is referred to as the effective horizontal operator.
\end{rem}

\begin{ex}
\label{HodgedeRhamGaussManinXzeta}
Let $(N_\zeta,g^t_\zeta)$ satisfy 
\begin{align}
    \underline{\mathbb{R}}\oplus\mathcal{H}^\bullet(N_\zeta/S)\cong\m{H}^\bullet(N_\zeta/S).
\end{align}
The adiabatic residue of the Hodge-de Rham operator on $(N_\zeta,g^t_\zeta)$ is given by the self-adjoint, elliptic, first order differential operator
\begin{align*}
    \mathfrak{D}_{GM}:\bigoplus_{p+q=\bullet}\Omega^p\left(S, \mathcal{H}^q(N_\zeta/S)\right)\rightarrow\bigoplus_{p+q=\bullet}\Omega^p\left(S, \mathcal{H}^q(N_\zeta/S)\right)
\end{align*}
given by 
\begin{align*}
    \mathfrak{D}_{GM}=\m{d}_{GM}\oplus\m{d}_{GM}^{*}.
\end{align*}
The Dirac operator
\begin{align*}
    \m{d}_{GM}\oplus\m{d}_{GM}^{*}
\end{align*}
associated to the vertical kernel Clifford bundle is given by the Gauß-Manin-Hodge-de Rham operator, i.e. the Hodge-de Rham operator twisted by the flat Gauß-Manin connection on $\mathcal{H}^\bullet(N_\zeta/S)\rightarrow S$.
\end{ex}

\begin{prop}
The analytic realisations of the adiabatic residues seen as a bounded maps
\begin{align*}
    \mathfrak{D}^t_{\mathcal{K};\beta}:&W^{k+1,p}_{t}\left(S,\mathcal{K}_{\m{AC};\beta}(\pi_\zeta/\nu_\zeta)\right)\rightarrow W^{k,p}_{t}\left(S, \mathcal{K}_{\m{AC};\beta}(\pi_\zeta/\nu_\zeta)\right)\\
    \mathfrak{D}^t_{\mathcal{K};\beta}:&  C^{k+1,\alpha}_{t}\left(S,\mathcal{K}_{\m{AC};\beta}(\pi_\zeta/\nu_\zeta)\right)\rightarrow  C^{k,\alpha}_{t}\left(S, \mathcal{K}_{\m{AC};\beta}(\pi_\zeta/\nu_\zeta)\right)\\
        \mathfrak{D}^t_{\mathcal{C}o\mathcal{K};\beta-1}:&W^{k+1,p}_{t}\left(S,\mathcal{C}o\mathcal{K}_{\m{AC};\beta-1}(\pi_\zeta/\nu_\zeta)\right)\rightarrow W^{k,p}_{t}\left(S, \mathcal{C}o\mathcal{K}_{\m{AC};\beta-1}(\pi_\zeta/\nu_\zeta)\right)\\
    \mathfrak{D}^t_{\mathcal{C}o\mathcal{K};\beta-1}:&  C^{k+1,\alpha}_{t}\left(S,\mathcal{C}o\mathcal{K}_{\m{AC};\beta-1}(\pi_\zeta/\nu_\zeta)\right)\rightarrow  C^{k,\alpha}_{t}\left(S,\mathcal{C}o \mathcal{K}_{\m{AC};\beta}(\pi_\zeta/\nu_\zeta)\right)
\end{align*}
are Fredholm. Elliptic regularity shows that the kernel and cokernels of these operators consist of smooth sections. Moreover, if $\phi\in\m{ker}(\mathfrak{D}_{\mathcal{K};\beta})^\perp$ and $\psi\in\m{coker}(\mathfrak{D}_{\mathcal{C}o\mathcal{K};\beta-1})^\perp$
\begin{align}
\label{DtKboundedbelow}
    \left|\left|\phi\right|\right|_{C^{k+1,\alpha}_t}\lesssim &\left|\left|\mathfrak{D}^t_{\mathcal{K};\beta}\phi\right|\right|_{C^{k,\alpha}_t}\\
\label{DtCoKboundedbelow}
    \left|\left|\psi\right|\right|_{C^{k+1,\alpha}_t}\lesssim &\left|\left|\mathfrak{D}^t_{\mathcal{C}o\mathcal{K};\beta-1}\psi\right|\right|_{C^{k,\alpha}_t}
\end{align}
\end{prop}

\begin{proof}
The first part of the statement follows from standard elliptic theory on compact manifolds. As $S$ is compact, the operators 
$\mathfrak{D}_{\mathfrak{K};\beta}$ and $\mathfrak{D}_{\mathcal{C}o\mathcal{K};\beta-1}$ are bounded from below on the complement of their kernels. By using 
\begin{align}
\label{DtK-DK}
    \left|\left|(\mathfrak{D}^t_{\mathcal{K};\beta}-\mathfrak{D}_{\mathcal{K};\beta})\phi\right|\right|_{C^{k,\alpha}_t}\lesssim &t\left|\left|\phi\right|\right|_{C^{k,\alpha}_t}\\
\label{DtCoK-DCoK}
    \left|\left|(\mathfrak{D}^t_{\mathcal{C}o\mathcal{K};\beta-1}-\mathfrak{D}_{\mathcal{C}o\mathcal{K};\beta-1})\psi\right|\right|_{C^{k,\alpha}_t}\lesssim &t\left|\left|\psi\right|\right|_{C^{k,\alpha}_t}
\end{align}
we conclude the statement.
\end{proof}

\begin{rem}
The above implies that 
\begin{align*}
    \m{ker}(\mathfrak{D}^t_{\mathcal{K};\beta})\hookrightarrow\m{ker }(\mathfrak{D}_{\mathcal{K};\beta})
\end{align*}
and 
\begin{align*}
    \m{coker}(\mathfrak{D}^t_{\mathcal{C}o\mathcal{K};\beta-1})\hookrightarrow\m{coker }(\mathfrak{D}_{\mathcal{C}o\mathcal{K};\beta-1}).
\end{align*}
\end{rem}

\begin{defi}
We will refer to the space 
\begin{align*}
    \mathfrak{Ker}_{\m{ACF};\beta}(\widehat{D}_\zeta)\coloneqq\m{ker}\left(\mathfrak{D}_{\mathcal{K};\beta}\right)\und{1.0cm}\mathfrak{CoKer}_{\m{ACF};\beta}(\widehat{D}^0_\zeta)\coloneqq\m{coker}\left(\mathfrak{D}_{\mathcal{C}o\mathcal{K};\beta-1}\right).
\end{align*}
as the \textbf{adiabatic kernel and adiabatic cokernel}. The \textbf{adiabatic index} is given by
\begin{align*}
    \mathfrak{ind}_{\m{ACF};\beta}\coloneqq \m{dim}\left(\mathfrak{Ker}_{\m{ACF};\beta}(\widehat{D}_\zeta)\right)-\m{dim}\left(\mathfrak{CoKer}_{\m{ACF};\beta-1}(\widehat{D}_\zeta)\right).
\end{align*}
\end{defi}

\begin{defi}
    The Dirac operator $\widehat{D}^t_{\zeta}$ will be called \textbf{isentropic}, if 
    \begin{align*}
        \mathfrak{Ker}_{\m{ACF};\beta}\left(\widehat{D}_\zeta\right)\cong\m{ker}_{\m{ACF};\beta}\left(\widehat{D}^0_\zeta\right)\und{0.5cm}\mathfrak{CoKer}_{\m{ACF};\beta-1}(\widehat{D}_\zeta)\cong\m{coker}_{\m{ACF};\beta-1}\left(\widehat{D}^0_\zeta\right).
    \end{align*}
\end{defi}

\begin{rem}
An isentropic process in statistical mechanics is a process that is adiabatic and reversible. Usually information about the Dirac operator is lost in the adiabatic limit. However, in the case of isentropic operators, an element in the kernel of the adiabatic residue determines an element in the kernel of Dirac operator for $t\neq0$.
\end{rem}

\begin{lem}[Adiabatic Wall-Crossing]
    Let $\beta_1>\beta_2$ be not critical rates. Then 
    \begin{align*}
        \mathfrak{ind}_{\m{ACF};\beta_2}(\widehat{D}^0_{\zeta})-\mathfrak{ind}_{\m{ACF};\beta_2}(\widehat{D}^0_{\zeta})=\sum_{\beta_2<\lambda<\beta_1}d_{\lambda+\frac{m-1}{2}-\delta(\widehat{E}_{0})}(\widehat{D}_0),
    \end{align*}
where $d_\lambda(\widehat{D}_0)$ as defined in Definition \ref{logCFdefi}.
\end{lem}

\begin{proof}
    This immediately follows from Lemma \ref{exactsequenceverticalbundles} and $\widehat{D}^t_{\zeta}\sim \widehat{D}^t_0$.
\end{proof}

\begin{rem}
\label{GMremark}
Let $(N_\zeta,g^t_\zeta)$ be as in Example \ref{HodgedeRhamGaussManinXzeta}. We can view the adiabatic residue of the Hodge-de Rham operator as the elliptic operator associated to the local system of 
\begin{align*}
    L^2GM^\bullet(\nu_\zeta:N_\zeta\rightarrow S)=(\Omega^\bullet(S,\mathcal{H}^\bullet(N_\zeta/S)),\m{d}_{GM}).
\end{align*}
Thus, the kernel of the adiabatic residue of $\widehat{D}^t_{\zeta}$ can be identified with the cohomology  
\begin{align*}
    \m{H}^\bullet(L^2GM^\bullet(\nu_\zeta:N_\zeta\rightarrow S)).
\end{align*}
If $\pi_1(S)=0$ then, Hurwitz-formula holds for the cohomology of $N_\zeta$. Furthermore, if $\pi_p(S)=0$ the above coincides with the group cohomology 
\begin{align*}
    \m{H}^p(L^2GM^\bullet(\nu_\zeta:N_\zeta\rightarrow S))=\m{H}^p(\pi_1(S),\mathcal{H}^\bullet(M_{\zeta(s)}))
\end{align*}
of the monodromy representation
\begin{align*}
    \varrho_{GM}:\pi_1(S)\rightarrow \bigtimes_{q=0}^m \m{GL}(\mathcal{H}^q(M_{\zeta(s)})).
\end{align*}
\end{rem}

Furthermore, standard results from algebraic topology link the adiabatic kernel of the Hodge-de Rham operator to the Leray-Serre spectral sequence.

\begin{thm}\cite[Thm. 14.18]{bott1982differential}
There exists a spectral sequence $E^{p,q}_N$ associated to the double complex
\begin{align*}
    (K^{\bullet_1,\bullet_2},\m{d})=(\check{C}^{\bullet_1}(\nu_\zeta^{-1}\mathcal{U},\Omega^{\bullet_2}),\delta+(-1)^{\bullet_1}\m{d}_{dR})
\end{align*}
converging to the cohomology of $N_\zeta$, i.e. $\tot{p+q=\bullet}{}{E^{p,q}_\infty}=\m{H}^\bullet(N_\zeta)$. Here $\mathcal{U}$ is a good cover of $S$. Moreover, the second page of this spectral sequence can be identified with 
\begin{align*}
    (\check{\m{H}}^{\bullet_1}(\nu_\zeta^{-1}\mathcal{U},\m{H}^{\bullet_2}(N_\zeta/S)),\m{d}^{p,q}_2).
\end{align*}
\end{thm}

\begin{prop}
Let us assume that $\Gamma\subset\m{SU}(m/2)$ and let $N_\zeta\rightarrow S$ be an ACF fibration whose fibres are ALE spaces resolving $V/\Gamma$. There exists an isomorphism 
\begin{align*}
    E^{p,q}_2\cong\left\{\begin{array}{ll}
         \m{H}^p(S)&\text{if } q=0 \\
         \m{H}^p(S,\mathcal{H}^q(N_\zeta/S))&\text{else}
    \end{array}\right.
\end{align*}
i.e. 
\begin{align*}
     E^{\bullet,\bullet}_2\cong\m{H}^\bullet(S)\oplus\m{H}^\bullet(L^2GM^\bullet(\nu_\zeta:N_\zeta\rightarrow S)),
\end{align*}
where $L^2GM^\bullet$ as in Remark \ref{GMremark}. Moreover, for $0<N$ the differentials $d_{2N}=0$.
\end{prop}

\begin{rem}
Assume $(N_{t^2\cdot\zeta},g_{t^2\cdot\zeta})$ as in Example \ref{NzetaasinResolutionsofSpin7Orbifolds} with $n=8$ and $m=4$. Then $\m{d}+\m{d}^{*_{g_{t^2\cdot\zeta}}}$ is isentropic if and only if the maps 
\begin{equation*}
\label{obstructionisentropicdim8codim4}
    \begin{tikzcd}
        \m{d}^{3,-2}:\m{H}^0(S,\mathcal{H}^2(N_\zeta/S))\arrow[r]&\m{H}^3(S)\\
        \m{d}^{3,-2}:\m{H}^1(S,\mathcal{H}^2(N_\zeta/S))\arrow[r]&\m{H}^4(S)\cong\mathbb{R}
    \end{tikzcd}
\end{equation*}
vanish.
\end{rem}

\subsection{Fredholm Theory for Adiabatic Families of Asymptotically Conically Fibred Dirac Operators}
\label{Fredholm Theory for Adiabatic Families of Asymptotically Conically Fibred Dirac Operators}

In the following we will use the adiabatic tools established in Section \ref{The Adiabatic Residues, Adiabatic Kernels and Cokernels} to study the analytic realisation of the family of ACF Dirac operators $\widehat{D}^t_{\zeta}$. We will construct certain \say{adapted adiabatic norms} and prove uniform elliptic estimates.\\

In the following we will construct adapted norms that allow uniform elliptic estimates. These norms will be essential in the uniform linear gluing results needed for the applications mentioned in the preface.\\

Before introducing our adapted Hölder norms, we first define the appropriate pushforward bundles which encode the asymptotic behaviour of our Dirac operator in the adiabatic regime.

\begin{defi}
\label{Ckalphapushforward}
We define the $C^{k,\alpha}_{\m{AC};\beta}$-pushforeword of the vector bundle $\widehat{E}_\zeta$ to be the infinite dimensional vector bundle
\begin{align*}
        C^{k,\alpha}_{\m{ACF};\beta;t}(U,(\nu_\zeta)_{*_{\m{AC};\beta}^{k}}\widehat{E}_\zeta)&:=  C^{k,\alpha}_{\m{ACF};\beta}(\nu^{-1}_\zeta(U),\widehat{E}_\zeta,\widehat{\nabla}^{h^t_\zeta})
\end{align*}
for all $U\subset S$ open. The fibres of these vector bundle are given by the vector spaces
\begin{align*} 
      C^{k,\alpha}_{\m{AC};\beta}(\nu_\zeta^{-1}(s),\widehat{E}_{\zeta(s)}).
\end{align*}
We will write 
\begin{align*}
(\nu_\zeta)_{*_{\m{AC};\beta}^{k,\alpha}}\pi_\zeta:(\nu_\zeta)_{*_{\m{AC};\beta}^{k,\alpha}}\widehat{E}_\zeta\rightarrow S. 
\end{align*}
and denote the $l\leq k$ norms on $(\nu_\zeta)_{*_{\m{AC};\beta}^{k,\alpha}}\widehat{E}_\zeta$ respectively by 
\begin{align*}
    \left|\psi\right|_{C^{k,\alpha}_{\m{AC};\beta;t}}(s)=&\left|\left|\psi(s)\right|\right|_{ C^{k,\alpha}_{\m{AC};\beta;t}(\nu^{-1}_\zeta(s),\widehat{E}_{\zeta(s)})}.
\end{align*}
\end{defi}

Having defined the relevant vector bundles, we now equip sections of them with Hölder norms that are sensitive to the adiabatic scaling. These norms form the analytic framework in which we derive uniform estimates.

\begin{defi}
We define the Hölder space of sections of $\nu_{*^{k,\alpha}_{\m{AC};\beta}}\widehat{E}_\zeta$
    \begin{align*}
        \left|\left|\psi\right|\right|_{ C^{(k,l),\alpha}_{\m{ACF};\beta;t}}=&\sum_{i=0}^k\underset{S}{\m{sup}}\left|(\widehat{\nabla}^{h^t_{\zeta} ;H})^i\psi\right|_{ C^{l-i,\alpha}_{\m{AC};\beta;t}}\\
        &+\underset{\underset{\m{dist}_{g_S}(s,s')<\m{inj}_{g_S}}{s,s'\in S,}}{\m{sup}}\left\{\frac{\left|(\widehat{\nabla}^{h^t_{\zeta} ;H})^i\psi(s)-\Pi^{s,s'}(\widehat{\nabla}^{h^t_{\zeta} ;H})^i\psi(y)\right|_{C^{l-i,\alpha}_{\m{AC};\beta;t}}}{d_{g_S}(s,s')^\alpha}\right\}
    \end{align*}
and set
\begin{align*}
    \left|\left|\psi\right|\right|_{ C^{k,\alpha}_{\m{ACF};\beta;t}}&=\left|\left|\psi\right|\right|_{ C^{(k,k),\alpha}_{\m{ACF};\beta;t}}.
\end{align*}
\end{defi}

By definition the Banach spaces 
\begin{align}
          C^{k+1,\alpha}_{\m{ACF};\beta;t}(N_\zeta,\widehat{E}_\zeta,\widehat{\nabla}^{\otimes_\zeta})\cong&C^{k+1,\alpha}_{\m{ACF};\beta;t}(S,\mathcal{C}o\mathcal{I}^{k+1,\alpha}_{\m{AC};\beta}(\pi_\zeta/\nu_\zeta)\oplus\mathcal{K}_{\m{AC};\beta}(\pi_\zeta/\nu_\zeta))\\
          C^{k,\alpha}_{\m{ACF};\beta-1;t}(N_\zeta,\widehat{E}_\zeta,\widehat{\nabla}^{\otimes_\zeta})\cong&C^{k,\alpha}_{\m{ACF};\beta;t}(S,\mathcal{I}^{k,\alpha}_{\m{AC};\beta-1}(\pi_\zeta/\nu_\zeta)\oplus\mathcal{C}o\mathcal{K}_{\m{AC};\beta-1}(\pi_\zeta/\nu_\zeta)),
\end{align}
where $\mathcal{CoI}^{k+1,\alpha}_{\m{AC},\beta}(\pi_\zeta/\nu_\zeta)$ and  $\mathcal{I}^{k+1,\alpha}_{\m{AC},\beta-1}(\pi_\zeta/\nu_\zeta$ are define analogously to Definition \ref{splittingdefi}. Notice, that by elliptic regularity, the bundles $\mathcal{K}_{\m{AC};\beta}(\pi_\zeta/\nu_\zeta)$ and $\mathcal{C}o\mathcal{K}_{\m{AC};\beta-1}(\pi_\zeta/\nu_\zeta)$ are independent of $k$ and $\alpha$.

\begin{manualassumption}{6}
\label{ass6}
In the following we will assume that $\widehat{E}_\zeta=(\widehat{E}_\zeta)_{-q}$, i.e. 
\begin{align*}
    \left|\psi\right|_{C^{k,\alpha}_{\m{AC};\beta;t}}=t^{-q-\beta}\cdot \left|\psi\right|_{C^{k,\alpha}_{\m{AC};\beta;1}},
\end{align*}
in order to keep the subsequent proofs readable and calculations short. 
\end{manualassumption}

The following lemma shows that commutators of horizontal and vertical derivatives are sufficiently small in the adiabatic limit, a crucial technical estimate used repeatedly in the Fredholm analysis.

\begin{lem}
\label{comnablaesti}
Let $\Phi\in C^{k+1,\alpha}_{\m{ACF};\beta;t}(N_\zeta,\widehat{E}_\zeta)$ as in Definition \ref{ACFHolderspaces}, $X\in\mathfrak{X}^{1,0}(N_\zeta)$ and $U\in\mathfrak{X}^{0,1}(N_\zeta)$ of unite norm with respect to $g^t_\zeta$. Then there exists a $\phi\geq 1$ 
    \begin{align*}
        \left|\left|\left[\widehat{\nabla}^{h^t_\zeta}_X,\widehat{\nabla}^{h^t_\zeta}_U\right]\Phi\right|\right|_{C^{k,\alpha}_{\m{ACF};\beta-1;t}}\lesssim  t^\phi\cdot        \left|\left|\Phi\right|\right|_{C^{k+1,\alpha}_{\m{ACF};\beta;t}}
    \end{align*}
\end{lem}

\begin{proof}
Recall \eqref{hatnablaotimestzeta}. The commutator of $\widehat{\nabla}^{h^t_{\zeta};V}$ and $\widehat{\nabla}^{h^t_\zeta;H}$ is given by 
    \begin{align*}
        [\widehat{\nabla}^{h^t_\zeta;H},\widehat{\nabla}^{h^t_\zeta;V}]=&F_{\widehat{\nabla}^{h^t_\zeta};HV}+\widehat{\nabla}^{h^t_\zeta}_{[H,V]}
    \end{align*}

Using that $\nabla^{\otimes_\zeta}=_{loc}\overline{\nabla}^{\oplus_\zeta}\otimes1+1\otimes\nabla^{tw.,0}$ we compute that
\begin{align*}
    [\nabla^{\otimes_\zeta}_X,\nabla^{\otimes_\zeta}_U]=\nabla^{\otimes_\zeta}_{[X,U]}
\end{align*}
for all $X\in\mathfrak{X}^{1,0}(N_\zeta)$ and $U\in\mathfrak{X}^{0,1}(N_\zeta)$. We deduce that 
    \begin{align*}
    \left[\widehat{\nabla}^{h^t_\zeta}_X,\widehat{\nabla}^{h^t_\zeta}_U\right]&=\left[\nabla^{\otimes_\zeta}_X,\nabla^{\otimes_\zeta}_U\right]+\left[\frac{1}{2}\m{cl}_{g^t_\zeta}(\Omega_{\zeta}(X))+\tau_{\zeta}(X),\nabla^{\otimes_\zeta}_U\right]\\\nonumber
    &+\left[\nabla^{\otimes_\zeta}_X,\frac{1}{2}\m{cl}_{g^t_\zeta}(\Omega_{\zeta}(U))+\tau_{\zeta}(U)\right]\\\nonumber
    &+\left[\frac{1}{2}\m{cl}_{g^t_\zeta}(\Omega_{\zeta}(X))+\tau_{\zeta}(X),\frac{1}{2}\m{cl}_{g^t_\zeta}(\Omega_{\zeta}(U))+\tau_{\zeta}(U)\right]\\\nonumber
    =&-\frac{1}{2}\m{cl}_{g^t_\zeta}(\nabla^{\oplus_\zeta}_U\Omega_{\zeta}(X))-\nabla^{\otimes_\zeta}_U\tau_{\zeta}(X)+\frac{1}{2}\m{cl}_{g^t_\zeta}(\nabla^{\oplus_\zeta}_X\Omega_{\zeta}(U))+\nabla^{\otimes_\zeta}_X\tau_{\zeta}(U)\\\nonumber
    &+\left[\frac{1}{2}\m{cl}_{g^t_\zeta}(\Omega_{\zeta}(X))+\tau_{\zeta}(X),\frac{1}{2}\m{cl}_{g^t_\zeta}(\Omega_{\zeta}(U))+\tau_{\zeta}(U)\right].
    \end{align*}

The statement follows from the above and the definition of the ACF Hölder norms.   
\end{proof}

Using the above, we decompose the space of sections into fiberwise image and kernel components.

\begin{lem}
\label{splittACF}
    The identifications
    \begin{align*}
          C^{k+1,\alpha}_{\m{ACF};\beta;t}(N_\zeta,\widehat{E}_\zeta)\cong&C^{k+1,\alpha}_{\m{ACF};\beta;t}(S,\mathcal{C}o\mathcal{I}^{k+1,\alpha}_{\m{AC};\beta}(\pi_\zeta/\nu_\zeta))\oplus C^{k+1,\alpha}_{\m{ACF};\beta;t}(S,\mathcal{K}_{\m{AC};\beta}(\pi_\zeta/\nu_\zeta))\\
          C^{k,\alpha}_{\m{ACF};\beta-1;t}(N_\zeta,\widehat{E}_\zeta)\cong&C^{k,\alpha}_{\m{ACF};\beta;t}(S,\mathcal{I}^{k,\alpha}_{\m{AC};\beta-1}(\pi_\zeta/\nu_\zeta))\oplus C^{k,\alpha}_{\m{ACF};\beta;t}(S,\mathcal{C}o\mathcal{K}_{\m{AC};\beta-1}(\pi_\zeta/\nu_\zeta))
    \end{align*}
    are isomorphisms of Banach spaces.
\end{lem}

\begin{proof}
    A moment's thought reveals that the only \say{off-diagonal} terms are dependent on the commutator estimate 
    \begin{align*}
        [\widehat{\nabla}^{h^t_\zeta;H},\widehat{D}_{\zeta;V}].
    \end{align*}
    
\end{proof}

\begin{rem}
    Notice that the bundle $\mathcal{C}o\mathcal{K}_{\m{AC};\beta-1}(\pi_\zeta/\nu_\zeta)\subset (\nu_\zeta)_{*_{\m{ACF};\beta-1}}\widehat{E}_\zeta$ and hence, there exists an embedding 
    \begin{align*}
        \mathcal{C}o\mathcal{K}_{\m{AC};\beta-1}(\pi_\zeta/\nu_\zeta)\hookrightarrow (\nu_\zeta)_{*^{k+1,\alpha}_{\m{ACF};\beta}}\widehat{E}_\zeta.
    \end{align*}
\end{rem}

To relate sections in these bundles, we define canonical inclusion and projection maps, and study their behaviour under the adiabatic scaling.

\begin{defi}
\label{defivarpiandiota}
We define the identifications
\begin{align*}
    \varpi_{\mathcal{K};\beta}:C^{k,\alpha}_{\m{ACF};\beta;t}(S,\mathcal{K}_{\m{AC};\beta}(\pi_\zeta/\nu_\zeta))&\cong C^{k,\alpha}_t(S,\mathcal{K}_{\m{AC};\beta}(\pi_\zeta/\nu_\zeta)):\iota_{\mathcal{K};\beta}\\
    \varpi_{\mathcal{C}o\mathcal{K};\beta-1}:C^{k,\alpha}_{\m{ACF};\beta-1;t}(S,\mathcal{C}o\mathcal{K}_{\m{AC};\beta-1}(\pi_\zeta/\nu_\zeta))&\cong  C^{k,\alpha}_t(S,\mathcal{C}o\mathcal{K}_{\m{AC};\beta-1}(\pi_\zeta/\nu_\zeta)):\iota_{\mathcal{C}o\mathcal{K};\beta-1}
\end{align*}
\end{defi}

The next lemma quantifies how these operations interact with the weighted Hölder norms.

\begin{lem}
\label{iotapiestimates}
Let $\beta>-m/2$ and $\beta+\mu_{\mathcal{K}}<\gamma$. Let $\psi\in  \Gamma^{k,\alpha}\left(S,\mathcal{K}_{\m{AC};\beta}(\pi_\zeta/\nu_\zeta)\right)$, $\Phi\in \Gamma^{k,\alpha}_{loc}(N_\zeta,\widehat{E}_\zeta)$ and let $\beta+\mu_{\mathcal{K}}<\gamma$. Then
\begin{align}
\label{iotaK}
    \left|\left|\iota_{\mathcal{K};\beta}\psi\right|\right|_{C^{k,\alpha}_{\m{ACF};\beta;t}}\lesssim& t^{-\beta-m/2}\cdot\left|\left|\psi\right|\right|_{C^{k,\alpha}_{t}}\\
\label{piK}
    \left|\left|\varpi_{\mathcal{K};\beta}\Phi\right|\right|_{C^{k,\alpha}_{t}}\lesssim& t^{m/2+\beta-\alpha}\cdot\left|\left|\pi_{\mathcal{K};\beta}\Phi\right|\right|_{C^{(k,0),\alpha}_{\m{ACF};\beta;t}}\\
    \label{iotaCoK}
    \left|\left|\iota_{\mathcal{C}o\mathcal{K};\beta-1}\psi\right|\right|_{C^{k,\alpha}_{\m{ACF};\beta-1;t}}\lesssim& t^{1-\beta-m/2}\cdot\left|\left|\psi\right|\right|_{C^{k,\alpha}_{t}}\\
\label{piCoK}
    \left|\left|\varpi_{\mathcal{C}o\mathcal{K};\beta-1}\Phi\right|\right|_{C^{k,\alpha}_{t}}\lesssim& t^{m/2+\beta-1-\alpha}\cdot\left|\left|\pi_{\mathcal{C}o\mathcal{K};\beta-1}\Phi\right|\right|_{C^{(k,0),\alpha}_{\m{ACF};\beta-1;t}}.
\end{align}
\end{lem}

\begin{proof}
The proof of the first two bounds are analogous to the second two. Thus, we will resume by proving the $t$-dependence of the operator norms of $\iota_{\mathcal{K};\beta}$ and $\varpi_{\mathcal{K};\beta}$.\\

If $k-m/2\geq k+\alpha\geq k-m/p$, $-m/2<\beta$ whenever $2>p$ or $-m/2\leq\beta$, and $\delta>\beta$, then by the weighted embedding theorems for AC-spaces \cite[Thm. 4.18]{marshal2002deformations} there exists continuous embeddings
\begin{align*}
    H^{k}_{\m{AC};\beta}\hookrightarrow  C^{k,\alpha}_{\m{AC};\beta}\hookrightarrow W^{k,p}_{\m{AC};\delta}.
\end{align*}
As such, we can bound the norms 
\begin{align*}
     \left|\psi\right|_{C^{k,\alpha}_{\m{AC};\beta;t}}\lesssim& \left|\psi\right|_{H^{k}_{\m{AC};\beta;t}}\lesssim \left|\psi\right|_{L^{2}_{\m{AC};\beta;t}}\lesssim t^{m/2+\beta}\cdot\left|\psi\right|_{L^{2}_{\m{AC};t}}.
\end{align*}
Let $e^1_{k}$ be an orthonormal basis of $(\mathcal{K}_{\m{AC};\beta}(\pi_\zeta/\nu_\zeta)_s)_{-q}$ with respect to $L^2_1$. Then 
\begin{align*}
    1=&\left|\left|e^t_{n}\right|\right|_{L^2_t}^2\\
    =&\int_{M_{\zeta(s)}}\left|e^t_{n}\right|^2_{h^t_{\zeta(s)}}\m{vol}_{t^2g_\zeta(s)}\\
    =&t^{m-2q}\int_{M_{\zeta(s)}}\left|e^t_{n}\right|^2_{h_{\zeta(s)}}\m{vol}_{g_\zeta(s)}\\
    =&t^{m-2q}\left|\left|e^t_{n}\right|\right|_{L^2_1}^2.
\end{align*}
Let $e^t_{n}=t^{q-m/2}e^1_{k}$ be a local, orthogonal, parallel frame of $\mathcal{K}_{\m{AC};\beta}(\pi_\zeta/\nu_\zeta)$. Then we can bound 
\begin{align*}
    \sum_{i=0}^k\left|(\widehat{\nabla}^{h^t_\zeta})^i\varpi_{\mathcal{K};\beta}\Phi\right|_{h^t_\zeta}(s)\lesssim&\sum_{i=0}^k\sum_{n}\left|\int_{\nu_{\zeta}^{-1}(s)}\left<(\widehat{\nabla}^{h^t_\zeta;H})^i\Phi(s),e^t_{n}\right>_{h^t_{\zeta},g^t_\zeta;V}\m{vol}_{g^t_{\zeta};V}\right|\\
    \lesssim&\sum_{i=0}^k\sum_{n}\int_{\nu_{\zeta}^{-1}(s)}\left|\left<(\widehat{\nabla}^{h^t_\zeta;H})^i\Phi(s),e^t_{n}\right>_{h^t_{\zeta},g^t_\zeta;V}\right|\m{vol}_{g^t_{\zeta};V}\\
    \lesssim&\sum_{i=0}^k\sum_{n}\int_{\nu_{\zeta}^{-1}(s)}w_{\m{ACF};\beta;t}\left|(\widehat{\nabla}^{h^t_\zeta;H})^i\Phi(s)\right|_{h^t_{\zeta},g^t_\zeta;V}\\
    &\cdot w_{\m{ACF};-\beta;t}\left|e^t_{n}\right|_{h^t_{\zeta},g^t_\zeta;V}\m{vol}_{g^t_{\zeta};V}\\
    \lesssim&t^{m/2+\beta}\left|\left|\Phi(s)\right|\right|_{C^{(l,0)}_{\m{ACF};\beta;t}}.
\end{align*}

Let $s,s'\in S$ and $x_s,x_{s'}\in N_\zeta$ be a horizontal pair, i.e. $x_{s'}$ is the end point of every horizontal path lifting a path connecting $s$ and $s'$. In particular, $\m{dist}_{g^t_\zeta}(x_s,x_{s'})=\m{dist}_{g_S}(s,s')$ and\\
\resizebox{1.0\linewidth}{!}{
\begin{minipage}{\linewidth}
\begin{align*}
    \left|\varpi_{\mathcal{K};\beta}\Phi(s)-\Pi^{\widehat{\nabla}^{h^t_\zeta}}_{s,s'}\varpi_{\mathcal{K};\beta}\Phi(s')\right|_{h^t_\zeta}\lesssim&\left|\int_{\nu_{\zeta}^{-1}(s)}\left<\Phi(x_s)-\Pi^{\widehat{\nabla}^{h^t_\zeta}}_{x_s,x_s'}\Phi(x_s'),e^t_{n}\right>_{h^t_{\zeta},g^t_\zeta;V}\m{vol}_{g^t_{\zeta};V}\right|\\
    \lesssim&\sum_{n}\int_{\nu_{\zeta}^{-1}(s)}\left|\Phi(x_s)-\Pi^{\widehat{\nabla}^{h^t_\zeta}}_{x_s,x_s'}\Phi(x_s')\right|_{h^t_{\zeta},g^t_\zeta;V}\cdot\left|e^t_{n}\right|_{h^t_{\zeta},g^t_\zeta;V}\m{vol}_{g^t_{\zeta};V}\\
    \lesssim&\sum_{n}\int_{\nu_{\zeta}^{-1}(s)}w_{\m{ACF};\beta-\alpha;t}(x_s,x_s')\frac{\left|\Phi(x_s)-\Pi^{\widehat{\nabla}^{h^t_\zeta}}_{x_s,x_s'}\Phi(x_s')\right|_{h^t_{\zeta},g^t_\zeta;V}}{\m{dist}_{g^t_\zeta}(x_s,x_s')^\alpha}\\
    &\cdot w_{\m{ACF};\alpha-\beta;t}(x_s,x_s')\cdot \m{dist}_{g^t_\zeta}(x_s,x_s')^\alpha\cdot \left|e^t_{n}\right|_{h^t_{\zeta},g^t_\zeta;V}\m{vol}_{g^t_{\zeta};V}\\
    \lesssim&\left[\Phi(s)\right]_{C^{0,\alpha}_{\m{ACF};\beta;t}}\sum_{n}\int_{\nu_{\zeta}^{-1}(s)} w_{\m{ACF};\alpha-\beta;t}(x_s,x_s')\cdot \m{dist}_{g^t_\zeta}(x_s,x_s')^\alpha\\
    &\cdot \left|e^t_{n}\right|_{h^t_{\zeta},g^t_\zeta;V}\m{vol}_{g^t_{\zeta};V}\\
    \lesssim&t^{m/2+\beta-\alpha}\cdot \left[\Phi(s)\right]_{C^{0,\alpha}_{\m{ACF};\beta;t}}\m{dist}_{g_S}(s,s')^{\alpha}
\end{align*}
\end{minipage}}
which concludes the proof.
\end{proof}

With these tools in place, we are now ready to state fiberwise Schauder estimates for the vertical Dirac operator. These are the foundation for all subsequent uniform elliptic estimates.

\begin{lem}\cite[Thm. 4.21]{marshal2002deformations}
The fibrewise asymptotically conical Schauder estimates 
\begin{align}
\label{schauderDzetaV}
   \left|\Phi\right|_{C^{k+1,\alpha}_{\m{AC};\beta;t}} &\lesssim \left|t^{-1}\cdot \widehat{D}_{\zeta;V}\Phi\right|_{C^{k,\alpha}_{\m{AC};\beta-1;t}}+\left|\Phi\right|_{C^{0}_{\m{AC};\beta;t}}
\end{align}
hold. By running a contradiction argument as in \cite[Prop. 8.7]{walpuski2012g_2} we can improve on it by 
\begin{align}
\label{improvedschauderDzetaV}
   \left|\Phi\right|_{C^{k+1,\alpha}_{\m{AC};\beta;t}} &\lesssim \left|t^{-1}\cdot \widehat{D}_{\zeta;V}\Phi\right|_{C^{k,\alpha}_{\m{AC};\beta-1;t}}\lesssim \left|\Phi\right|_{C^{k+1,\alpha}_{\m{AC};\beta;t}}
\end{align}
where \eqref{improvedschauderDzetaV} holds for all $\Phi\in C^{0,\alpha}_{\m{ACF};\beta;t}(S,\mathcal{C}o\mathcal{I}^{k,\alpha}_{\m{AC};\beta}(\pi_\zeta/\nu_\zeta))$.
\end{lem}

In the following we will improve these Schauder estimates, by bounding individual components of  
\begin{align*}
    \widehat{D}^t_{\zeta;\beta}=\left(\begin{array}{cc}
         \widehat{D}^t_{\zeta;\beta;\mathcal{C}o\mathcal{II}}&\widehat{D}^t_{\zeta;\beta;\mathcal{KI}}\\
         \widehat{D}^t_{\zeta;\beta;\mathcal{C}o\mathcal{I }\mathcal{C}o\mathcal{K}}&\widehat{D}^t_{\zeta;\beta;\mathcal{K}\mathcal{C}o\mathcal{K}}\\
    \end{array}\right).
\end{align*}

As mentioned previously, the horizontal Dirac operator restricts to a differential operator on $(\nu_\zeta)_{*_{\m{ACF};\beta}}\widehat{E}_\zeta$ and $(\nu_\zeta)_{*_{\m{ACF};\beta-1}}\widehat{E}_\zeta$. In order to improve the Schauder estimate for the components of $\widehat{D}^t_{\zeta;\beta}$ we need to bound the \say{naive} off-diagonal elements. 

\begin{lem}
Let 
\begin{align*}
    \iota_{\mathcal{K}}\geq\phi-m/2-\beta&\und{1.0cm}\iota_{\mathcal{C}o\mathcal{K}}\geq\phi-m/2-\beta\\
    \pi_{\mathcal{I}}\geq \phi+ m/2+\beta-1-\alpha&\und{1.0cm}\pi_{\mathcal{C}o\mathcal{I}}\geq \phi+m/2+\beta-1-\alpha.
\end{align*}
Then the following estimates 
\begin{align}
\label{schauderDtzetapushKCoI}
    \left|\left|\widehat{D}^t_{\zeta;H;\beta;\mathcal{K}\mathcal{C}o\mathcal{I}}\Phi\right|\right|_{C^{k,\alpha}_{\m{ACF};\beta-1;t}}\lesssim& t^{\iota_{\mathcal{K}}}\cdot\left|\left|\varpi_{\mathcal{K};\beta}\Phi\right|\right|_{C^{k+1,\alpha}_{t}}\\ 
\label{schauderDtzetapushCoIK}
     \left|\left|\varpi_{\mathcal{K};\beta}\widehat{D}^t_{\zeta;H;\beta;\mathcal{C}o\mathcal{I}\mathcal{K}}\Phi\right|\right|_{C^{k,\alpha}_{t}}\lesssim& t^{\pi_{\mathcal{C}o\mathcal{K}}}\cdot\left|\left|\pi_{\mathcal{C}o\mathcal{I};\beta}\Phi\right|\right|_{C^{k+1,\alpha}_{\m{ACF};\beta;t}}\\
\label{schauderDtzetapushCoKI}
    \left|\left|\widehat{D}^t_{\zeta;H;\beta;\mathcal{C}o\mathcal{KI}}\Phi\right|\right|_{C^{k,\alpha}_{\m{ACF};\beta-1;t}}\lesssim& t^{\iota_{\mathcal{C}o\mathcal{K}}}\cdot\left|\left|\varpi_{\mathcal{C}o\mathcal{K};\beta-1}\Phi\right|\right|_{C^{k+1,\alpha}_{t}}\\ 
\label{schauderDtzetapushICoK}
     \left|\left|\varpi_{\mathcal{C}o\mathcal{K};\beta-1}\widehat{D}^t_{\zeta;H;\beta;\mathcal{I}\mathcal{C}o\mathcal{K}}\Phi\right|\right|_{C^{k,\alpha}_{t}}\lesssim& t^{\pi_{\mathcal{K}}}\cdot\left|\left|\pi_{\mathcal{I};\beta-1}\Phi\right|\right|_{C^{k+1,\alpha}_{\m{ACF};\beta;t}}
\end{align}
hold true.
\end{lem}

\begin{proof}
Using the \eqref{coimbeta} and \eqref{imbeta} we bound 
    \begin{align*}
        \left|\left|\widehat{D}^t_{\zeta;H;\beta;\mathcal{K}\mathcal{C}o\mathcal{I}}\Phi\right|\right|_{C^{k,\alpha}_{\m{ACF};\beta-1;t}}=& \left|\left|R_{V;\beta}\{\widehat{D}_{\zeta;V},\widehat{D}^t_{\zeta;H}\}\pi_{\mathcal{K};\beta}\Phi\right|\right|_{C^{k,\alpha}_{\m{ACF};\beta-1;t}}\\
        \lesssim &t^{\iota_{\mathcal{K}}}\cdot\left|\left|\varpi_{\mathcal{K};\beta}\Phi\right|\right|_{C^{k+1,\alpha}_{t}}\\ 
    \left|\left|\varpi_{\mathcal{K};\beta}\widehat{D}^t_{\zeta;H;\beta;\mathcal{C}o\mathcal{I}\mathcal{K}}\Phi\right|\right|_{C^{k,\alpha}_{t}}=& \left|\left|\varpi_{\mathcal{K};\beta}\{\widehat{D}_{\zeta;V},\widehat{D}^t_{\zeta;H}\}L_{V;-\beta-m-1}\Phi\right|\right|_{C^{k,\alpha}_{t}}\\
    \lesssim &t^{\pi_{\mathcal{C}o\mathcal{I}}}\cdot\left|\left|\pi_{\mathcal{C}o\mathcal{I};\beta}\Phi\right|\right|_{C^{k+1,\alpha}_{\m{ACF};\beta;t}}\\
    \left|\left|\widehat{D}^t_{\zeta;H;\beta;\mathcal{C}o\mathcal{KI}}\Phi\right|\right|_{C^{k,\alpha}_{\m{ACF};\beta-1;t}}=& \left|\left|L_{V;-\beta-m-1}\{\widehat{D}_{\zeta;V},\widehat{D}^t_{\zeta;H}\}\pi_{\mathcal{C}o\mathcal{K};\beta-1}\Phi\right|\right|_{C^{k,\alpha}_{\m{ACF};\beta-1;t}}\\
    \lesssim &t^{\iota_{\mathcal{C}o\mathcal{K}}}\cdot\left|\left|\varpi_{\mathcal{C}o\mathcal{K};\beta-1}\Phi\right|\right|_{C^{k+1,\alpha}_{t}}\\
    \left|\left|\varpi_{\mathcal{C}o\mathcal{K};\beta-1}\widehat{D}^t_{\zeta;H;\beta;\mathcal{I}\mathcal{C}o\mathcal{K}}\Phi\right|\right|_{C^{k,\alpha}_{t}}=& \left|\left|\varpi_{\mathcal{C}o\mathcal{K};\beta-1}\{\widehat{D}_{\zeta;V},\widehat{D}^t_{\zeta;H}\}R_{V;\beta}\Phi\right|\right|_{C^{k,\alpha}_{t}}\\
    \lesssim& t^{\pi_{\mathcal{I}}}\cdot\left|\left|\pi_{\mathcal{I};\beta-1}\Phi\right|\right|_{C^{k+1,\alpha}_{\m{ACF};\beta;t}}.
\end{align*}
\end{proof}

In the following example we will see that the value of $\phi\geq 1$. This example will be later used in the construction of torsion-free $\m{Spin}(7)$-structures.

\begin{lem}
\label{valueofphi}
    Let $D=\m{d}+\m{d}^*$ be the Hodge-de Rham operator and $\Psi\in C^{k+1,\alpha}_{\m{ACF};\beta}$. We can bound
    \begin{align*}
        \left|\left|L_{\m{AC};-\beta-m-1}\{\widehat{D}_{\zeta;V};\widehat{D}^t_{\zeta;H}\}\pi_{\mathcal{K};\beta}\Psi\right|\right|_{C^{k,\alpha}_{\m{ACF};\beta-1}}\lesssim t^\phi\cdot\left|\left|\pi_{\mathcal{K};\beta}\Psi\right|\right|_{C^{k+1,\alpha}_{\m{ACF};\beta}},
    \end{align*}
for $\phi\geq 1$.
\end{lem}

\begin{proof}
    Notice, that the de Rham differential $\m{d}=\m{d}^{1,0}+\m{d}^{0,1}+\m{d}^{2,-1}$ on a fibration satisfies the identities 
    \begin{align*}
        \{\m{d}^{1,0},\m{d}^{1,0}\}=&-2\{\m{d}^{0,1},\m{d}^{2,-1}\},\hspace{0.3cm}\{\m{d}^{0,1},\m{d}^{0,1}\}=0,\hspace{0.3cm}\{\m{d}^{2,-1},\m{d}^{2,-1}\}=0\\
        \{\m{d}^{1,0},\m{d}^{0,1}\}=&0,\hspace{0.3cm}\{\m{d}^{1,0},\m{d}^{2,-1}\}=0.
    \end{align*}
    The same identities hold for the codifferential $\m{d}^*=\m{d}^{-1,0}+\m{d}^{0,-1}+\m{d}^{-2,1}$. The statement follows from a direct computation for $L_{\m{AC};-\beta-m-1}\{\widehat{D}_{\zeta;V};\widehat{D}^t_{\zeta;H}\}$.
\end{proof}

Building on these component estimates, we now prove a collection of improved Schauder estimates that reflects the almost diagonal structure of the Dirac operator in the adiabatic limit.

\begin{prop}[Improved ACF Schauder Estimates]
\label{mainpropadiabaticfamiliesofdirac}
Let in the following 
\begin{align*}
        \iota\geq&\phi-m/2-\beta\und{1.0cm}\pi\geq\phi+ m/2+\beta-1-\alpha.
\end{align*}
The following estimates 
\begin{align}
\label{schauderDtzetapushII}
   \left|\left|\pi_{\mathcal{C}o\mathcal{I};\beta}\Phi\right|\right|_{C^{k+1,\alpha}_{\m{ACF};\beta;t}(\widehat{D}^t_\zeta)} \lesssim& \left|\left|\widehat{D}^t_{\zeta;\beta;\mathcal{C}o\mathcal{I}\mathcal{I}}\pi_{\mathcal{C}o\mathcal{I};\beta}\Phi\right|\right|_{C^{k,\alpha}_{\beta-1;t}}\\
\label{schauderDtzetapushCoKCoI}
    \left|\left|\widehat{D}^t_{\zeta;\beta;\mathcal{KI}}\pi_{\mathcal{K};\beta}\Phi\right|\right|_{C^{k,\alpha}_{\beta-1;t}}\lesssim& t^{\iota}\cdot\left|\left|\varpi_{\mathcal{K};\beta}\Phi\right|\right|_{C^{k+1,\alpha}_{t}}\\ 
\label{schauderDtzetapushIK}
     \left|\left|\varpi_{\mathcal{C}o\mathcal{K};\beta-1}\widehat{D}^t_{\zeta;\beta;\mathcal{C}o\mathcal{I}\mathcal{C}o\mathcal{K}}\pi_{\mathcal{C}o\mathcal{I};\beta}\Phi\right|\right|_{C^{k,\alpha}_{t}}\lesssim& t^{\pi}\cdot\left|\left|\pi_{\mathcal{C}o\mathcal{I};\beta}\Phi\right|\right|_{C^{k+1,\alpha}_{\m{ACF};\beta;t}}
\end{align}
hold. Furthermore, the map 
    \begin{align*}
        \widehat{D}^t_{\zeta;\beta;\mathcal{K}\mathcal{C}o\mathcal{K}}:C^{k+1,\alpha}_{t}(S,\mathcal{K}_{\m{AC};\beta}(\pi_\zeta/\nu_\zeta))\cap C^{k+1,\alpha}_{\m{ACF};\beta;t}(\widehat{D}^t_\zeta) \rightarrow C^{k,\alpha}_{t}(S,\mathcal{C}o\mathcal{K}_{\m{AC};\beta-1}(\pi_\zeta/\nu_\zeta)) 
    \end{align*}
    satisfies the Schauder estimate
    \begin{align}
    \label{schauderDtzetapushCoII}
    \left|\left|\varpi_{\mathcal{K};\beta}\Phi\right|\right|_{C^{k+1,\alpha}_{t}}\lesssim& \left|\left|\widehat{D}^t_{\zeta;\beta;\mathcal{K}\mathcal{C}o\mathcal{K}}\varpi_{\mathcal{K};\beta}\Phi\right|\right|_{C^{k,\alpha}_{t}}+\left|\left|\varpi_{\mathcal{K};\beta}\Phi\right|\right|_{C^{0}_{t}},
    \end{align}
    is Fredholm, and its kernel and cokernel coincide with 
    \begin{align*}
        \m{ker}(\widehat{D}^t_{\zeta;\beta;\mathcal{K}\mathcal{C}o\mathcal{K}})=\m{ker}(\widehat{D}^t_{\zeta;\beta;\mathcal{K}\mathcal{K}})\und{1.0cm}\m{coker}(\widehat{D}^t_{\zeta;\beta;\mathcal{K}\mathcal{C}o\mathcal{K}})=\m{coker}(\widehat{D}^t_{\zeta;\beta;\mathcal{C}o\mathcal{K}\mathcal{C}o\mathcal{K}}). 
    \end{align*}
\end{prop}

\begin{rem}
\label{E2equalH2}
    Notice that if $(N_\zeta,g_\zeta)=(N\times M_\zeta,g_N\times g_{\zeta;V})/H\rightarrow (N,g_N)/H=(S,g_S)$ for some finite subgroup $H\subset\m{SO}(n)$, then both 
    \begin{align*}
            \pi=\iota=\infty.
    \end{align*}    
 This will be relevant in constructing examples of compact torsion free $\m{Spin}(7)$-manifolds as most resolutions of orbifold singularities will be of this form.
\end{rem}

\begin{rem}
In order to prove \eqref{schauderDtzetapushII} we will follow the proof of \cite[Prop. 8.3]{walpuski2012g_2}. It suffices to prove the statement for $k=0$.
\end{rem}

\begin{proof}[Proof of \eqref{schauderDtzetapushII}]
Following the above proof we only need to prove that
\begin{align}
\label{schauderDtzetaimproved2}
\left|\left|\pi_{\mathcal{C}o\mathcal{I};\beta}\Phi\right|\right|_{C^{0,\alpha}_{\m{ACF};\beta;t}} \lesssim& \left|\left|\widehat{D}^t_{\zeta;\beta;{\mathcal{II}}}\pi_{\mathcal{C}o\mathcal{I};\beta}\Phi\right|\right|_{C^{0,\alpha}_{\m{ACF};\beta-1;t}}.
\end{align}
holds. Assume that \eqref{schauderDtzetaimproved2} is false and that there exists a sequence $(t_i, \Phi_i, x_i)_i$ such that 
\begin{align*}
    t_i\to 0\hspace{1.0cm}\left|\left|\Phi_i\right|\right|_{C^{1,\alpha}_{\m{ACF};\beta;t_i}}\leq C\hspace{1.0cm}
    w^{t_i}_{\zeta;\beta}|\Phi_i|_{h^{t_i}_\zeta}(x_i)=1\\
    \left|\left|(D^{\otimes^{t_i}_\zeta}_{H;\mathcal{II}}+t_i^{-1}\cdot \widehat{D}_{\zeta;V})\Phi_i\right|\right|_{C^{0,\alpha}_{\m{ACF};\beta-1;t_i}}\to 0.    
\end{align*}
In contrast to the previous case we need to distinguish between two cases.\\

\textbf{Case 1}: The sequence $x_i$ has a bounded subsequence and consequently by passing through a subsequence we can assume that $x_i$ converges in the exceptional set $\Upsilon_\zeta$. We define a new sequence by 
\begin{align*}
    \pi_{\mathcal{C}o\mathcal{I};\beta}\hat{\Phi}_i=(t_i)^{q+\beta}\pi_\mathcal{I}\Phi_i.
\end{align*}
By using Arzelà-Ascoli and after passing through another subsequence we can extract a limit
\begin{align*}
    \left|\left|\pi_{\mathcal{C}o\mathcal{I};\beta}\hat{\Phi}_\infty\right|\right|_{C^{0,\alpha}_{\m{ACF};\beta;1}}\leq C\hspace{1.0cm}
    \left|\left| \widehat{D}_{\zeta;V}\pi_{\mathcal{C}o\mathcal{I};\beta}\hat{\Phi}_\infty\right|\right|_{C^{0,\alpha}_{\m{ACF};\beta-1;1}}=0.    
\end{align*} 
As $\widehat{D}_{\zeta;V}$ is injective on sections of $\mathcal{I}^{1,\alpha}_{\m{AC};\beta}(\pi_\zeta/\nu_\zeta)$ by \eqref{improvedschauderDzetaV} and as
\begin{align*}
    \left|\pi_{\mathcal{C}o\mathcal{I};\beta}\hat{\Phi}_\infty\right|_{h^1_\zeta}(\hat{x}_\infty)=1
\end{align*}
holds, we arrived at a contradiction.\\

\textbf{Case 2}: The sequence $x_i\to \infty$, i.e. it leaves every compact subset of $N_\zeta$. We know that for $x_i\to \infty$
\begin{align*}
    (\widehat{E}_{\zeta},\m{cl}_{g^t_\zeta},h^t_\zeta, \widehat{\nabla}^{h^t_\zeta})(x_i)\to (\widehat{E}_{0},\m{cl}_{g^t_0},h^t_0, \nabla^{\otimes_0})(\rho_\zeta(x_i))
\end{align*}
uniform. Hence, we define a new sequence by 
\begin{align*}
    \hat{\Phi}_i=(t_ir(x_i))^{q+\beta}\Psi^*_{r(x_i)}\pi_{\mathcal{C}o\mathcal{I};\beta}\Phi_i\und{1.0cm}\hat{x}_i=\frac{\rho_\zeta(x_i)}{r(\rho_\zeta(x_i))}.
\end{align*}
on $N_\zeta\backslash \widehat{E}_\zeta\cong N_0\backslash S$. By using Arzelà-Ascoli and after passing through another subsequence we can extract a limit
\begin{align*}
    \left|\left|\hat{\Phi}_\infty\right|\right|_{C^{0,\alpha}_{\m{CF};\beta;1}}\leq C\hspace{1.0cm}
    \left|\left| \widehat{D}_{0;V}\hat{\Phi}_\infty\right|\right|_{C^{0,\alpha}_{\m{CF};\beta-1;1}}=0.    
\end{align*} 
By the injectivity of the operator $\widehat{D}_{0;V}$ and
\begin{align*}
    \left|\hat{\Phi}_\infty\right|_{h^1_\zeta}(\hat{x}_\infty)=1,
\end{align*}
we arrive at a contradiction.
\end{proof}
\begin{proof}[Proof of \eqref{schauderDtzetapushCoKCoI}]
By Lemma \ref{comnablaesti} and \eqref{imbeta} we can bound 
\begin{align*}
        \left|\left|\pi_{\mathcal{I};\beta-1}\widehat{D}^t_{\zeta;\beta;\mathcal{KI}}\pi_{\mathcal{K};\beta}\Phi\right|\right|_{C^{k,\alpha}_{\beta-1;t}}=&\left|\left|L_{V;-\beta-m-1}\{\widehat{D}_{\zeta;V},\widehat{D}^t_{\zeta;H}\}\pi_{\mathcal{K};\beta}\Phi\right|\right|_{C^{k,\alpha}_{\m{ACF};\beta-1;t}}\\
        \lesssim &t^{\phi}\left|\left|\varpi_{\mathcal{K};\beta}\Phi\right|\right|_{C^{k+1,\alpha}_{\m{ACF};\beta;t}}\\
        \lesssim &t^{\phi-\beta-m/2}\left|\left|\varpi_{\mathcal{K};\beta}\Phi\right|\right|_{C^{k+1,\alpha}_{t}}.
\end{align*}

\end{proof}
\begin{proof}[Proof of \eqref{schauderDtzetapushIK}]
By Lemma \ref{comnablaesti} and \eqref{coimbeta} we can bound 
\begin{align*}
    \left|\left|\varpi_{\mathcal{C}o\mathcal{K};\beta-1}\widehat{D}^t_{\zeta;\beta;\mathcal{C}o\mathcal{I }\mathcal{C}o\mathcal{K}}\Phi\right|\right|_{C^{k,\alpha}_{t}}=&\left|\left|\varpi_{\mathcal{C}o\mathcal{K};\beta-1}\{\widehat{D}^t_{\zeta;H},\widehat{D}_{\zeta;V}\}L_{V;-\beta-m-1}\Phi\right|\right|_{C^{k,\alpha}_{t}}\\
    \lesssim&t^{m/2+\beta-1-\alpha}\cdot\left|\left|\{\widehat{D}^t_{\zeta;H},\widehat{D}_{\zeta;V}\}L_{V;-\beta-m-1}\Phi\right|\right|_{C^{k,\alpha}_{\m{ACF};\beta-1;t}}\\
    \lesssim&t^{\phi+m/2+\beta-1-\alpha}\cdot\left|\left|\pi_{\mathcal{C}o\mathcal{I};\beta}\Phi\right|\right|_{C^{k+1,\alpha}_{\m{ACF};\beta;t}}.
\end{align*}

\end{proof}

\begin{proof}[Proof of \eqref{schauderDtzetapushCoII}]
Combining \eqref{schauderDtzetapushCoII} and \eqref{schauderDtzetapushCoKCoI} we deduce that \eqref{schauderDtzetapushCoIK}
\begin{align*}
            \left|\left|\varpi_{\mathcal{K};\beta}\Phi\right|\right|_{C^{k+1,\alpha}_{t}}&\lesssim \left|\left|\varpi_{\mathcal{K};\beta}\widehat{D}^t_\zeta\pi_{\mathcal{K};\beta}\Phi\right|\right|_{C^{k,\alpha}_{t}}+\left|\left|\varpi_{\mathcal{K};\beta}\Phi\right|\right|_{C^{0}_{t}}\\
            &\lesssim \left|\left|\widehat{D}^t_{\zeta;\beta;\mathcal{K}\mathcal{C}o\mathcal{K}}\varpi_{\mathcal{K};\beta}\Phi\right|\right|_{C^{k,\alpha}_{t}}+\underbrace{\left|\left|\varpi_{\mathcal{C}o\mathcal{I};\beta}\widehat{D}^t_\zeta\pi_{\mathcal{K};\beta}\Phi\right|\right|_{C^{k,\alpha}_{\m{ACF};\beta-1;t}}}_{\subset \left|\left|\varpi_{\mathcal{K};\beta}\Phi\right|\right|_{C^{k+1,\alpha}_{t}}}\\
            &+ \underbrace{\left|\left|\pi_{\mathcal{I};\beta-1}\widehat{D}^t_{\zeta}\pi_{\mathcal{K};\beta}\Phi\right|\right|_{C^{k,\alpha}_{\m{ACF};\beta-1;t}}}_{\subset \left|\left|\varpi_{\mathcal{K};\beta}\Phi\right|\right|_{C^{k+1,\alpha}_{t}}}+\left|\left|\varpi_{\mathcal{K};\beta}\Phi\right|\right|_{C^{0}_{t}}.
    \end{align*}
\end{proof}

To control the operator uniformly in $t$, we now define adiabatic norms that reflect the asymptotic decoupling of kernel and cokernel components. These norms are tailored to the structure of our problem.

\begin{defi}
Recall Definition \ref{projectionbundledefi} and Definition \ref{defivarpiandiota}. Let $-\iota<\kappa<\pi$. We define the \textbf{adiabatic norms} adapted to the adiabatic family of Dirac operators $\widehat{D}^t_{\zeta}$ to be 
\begin{align}    
\label{adiabaticnormszeta}
    \left|\left|\Phi\right|\right|_{\mathfrak{D}^{k+1,\alpha}_{\m{ACF};\beta;t}}=&\left|\left|\pi_{\mathcal{C}o\mathcal{I};\beta}\Phi\right|\right|_{C^{k+1,\alpha}_{\m{ACF};\beta;t}}+t^{-\kappa}\cdot \left|\left|\varpi_{\mathcal{K};\beta}\Phi\right|\right|_{C^{k+1,\alpha}_{t}}\\
    \left|\left|\Phi\right|\right|_{\mathfrak{C}^{k,\alpha}_{\m{ACF};\beta;t}}=&\left|\left|\pi_{\mathcal{I};\beta-1}\Phi\right|\right|_{C^{k,\alpha}_{\m{ACF};\beta-1;t}}+t^{-\kappa}\cdot \left|\left|\varpi_{\mathcal{C}o\mathcal{K};\beta-1}\Phi\right|\right|_{C^{k,\alpha}_{t}}.
\end{align}
\end{defi}

With these norms, we define the analytic domain of the Dirac operator equipped with the corresponding graph norm. This is the appropriate functional-analytic setting for proving Fredholm properties.

\begin{defi}
We define the space 
\begin{align*}
      \mathfrak{D}^{k+1,\alpha}_{\m{ACF};\beta;t}(\widehat{D}^t_{\zeta};\m{APS}):=\left(\m{dom}_{  \mathfrak{D}^{k+1,\alpha}_{\m{ACF};\beta;t}}(\widehat{D}^t_{\zeta};\m{APS}),\left|\left|.\right|\right|_{  \mathfrak{D}^{k+1,\alpha}_{\m{ACF};\beta;t}(\widehat{D}^t_{\zeta};\m{APS})}\right)
\end{align*}
equipped with the graph norm  
\begin{align*}
    \left|\left|\Phi\right|\right|_{  \mathfrak{D}^{k+1,\alpha}_{\m{ACF};\beta;t}(\widehat{D}^t_{\zeta};\m{APS})}=\left|\left|\Phi\right|\right|_{\mathfrak{D}^{k+1,\alpha}_{\m{ACF};\beta;t}}+\left|\left|\widehat{D}^t_{\zeta}\Phi\right|\right|_{\mathfrak{C}^{k,\alpha}_{\m{ACF};\beta-1;t}}.
\end{align*}
\end{defi}

Using the direct sum decomposition of the bundle and elliptic regularity, we confirm that this domain splits naturally into image and kernel components, as expected in a well-behaved Fredholm theory.

\begin{cor}
    The space $\mathfrak{D}^{k+1,\alpha}_{\m{ACF};\beta;t}(\widehat{D}^t_\zeta;\m{APS})$ decomposes into 
    \begin{align*}
        \mathfrak{D}^{k+1,\alpha}_{\m{ACF};\beta;t}(\widehat{D}^t_\zeta;\m{APS})\cong \m{coim}_{\mathfrak{D}^{k+1,\alpha}_{\m{ACF};\beta;t}(\widehat{D}^t_\zeta;\m{APS})}(\widehat{D}^t_\zeta)\oplus\mathfrak{Ker}_{\m{ACF};\beta}(\widehat{D}_\zeta).
    \end{align*}
\end{cor}

\begin{proof}
    The space $C^{k+1,\alpha}(S,\mathcal{K}_{\m{AC};\beta}(\pi_\zeta/\nu_\zeta))$ splits into $\m{coim}_{C^{k+1,\alpha}}(\mathfrak{D}_{\mathcal{K};\beta})\oplus\m{ker}(\mathfrak{D}_{\mathcal{K};\beta})$. By Lemma \ref{splittACF} we conclude the statement.
\end{proof}

We will now deduce that the ACF Dirac operator is Fredholm when realised on the adapted function spaces. The index is shown to be independent of the Hölder parameters and continuous in the adiabatic limit.

\begin{prop}
\label{FredholmACFDC}
    The space $\mathfrak{D}^{k+1,\alpha}_{\m{ACF};\beta;t}(\widehat{D}^t_{\zeta};\m{APS})$ is a Banach space and the map 
    \begin{align*}
        \widehat{D}^t_{\zeta}:\mathfrak{D}^{k+1,\alpha}_{\m{ACF};\beta;t}(\widehat{D}^t_{\zeta};\m{APS})\rightarrow \mathfrak{C}^{k,\alpha}_{\m{ACF};\beta-1;t}(N_\zeta,\widehat{E}_\zeta)
    \end{align*}
    is Fredholm. The kernel and cokernel are smooth, i.e.  independent of $\alpha$ and $k$. Moreover, 
    \begin{align*}
        \m{ind}_{\m{ACF};\beta}(\widehat{D}^t_{\zeta})=\mathfrak{ind}_{\m{ACF};\beta}(\widehat{D}_{\zeta})
    \end{align*}
    and hence, if $\beta_2<\beta_1$ is a critical rate, then 
    \begin{align*}
       \m{ind}_{\m{ACF};\beta_2}(\widehat{D}^t_{\zeta})-\m{ind}_{\m{ACF};\beta_1}(\widehat{D}^t_{\zeta})=\sum_{\beta_2<\lambda<\beta_1}d_{\lambda+\frac{m-1}{2}-\delta(\widehat{E}_{0})}(\widehat{D}_0)
    \end{align*}
\end{prop}

\begin{proof}
Again this statement follows directly from Proposition \ref{almostsemifredholmACF}, elliptic bootstrapping and the proof of Proposition \ref{FredholmCFS}.
\end{proof}

Finally, we prove a key uniform estimate that bounds the adapted norm of a section in terms of the norm of its image under the Dirac operator, up to a controllable error term. This result is central for the gluing arguments in the applications that follow.

\begin{prop}
\label{uniformboundszetaprop}
Assume that $-\iota<\kappa<\pi$. Then the operator
\begin{align*}
    \widehat{D}^t_{\zeta}:\mathfrak{D}^{k+1,\alpha}_{\m{ACF};\beta;t}(\widehat{D}^t_{\zeta};\m{APS})\rightarrow \mathfrak{C}^{k,\alpha}_{\m{ACF};\beta-1;t}(N_\zeta,\widehat{E}_\zeta)
\end{align*}
satisfies
\begin{align}
\label{almostsemifredholmACF}    \left|\left|\Phi\right|\right|_{\mathfrak{D}^{k+1,\alpha}_{\m{ACF};\beta;t}}\lesssim \left|\left|\widehat{D}^t_{\zeta}\Phi\right|\right|_{\mathfrak{C}^{k,\alpha}_{\m{ACF};\beta-1;t}}+t^{-\kappa}\cdot \left|\left|\pi_{\mathcal{K};\beta}\Phi\right|\right|_{C^{0}_{t}}
\end{align}

Further, let $\Phi\in\mathfrak{Ker}_{\m{ACF};\beta}(\widehat{D}_\zeta)^{\perp}$. Then 
\begin{align}
\label{uniformboundsesti}
\left|\left|\Phi\right|\right|_{\mathfrak{D}^{k+1,\alpha}_{\m{ACF};\beta;t}}\lesssim&\left|\left|\widehat{D}^t_{\zeta}\Phi\right|\right|_{\mathfrak{C}^{k,\alpha}_{\m{ACF};\beta-1;t}}.
\end{align}
\end{prop}

\begin{proof}[Proof of Proposition \ref{uniformboundszetaprop}]
We can bound the first term by
\begin{align*}
    \left|\left|\pi_{\mathcal{C}o\mathcal{I};\beta}\Phi\right|\right|_{C^{k+1,\alpha}_{\m{ACF};\beta;t}}\lesssim&\left|\left|\pi_{\mathcal{I};\beta-1-1}\widehat{D}^t_{\zeta}\pi_{\mathcal{C}o\mathcal{I};\beta}\Phi\right|\right|_{C^{k,\alpha}_{\m{ACF};\beta-1;t}}\\
    \lesssim&\left|\left|\pi_{\mathcal{I};\beta-1-1}\widehat{D}^t_{\zeta}\Phi\right|\right|_{C^{k,\alpha}_{\m{ACF};\beta;t}}+\left|\left|\pi_{\mathcal{I};\beta-1-1}\widehat{D}^t_{\zeta}\pi_{\mathcal{K};\beta}\Phi\right|\right|_{C^{k,\alpha}_{\m{ACF};\beta-1;t}}\\
    \lesssim&\left|\left|\pi_{\mathcal{I};\beta-1-1}\widehat{D}^t_{\zeta}\Phi\right|\right|_{C^{k,\alpha}_{\m{ACF};\beta-1;t}}+\underbrace{t^{\iota}\left|\left|\varpi_{\mathcal{K};\beta}\Phi\right|\right|_{C^{k+1,\alpha}_{t}}}_{\subset t^{-\kappa}\left|\left|\varpi_{\mathcal{K};\beta}\Phi\right|\right|_{C^{k+1,\alpha}_{t}}}
\end{align*}

\begin{align*}
    t^{-\kappa}\left|\left|\varpi_{\mathcal{K};\beta}\Phi\right|\right|_{C^{k+1,\alpha}_{t}}\lesssim&t^{-\kappa}\left|\left|\varpi_{\mathcal{C}o\mathcal{K};\beta-1}\widehat{D}^t_{\zeta}\pi_{\mathcal{K};\beta}\Phi\right|\right|_{C^{k,\alpha}_{t}}+t^{-\kappa}\left|\left|\varpi_{\mathcal{K};\beta}\Phi\right|\right|_{C^{0}_{t}}\\
    \lesssim&t^{-\kappa}\left|\left|\varpi_{\mathcal{C}o\mathcal{K};\beta-1}\widehat{D}^t_{\zeta}\Phi\right|\right|_{C^{k,\alpha}_{t}}+\underbrace{t^{-\kappa}\left|\left|\varpi_{\mathcal{C}o\mathcal{K};\beta-1}\widehat{D}^t_{\zeta}\pi_{\mathcal{C}o\mathcal{I};\beta}\Phi\right|\right|_{C^{k,\alpha}_{t}}}_{\subset\left|\left|\pi_{\mathcal{C}o\mathcal{I};\beta}\Phi\right|\right|_{C^{k,\alpha}_{\m{ACF};\beta;t}}}\\
    &+t^{-\kappa}\left|\left|\varpi_{\mathcal{K};\beta}\Phi\right|\right|_{C^{0}_{t}}
\end{align*}
where the second term can also be absorbed into left hand side of \eqref{uniformboundsesti}. The remaining terms sum up to the $\left|\left|\widehat{D}^t_{\zeta}\Phi\right|\right|_{\mathfrak{C}^{0,\alpha}_{\m{ACF};\beta-1;t}}$ which completes the proof.
\end{proof}

\section{Uniform Elliptic Theory for Dirac Operators on Orbifold Resolutions of Type A}
\label{Uniform Elliptic Theory for Dirac Operators on Orbifold Resolutions of type A}

In the following section we will proceed with constructing the uniform elliptic theory for the Dirac operator $D_t$ on $X_t$, which will be derived by combining our understanding of the adiabatic limits of families of Dirac operators on $N_\zeta$ and the elliptic theory of Dirac operators on orbifolds developed in Section \ref{Fredholm Theory for Dirac Operators on Orbifolds}.

\subsection{Weighted Function Spaces on Resolutions of Orbifolds}
\label{Weighted Function Space on Resolutions of Orbifolds}

In the following we will introduce a family of weighted Banach spaces of sections of the Dirac bundle on the orbifold resolution. As the orbifold resolution is assumed to be a compact manifold, these weights will only contribute to the operator norms and not the Fredholm properties of the analytic realisations. The upshot of choosing these weights is that they can be used to prove uniform operator bounds which are necessary for proving gluing theorems on the orbifold resolution. 

\begin{nota}
    In the following we will need to vary the width of the tubular neighbourhood of the singular stratum. Hence, we will set 
\begin{align*}
    \epsilon\sim t^\lambda
\end{align*}
for a $1>\lambda\geq 0$.
\end{nota}

Building on the uniform Fredholm realizations of the normal operator $\widehat{D}_0$ on the anti-gluing space, i.e. $N_0$, as well as $\widehat{D}_\zeta$ and $D$, we now construct weighted Hölder spaces of sections on the gluing space, i.e. the orbifold resolution $X_t$. Following the strategy outlined in Section \ref{Linear Gluing}, these spaces are defined by cutting and pasting local function spaces associated to the ACF and CFS regions, with weights adapted to the degenerating geometry and stratified singularities. This construction ensures compatibility with the asymptotic behaviour of solutions and sets the stage for uniform operator estimates across the family $D_t$.

\begin{defi}
    We define the function spaces 
\begin{align*}
    \mathfrak{D}^{k,\alpha}_{\beta;t}(X_t,E_t)\und{1.0cm}\overline{\mathfrak{D}}^{k,\alpha}_{\beta;t}(\overline{D}_t;\m{APS})
\end{align*}
by completing $\Gamma^{k,\alpha}(X^t,E^t)$ and $\Gamma^{k,\alpha}_{c}(\overline{X}^t,\overline{E}^t)$ with respect to the norms
    \begin{align*}
        \left|\left|\Phi\right|\right|_{\mathfrak{D}^{k,\alpha}_{\beta;t}}=\left|\left|\delta^*_t\left((1-\chi_4)\cdot\Phi\right)\right|\right|_{\mathfrak{D}^{k,\alpha}_{\m{ACF};\beta;t}}+\left|\left|\chi_2\cdot\Phi\right|\right|_{C^{k,\alpha}_{\m{CFS};\beta;\epsilon}}
    \end{align*}
    and 
    \begin{align*}
        \left|\left|\Phi\right|\right|_{\overline{\mathfrak{D}}^{k,\alpha}_{\beta;t}(\overline{D}_t)}=\left|\left|\delta^*_t\left(\chi_2\cdot\Phi\right)\right|\right|_{\mathfrak{D}^{k,\alpha}_{\m{ACF};\beta;t}(\widehat{D}^t_\zeta)}+\left|\left|(1-\chi_4)\cdot\Phi\right|\right|_{C^{k,\alpha}_{\m{CFS};\beta;\epsilon}}.
    \end{align*}
\end{defi}

The following lemma shows that the graph norm of $D^t$ is uniformly equivalent to the pasted domain norm on the compact resolution $X_t$.

\begin{lem}
    The norms $\mathfrak{D}^{k+1,\alpha}_{\beta;t}(D^t)$ and $\mathfrak{D}^{k+1,\alpha}_{\beta;t}(X^t,E^t)$ are $t$-independent equivalent.
\end{lem}

\begin{proof}
    It suffices to show that 
    \begin{align*}
        \left|\left|D^t\Phi\right|\right|_{\mathfrak{C}^{k,\alpha}_{\beta-1;t}}\lesssim \left|\left|\Phi\right|\right|_{\mathfrak{D}^{k+1,\alpha}_{\beta;t}}.
    \end{align*}
    We begin with 
    \begin{align*}
        \left|\left|(1-\chi_4^t)D^t\delta^*_t\Phi\right|\right|_{\mathfrak{C}^{k,\alpha}_{\m{ACF};\beta-1;t}}\lesssim& \left|\left|D^t(1-\chi^t_4)\delta^*_t\Phi\right|\right|_{\mathfrak{C}^{k,\alpha}_{\m{ACF};\beta-1;t}}+\left|\left|\m{cl}_{g^t}(\chi^t_4)\delta^*_t\Phi\right|\right|_{\mathfrak{C}^{k,\alpha}_{\m{ACF};\beta-1;t}}\\
        \lesssim &\left|\left|D^t(1-\chi^t_4)\delta^*_t\Phi\right|\right|_{\mathfrak{C}^{k,\alpha}_{\m{ACF};\beta-1;t}}+\left|\left|(1-\chi^t_5)\chi_3^t\Phi\right|\right|_{\mathfrak{C}^{k,\alpha}_{\m{ACF};\beta-1;t}}.
    \end{align*}
    Further, we bound 
\begin{align*}
    \left|\left|D^t(1-\chi^t_4)\delta^*_t\Phi\right|\right|_{\mathfrak{C}^{k,\alpha}_{\m{ACF};\beta-1;t}}=&\left|\left|\pi_{\mathcal{I};\beta-1}D^t(1-\chi^t_4)\delta^*_t\Phi\right|\right|_{C^{k,\alpha}_{\m{ACF};\beta-1;t}}\\
    &+t^{-\kappa}\left|\left|\pi_{\mathcal{C}o\mathcal{K};\beta-1}D^t(1-\chi^t_4)\delta^*_t\Phi\right|\right|_{C^{k,\alpha}_{t}}\\
    \lesssim&\left|\left|\pi_{\mathcal{C}o\mathcal{I};\beta}(1-\chi^t_4)\delta^*_t\Phi\right|\right|_{C^{k+1,\alpha}_{\m{ACF};\beta;t}}+t^{-\kappa}\left|\left|\pi_{\mathcal{K};\beta}(1-\chi^t_4)\delta^*_t\Phi\right|\right|_{C^{k+1,\alpha}_{t}}\\
\end{align*}
by using $\left|\left|w^t_{\m{ACF};-1}\right|\right|_{C^{0}(B_{4t^{-1}\epsilon}(N_\zeta))}\lesssim t^\lambda$. We further bound
    \begin{align*}
        \left|\left|\chi_2D_t\Phi\right|\right|_{C^{k,\alpha}_{\m{CFS};\beta-1;\epsilon}}\lesssim& \left|\left|\nabla^{h^t}\chi_2\Phi\right|\right|_{C^{k,\alpha}_{\m{CFS};\beta-1;\epsilon}}+\left|\left|\m{cl}_{g_t}(\m{d}\chi_2)\Phi\right|\right|_{C^{k,\alpha}_{\m{CFS};\beta-1;\epsilon}}\\
        \lesssim&\left|\left|\nabla^{h^t;H}\chi_2\Phi\right|\right|_{C^{k,\alpha}_{\m{CFS};\beta;\epsilon}}+\left|\left|\nabla^{h^t;V}\chi_2\Phi\right|\right|_{C^{k,\alpha}_{\m{CFS};\beta-1;\epsilon}}\\
        &+\left|\left|\chi_1(1-\chi_3)\Phi\right|\right|_{C^{k,\alpha}_{\m{CFS};\beta-1;\epsilon}}\\
        \lesssim&\left|\left|\chi_2\Phi\right|\right|_{C^{k+1,\alpha}_{\m{CFS};\beta;\epsilon}}+\left|\left|\chi_1(1-\chi_3)\Phi\right|\right|_{C^{k,\alpha}_{\m{CFS};\beta-1;\epsilon}}
    \end{align*}
\end{proof}

To construct global solutions from local model data, we now introduce cut-and-paste maps that glue sections from the CFS and ACF regions, as well as from the gluing and anti-gluing spaces. These maps are carefully designed to respect the structure of the weighted norms and will play a central role in the analytic gluing arguments developed in the following subsections.

\begin{defi}
    The projection 
    \begin{align*}
        \cup^t:\mathfrak{D}^{k+1,\alpha}_{\m{ACF};\beta;t}(\widehat{D}^t_\zeta;\m{APS})\oplus C^{k+1,\alpha}_{\m{CFS};\beta;\epsilon}(X,E;\m{APS})\twoheadrightarrow\mathfrak{D}^{k+1,\alpha}_{\beta;t}(X^t,E^t),\,(\Psi\oplus\Phi)\mapsto(\Psi\cup^t\Phi)
    \end{align*}
    has a section 
    \begin{align*}
        \cap^t:\mathfrak{D}^{k+1,\alpha}_{\beta;t}(X^t,E^t)\hookrightarrow\mathfrak{D}^{k+1,\alpha}_{\m{ACF};\beta;t}(\widehat{D}^t_\zeta;\m{APS})\oplus C^{k+1,\alpha}_{\m{CFS};\beta;\epsilon}(X,E;\m{APS}),\,\Psi\mapsto\left(\delta^*_t(1-\chi_4)\Phi,\chi_2\Phi\right)
    \end{align*}
    as well as the projection $\overline{\cup}^t$ has a section $\overline{\cap}^t$ defined in an analogous way.    
\end{defi}

We now relate the adapted weighted norms to both the standard Hölder norms on sections of $E_t$ and the conically fibred (CF) norms on the model spaces. In particular, we show that the resulting function spaces are (non-uniformly) equivalent as Banach spaces.

\begin{lem}
 The Banach spaces $\mathfrak{D}^{k+1,\alpha}_{\beta;t}(X^t,E^t)$ and $C^{k+1,\alpha}_t(X^t,E^t)$ are homeomorphic; the Banach spaces $\overline{\mathfrak{D}}^{k+1,\alpha}_{\beta;t}(\overline{D}^t;\m{APS})$ and the Banach spaces $C^{k+1,\alpha}_{\m{CF};\beta}(\widehat{D}_0;\m{APS})$ are homeomorphic. Moreover, 
    \begin{align*}
        \min\left\{1,t^{\lambda(k+\alpha-\beta)},t^{\kappa-m/2}\right\}\left|\left|\Phi\right|\right|_{C^{k+1,\alpha}_t}&\lesssim \left|\left|\Phi\right|\right|_{\mathfrak{D}^{k,\alpha}_{\beta;t}}\lesssim \max\left\{1,t^{-\lambda\beta},t^{-\kappa+m/2}\right\}\left|\left|\Phi\right|\right|_{C^{k+1,\alpha}_t}\\
        \left|\left|\delta^*_t\Phi\right|\right|_{C^{k+1,\alpha}_{\m{CF};\beta;t}}=\left|\left|\Phi\right|\right|_{C^{k+1,\alpha}_{\m{CF};\beta}}&\lesssim \left|\left|\Phi\right|\right|_{\overline{\mathfrak{D}}^{k,\alpha}_{\beta;t}}\lesssim  \left|\left|\Phi\right|\right|_{C^{k+1,\alpha}_{\m{CF};\beta}}
    \end{align*}
holds. The analogous statements hold for $\mathfrak{C}^{k,\alpha}_{\beta-1;t}(X^t,E^t)$ and $C^{k,\alpha}_t(X^t,E^t)$ as well as $\overline{\mathfrak{C}}^{k,\alpha}_{\beta-1;t}(\overline{X}^t,\overline{E}^t)$ and $C^{k,\alpha}_{\m{CF};\beta-1;}(N_0,\widehat{E}_{0})$. Moreover, given a section $\Phi\in \Gamma^{k+1,\alpha}_{loc}(X^t,E^t)$ with support in $A_{(2t^{-1}\epsilon,3t^{-1}\epsilon)}(S)$, then 
\begin{align*}
        \left|\left|\Phi\right|\right|_{C^{k+1,\alpha}_{\m{CF};\beta}}&\lesssim \left|\left|\Phi\right|\right|_{\mathfrak{D}^{k+1,\alpha}_{\beta;t}}\lesssim \left|\left|\Phi\right|\right|_{C^{k+1,\alpha}_{\m{CF};\beta}}\\
        \left|\left|\Phi\right|\right|_{C^{k,\alpha}_{\m{CF};\beta-1}}&\lesssim \left|\left|\Phi\right|\right|_{\mathfrak{C}^{k,\alpha}_{\beta-1;t}}\lesssim \left|\left|\Phi\right|\right|_{C^{k,\alpha}_{CF-1;\beta}}
    \end{align*}
\end{lem}

\begin{proof}
By the definition of the weighted CFS-norms, we conclude that 
\begin{align*}
    \left|\left|\chi_2\Phi\right|\right|_{C^{k+1,\alpha}_{\m{CFS};\beta;\epsilon}}\lesssim \max\{1,t^{-\lambda\beta}\}\left|\left|\Phi\right|\right|_{C^{k+1,\alpha}_t(X^t\backslash B_{t^{-1}\epsilon}(N_\zeta))}
\end{align*}
and
\begin{align*}
    \left|\left|\Phi\right|\right|_{C^{k+1,\alpha}_t(X^t\backslash B_{2t^{-1}\epsilon}(N_\zeta))}\lesssim \max\{1,t^{\lambda(\beta-k-\alpha)}\}\left|\left|\chi_2\Phi\right|\right|_{C^{k+1,\alpha}_{\m{CFS};\beta;\epsilon}}.
\end{align*}
Further, using the definition of the adiabatic ACF norms, we conclude that 
\begin{align*}
    \left|\left|\delta^*_t(1-\chi_4)\Phi\right|\right|_{\mathfrak{D}^{k+1,\alpha}_{\m{ACF};\beta;t}}\lesssim \max\left\{1,t^{-\lambda\beta},t^{-\kappa+m/2}\right\}\left|\left|\Phi\right|\right|_{C^{k+1,\alpha}_t(B_{4t^{-1}\epsilon}(N_\zeta))}
\end{align*}
and 
\begin{align*}
   \min\left\{1,t^{\lambda(k+\alpha-\beta)},t^{\kappa-m/2}\right\}\left|\left|\Phi\right|\right|_{C^{k+1,\alpha}_t(B_{3t^{-1}\epsilon}(N_\zeta))}\lesssim \left|\left|\delta^*_t(1-\chi_4)\Phi\right|\right|_{\mathfrak{D}^{k+1,\alpha}_{\m{ACF};\beta;t}}.
\end{align*}
We further use 
\begin{align*}
    \left|\left|\pi_{\mathcal{C}o\mathcal{I};\beta}(1-\chi^t_4)\Phi\right|\right|_{C^{k+1,\alpha}_{\m{ACF};\beta;t}}\lesssim&\left|\left|(1-\chi^t_4)\right|\right|_{C^{k+1,\alpha}_{\m{ACF};0;t}}\left|\left|\Phi\right|\right|_{C^{k+1,\alpha}_{\m{ACF};\beta;t}(B_{4t^{-1}\epsilon}(N_\zeta))}\\
    \lesssim&\left|\left|\Phi\right|\right|_{C^{k+1,\alpha}_{\m{ACF};\beta;t}( B_{4t^{-1}\epsilon}(N_\zeta))}\\
    \lesssim&\max\left\{1,t^{-\lambda\beta}\right\}\left|\left|\Phi\right|\right|_{C^{k+1,\alpha}_{t}( B_{4t^{-1}\epsilon}(N_\zeta))}\\
        t^{-\kappa}\left|\left|\varpi_{\mathcal{K};\beta}(1-\chi^t_4)\Phi\right|\right|_{C^{k+1,\alpha}_{t}}&\lesssim t^{-\kappa+m/2}\left|\left|\Phi\right|\right|_{C^{k+1,\alpha}_{t}( B_{4t^{-1}\epsilon}(N_\zeta))}
\end{align*}

In order to show the uniform equivalence of the norms on the gluing neck, we observe that 

\begin{align*}
    \left|\left|\pi_{\mathcal{C}o\mathcal{I};\beta}\chi^t_2\Phi\right|\right|_{C^{k+1,\alpha}_{\m{ACF};\beta;t}}&\lesssim\left|\left|\chi^t_2\right|\right|_{C^{k+1,\alpha}_{\m{ACF};0;t}}\left|\left|\Phi\right|\right|_{C^{k+1,\alpha}_{\m{ACF};\beta;t}(N_\zeta\backslash B_{t^{-1}\epsilon}(N_\zeta))}\lesssim \left|\left|\Phi\right|\right|_{C^{k+1,\alpha}_{\m{ACF};\beta;t}(N_\zeta\backslash B_{t^{-1}\epsilon}(N_\zeta))}\\
\end{align*}
and since
\begin{align*}
    \nabla^{\oplus_0}\chi_i=\epsilon^{-1}\partial_r\chi_i\m{d}r
\end{align*}
we deduce that\\ 
\resizebox{1.0\linewidth}{!}{
\begin{minipage}{\linewidth}
\begin{align*}
    =&\left|(\widehat{\nabla}^{h^t_\zeta;H})^i\pi_{\mathcal{K};\beta}\chi^t_2\Phi\right|_{C^{k-i}_{\m{AC};\beta;t}}(s)\\
    \lesssim&\sum_{j=0}^{i}\sum_{n}\int_{\nu_{\zeta}^{-1}(s)}\left|\left<(\widehat{\nabla}^{g^t_\zeta;H})^j({\chi^t_2})\wedge(\widehat{\nabla}^{h^t_\zeta;H})^{i-j}\Phi(s),e^t_{n}\right>_{h^t_{\zeta;V}}\right|\m{vol}_{t^2g_{\zeta;V}}\left|e^t_n\right|_{C^{k-i,\alpha}_{\m{AC};\beta;t}}\\
    \lesssim&t^{-m/2-\beta}\cdot\sum_{j=0}^{i}\sum_{n}\int_{\nu_{\zeta}^{-1}(s)}\left|\left<(\widehat{\nabla}^{g^t_\zeta;H})^j({\chi^t_2})\wedge(\widehat{\nabla}^{h^t_\zeta;H})^{i-j}\Phi(s),e^t_{n}\right>_{h^t_{\zeta;V}}\right|\m{vol}_{t^2g_{\zeta;V}}\\
    \lesssim&t^{-m/2-\beta}\cdot\sum_{j=0}^{i}\left|(\widehat{\nabla}^{h^t_\zeta})^{i-j}\Phi(s)\right|_{C^{0}_{\m{AC};\beta;t}(N_\zeta\backslash B_{t^{-1}\epsilon}(N_\zeta))}\int^\infty_{t^{\lambda-1}}(rt)^\beta t^{j(1-\lambda)}r^{1-m}r^{\mu_{\mathcal{K}}}t^{m} r^{m-1}\m{d}r\\
    \lesssim&t^{m/2+(\lambda-1)(\beta+\mu_{\mathcal{K}}+1)}\left|(\widehat{\nabla}^{h^t_\zeta})^{i}\Phi(s)\right|_{C^{k-i,\alpha}_{\m{AC};\beta;t}(N_\zeta\backslash B_{t^{-1}\epsilon}(N_\zeta))}.
\end{align*}
\end{minipage}}\\
A moment's thought reveals the additional exponent $(\lambda-1)\alpha$ in the presence of the Hölder exponent. Hence, we can deduce 
\begin{align*}
    \left|\left|\pi_{\mathcal{K};\beta}\chi^t_2\Phi\right|\right|_{C^{k,\alpha}_{\m{ACF};\beta;t}}\lesssim&t^{m/2+(\lambda-1)(\beta+\mu_{\mathcal{K}}+1+\alpha)}\left|\left|\Phi\right|\right|_{C^{k,\alpha}_{\m{ACF};\beta;t}(N_\zeta\backslash B_{t^{-1}\epsilon}(N_\zeta))}.
\end{align*}
The other estimates follow from the definition of the norms.
\end{proof}

We now use the cutting and pasting morphisms to define the approximate kernel and cokernel of $D^t$ by combining kernel and cokernel data from the ACF and CFS regions. Following the strategy developed in Section \ref{Linear Gluing}, these spaces are constructed to approximate the true kernel and cokernel of $D^t$ in a controlled and uniform way, reflecting the limiting behaviour of solutions across the glued geometry.

\begin{defi}
    Let $\beta\notin\mathcal{C}(\widehat{D}_0)$. We define the \textbf{approximate kernel and cokernel} of $D^t$ by
    \begin{align*}
        \m{xker}_{\beta}(D^t)=\begin{array}{c}
             \m{ker}\left(\mathfrak{D}_{\mathcal{K};\beta}: C^{\bullet,\alpha}_{t}\left(S,\mathcal{K}_{\m{AC};\beta}(\pi_\zeta/\nu_\zeta)\right)\rightarrow  C^{\bullet-1,\alpha}_{t}\left(S, \mathcal{K}_{\m{AC};\beta}(\pi_\zeta/\nu_\zeta)\right)\right)\\
             \oplus\\
             \m{ker}\left(D:C^{\bullet,\alpha}_{\m{CFS};\beta;\epsilon}(X,E;\m{APS})\rightarrow C^{\bullet-1,\alpha}_{\m{CFS};\beta-1;\epsilon}(X,E)\right)
        \end{array}
    \end{align*}
    and\\
    \resizebox{1.0\linewidth}{!}{
\begin{minipage}{\linewidth}
    \begin{align*}
        \m{xcoker}_{\beta-1}(D^t)=\begin{array}{c}
             \m{coker}\left(\mathfrak{D}_{\mathcal{C}o\mathcal{K};\beta-1}: C^{\bullet,\alpha}_{t}\left(S,\mathcal{C}o\mathcal{K}_{\m{AC};\beta-1}(\pi_\zeta/\nu_\zeta)\right)\rightarrow  C^{\bullet-1,\alpha}_{t}\left(S,\mathcal{C}o \mathcal{K}_{\m{AC};\beta}(\pi_\zeta/\nu_\zeta)\right)\right)\\
             \oplus\\
             \m{coker}\left(D:C^{\bullet,\alpha}_{\m{CFS};\beta;\epsilon}(X,E;\m{APS})\rightarrow C^{\bullet-1,\alpha}_{\m{CFS};\beta-1;\epsilon}(X,E)\right).
        \end{array}
    \end{align*}
    \end{minipage}}\\
    
\end{defi}

\begin{ex}
    Let $D=\m{d}+\m{d}^*$ and $(N_\zeta,g_\zeta)$ be given by a fibration of ALE spaces. Then the approximate kernel for large enough\footnote{Here, \say{large enough} means that there exists a $\beta<-m/2$ such that there is no critical rate $\lambda\in\mathcal{C}(\widehat{D}_0)$ with $\beta<\lambda<-m/2$.} $\beta<-m/2$ is isomorphic 
    \begin{align*}
        \m{xker}_\beta(D^t)\cong \m{H}^\bullet(S,\mathcal{H}^\bullet(N_\zeta/S))\oplus\m{H}^\bullet(X).
    \end{align*}
\end{ex}

\begin{rem}
    Notice, that if $\beta\in\mathcal{C}(\widehat{D}_0)$ then $\mathfrak{Ker}_{\m{ACF};\beta}(\widehat{D}_\zeta)$ and $\m{ker}_{\m{CFS};\beta}(D)$ are still finite dimensional vector spaces and we can still define the approximate kernel of $D_{t}$. As the operator $\widehat{D}_0:C^{k+1,\alpha}_{\m{CF};\beta}(\widehat{D}_0;\m{APS})\rightarrow C^{k,\alpha}_{\m{CF};\beta-1}(N_0,\widehat{E}_0)$ has a finite dimensional kernel the approximate kernel $\m{xker}_\beta(D_t)$ is identified with the intersection 
    \begin{equation*}
        \begin{tikzcd}
{            \m{xker}_\beta(D^t)} && {            \mathfrak{Ker}_{\m{ACF};\beta}(\widehat{D}_\zeta)} \\
	\\
	{     \m{ker}_{\m{CFS};\beta}(D)} && {           \m{ker}_{\m{CF};\beta}(\widehat{D}_0).}
	\arrow[from=1-1, to=1-3]
	\arrow[from=1-1, to=3-1]
	\arrow["\lrcorner"{anchor=center, pos=0.125}, draw=none, from=1-1, to=3-3]
	\arrow["{\m{res}_{\partial_\infty N_\zeta}}"{description}, from=1-3, to=3-3]
	\arrow["{\m{res}_{S}}"{description}, from=3-1, to=3-3]
        \end{tikzcd}
    \end{equation*}
\end{rem}

\subsection{Uniform Linear Gluing for Dirac Operators on Orbifold Resolutions}
\label{Uniform Linear Gluing for Dirac Operators on Orbifold Resolutions}

In Section~\ref{Linear Gluing}, we established that solving the global Dirac equation
\begin{align*}
    D^t(\Psi \cup^t \Phi) = 0
\end{align*}
is equivalent to solving a coupled system on the local models, expressed as
\begin{align*}
     \Theta^t_{\mathrm{ACF}}(\Psi \oplus \Phi) \cup^t \Theta_{\mathrm{CFS};t}(\Psi \oplus \Phi) = 0,
\end{align*}
where the two components of the system are defined by
\begin{align*}
    \Theta^t_{\mathrm{ACF}}(\Psi \oplus \Phi) 
    \coloneqq \widehat{D}^t_{\zeta} \Psi 
    + (D^t - \widehat{D}^t_{\zeta})(1 - \chi^t_5)\Psi 
    + \mathrm{cl}_{g^t}(\mathrm{d}\chi^t_3)(\delta_t^*\Phi - \Psi),
\end{align*}
and
\begin{align*}
    \Theta_{\mathrm{CFS};t}(\Psi \oplus \Phi) 
    \coloneqq\ D\Phi 
    + (D_t - D)\chi_1 \Phi 
    + \mathrm{cl}_{g_t}(\mathrm{d}\chi_3)(\Phi - \delta_{t^{-1}}^* \Psi).
\end{align*}

Using the uniform Fredholm theories for $D$ and $\widehat{D}^t_{\zeta}$ developed in Section \ref{Elliptic Theory for Dirac Operators on Orbifolds as CFS Spaces} and \ref{Uniform Elliptic Theory of Adiabatic Families of ACF Dirac Operators} we are able to prove the following statement.

\begin{thm}[Uniform Linear Gluing]
\label{lineargluingthm}
The constants $\alpha$, $\beta\notin\mathcal{C}(\widehat{D}_0)$, $\lambda$, $\kappa$ can be chosen such that 
\begin{align*}
            -\phi+m/2+\beta<\kappa<\phi+ m/2+\beta-1-\alpha.
\end{align*}
The operator $D_t$ satisfies 
\begin{empheq}[box=\colorbox{blue!10}]{align}
        \label{unifromestimate}
        \left|\left|\Phi\right|\right|_{\mathfrak{D}^{k+1,\alpha}_{\beta;t}}\lesssim \left|\left|D_t\Phi\right|\right|_{\mathfrak{C}^{k,\alpha}_{\beta-1;t}}
    \end{empheq}
for all $\Phi\in \m{xker}(D_t)^\perp$. Moreover, there exists an exact sequence 
\begin{equation}
\label{lineargluingexactsequence}
    \begin{tikzcd}
        \m{ker}(D_t)\arrow[r,"i_{\beta;t}",hook]&
             \m{xker}_{\beta}(D_t)\arrow[r,"\m{ob}_{\beta;t}"]&\m{xcoker}_{\beta-1}(D_t)\arrow[r,"p_{\beta;t}",two heads]&\m{coker}(D_t)
    \end{tikzcd}
\end{equation}
which implies the existence of a uniform bounded right-inverse of $D_t$ if $\m{ob}_{\beta;t}=0$.
\end{thm}

\begin{rem}
    By taking the Euler characteristic of the sequence, we deduce that $\m{ind}\circ \cup_t=+\circ\mathfrak{ind}$, i.e. 
    \begin{align*}
        \m{ind}(D_t)=\m{ind}_{\m{CFS};\beta}(D)+\mathfrak{ind}_{\m{ACF};\beta}(\widehat{D}_\zeta).
    \end{align*}
\end{rem}

\begin{rem}
    Notice that \eqref{lineargluingexactsequence} is an exact sequence of finite dimensional vector spaces. If 
    \begin{align*}
        \m{dim}(\m{ker}(D_t))=\m{dim}(\m{xker}_\beta(D_t))
    \end{align*}
    for some $\beta\notin\mathcal{C}(\widehat{D}_0)$ then $\m{ob}_{\beta;t}=0$.
\end{rem}

\subsubsection{Proof of Theorem \ref{lineargluingthm}}
\label{Proof of Theorem lineargluingthm}

The following section is devoted to the proof of the Linear-Gluing-Theorem \ref{lineargluingthm} for Dirac operators on orbifold resolutions.\\

Before we start, we need to compare the operator $D_t$ with $\widehat{D}^t_{\zeta}$ and $D$. These are bounds on the non-leading order terms of $\Theta^t_{\m{ACF}}$ and $\Theta_{\m{CFS};t}$. In order to prove these estimates we will begin with the following technical lemmas.

\begin{prop}
    The commutator estimates \\
    \resizebox{1.0\linewidth}{!}{
\begin{minipage}{\linewidth}
    \begin{align}
    \label{commcap}
        \left|\left|\left(\cap^t\circ D_t-\left(\widehat{D}^t_{\zeta}\oplus D\right)\circ \cap^t\right)\Phi\right|\right|_{\mathfrak{C}^{k,\alpha}_{\m{ACF};\beta-1;t}\oplus C^{k,\alpha}_{\m{CFS};\beta-1;\epsilon}}\lesssim& \max\left\{t^{(\lambda-1)(\eta+1)},t^\lambda\right\}\cdot\left|\left|\Phi\right|\right|_{\mathfrak{D}^{k+1,\alpha}_{\beta;t}}\\\nonumber
        &+\left|\left|\Phi\right|\right|_{C^{k,\alpha}_{\m{CF};\beta}(A_{2\epsilon,3\epsilon}(S))}\\
    \label{commcup}
        \left|\left|\left(D_t\circ \cup^t-\cup^t\circ \left(\widehat{D}^t_{\zeta}\oplus D\right)\right)(\Psi\oplus\Phi)\right|\right|_{\mathfrak{C}^{k,\alpha}_{\beta-1;t}}\lesssim& \max\left\{t^{(\lambda-1)(\eta+1)},t^\lambda\right\}\cdot \left|\left|(\Psi\oplus\Phi)\right|\right|_{\mathfrak{D}^{k+1,\alpha}_{\m{ACF};\beta;t}\oplus C^{k+1,\alpha}_{\m{CFS};\beta;\epsilon}}\\\nonumber
        &+\left|\left|\delta^*_{t^{-1}}\Psi-\Phi\right|\right|_{C^{k,\alpha}_{\m{CF};\beta}(A_{2\epsilon,3\epsilon}(S))}\\
    \label{commcapanti}
        \left|\left|\left( \overline{\cap}^t\circ\overline{D}_t- \left(\widehat{D}^t_{\zeta}\oplus D\right)\circ\overline{\cap}^t\right)\Phi\right|\right|_{\mathfrak{C}^{k,\alpha}_{\m{ACF};\beta-1;t}\oplus C^{k,\alpha}_{\m{CFS};\beta-1;\epsilon}}\lesssim &\max\left\{t^{(\lambda-1)(\eta+1)},t^\lambda\right\}\cdot\left|\left|\Phi\right|\right|_{\overline{\mathfrak{D}}^{k+1,\alpha}_{\beta;t}(\overline{D}^t)}\\\nonumber
        &+\left|\left|\Phi\right|\right|_{C^{0,\alpha}_{\m{CF};\beta}(A_{2\epsilon,3\epsilon}(S))}\\
    \label{commcupanti}
        \left|\left|\left(  \overline{D}_t\circ\overline{\cup}^t-\overline{\cup}^t \circ\left(\widehat{D}^t_{\zeta}\oplus D\right)\right)(\Psi\oplus\Phi)\right|\right|_{\overline{\mathfrak{C}}^{k,\alpha}_{\beta-1;t}}\lesssim& \max\left\{t^{(\lambda-1)(\eta+1)},t^\lambda\right\}\cdot \left|\left|(\Psi\oplus\Phi)\right|\right|_{\mathfrak{D}^{k+1,\alpha}_{\m{ACF};\beta;t}\oplus C^{k+1,\alpha}_{\m{CFS};\beta;\epsilon}}\\\nonumber
        &+\left|\left|\delta^*_{t^{-1}}\Psi-\Phi\right|\right|_{C^{k,\alpha}_{\m{CF};\beta}(A_{2\epsilon,3\epsilon}(S))}
    \end{align}
\end{minipage}}
    hold.
\end{prop}

\begin{proof}
In order to proof \eqref{commcap} we observe that 
\begin{align*}
     LHS=&\left|\left|\left((1-\chi^t_3)D^t-\widehat{D}^t_{\zeta}(1-\chi^t_3)\right)\delta^*_t\Phi\right|\right|_{\mathfrak{C}^{k,\alpha}_{\m{ACF};\beta-1;t}}+\left|\left|\left(\chi_3 D_t- D\chi_3 \right)\Phi\right|\right|_{ C^{k,\alpha}_{\m{CFS};\beta-1;\epsilon}}\\
        \lesssim&\underbrace{\left|\left|(D^t-\widehat{D}^t_\zeta)(1-\chi^t_3)\delta^*_t\Phi\right|\right|_{\mathfrak{C}^{k,\alpha}_{\m{ACF};\beta-1;t}}}_{(i)}+\underbrace{\left|\left|\m{cl}_{g^t}(\m{d}\chi^t_3)\delta^*_t\Phi\right|\right|_{\mathfrak{C}^{k,\alpha}_{\m{ACF};\beta-1;t}}}_{\lesssim \left|\left|\delta^*_t\Phi\right|\right|_{C^{k,\alpha}_{\m{CF};\beta;t}(A_{2t^{-1}\epsilon,3t^{-1}\epsilon}(S))}}\\
        &+\underbrace{\left|\left|(D_t-D)\chi_3\Phi\right|\right|_{ C^{k,\alpha}_{\m{CFS};\beta-1;\epsilon}}}_{(ii)}+\underbrace{\left|\left|\m{cl}_{g_t}(\m{d}\chi_3)\Phi\right|\right|_{ C^{k,\alpha}_{\m{CFS};\beta-1;\epsilon}}}_{\lesssim \left|\left|\Phi\right|\right|_{C^{k,\alpha}_{\m{CF};\beta}(A_{2\epsilon,3\epsilon}(S))}}\\
        \lesssim&\max\left\{t^{(\lambda-1)(\eta+1)},t^\lambda\right\}\cdot\left|\left|\Phi\right|\right|_{\mathfrak{D}^{k+1,\alpha}_{\beta;t}}+\left|\left|\Phi\right|\right|_{C^{k,\alpha}_{\m{CF};\beta}(A_{2\epsilon,3\epsilon}(S))}
\end{align*}
whereby $(i)$ and $(ii)$ can be bounded by 
\begin{align*}
    (i)=&\left|\left|(D^t-\widehat{D}^t_\zeta)(1-\chi^t_3)\delta^*_t\Phi\right|\right|_{\mathfrak{C}^{k,\alpha}_{\m{ACF};\beta-1;t}}\\
    =&\left|\left|\left((\m{cl}_{g^t_\zeta}-\m{cl}_{g^t})\circ\widehat{\nabla}^{h^t_\zeta}-\m{cl}_{g^t}\circ(\widehat{\nabla}^{h^t_\zeta}-\nabla^{h^t})\right)(1-\chi^t_3)\delta^*_t\Phi\right|\right|_{\mathfrak{C}^{k,\alpha}_{\m{ACF};\beta-1;t}}\\
    \lesssim& t^\lambda \left|\left|(1-\chi^t_4)\delta^*_t\Phi\right|\right|_{\mathfrak{C}^{k+1,\alpha}_{\m{ACF};\beta;t}}\\
    (ii)=&\left|\left|(D_t-D)\chi_3\Phi\right|\right|_{C^{k,\alpha}_{\m{CFS};\beta-1;\epsilon}}\\
    =&\left|\left|\left((\m{cl}_{g}-\m{cl}_{g_t})\circ\widehat{\nabla}^{h}-\m{cl}_{g_t}\circ(\widehat{\nabla}^{h}-\nabla^{h_t})\right)\chi_3\Phi\right|\right|_{C^{k,\alpha}_{\m{CFS};\beta-1;\epsilon}}\\
    \lesssim&t^{(\lambda-1)(\eta+1)}\left|\left|\chi_2\Phi\right|\right|_{C^{k+1,\alpha}_{\m{CFS};\beta;\epsilon}}.
\end{align*}
In a similar manner we bound \eqref{commcup} by\\
\resizebox{1.0\linewidth}{!}{
\begin{minipage}{\linewidth}
\begin{align*}
    LHS=&\left|\left|\left(D_t-\widehat{D}_{t^2\cdot\zeta}\right)(1-\chi_3)\delta^*_{t^{-1}}\Psi+\left( D_t- D\right)\chi_3\Phi+\m{cl}_{g_{t^2\cdot\zeta}}(\m{d}\chi_3)\delta^*_{t^{-1}}\Psi-\m{cl}_{g}(\m{d}\chi_3)\Phi\right|\right|_{ \mathfrak{C}^{k,\alpha}_{\beta-1;t}}\\
        \lesssim&\left|\left|(1-\chi^t_4)\left(D^t-\widehat{D}^t_{\zeta}\right)(1-\chi^t_3)\Psi\right|\right|_{ \mathfrak{C}^{k,\alpha}_{\m{ACF};\beta-1;t}}+\left|\left|\chi_2\left( D_t- D\right)\chi_3\Phi\right|\right|_{C^{k,\alpha}_{\m{CFS};\beta-1;\epsilon}}\\
        &+\left|\left|\chi_2\left(D_t-\widehat{D}_{t^2\cdot\zeta}\right)(1-\chi_3)\delta_{t^{-1}}^*\Psi\right|\right|_{ C^{k,\alpha}_{\m{CFS};\beta-1;\epsilon}}+\left|\left|(1-\chi^t_4)\left( D^t- \delta_t^*D\right)\chi^t_3\delta^*_t\Phi\right|\right|_{\mathfrak{C}^{k,\alpha}_{\m{ACF};\beta-1;t}}\\
         &+\underbrace{\left|\left|(\m{cl}_{g^t_{\zeta}}(\m{d}\chi^t_3)-\m{cl}_{g^t}(\m{d}\chi^t_3))\Psi\right|\right|_{ \mathfrak{C}^{k,\alpha}_{\m{ACF};\beta-1;t}}+\left|\left|(\m{cl}_{g}(\m{d}\chi_3)-\m{cl}_{g_t}(\m{d}\chi_3))\Phi\right|\right|_{ C^{k,\alpha}_{\m{CFS};\beta-1;\epsilon}}}_{\lesssim \max\left\{t^{(\lambda-1)(\eta+1)},t^\lambda\right\}\left|\left|(\Psi\oplus\Phi)\right|\right|_{\mathfrak{D}^{k+1,\alpha}_{\m{ACF};\beta;t}\oplus C^{k+1,\alpha}_{\m{CFS};\beta;\epsilon}}}\\
         &+\underbrace{\left|\left|\m{cl}_{g^t}(\m{d}\chi_3)(\Psi-\delta^*_{t}\Phi)\right|\right|_{ \mathfrak{C}^{k,\alpha}_{\m{ACF};\beta-1;t}}+\left|\left|\m{cl}_{g_t}(\m{d}\chi_3)(\delta^*_{t^{-1}}\Psi-\Phi)\right|\right|_{ C^{k,\alpha}_{\m{CFS};\beta-1;\epsilon}}}_{\lesssim \left|\left|\delta^*_{t^{-1}}\Psi-\Phi\right|\right|_{C^{k+1,\alpha}_{\m{CF};\beta}(A_{2\epsilon,3\epsilon}(S))}}\\
        \lesssim&\max\left\{t^{(\lambda-1)(\eta+1)},t^\lambda\right\}\cdot\left|\left|(\Psi\oplus\Phi)\right|\right|_{\mathfrak{D}^{k+1,\alpha}_{\m{ACF};\beta;t}\oplus C^{k,\alpha}_{\m{CFS};\beta;\epsilon}}+\left|\left|\delta^*_{t^{-1}}\Psi-\Phi\right|\right|_{C^{k+1,\alpha}_{\m{CF};\beta}(A_{2\epsilon,3\epsilon}(S))}.\\
\end{align*}
\end{minipage}}
Again, by a similar argument we bound \eqref{commcapanti} by
\begin{align*}
     LHS=&\left|\left|\left(\chi^t_3\overline{D}^t -\widehat{D}^t_{\zeta} \chi^t_3 \right)\delta^*_t\Phi\right|\right|_{\mathfrak{C}^{k,\alpha}_{\m{ACF};\beta-1;t}(\widehat{D}^t_\zeta) }+\left|\left|\left((1-\chi_3)\overline{D}_t-D(1-\chi_3)\right)\Phi\right|\right|_{C^{k,\alpha}_{\m{CFS};\beta-1;\epsilon}}\\
        \lesssim&\left|\left|(\widehat{D}^t_{\zeta}-\widehat{D}^t_0)\chi^t_3 \delta_t^*\Phi\right|\right|_{\mathfrak{C}^{k,\alpha}_{\m{ACF};\beta-1;t}(\widehat{D}^t_\zeta) }+\left|\left|(D-\widehat{D}_0)(1-\chi_3)\Phi\right|\right|_{C^{k,\alpha}_{\m{CFS};\beta-1;\epsilon}}\\
        &+\underbrace{\left|\left|\m{cl}_{g^t_0}(\m{d}\chi^t_3)\delta_t^*\Phi\right|\right|_{ \mathfrak{C}^{k,\alpha}_{\m{ACF};\beta-1;t}(\widehat{D}^t_\zeta) }+\left|\left|\m{cl}_{g_0}(\m{d}\chi_3)\Phi\right|\right|_{ C^{k,\alpha}_{\m{CFS};\beta-1;\epsilon}}}_{\lesssim\left|\left|\Phi\right|\right|_{C^{k,\alpha}_{\m{CF};\beta}(A_{2\epsilon,3\epsilon}(S))}}\\
        \lesssim&\max\left\{t^{(\lambda-1)(\eta+1)},t^\lambda\right\}\cdot\left|\left|\Phi\right|\right|_{\overline{\mathfrak{D}}^{k+1,\alpha}_{\beta;t}(\overline{D}^t)}+\left|\left|\Phi\right|\right|_{C^{k,\alpha}_{\m{CF};\beta}(A_{2\epsilon,3\epsilon}(S))}.
\end{align*}
Lastly, \eqref{commcupanti} is bounded by\\
\resizebox{1.0\linewidth}{!}{
\begin{minipage}{\linewidth}
\begin{align*}
     LHS=&\left|\left|\left( \overline{D}_t- D\right)(1-\chi_3)\Psi+\left(\overline{D}_t-\widehat{D}_{t^2\cdot\zeta}\right)\chi_3\delta^*_{t^{-1}}\Phi+\m{cl}_{g_{t^2\cdot\zeta}}(\m{d}\chi_3)\delta^*_{t^{-1}}\Phi-\m{cl}_{g}(\m{d}\chi_3)\Psi\right|\right|_{ C^{k,\alpha}_{\m{CF};\beta-1}}\\
        \lesssim&\left|\left|\chi^t_2\left(\widehat{D}^t_0-\widehat{D}^t_{\zeta}\right)\chi^t_3\Phi\right|\right|_{ \mathfrak{C}^{k,\alpha}_{\m{ACF};\beta-1;t}}+\left|\left|(1-\chi_4)\left( \overline{D}_t- D\right)(1-\chi_3)\Psi\right|\right|_{ C^{k,\alpha}_{\m{CFS};\beta-1;\epsilon}}\\
        &+\left|\left|(1-\chi_4)\left(\widehat{D}_0-\widehat{D}_{t^2\cdot\zeta}\right)\chi_3\delta_{t^{-1}}^*\Phi\right|\right|_{ C^{k,\alpha}_{\m{CFS};\beta-1;\epsilon}}+\left|\left|\chi^t_2\left( \overline{D}^t- \delta^t_*D\right)(1-\chi^t_3)\delta^*_t\Psi\right|\right|_{ \mathfrak{C}^{k,\alpha}_{\m{ACF};\beta-1;t}}\\
        &+\underbrace{\left|\left|(\m{cl}_{g^t_\zeta}(\m{d}\chi^t_3)-\m{cl}_{g^t_0}(\m{d}\chi^t_3))\Phi\right|\right|_{ C^{k,\alpha}_{\m{CF};\beta-1;t}}+\left|\left|(\m{cl}_{g}(\m{d}\chi_3)-\m{cl}_{g_0}(\m{d}\chi_3))\Psi\right|\right|_{ C^{k,\alpha}_{\m{CF};\beta-1}}}_{\lesssim\max\left\{t^{(\lambda-1)(\eta+1)},t^\lambda\right\} \left|\left|(\Psi\oplus\Phi)\right|\right|_{\mathfrak{D}^{k+1,\alpha}_{\m{ACF};\beta;t}\oplus C^{k+1,\alpha}_{\m{CFS};\beta;\epsilon}}}\\
        &+\underbrace{\left|\left|\m{cl}_{g^t_0}(\m{d}\chi_3)(\Psi-\delta^*_{t}\Phi)\right|\right|_{ \mathfrak{C}^{k,\alpha}_{\m{ACF};\beta-1;t}}+\left|\left|\m{cl}_{g_0}(\m{d}\chi_3)(\delta^*_{t^{-1}}\Psi-\Phi)\right|\right|_{ C^{k,\alpha}_{\m{CFS};\beta-1;\epsilon}}}_{\lesssim \left|\left|\delta^*_{t^{-1}}\Psi-\Phi\right|\right|_{C^{k+1,\alpha}_{\m{CF};\beta}(A_{\epsilon,2\epsilon}(S))}}\\
        \lesssim&\max\left\{t^{(\lambda-1)(\eta+1)},t^\lambda\right\}\cdot\left|\left|(\Psi\oplus\Phi)\right|\right|_{\mathfrak{D}^{k+1,\alpha}_{\m{ACF};\beta;t}\oplus C^{k+1,\alpha}_{\m{CFS};\beta;\epsilon}}+\left|\left|\delta^*_{t^{-1}}\Psi-\Phi\right|\right|_{C^{k+1,\alpha}_{\m{CF};\beta}(A_{\epsilon,2\epsilon}(S))}
\end{align*}
\end{minipage}}
\end{proof}

\begin{prop}
\label{unifromboundednessfrombelowDtprop}
    The map 
    \begin{align}
        \label{Dt}D_t:\mathfrak{D}^{k+1,\alpha}_{\beta;t}(X^t,E^t)\rightarrow \mathfrak{C}^{k,\alpha}_{\beta-1;t}(X^t,E^t)
    \end{align}
    is Fredholm and its kernel and cokernel do neither depend on $k$, $\alpha$, nor on $\beta$. Furthermore, it satisfies 
    \begin{align}
        \label{SchauderestimateDt}\left|\left|\Phi\right|\right|_{\mathfrak{D}^{k+1,\alpha}_{\beta;t}}\lesssim \left|\left|D_t\Phi\right|\right|_{\mathfrak{C}^{k,\alpha}_{\beta-1;t}}+\left|\left|\Phi\right|\right|_{\mathfrak{C}^{0}_{\beta;t}}
    \end{align}
    which can be improved to 
        \begin{align}
        \label{unifromboundednessfrombelowDt}
        \left|\left|\Phi\right|\right|_{\mathfrak{D}^{k+1,\alpha}_{\beta;t}}\lesssim \left|\left|D_t\Phi\right|\right|_{\mathfrak{C}^{k,\alpha}_{\beta-1;t}}
    \end{align}
    for all $\Phi\in \m{xker}(D^t)^\perp$. Consequently, the projections
    \begin{align*}
        \Pi_{L^2_t}:\m{xker}_{\beta}(D_t)\twoheadrightarrow\m{ker}(D_t)\und{1.0cm}\Pi_{L^2_t}:\m{xcoker}_{\beta-1}(D_t)\twoheadrightarrow\m{coker}(D_t)
    \end{align*}
    are surjective. 
\end{prop}

\begin{proof}
    Notice, that the  $\mathfrak{D}^{k+1,\alpha}_{\beta;t}$/$\mathfrak{C}^{k,\alpha}_{\beta-1;t}$-norms are $t$-dependent equivalent to the Hölder spaces of sections of $E^t\rightarrow X^t$. Hence, the first statement follows from standard elliptic regularity of $D_t$.\\
    By Proposition \ref{SchauderD} and Proposition \ref{SchauderDtzeta} we conclude
    \begin{align*}
        \left|\left|\Phi\right|\right|_{\mathfrak{D}^{k+1,\alpha}_{\beta;t}}=&\left|\left|(1-\chi^t_4)\cdot\delta^*_t\Phi\right|\right|_{\mathfrak{D}^{k+1,\alpha}_{\m{ACF};\beta;t}}+\left|\left|\chi_2\cdot\Phi\right|\right|_{C^{k+1,\alpha}_{\m{CFS};\beta;t}}\\
        \lesssim&\left|\left|\widehat{D}^t_{\zeta}((1-\chi^t_4)\cdot \Phi)\right|\right|_{\mathfrak{C}^{k,\alpha}_{\m{ACF};\beta-1,t}}+\left|\left|D(\chi_2\cdot\Phi)\right|\right|_{C^{k,\alpha}_{\m{CFS};\beta-1,t}}\\
        &+\left|\left|(1-\chi^t_4) \cdot\Phi\right|\right|_{\mathfrak{C}^{0}_{\m{ACF};\beta;t}}+\left|\left|\chi_2\cdot\Phi\right|\right|_{C^{0}_{\m{CFS};\beta;t}}\\
        \lesssim&\left|\left|D_t\Phi\right|\right|_{\mathfrak{D}^{k,\alpha}_{\beta-1,t}}+\underbrace{\left|\left|\left(\cap^t\circ D_t-\left(\widehat{D}^t_{\zeta}\oplus D\right)\circ \cap^t\right)\Phi\right|\right|_{\mathfrak{C}^{k,\alpha}_{\m{ACF};\beta-1;t}\oplus C^{k,\alpha}_{\m{CFS};\beta-1;\epsilon}}}_{\left\{\subset\left|\left|\Phi\right|\right|_{\mathfrak{D}^{k+1,\alpha}_{\beta;t}}\right\}+\left|\left|\Phi\right|\right|_{C^{k,\alpha}_{\m{CF};\beta;t}(A_{(t^{-1}\epsilon,2t^{-1}\epsilon)}(S))}}\\
        &+\left|\left|\Phi\right|\right|_{\mathfrak{C}^{0}_{\beta;t}}        
    \end{align*}
In order to show the uniform Schauder estimate 
        \begin{align*}
        \left|\left|\Phi\right|\right|_{\mathfrak{D}^{k+1,\alpha}_{\beta;t}}\lesssim&\left|\left|D_t\Phi\right|\right|_{\mathfrak{C}^{k,\alpha}_{\beta-1,t}} +\left|\left|\Phi\right|\right|_{\mathfrak{C}^{0}_{\beta;t}}           
    \end{align*}
it remains to prove that 
\begin{align*}
    \left|\left|\Phi\right|\right|_{C^{k,\alpha}_{\m{CF};\beta;t}(A_{(t^{-1}\epsilon,2t^{-1}\epsilon)}(S))}\lesssim \left|\left|D_t\Phi\right|\right|_{\mathfrak{C}^{k,\alpha}_{\beta-1;t}}.
\end{align*}
By Lemma \ref{D-D0lemma} and Lemma \ref{Dzeta-D0lemma} we know that\\ 
\resizebox{1.0\linewidth}{!}{
\begin{minipage}{\linewidth}
\begin{align*}
    \left|\left|\Phi\right|\right|_{C^{k,\alpha}_{\m{CF};\beta;t}(A_{(t^{-1}\epsilon,2t^{-1}\epsilon)}(S))}\lesssim& \left|\left|\widehat{D}^t_0\Phi\right|\right|_{C^{k,\alpha}_{\m{CF};\beta-1;t}(A_{(2t^{-1}\epsilon,3t^{-1}\epsilon)}(S))}+\left|\left|\Phi\right|\right|_{C^{0}_{\m{CF};\beta;t}(A_{(2t^{-1}\epsilon,3t^{-1}\epsilon)}(S))}\\
    \lesssim& \left|\left|D_t\Phi\right|\right|_{C^{k,\alpha}_{\m{CF};\beta-1;t}(A_{(2t^{-1}\epsilon,3t^{-1}\epsilon)}(S))}+\underbrace{\left|\left|(D^t-\widehat{D}^t_0)\Phi\right|\right|_{C^{k,\alpha}_{\m{CF};\beta;t}(A_{(2t^{-1}\epsilon,3t^{-1}\epsilon)}(S))}}_{(i)}\\
    &+\left|\left|\Phi\right|\right|_{C^{0}_{\m{CF};\beta;t}(A_{(2t^{-1}\epsilon,3t^{-1}\epsilon)}(S))}.
\end{align*}
\end{minipage}}\\
and as $(i)$ can be bounded by\\
\resizebox{1.0\linewidth}{!}{
\begin{minipage}{\linewidth}
\begin{align*}
    (i)=&\left|\left|\left(\m{cl}_{g^t}\circ(\nabla^{h^t}-\nabla^{\otimes_0})+(\m{cl}_{g^t}-\m{cl}_{g_0})\circ \nabla^{\otimes_0}\right)\Phi\right|\right|_{C^{k,\alpha}_{\m{CF};\beta;t}(A_{(2t^{-1}\epsilon,3t^{-1}\epsilon)}(S))}\\
    \lesssim&\left|\left|\m{cl}_{g^t}\circ(\nabla^{h^t}-\nabla^{\otimes_0})\Phi\right|\right|_{C^{k,\alpha}_{\m{CF};\beta;t}(A_{(2t^{-1}\epsilon,3t^{-1}\epsilon)}(S))}+\left|\left|(\m{cl}_{g^t}-\m{cl}_{g_0})\circ \nabla^{\otimes_0}\Phi\right|\right|_{C^{k,\alpha}_{\m{CF};\beta;t}(A_{(2t^{-1}\epsilon,3t^{-1}\epsilon)}(S))}\\
    \lesssim& \max\left\{t^{(\lambda-1)(\eta+1)},t^{\lambda}\right\}\left|\left|\Phi\right|\right|_{\mathfrak{D}^{k+1,\alpha}_{\beta;t}}\\
\end{align*}
\end{minipage}}
it can be absorbed into the left-hand side of the Schauder estimate \eqref{SchauderestimateDt}.\\

We will now proceed with proving the uniform lower boundedness of $D_t$ on $\m{xker}_{\beta}(D_t)^\perp$. By using \eqref{SchauderestimateDt} it suffices to prove that 
\begin{align*}
    \left|\left|\Phi\right|\right|_{\mathfrak{C}^{0}_{\beta;t}}\lesssim \left|\left|D_t\Phi\right|\right|_{\mathfrak{C}^{0,\alpha}_{\beta-1,t}}.
\end{align*}
We will prove this estimate by contradiction. Assume that \eqref{unifromboundednessfrombelowDt} is false and that there exist sequences $t_i\to 0$ and  $\Phi_i\perp\m{xker}_{\beta}(D_t)$ such that 
\begin{align*}
    \left|\left|\Phi_i\right|\right|_{\mathfrak{C}^{0}_{\beta;t_i}}=1,
    \und{1.0cm}\left|\left|D^{t^i}\Phi_i\right|\right|_{\mathfrak{C}^{0,\alpha}_{\beta-1;t_i}}\to 0
\end{align*}
Using the uniform Schauder estimate we know that 
\begin{align*}
    \left|\left|\Phi_i\right|\right|_{\mathfrak{C}^{1,\alpha}_{\beta;t_i}}\leq 1
\end{align*}
is bounded and by passing to a subsequence we can arrange that 
\begin{align*}
    \left|\left|(1-\chi^t_4)\delta_t^*\Phi_i\right|\right|_{\mathfrak{C}^{0,\alpha}_{\m{ACF};\beta;t_i}}+\left|\left|\chi_2\Phi_i\right|\right|_{C^{0,\alpha}_{\m{CFS};\beta;\epsilon_i}}\lesssim C
\end{align*}
Further, we pick a sequence $x_i\in X^{t_i}$ such that 
\begin{align*}
        |w_{\beta;t_i}\Phi_i|_{h^{t_i}}(x_i)=1.
\end{align*}
After passing through a subsequence, we can reduce the contradiction argument to one of the following cases:
\begin{itemize}
    \item[1.] The sequence $x_i$ accumulates near the exceptional set of the resolution $\rho^t:X^t\dashrightarrow X$ $\Rightarrow$ ACF regime
    \item[2.] The sequence $x_i$ accumulates in the gluing neck of the space $X^t$ $\Rightarrow$ CF regime
    \item[3.] The sequence $x_i$ accumulates on the regular part of the resolution $\rho^t:X^t\dashrightarrow X$ $\Rightarrow$ CFS regime
\end{itemize}

In the following we will show that all three cases will lead to a contradiction.\\

\textbf{Case 1:} By passing to a subsequence and by using Arzelà-Ascoli we can extract a limit $x_\infty\in N_\zeta$ and a $\varpi_{\mathcal{K};\beta}\Phi_\infty\in C^{1,\alpha/2}(S,\mathcal{K}_{\m{AC};\beta}(\pi_{\zeta}/\nu_{\zeta}))$ satisfying 
\begin{align*}
    1\lesssim|\varpi_{\mathcal{K};\beta}\Phi_{\infty}(\nu_\zeta(x_\infty))|_{(\nu_{\zeta})_*h_\zeta}
\end{align*}
and
\begin{align*}
    \mathfrak{D}_{\mathcal{K};\beta}\varpi_{\mathcal{K};\beta}\Phi_\infty=0
\end{align*}
which contradicts $\Phi_\infty\in \m{xker}_{\beta}(D_t)^\perp$.\\

\textbf{Case 2:} Assume that the sequence $x_i$ accumulates in the neck of $\rho^{t_i}:X^{t_i}\dashrightarrow X$. By using Lemma \ref{D-D0lemma} and Lemma \ref{Dzeta-D0lemma} as well as a rescaling argument, and by passing through a subsequence, we can apply Arzelà-Ascoli to extract a limit $x_\infty\in N_0$ and $\Phi_\infty \in C^{1,\alpha/2}_{\m{CF};\beta}(N_0,\widehat{E}_{0})$ with 
\begin{align*}
    1\lesssim |\Phi_\infty(x_\infty)|_{h_0}
\end{align*}
and 
\begin{align*}
    \widehat{D}_0\Phi_\infty=0
\end{align*}
which is a contradiction to Proposition \ref{widehatD0iso}.\\

\textbf{Case 3:} By passing through a subsequence and by Arzelà-Ascoli we can $x_\infty\in N_\zeta$ and a $\Phi_\infty\in C^{1,\alpha/2}_{\m{CFS};\beta;\epsilon_i}(X,E)$ satisfying 
\begin{align*}
    1\lesssim|\Phi_{\infty}(x_\infty)|_{h}
\end{align*}
and
\begin{align*}
    D\Phi_\infty=0
\end{align*}
which contradicts $\Phi\in \m{xker}_{\beta}(D_t)^\perp$.
\end{proof}

In the following, we will adapt the discussion of Hutchings and Taubes in \cite[section 9]{hutchingstaubesII} on the linear gluing exact sequence to our setup. We will begin, by proving the existence of solutions to the so-called anti-gluing equation.

\begin{lem}
\label{antigluing}
The so called anti-gluing map 
\begin{equation*}
\adjustbox{scale=0.8,center}{
    \begin{tikzcd}
        \begin{array}{c}
        C^{k+1,\alpha}_{\m{ACF};\beta;t}(N_\zeta\backslash B_{2t^{-1}\epsilon}(N_\zeta),\widehat{E}_\zeta)\\
        \oplus\\
        C^{k+1,\alpha}_{\m{CFS};\beta;\epsilon}(\m{Tub}_{4\epsilon}(S),E)\end{array}\arrow[r,"\overline{\Theta}^t"]&\begin{array}{c}
              C^{k,\alpha}_{\m{ACF};\beta-1;t}(N_\zeta\backslash B_{2t^{-1}\epsilon}(N_\zeta),\widehat{E}_\zeta)\\
              \oplus\\
              C^{k+1,\alpha}_{2t^{-1}\epsilon}(Y,\widehat{E}_Y)\\ 
            \oplus\\
            C^{k,\alpha}_{\m{CFS};\beta-1;\epsilon}(\m{Tub}_{4\epsilon}(S),E)\\
            \oplus\\
            C^{k+1,\alpha}_{4\epsilon}(Y,\widehat{E}_Y)
        \end{array}&
        \begin{array}{c}
             \Psi\\
             \oplus\\
             \Phi
        \end{array}\arrow[r,mapsto]&\begin{array}{c}
             \Theta^t_{\m{ACF}}(\Psi\oplus\Phi) \\
             \oplus\\
             \m{res}_{2t^{-1}\epsilon}(\Psi)\\
        \oplus \\
            \Theta_{\m{CFS};t}(\Psi\oplus\Phi) \\
             \oplus\\
             \m{res}_{4\epsilon}(\Phi)
        \end{array}
    \end{tikzcd}}
\end{equation*}
is an isomorphism of Banach spaces.
\end{lem}

\begin{proof}
    Notice that following Proposition \ref{widehatD0iso} and Lemma \ref{differentchoicesofrestriciton} the map
    \begin{equation*}
    \adjustbox{scale=0.8,center}{
    \begin{tikzcd}
        \begin{array}{c}
        C^{k+1,\alpha}_{\m{CF};\beta;t}(N_0\backslash \m{Tub}_{2t^{-1}\epsilon},\widehat{E}_{0})\\
        \oplus\\
        C^{k+1,\alpha}_{\m{CF};\beta}(\m{Tub}_{4\epsilon}(S),\widehat{E}_{0})\end{array}\arrow[r]&\begin{array}{c}
              C^{k,\alpha}_{\m{CF};\beta-1;t}(N_0\backslash \m{Tub}_{2t^{-1}\epsilon}(S),\widehat{E}_{0})\\
              \oplus\\
              C^{k+1,\alpha}_{2t^{-1}\epsilon}(Y,\widehat{E}_Y)\\ 
            \oplus\\
            C^{k,\alpha}_{\m{CF};\beta-1}(\m{Tub}_{4\epsilon}(S),\widehat{E}_{0})\\
            \oplus\\
            C^{k+1,\alpha}_{4\epsilon}(Y,\widehat{E}_Y)
        \end{array}&
        \begin{array}{c}
             \Psi\\
             \oplus\\
             \Phi
        \end{array}\arrow[r,mapsto]&\begin{array}{c}
             \widehat{D}^t_0\Psi \\
             \oplus\\
             \m{res}_{2t^{-1}\epsilon}(\Psi)\\
        \oplus \\
            \widehat{D}_0 \\
             \oplus\\
             \m{res}_{4\epsilon}(\Phi)
        \end{array}
    \end{tikzcd}}
\end{equation*}
defines an isomorphism of Banach spaces with a uniform bounded inverse. Using a left-inverse, we can write the anti-gluing map as (dropping the restriction maps)
\begin{align*}
    \widehat{L}_0\Theta^t=&\left(\begin{array}{cc}
         1-\widehat{L}^t_0\m{cl}_{g^t}(\m{d}\chi_3^t)&\widehat{L}^t_0\m{cl}_{g^t}(\m{d}\chi_3^t)\delta^*_t  \\
         \widehat{L}_0\m{cl}_{g_t}(\m{d}\chi_3)\delta^*_{t^{-1}}&1+\widehat{L}_0\m{cl}_{g_t}(\m{d}\chi_3)\delta^*_{t^{-1}} 
    \end{array}\right)\\
    &+\left(\begin{array}{cc}
         \widehat{L}^t_{0}\left((D^t-\widehat{D}^t_\zeta)(1-\chi^t_5)+(\widehat{D}^t_\zeta-\widehat{D}^t_0)\right)&  0\\
         0& \widehat{L}_{0}\left((D_t-D)\chi_1+(D-\widehat{D}_0)\right)
    \end{array}\right)\\
    =&T+S.
\end{align*}
Now, $T$ is invertible with $T^{-1}=2-T$. Using estimate \eqref{commcupanti} we deduce that
\begin{align*}
    \left|\left|S(\Psi\oplus\Phi)\right|\right|_{C^{k,\alpha}_{\m{ACF};\beta-1;t}\oplus
        C^{k,\alpha}_{\m{CFS};\beta-1;\epsilon}}\lesssim \m{max}\left\{t^\lambda,t^{(1-\lambda)(1+\eta)}\right\}\left|\left|\Psi\oplus\Phi\right|\right|_{C^{k+1,\alpha}_{\m{ACF};\beta;t}\oplus
        C^{k+1,\alpha}_{\m{CFS};\beta;\epsilon}}.
\end{align*}
\end{proof}

Using the existence result for solution of the anti-gluing map we can construct the map $i_{\beta;t}$ in the following Lemma.

\begin{lem}
    Let $\Sigma\in \mathfrak{D}^{k+1,\alpha}_{\beta;t}(X^t,E^t)$, $\Xi\oplus \m{H}\in \mathfrak{C}^{k,\alpha}_{\m{ACF};\beta-1;t}(N_\zeta,\widehat{E}_\zeta)\oplus C^{k,\alpha}_{\m{CFS};\beta-1;\epsilon}(X,E)$ such that 
    \begin{align*}
        D_t\Sigma=\Xi\cup^t \m{H}.
    \end{align*}
    Then there exists a unique $\Psi\oplus\Phi\in \mathfrak{D}^{k+1,\alpha}_{\m{ACF};\beta;t}(\widehat{D}^t_\zeta)\oplus C^{k+1,\alpha}_{\m{CFS};\beta;\epsilon}(X,E;\m{APS})$ such that 
    \begin{align*}
        \Sigma=\Psi\cup^t \Phi\und{1.0cm}\Theta^t_{\m{ACF}}(\Psi\oplus\Phi)=\Xi\und{0.5cm}\Theta_{\m{CFS};t}(\Psi\oplus\Phi)=\m{H}.
    \end{align*}
    In particular, the homomorphism 
    \begin{align*}
        i_{\beta;t}:\m{ker}(D_t)\rightarrow\m{xker}_{\beta}(D_t)
    \end{align*}
    is well-defined.
\end{lem}

\begin{proof}
We begin with an arbitrary decomposition of $\Sigma=\Psi'\cup^t\Phi'$. Then 
\begin{align*}
    (\Xi-\Theta^t_{\m{ACF}}(\Psi'\oplus\Phi')\cup^t(\m{H}-\Theta_{\m{CFS};t}(\Psi'\oplus\Phi'))=0
\end{align*}
and by writing $\Psi\oplus\Phi=(\Psi'+\Psi'')\oplus(\Phi'+\Phi'')$ we deduce that 
\begin{align*}
    \Theta^t_{\m{ACF}}(\Psi''\oplus\Phi'')=&\Xi-\Theta^t_{\m{ACF}}(\Psi'\oplus\Phi')\\
    \Theta_{\m{CFS};t}(\Psi''\oplus\Phi'')=&\m{H}-\Theta_{\m{CFS};t}(\Psi'\oplus\Phi').
\end{align*}
By Lemma \ref{antigluing}, this indeed has a solution.
\end{proof}

Using Proposition \ref{FredholmCFS} and Theorem \ref{FredholmACF}, as well as the ACF- and CFS-wall-crossing formulae we are able to make the following observation.

\begin{prop}
\label{DzetaplusD}
The map 
\begin{align*}
    \begin{array}{c}
         \widehat{D}^t_{\zeta} \\
         \oplus  \\
         D
    \end{array}:\begin{array}{c}
         \mathfrak{D}^{k+1,\alpha}_{\m{ACF};\beta;t}(\widehat{D}^t_{\zeta};\m{APS})  \\
         \oplus\\
          C^{k+1,\alpha}_{\m{CFS};\beta;\epsilon}(X,E;\m{APS})
    \end{array}\longrightarrow \begin{array}{c}
         \mathfrak{C}^{k,\alpha}_{\m{ACF};\beta-1;t}(N_\zeta,\widehat{E}_\zeta)  \\
         \oplus\\
         C^{k,\alpha}_{\m{CFS};\beta-1;\epsilon}(X,E)
    \end{array}
\end{align*}
is Fredholm and its kernel and cokernel do not depend on $k$ nor $\alpha$. Moreover, its index does not depend on $\beta$.
\end{prop}

Using the above we are able to construct the linear obstruction map $\m{ob}_{\beta;t}$ present in the centre of the linear gluing exact sequence.

\begin{lem}
The homomorphism  
\begin{align*}
    \m{ob}_{\beta;t}:\m{xker}_{\beta}(D_t)\rightarrow\m{xcoker}_{\beta-1}(D_t)
\end{align*}
is well-defined.
\end{lem}

\begin{proof}
By Proposition \ref{DzetaplusD} the map 
\begin{align*}
    \begin{array}{c}
         \widehat{D}^t_{\zeta} \\
         \oplus  \\
         D
    \end{array}:\begin{array}{c}
         \mathfrak{D}^{k+1,\alpha}_{\m{ACF};\beta;t}(\widehat{D}^t_{\zeta};\m{APS})  \\
         \oplus\\
          C^{k+1,\alpha}_{\m{CFS};\beta;\epsilon}(X,E;\m{APS})
    \end{array}\longrightarrow \begin{array}{c}
         \mathfrak{C}^{k,\alpha}_{\m{ACF};\beta-1;t}(N_\zeta,\widehat{E}_\zeta)  \\
         \oplus\\
         C^{k,\alpha}_{\m{CFS};\beta-1;\epsilon}(X,E)
    \end{array}
\end{align*}
satisfies 
\begin{align*}
    \left|\left|(\Psi\oplus\Phi)\right|\right|_{\mathfrak{D}^{k+1,\alpha}_{\m{ACF};\beta;t}(\widehat{D}^t_{\zeta};\m{APS})\oplus C^{k+1,\alpha}_{\m{CFS};\beta;\epsilon}}\lesssim \left|\left|(\widehat{D}^t_{\zeta}\Psi,D\Phi)\right|\right|_{\mathfrak{C}^{k,\alpha}_{\m{ACF};\beta-1;t}\oplus C^{k,\alpha}_{\m{CFS};\beta-1;\epsilon}}
\end{align*}
on $\m{xker}_{\beta}(D_t)^\perp$. Thus there exists a uniform bounded left-inverse 
\begin{align*}
    \widehat{L}^t_{\zeta}\oplus L:\m{xcoker}_{\beta-1}(D_t)^\perp\rightarrow\m{xker}_{\beta}(D_t)^\perp
\end{align*}

By estimate \eqref{commcup} and by a similar argument as in the proof of Lemma \ref{antigluing}, we deduce that the map
\begin{align*}
    (\Psi\oplus\Phi)\mapsto \left(\widehat{L}^t_\zeta\Theta^t_{\m{ACF}}(\Psi\oplus\Phi),L\Theta_{\m{CFS};t}(\Psi\oplus\Phi)\right)
\end{align*}
is invertible. Hence, we deduce that for every $\Xi\oplus \m{H}\in\m{xker}_\beta(D_t)$ there exists unique $\Psi\oplus\Phi\in\m{xker}_\beta(D_t)^\perp$
such that 
\begin{align}
\label{existenceofob}
    \Theta^t_{\m{ACF}}(\Xi+\Psi,H+\Phi)\oplus \Theta_{\m{CFS};t}(\Xi+\Psi,H+\Phi)\in\m{xcoker}_\beta(D_t).
\end{align}
The linear obstruction map is defined as
\begin{align*}
    \m{ob}_{\beta;t}(\Xi\oplus \m{H})=\Theta^t_{\m{ACF}}((\Xi+\Psi)\oplus(H+\Phi))\oplus \Theta_{\m{CFS};t}((\Xi+\Psi)\oplus(H+\Phi)).
\end{align*}
\end{proof}

Finally, we are able to construct the projection $p_{\beta;t}$.

\begin{lem}
The homomorphism 
\begin{align*}
    p_{\beta;t}:\m{xcoker}_{\beta-1}(D_t) \rightarrow \m{coker}(D_t)
\end{align*}
is well defined.
\end{lem}

\begin{proof}
    This statement follows from the self-adjointness of $D$, $\widehat{D}^t_\zeta$ and $D_t$ as well as Proposition \ref{unifromboundednessfrombelowDtprop} and by $\beta\to-m-\beta-1$.
\end{proof}

Finally, we prove the exactness of the sequence.
\begin{lem}\cite{hutchingstaubesII}
    The sequence \eqref{lineargluingexactsequence} is exact. 
\end{lem}

\begin{proof}
We begin by proving the injectivity of $i_{\beta;t}$. Consider $\Sigma\in\m{ker}(D_t)$ such that $i_{\beta;t}(\Sigma)=0$. This implies that $\Sigma=\Psi\cup^t\Phi$ for $\Psi\oplus\Phi\in \m{xker}_\beta(D_t)^\perp$ and $\Theta^t_{ACF/CFS}(\Psi\oplus\Phi)=0$. By the uniqueness of the solution to \eqref{existenceofob} with $\Xi\oplus\m{H}=0$ we deduce that $\Psi\oplus\Phi=0$ which implies $\Sigma=0$. Hence, $i_{\beta;t}$ is injective. \\

Next, we are going to show that $\m{ker}(\m{ob}_{\beta;t})=\m{im}(i_{\beta;t})$. By definition $\m{ker}(\m{ob}_{\beta;t})\supseteq\m{im}(i_{\beta;t})$ and hence, it suffice to prove the converse inclusion $\m{ker}(\m{ob}_{\beta;t})\subseteq\m{im}(i_{\beta;t})$.\\
Suppose that $\m{ob}_{\beta;t}(\Psi\oplus\Phi)=0$. Then there exists $\Xi\oplus\m{H}\in\m{xker}_\beta(D_t)^\perp$ such that 
\begin{align*}
    \Theta^t_{\m{ACF}}((\Xi+\Psi)\oplus(\m{H}+\Phi))=0\und{1.0cm}\Theta^t_{\m{ACF}}((\Xi+\Psi)\oplus(\m{H}+\Phi)).
\end{align*}
Consequently, $\Sigma=(\Xi+\Psi)\cup^t(\m{H}+\Phi)\in\m{ker}(D_t)$ and hence, $i_{\beta;t}(\Sigma)=\Xi\oplus \m{H}$.\\

Next, we are going to show that $\m{ker}(p_{\beta;t})=\m{im}(\m{ob}_{\beta;t})$. By the definition of the maps $\m{ob}_{\beta;t}$ and $p_{\beta;t}$ it suffices to prove that $\m{im}(\m{ob}_{\beta;t})\supseteq\m{ker}(p_{\beta;t})$. Let $\Psi\oplus\Phi\in\m{xcoker}_\beta(D_t)$ and suppose $p_{\beta;t}(\Psi\oplus\Phi)=0$. Hence, there exists a $\Sigma\in \mathfrak{D}^{k+1,\alpha}_{\beta;t}(D_t)$ such that 
\begin{align*}
    D_t\Sigma=\Psi\cup^t\Phi.
\end{align*}
Lemma \ref{antigluing} implies the existence of $\Xi\oplus \m{H}\in \mathfrak{D}^{k+1,\alpha}_{\m{ACF};\beta;t}(\widehat{D}^t_\zeta;\m{APS})\oplus C^{k+1,\alpha}_{\m{CFS};\beta;\epsilon}(X,E;\m{APS})$ such that 
\begin{align*}
    \Theta^t_{\m{ACF}}(\Xi\oplus\m{H})\oplus\Theta_{\m{CFS};t}(\Xi\oplus\m{H})=\Psi\oplus\Phi
\end{align*}
Then by definition $\m{ob}_{\beta;t}(\Pi_{\m{xker}_{\beta}(D_t)}(\Xi\oplus\m{H})=\Psi\oplus\Phi$.\\

Lastly, the surjectivity of $p_{\beta;t}$ follows from Proposition \ref{unifromboundednessfrombelowDtprop}.
    
\end{proof}

\subsection{Uniform Elliptic Theory for Dirac Operators on Tame Resolutions}
\label{Unifrom Ellitpic Theory for Dirac Operators on Tame Resolutions}

For interpolating resolutions of Dirac bundles, we have already developed analytic estimates that are robust under degeneration. The notion of $k$-tameness provides a quantitative framework for comparing the Dirac bundle $(E_t,\m{cl}_{g_t},h_t,\nabla^{h_t})$ with such an interpolating model. This allows us to transfer the analytic results from the interpolating setting to the $k$-tame world and, in particular, to perturb uniform estimates from the preglued family to the actual resolution.\\

Earlier in the paper we introduced the notion of $k$-tame resolutions in a heuristic way. Now that the relevant function spaces have been constructed, we return to give the precise definition. The construction proceeds by comparing the actual resolution bundle $(E_t,\m{cl}_{g_t},h_t,\nabla^{h_t})$ with a preglued family $(E^{pre}_t,\m{cl}_{g^{pre}_t},h^{pre}_t,\nabla^{h^{pre}_t})$ interpolating between the orbifold Dirac bundle and an asymptotically conical fibred (ACF) model with torsion. Analytic results were already established for such interpolating families, and the $k$-tame framework provides a quantitative way to transfer these estimates to the true resolution. In particular, this allows us to extend the uniform elliptic theory from preglued models to the $k$-tame world.\\

In this section, we state the precise definition of $k$-tame resolutions and prove the fundamental uniform elliptic estimate for Dirac operators in this setting.

\begin{defi}
\label{ktameDiracbundle}
    A smooth Gromov-Hausdorff resolution of an orbifold Dirac bundle
\begin{equation*}
    \begin{tikzcd}
	{(E_t,\m{cl}_{g_t},h_t,\nabla^{h_t})} && {(E,\m{cl}_{g},h,\nabla^{h})} \\
	\\
	{(X_t,g_t)} && {(X,g)}
	\arrow["{\hat{\rho}_t}"{description}, dashed, from=1-1, to=1-3]
	\arrow["{\pi_t}"{description}, from=1-1, to=3-1]
	\arrow["\pi"{description}, from=1-3, to=3-3]
	\arrow["{\rho_{t}}"{description}, dashed, from=3-1, to=3-3]
\end{tikzcd}
\end{equation*}
    is called \textbf{$k$-tame}, if there exists a smooth Gromov-Hausdorff resolution of Dirac bundles 
\begin{equation*}
    \begin{tikzcd}
	{(E^{pre}_t,\m{cl}_{g^{pre}_t},h^{pre}_t,\nabla^{h^{pre}_t})} && {(E,\m{cl}_{g},h,\nabla^{h})} \\
	\\
	{(X^{pre}_t,g^{pre}_t)} && {(X,g)}
	\arrow["{\hat{\rho}^{pre}_t}"{description}, dashed, from=1-1, to=1-3]
	\arrow["{\pi^{pre}_t}"{description}, from=1-1, to=3-1]
	\arrow["\pi"{description}, from=1-3, to=3-3]
	\arrow["{\rho^{pre}_{t}}"{description}, dashed, from=3-1, to=3-3]
\end{tikzcd}
\end{equation*}
constructed from interpolating the orbifold Dirac bundle and a family ACF-Dirac bundle with torsion, and isomorphisms 
\begin{equation*}
    \begin{tikzcd}
	{E_t} && {E^{pre}_t} \\
	\\
	{X_t} && {E_t}
	\arrow["{\hat{f}_t}"{description}, from=1-1, to=1-3]
	\arrow["{\pi_t}"{description}, from=1-1, to=3-1]
	\arrow["{\pi^{pre}_t}"{description}, from=1-3, to=3-3]
	\arrow["{f_t}"{description}, from=3-1, to=3-3]
\end{tikzcd}
\end{equation*}
such that 
    \begin{align*}
       \left|\left|\cup_t\left((\m{cl}^{pre}_{g_t},h^{pre}_t,\nabla^{h^{pre}_t})-(\hat{f}_t)_*(\m{cl}_{g_t},h_t,\nabla^{h_t})\right)\right|\right|_{C^{k}_{\m{ACF};-1;t}\oplus C^{k}_{\m{CFS};-1;\epsilon }}\xrightarrow[t\to 0]{}0.
    \end{align*}    
\end{defi}

This definition formalizes $k$-tameness by requiring that the resolution family can be compared, up to controlled error, with a preglued family of Dirac bundles for which uniform elliptic theory is already available. This control ensures that the difference between the corresponding Dirac operators is negligible in the weighted Hölder norms of interest. As a consequence, the analytic estimates established in the preglued setting extend directly to the $k$-tame case. The following proposition records the uniform elliptic estimate in this general framework.

\begin{prop}
\label{tameuniformellitic}
Let $\hat{\rho}_t\colon(E_t,\m{cl}_{g_t},h_t,\nabla^{h_t})\dashrightarrow (E,\m{cl}_g,h,\nabla^h)$ be a $(k+1)$-tame smooth Gromov-Hausdorff resolution with respect to $\hat{\rho}^{pre}_t\colon(E^{pre}_t,\m{cl}_{g^{pre}_t},h^{pre}_t,\nabla^{h^{pre}_t})\dashrightarrow (E,\m{cl}_g,h,\nabla^h)$. For all $\Phi\in\m{xker}_{\beta}(D^{pre}_t)^\perp$, the estimate 
\begin{align}
    \left|\left|\Phi\right|\right|_{\mathfrak{D}^{k+1,\alpha}_{\beta;t}}\lesssim  \left|\left|D_t \Phi\right|\right|_{\mathfrak{C}^{k,\alpha}_{\beta-1;t}}.
\end{align}
\end{prop}

\begin{proof}
We bound the difference 
    \begin{align*}
         \left|\left|(D^{pre}_{\zeta;t}-D_{t})\Psi\right|\right|_{\mathfrak{C}^{k,\alpha}_{\beta-1;t}}\lesssim&
         \left|\left|\m{cl}_{g_t}\circ(\nabla^{g^{pre}_{\zeta;t}}-\nabla^{g_{t}})\Psi\right|\right|_{\mathfrak{C}^{k,\alpha}_{\beta-1;t}}\\
         &+\left|\left|(\m{cl}_{g^{pre}_{\zeta;t}}-\m{cl}_{g_{t}})\circ\nabla^{g^{pre}_{\zeta;t}}\Psi\right|\right|_{\mathfrak{C}^{k,\alpha}_{\beta-1;t}}\\
         \lesssim& \,o(1)\cdot \left|\left|\Psi\right|\right|_{\mathfrak{D}^{k+1,\alpha}_{\beta;t}}
    \end{align*}
can be absorbed into the right-hand side of \eqref{unifromestimate}.
\end{proof}

\appendix
\section{The Spectrum of Dirac Operators on Spheres}
\label{The Spectrum of Dirac Operators on Spheres}

Using the results of Bär \cite{baer1996dirac} and Ikeda and Taniguchi \cite{ikeda1978spectra} we are able to compute the spectra of geometric Dirac operators such as the trivially twisted Spin-Dirac operator or the Hodge-de Rham operator on (orbifold) space-forms.\\

The key idea in these constructions is that (orbifold) space forms are homogeneous spaces. Thus, the spectral theory boils down to representation theory, and the eigenvalues and their multiplicity can be combinatorially determined.\\

These eigenvalues and their multiplicity are needed in Section \ref{Polyhomogenous Solutions of the Normal Operator} to determine the critical rates of the normal operator $\widehat{D}_0$ and accordingly of the operators $D$ and $\widehat{D}^t_\zeta$.

\subsection{The Spectrum of the Hodge-de Rham Operator}
\label{The Spectrum of the Hodge-de Rham Operator}

In the following, we will regard the round $m-1$-sphere as the homogeneous space. 
\begin{align*}
    \mathbb{S}^{m-1}=\m{SO}(m)/\m{SO}(m-1).
\end{align*}
Hence, the spectral theory of $\mathbb{S}^{m-1}$ is related to the Kasimir operator of $\m{SO}(m)$. The following theorem summarises the results in \cite{ikeda1978spectra}.

\begin{thm}\cite{ikeda1978spectra}
Let $\mathbb{S}^{m-1}$ be equipped with the round metric. The $L^2$-space of complex valued forms decomposes into irreducible $\m{SO}(m)$-representations of eigenforms. The decomposition is given by 
\begin{itemize}
    \item ($m-1=2n$) 
\begin{align*}
    L^2\Omega^q(\mathbb{S}^{m-1},\mathbb{C})=\bigoplus_{k\in\mathbb{N}}\left\{\begin{array}{rcll}
        \Phi^0_k&\oplus&\Phi^{0,1}_k&q=0\\
         \Phi^{q,0}_k&\oplus& \Phi^{q,1}_k&1\leq q\leq n-2  \\
         \Phi^{q,0}_k&\oplus&\Psi^{q,1}_k& q=n-1\\
         &\Psi^{q,0}_k&& q=n
    \end{array}\right.
\end{align*}
\begin{table}[!htt]
    \centering
    \begin{tabular}{|c|c|}\hline
         $\m{SO}(2n+1)$-representation& eigenvalue \\\hline\hline
         
         $\Phi^{0}_k$&$k(k+n-1)$ \\\hline
         $\Phi^{q,0}_k$&$(k+q)(k+n+1-q)$ \\\hline
         $\Phi^{q,1}_k$&$(k+q+1)(k+n-q)$ \\\hline
         $\Psi^{q,0}_k$&$(k+q)(k+q+1)$ \\\hline
         $\Psi^{q,1}_k$&$(k+q+1)(k+q+2)$ \\\hline
    \end{tabular}
\end{table}

\item ($m-1=2n-1$)
\begin{align*}
    L^2\Omega^q(\mathbb{S}^{m-1},\mathbb{C})=\bigoplus_{k\in\mathbb{N}}\left\{\begin{array}{cccccl}
        &\Phi^0_k&\oplus&\Phi^{0,1}_k&&q=0\\
         &\Phi^{q,0}_k&\oplus&\Phi^{q,1}_k&& 1\leq q\leq n-3\\
         &\Phi^{q,0}_k&\oplus&\Psi^{q,1}_k&& q=n-2\\
         \vartheta^{q,0}_k&\oplus&\Psi^{q,0}_k&\oplus&\vartheta^{q,1}_k&q=n-1
    \end{array}\right.
\end{align*}
\begin{table}[!ht]
    \centering
    \begin{tabular}{|c|c|}\hline
         $\m{SO}(2n)$-representation& eigenvalue \\\hline\hline
         $\Phi^{0}_k$&$k(k+n-1)$ \\\hline
         $\Phi^{q,0}_k$&$(k+q)(k+n+1-q)$ \\\hline
         $\Phi^{q,1}_k$&$(k+q+1)(k+n-q)$ \\\hline
         $\vartheta^{q,0}_k$&$(k+q)^2$\\\hline
         $\vartheta^{q,1}_k$& $(k+q)^2$ \\\hline
         $\Psi^{q,0}_k$&$(k+q-1)(k+q+1)$ \\\hline
         $\Psi^{q,1}_k$&$(k+q-1)(k+q+1)$ \\\hline
    \end{tabular}
\end{table}
\end{itemize}
\end{thm}

\begin{thm}
Let $\Gamma$ be a finite group acting on $\mathbb{R}^m$. Then the critical rates of the Laplace operator on $q$ forms on  $\mathbb{R}^m/\Gamma$ are a subset of 
     \begin{align*}
         \beta_{\pm}\in&\left\{-\frac{m-2}{2}\pm\sqrt{\left(\frac{m-2}{2}\right)^2+(q-2)(m-q)+\lambda_{q-1}},\right.\\
         &-\frac{m}{2}\pm\sqrt{\left(\frac{m}{2} sy \right)^2+q(m-q)+\lambda_{q-1,q}},\\
         &-\frac{m-4}{2}\pm\sqrt{\left(\frac{m-4}{2}\right)^2+(q-2)(m-q-2)+\lambda_{q-1,q}},\\
         &\left.-\frac{m-2}{2}\pm\sqrt{\left(\frac{m-2}{2}\right)^2+(m-q-2)q+\lambda_{q}}\right\}
     \end{align*}
where $\lambda_q\in\m{Spec}(\Delta_{g_{\mathbb{S}^{m-1}}}\colon \Omega^q\rightarrow\Omega^q)$, i.e.  
    \begin{table}[!h]
        \centering
        \begin{tabular}{|l|l|}\hline
             $q$&$\lambda_{q-1}$ \\\hline\hline
             $q=1$&$k(k+n-1)$, $(k+1)(k+n)$ \\\hline
             $1<q\leq n-2$&$(k+q-1)(k+n-q-2)$, $(k+q)(k+n-q+1)$ \\\hline
             $q=n-1$& $(k+n-2)(k+4)$,$(k+n-1)(k+2)$\\\hline
             $q=n$&$(k+n-1)(k+2)$,$(k+n)(k+n+1)$ \\\hline\hline
             $q$& $\lambda_{q}$ \\\hline\hline
             $q=1$&$(k+1)(k+n)$, $(k+2)(k+n-1)$ \\\hline
             $1<q\leq n-2$&$(k+q)(k+n+1-q)$, $(k+q+1)(k+n-q)$ \\\hline
             $q=n-1$&$(k+n-1)(k+2)$,$(k+n)(k+n+1)$\\\hline
             $q=n$&$(k+n)(k+n+1)$ \\\hline
        \end{tabular}
        \caption{$m-1=2n$}
    \end{table}
        \begin{table}[!h]
        \centering
        \begin{tabular}{|l|l|}\hline
             $q$&$\lambda_{q-1}$ \\\hline\hline
             $q=1$&$k(k+1)$, $(k+1)(k+1)$  \\\hline
             $1<q\leq n-3$& $(k+q-1)(k+n-q+2)$, $(k+q)(k+n-q+1)$ \\\hline
             $q=n-2$& $(k+n-3)(k+4)$, $(k+n-2)(k+3)$ \\\hline
             $q=n-1$&$(k+n-2)(k+3)$, $(k+n-3)(k+n-1)$ \\\hline\hline
             $q$& $\lambda_{q}$ \\\hline\hline
             $q=1$&$(k+1)(k+2)$ \\\hline
             $1<q\leq n-3$& $(k+q)(k+n-q+1)$, $(k+q+1)(k+n-q)$ \\\hline
             $q=n-2$&$(k+n-2)(k+n+3)$, $(k+n-3)(k+n-1)$ \\\hline
             $q=n-1$&$(k+n-1)^2$, $(k+n-2)(k+n-1)$ \\\hline
        \end{tabular}
        \caption{$m-1=2n-1$}
    \end{table}
\end{thm}

\subsection{The Spectrum of the Trivially Twisted Spin-Dirac Operator}
\label{The Spectrum of the triviially twised Spin-Dirac Operator}

Let $S\mathbb{S}^{m-1}$ denote the spinor bundle on
\begin{align*}
    \mathbb{S}^{m-1}=\m{Spin}(m)/\m{Spin}(m-1),
\end{align*}
$\overline{\nabla}^{g_{\mathbb{S}^{m-1}}}$ the lift of the Levi-Civita, $(F,\nabla^F)$ a trivial rank $r$-vector bundle and 
\begin{align*}
    D=\m{cl}_{g_{\mathbb{S}^{m-1}}}\circ(\overline{\nabla}^{g_{\mathbb{S}^{m-1}}}\otimes1+1\otimes\nabla^F)
\end{align*}
the Dirac operator on $E=S\mathbb{S}^{m-1}\otimes F$. Let $\gamma=\pm\frac{1}{2}$ and let $\Phi\in \Gamma(\mathbb{S}^{m-1},E)$. We denote by 
\begin{align*}
    \nabla\coloneqq\overline{\nabla}^{g_{\mathbb{S}^{m-1}}}\otimes1+1\otimes\nabla^F.
\end{align*}
Then $\Phi$ is called a $\gamma$-Killing spinor if
\begin{align*}
    \nabla^\gamma_X\Phi=\nabla_X\Phi-\gamma\m{cl}_{g_{\mathbb{S}^{m-1}}}(X)\Phi=0
\end{align*}
holds for all $X\in\mathfrak{X}(\mathbb{S}^{m-1})$.

\begin{lem}\cite{baer1996dirac}
    The bundle $E$ is $\nabla^\gamma$-trivialisable. Moreover the shifted Weitzen-böck-formula  
    \begin{align*}
        (D+\gamma)^2=(\nabla^\gamma)^*\nabla^\gamma+(m-2)^2
    \end{align*}
    holds.
\end{lem}

\begin{lem}
Let $f_i$ denote the eigenfunctions of $\Delta_{g_{\mathbb{S}^{m-1}}}$ of eigenvalue $\lambda_i$. Then $\Phi^j_i=f_i\Phi^j$ forms a basis of 
\begin{align*}
    L^2(\mathbb{S}^{m-1},E)
\end{align*}
where $\Phi^j$ is a basis of Killing spinors.
\end{lem}

\begin{cor}
    The eigenvalues of $D$ are 
    \begin{align*}
        \pm\lambda_k=\pm\left(\frac{m-1}{2}+k\right)
    \end{align*}
    and are of multiplicity
    \begin{align*}
        \m{m}(\lambda_k)=r\cdot 2^{[\frac{m-1}{2}]}\left(\begin{array}{c}
             k+m-2  \\
             k 
        \end{array}\right).
    \end{align*}
\end{cor}

In the following, let $m$ be even and $\Gamma\subset\m{SO}(m)$ be a finite subgroup. 
\begin{lem}
    The spin structures on $\mathbb{S}^{m-1}/\Gamma$ are in one-to-one correspondence with the lifts \begin{align*} \epsilon\colon \Gamma\rightarrow \m{Spin}(m).
    \end{align*}
    and the bundle $E$ descends to $E/\epsilon(\Gamma)$ on $\mathbb{S}^{m-1}$.
\end{lem}

\begin{thm}\cite[Thm. 2]{baer1996dirac}
    The eigenvalues of the Dirac operator on $E/\epsilon(\Gamma)$ are 
    \begin{align*}
        \pm\lambda_k=\pm\left(\frac{m-1}{2}+k\right)
    \end{align*}
    and are of multiplicity 
    \begin{align*}
        \m{m}(\pm\lambda_k)
    \end{align*}
    determined by the generating functions 
    \begin{align*}
        F_{\pm}(z)=\sum_{k=0}^\infty \m{m}\left(\pm\left(\frac{m-1}{2}+k\right),D\right)z^k=\frac{r}{|\Gamma|}\sum_{g\in\Gamma}\frac{\chi_\mp(\epsilon(g))-z\cdot \chi_\pm(\epsilon(g))}{\m{det}(1_{m-2}-z\cdot g)}
    \end{align*}
    where $\chi_{\pm}$ are the characters of the irreducible half-spin representations 
    \begin{align*}
        \varrho_\pm\colon \m{Spin}(m)\rightarrow \m{Aut}(\Sigma^\pm). 
    \end{align*}
\end{thm}

\newpage
\bibliographystyle{alpha}
\bibliography{literature}

\end{document}

%% file: Gluing4.pdf_tex
\begingroup%
  \makeatletter%
  \providecommand\color[2][]{%
    \errmessage{(Inkscape) Color is used for the text in Inkscape, but the package 'color.sty' is not loaded}%
    \renewcommand\color[2][]{}%
  }%
  \providecommand\transparent[1]{%
    \errmessage{(Inkscape) Transparency is used (non-zero) for the text in Inkscape, but the package 'transparent.sty' is not loaded}%
    \renewcommand\transparent[1]{}%
  }%
  \providecommand\rotatebox[2]{#2}%
  \newcommand*\fsize{\dimexpr\f@size pt\relax}%
  \newcommand*\lineheight[1]{\fontsize{\fsize}{#1\fsize}\selectfont}%
  \ifx\svgwidth\undefined%
    \setlength{\unitlength}{595.27559055bp}%
    \ifx\svgscale\undefined%
      \relax%
    \else%
      \setlength{\unitlength}{\unitlength * \real{\svgscale}}%
    \fi%
  \else%
    \setlength{\unitlength}{\svgwidth}%
  \fi%
  \global\let\svgwidth\undefined%
  \global\let\svgscale\undefined%
  \makeatother%
  \begin{picture}(1,1.00793653)%
    \lineheight{1}%
    \setlength\tabcolsep{0pt}%
    \put(0,0){\includegraphics[width=\unitlength,page=1]{Gluing4.pdf}}%
    \put(0.31588644,0.82242144){\color[rgb]{0,0,0}\makebox(0,0)[lt]{\lineheight{1.25}\smash{\begin{tabular}[t]{l}\huge{CFS}\end{tabular}}}}%
    \put(0.64440735,0.82433742){\color[rgb]{0,0,0}\makebox(0,0)[lt]{\lineheight{1.25}\smash{\begin{tabular}[t]{l}\huge{ACF}\end{tabular}}}}%
    \put(0.142889,0.19437031){\color[rgb]{0,0,0}\makebox(0,0)[lt]{\lineheight{1.25}\smash{\begin{tabular}[t]{l}\huge{compact}\end{tabular}}}}%
    \put(0.82128993,0.19050845){\color[rgb]{0,0,0}\makebox(0,0)[lt]{\lineheight{1.25}\smash{\begin{tabular}[t]{l}\huge{CF}\end{tabular}}}}%
    \put(0.3854248,0.512258){\color[rgb]{0,0,0}\makebox(0,0)[lt]{\lineheight{1.25}\smash{\begin{tabular}[t]{l}\huge{gluing}\end{tabular}}}}%
    \put(0.58925523,0.51040497){\color[rgb]{0,0,0}\makebox(0,0)[lt]{\lineheight{1.25}\smash{\begin{tabular}[t]{l}\huge{anti-gluing}\end{tabular}}}}%
  \end{picture}%
\endgroup%

%% file: main.bbl
\newcommand{\etalchar}[1]{$^{#1}$}
\begin{thebibliography}{PTW17}

\bibitem[AB98]{ammann1998dirac}
Bernd Ammann and Christian B{\"a}r.
\newblock The dirac operator on nilmanifolds and collapsing circle bundles.
\newblock {\em Annals of Global Analysis and Geometry}, 16:221--253, 1998.

\bibitem[AGR16]{Albin2016Index}
Pierre Albin and Jesse Gell-Redman.
\newblock The index of {Dirac} operators on incomplete edge spaces.
\newblock {\em SIGMA, Symmetry Integrability Geom. Methods Appl.}, 12:paper 089, 45, 2016.

\bibitem[AGR23]{albin2023index}
Pierre Albin and Jesse Gell-Redman.
\newblock The index formula for families of dirac type operators on pseudomanifolds.
\newblock {\em Journal of Differential Geometry}, 125(2):207--343, 2023.

\bibitem[ALR07]{adem2007orbifolds}
Alejandro Adem, Johann Leida, and Yongbin Ruan.
\newblock {\em Orbifolds and stringy topology}, volume 171.
\newblock Cambridge University Press, 2007.

\bibitem[AN61]{agmon1961properties}
Shmuel Agmon and Louis Nirenberg.
\newblock Properties of solutions of ordinary differential equations in banach space.
\newblock Technical report, Army Research Office Durham NC, 1961.

\bibitem[B{\"a}r96]{baer1996dirac}
Christian B{\"a}r.
\newblock The dirac operator on space forms of positive curvature.
\newblock {\em Journal of the Mathematical Society of Japan}, 48(1):69--83, 1996.

\bibitem[BC89]{bismut1989eta}
Jean-Michel Bismut and Jeff Cheeger.
\newblock $\eta$-invariants and their adiabatic limits.
\newblock {\em Journal of the American Mathematical Society}, 2(1):33--70, 1989.

\bibitem[Ber23]{bera2023deformCS}
Gorapada Bera.
\newblock Deformations and desingularizations of conically singular associative submanifolds, 2023.

\bibitem[BGV92]{berline1992heat}
Nicole Berline, Ezra Getzler, and Mich{\`e}le Vergne.
\newblock {\em Heat kernels and {Dirac} operators}, volume 298 of {\em Grundlehren Math. Wiss.}
\newblock Berlin etc.: Springer-Verlag, 1992.

\bibitem[Bis86]{bismut1986}
Jean-Michel Bismut.
\newblock The {Atiyah}-{Singer} index theorem for families of {Dirac} operators: {Two} heat equation proofs.
\newblock {\em Invent. Math.}, 83:91--151, 1986.

\bibitem[BKN89]{bando1989construction}
Shigetoshi Bando, Atsushi Kasue, and Hiraku Nakajima.
\newblock On a construction of coordinates at infinity on manifolds with fast curvature decay and maximal volume growth.
\newblock {\em Inventiones mathematicae}, 97(2):313--349, 1989.

\bibitem[BL91]{bismut1991}
Jean-Michel Bismut and Gilles Lebeau.
\newblock Complex immersions and {Quillen} metrics.
\newblock {\em Publ. Math., Inst. Hautes {\'E}tud. Sci.}, 74:1--297, 1991.

\bibitem[BL95]{bismut1995}
Jean-Michel Bismut and John Lott.
\newblock Flat vector bundles, direct images and higher real analytic torsion.
\newblock {\em Journal of the American Mathematical Society}, 8(2):291--363, 1995.

\bibitem[BT{\etalchar{+}}82]{bott1982differential}
Raoul Bott, Loring~W Tu, et~al.
\newblock {\em Differential forms in algebraic topology}, volume~82.
\newblock Springer, 1982.

\bibitem[CG86]{cheeger1986collapsing}
J~Cheeger and M~Gromov.
\newblock Collapsing riemannian manifolds while keeping their curvature bounded.
\newblock {\em Journal of Differential Geometry}, 23(3):309--346, 1986.

\bibitem[CG90]{cheeger1990collapsing}
Jeff Cheeger and Mikhael Gromov.
\newblock {Collapsing Riemannian manifolds while keeping their curvature bounded. II}.
\newblock {\em Journal of Differential Geometry}, 32(1):269--298, 1990.

\bibitem[Dai91]{dai1991adiabatic}
Xianzhe Dai.
\newblock Adiabatic limits, nonmultiplicativity of signature, and leray spectral sequence.
\newblock {\em Journal of the American Mathematical Society}, 4(2):265--321, 1991.

\bibitem[DM18]{mazzeo2018fredholm}
Anda Degeratu and Rafe Mazzeo.
\newblock Fredholm theory for elliptic operators on quasi-asymptotically conical spaces.
\newblock {\em Proc. Lond. Math. Soc. (3)}, 116(5):1112--1160, 2018.

\bibitem[Don83]{donaldson1983application}
Simon~K Donaldson.
\newblock An application of gauge theory to four-dimensional topology.
\newblock {\em Journal of Differential Geometry}, 18(2):279--315, 1983.

\bibitem[DT98]{donaldson1998gauge}
S.~K. Donaldson and R.~P. Thomas.
\newblock Gauge theory in higher dimensions.
\newblock In {\em The geometric universe: science, geometry, and the work of Roger Penrose. Proceedings of the symposium on geometric issues in the foundations of science, Oxford, UK, June 1996 in honour of Roger Penrose in his 65th year}, pages 31--47. Oxford: Oxford University Press, 1998.

\bibitem[Goe14]{goette2014adiabatic}
Sebastian Goette.
\newblock Adiabatic limits of {Seifert} fibrations, {Dedekind} sums, and the diffeomorphism type of certain 7-manifolds.
\newblock {\em J. Eur. Math. Soc. (JEMS)}, 16(12):2499--2555, 2014.

\bibitem[Gut23]{gutwein2023coassociative}
Dominik Gutwein.
\newblock Coassociative submanifolds in joyce's generalised kummer constructions.
\newblock {\em arXiv preprint arXiv:2305.05561}, 2023.

\bibitem[GW09]{gromoll2009metric}
Detlef Gromoll and Gerard Walschap.
\newblock {\em Metric foliations and curvature}, volume 268.
\newblock Springer Science \& Business Media, 2009.

\bibitem[HL82]{Harvey1982}
Reese Harvey and H.~Blaine Lawson.
\newblock Calibrated geometries.
\newblock {\em Acta Math.}, 148:47--157, 1982.

\bibitem[HT07]{hutchingstaubesI}
Michael Hutchings and Clifford~Henry Taubes.
\newblock Gluing pseudoholomorphic curves along branched covered cylinders. {I}.
\newblock {\em J. Symplectic Geom.}, 5(1):43--137, 2007.

\bibitem[HT09]{hutchingstaubesII}
Michael Hutchings and Clifford~Henry Taubes.
\newblock Gluing pseudoholomorphic curves along branched covered cylinders. {II}.
\newblock {\em J. Symplectic Geom.}, 7(1):29--133, 2009.

\bibitem[IT78]{ikeda1978spectra}
Akira Ikeda and Yoshiharu Taniguchi.
\newblock Spectra and eigenforms of the laplacian on $\mathbb{n}^n$ and $\mathbb{P}^n$.
\newblock {\em Osaka J. Math}, 15(3):515--546, 1978.

\bibitem[JK21]{joyce2017new}
Dominic Joyce and Spiro Karigiannis.
\newblock A new construction of compact torsion-free {{\(\m{G}_2\)}}-manifolds by gluing families of {Eguchi}-{Hanson} spaces.
\newblock {\em J. Differ. Geom.}, 117(2):255--343, 2021.

\bibitem[Joy96a]{Joyce1996b}
Dominic~D. Joyce.
\newblock Compact {$8$}--manifolds with holonomy {${\rm Spin}(7)$}.
\newblock {\em Invent. Math.}, 123(3):507--552, 1996.

\bibitem[Joy96b]{joyce1996aI}
Dominic~D. Joyce.
\newblock {Compact Riemannian 7-manifolds with holonomy $ \m{G}_2 $. I}.
\newblock {\em Journal of differential geometry}, 43(2):291--328, 1996.

\bibitem[Joy00]{joyce2000compact}
Dominic~D. Joyce.
\newblock {\em Compact manifolds with special holonomy}.
\newblock Oxford University Press on Demand, 2000.

\bibitem[LMO85]{lockhart1985elliptic}
Robert~B Lockhart and Robert~C Mc~Owen.
\newblock {Elliptic differential operators on noncompact manifolds}.
\newblock {\em Annali della Scuola Normale Superiore di Pisa-Classe di Scienze}, 12(3):409--447, 1985.

\bibitem[Loc87]{lockhart1987fredholm}
Robert Lockhart.
\newblock {Fredholm, Hodge and Liouville theorems on noncompact manifolds}.
\newblock {\em Transactions of the American Mathematical Society}, 301(1):1--35, 1987.

\bibitem[Maj25a]{majewskithesis}
Viktor~F. Majewski.
\newblock {\em {Dirac Operators on Orbifold Resolutions and their Applications to the Construction of Exceptional Holonomy Metrics}}.
\newblock PhD thesis, {Humboldt University of Berlin}, 2025.

\bibitem[Maj25b]{majewski2025spin7orbifoldresolutions}
Viktor~F. Majewski.
\newblock {Spin(7)-Orbifold Resolutions}, 2025.
\newblock \href{https://arxiv.org/abs/2509.16057}{arXiv:2509.16057 [math.DG]}.

\bibitem[Mar02]{marshal2002deformations}
Stephen~P Marshal.
\newblock {\em Deformations of special Lagrangian submanifolds}.
\newblock PhD thesis, Citeseer, 2002.

\bibitem[Maz91]{mazzeo1991elliptic}
Rafe Mazzeo.
\newblock Elliptic theory of differential edge operators i.
\newblock {\em Communications in Partial Differential Equations}, 16(10):1615--1664, 1991.

\bibitem[McL98]{mclean1998deformations}
Robert~C McLean.
\newblock Deformations of calibrated submanifolds.
\newblock {\em Communications in analysis and geometry}, 6(4):705--747, 1998.

\bibitem[Mel91]{melrose1991Pseudo}
Richard~B. Melrose.
\newblock Pseudodifferential operators, corners and singular limits.
\newblock In {\em Proceedings of the international congress of mathematicians (ICM), August 21--29, 1990, Kyoto, Japan. Volume I}, pages 217--234. Tokyo etc.: Springer-Verlag, 1991.

\bibitem[Mel92]{melrose1992atiyah}
Richard~B Melrose.
\newblock The atiyah-patodi-singer index theorem.
\newblock {\em Wellesley, Massachusetts}, 1992.

\bibitem[MM90]{mazzeo1990adiabatic}
Rafe~R Mazzeo and Richard~B Melrose.
\newblock {The adiabatic limit, Hodge cohomology and Leray's spectral sequence for a fibration}.
\newblock {\em Journal of Differential Geometry}, 31(1):185--213, 1990.

\bibitem[Moe02]{moerdijk2002orbifolds}
Ieke Moerdijk.
\newblock Orbifolds as groupoids: an introduction.
\newblock {\em arXiv preprint math/0203100}, 2002.

\bibitem[MV14]{mazzeo2014elliptic}
Rafe Mazzeo and Boris Vertman.
\newblock Elliptic theory of differential edge operators, ii: Boundary value problems.
\newblock {\em Indiana University Mathematics Journal}, pages 1911--1955, 2014.

\bibitem[Par23]{parker2023deformations}
Gregory~J. Parker.
\newblock Deformations of $\mathbb{Z}_2$-harmonic spinors on 3-manifolds, 2023.

\bibitem[Pfl01]{pflaum2001analytic}
Markus~J. Pflaum.
\newblock {\em Analytic and geometric study of stratified spaces}, volume 1768 of {\em Lect. Notes Math.}
\newblock Berlin: Springer, 2001.

\bibitem[Pla20]{platt}
Daniel Platt.
\newblock {Improved Estimates for $\m{G}_2$-structures on the Generalised Kummer Construction}, 2020.

\bibitem[PTW17]{posthuma2017resolutionsproperriemannianlie}
Hessel~B. Posthuma, Xiang Tang, and Kirsten J.~L. Wang.
\newblock {Resolutions of proper Riemannian Lie groupoids}, 2017.

\bibitem[Sch91]{Schulze1991Pseudo}
Bert-Wolfgang Schulze.
\newblock {\em Pseudo-differential operators on manifolds with singularities}, volume~24 of {\em Stud. Math. Appl.}
\newblock Amsterdam etc.: North-Holland, 1991.

\bibitem[Sch98]{Schulze1998Boundary}
Bert-Wolfgang Schulze.
\newblock {\em Boundary value problems and singular pseudo-differential operators}.
\newblock Pure Appl. Math., Wiley-Intersci. Ser. Texts Monogr. Tracts. Chichester: John Wiley \& Sons, 1998.

\bibitem[Sea91]{seaman1991}
Walter Seaman.
\newblock Harmonic two-forms in four dimensions.
\newblock {\em Proc. Am. Math. Soc.}, 112(2):545--548, 1991.

\bibitem[Tia00]{tian2000gauge}
Gang Tian.
\newblock {Gauge theory and calibrated geometry, I}.
\newblock {\em Annals of mathematics}, 151(1):193--268, 2000.

\bibitem[UY86]{uhlenbeck1986existence}
Karen Uhlenbeck and Shing-Tung Yau.
\newblock {On the existence of Hermitian-Yang-Mills connections in stable vector bundles}.
\newblock {\em Communications on Pure and Applied Mathematics}, 39(S1):S257--S293, 1986.

\bibitem[Wal17a]{walpuski2012g_2}
Thomas Walpuski.
\newblock {{\(\m{G}_2\)}}-instantons, associative submanifolds and {Fueter} sections.
\newblock {\em Commun. Anal. Geom.}, 25(4):847--893, 2017.

\bibitem[Wal17b]{walpuski2017spin}
Thomas Walpuski.
\newblock Spin (7)-instantons, $\m{C}$ayley submanifolds and $\m{F}$ueter sections.
\newblock {\em Communications in Mathematical Physics}, 352(1):1--36, 2017.

\bibitem[Y{\etalchar{+}}07]{yang2007dirac}
Fangyun Yang et~al.
\newblock {\em Dirac operators and monopoles with singularities}.
\newblock PhD thesis, Massachusetts Institute of Technology, 2007.

\end{thebibliography}
